\documentclass{article}

\usepackage[T1]{fontenc}
\usepackage[utf8]{inputenc}

\usepackage{lmodern}

\usepackage{latexsym}
\usepackage{amsmath,amssymb,amsthm,amstext,amsfonts,pict2e,amscd}
\usepackage{color}

\theoremstyle{plain}
\newtheorem{thm}{Theorem}[section]
\newtheorem{lem}[thm]{Lemma}
\newtheorem{pro}[thm]{Proposition}
\newtheorem{cor}[thm]{Corollary}

\theoremstyle{definition}
\newtheorem{defn}[thm]{Definition}
\newtheorem{rem}[thm]{Remark}
\newtheorem{que}[thm]{Question}

\newcommand{\cartbig}
{
	\begin{picture}(0,0)
		\put(0,-2){\framebox(10,10){}}
	\end{picture}
	\hspace{0.4cm}
}
\newcommand{\pr} {\mathrm{pr}}

\begin{document}

\title{
       Peano partial cubes}
\author{Norbert Polat\\
        I.A.E., Universit\'{e} Jean Moulin (Lyon 3)\\
         6 cours Albert Thomas\\
         69355 Lyon Cedex 08, France\\
         \texttt{norbert.polat@univ-lyon3.fr}}
\date{}
\maketitle

\begin{abstract}
Peano partial cubes are the partial cubes that have the Pash and Peano properties, and thus they are the bipartite graphs whose geodesic interval spaces are (closed) join spaces.  These graphs are the partial cubes all of whose finite convex subgraphs have a pre-hull number which is at most $1$.  Special Peano partial cubes are median graphs, cellular bipartite graphs and netlike partial cubes.  Analogous properties of these graphs are satisfied by Peano partial cubes.  In particular the convex hull of any isometric cycle of such a graph is a gated quasi-hypertori (i.e., the Cartesian product of copies of $K_2$ and even cycles).  The finite quasi-hypertori are the finite regular Peano partial cubes, and they turn out to be the Peano partial cubes that are antipodal.  Moreover, for any Peano partial cubes $G$ that contains no isometric rays, there exists a finite qasi-hypertorus which is fixed by all automorphisms of $G$, and any self-contraction of $G$ fixes some finite quasi-hypertorus.  A Peano partial cube $G$ is called a hyper-median partial cube if any triple of vertices of $G$ has either a median or a hyper-median, that is, a quasi-median whose convex-hull induces a hypertorus (i.e., the Cartesian product of even cycles such that at least one of them has length greater than $4$).  These graphs have several properties similar to that of median graphs.  In particular a graph is a hyper-median partial cube if and only if all its finite convex subgraphs are obtained by a sequence of gated amalgamations from finite quasi-hypertori.  Also a finite graph is a hyper-median partial cube if and only if it can be obtained from $K_1$ by a sequence of special expansions.  The class of Peano partial cubes and that of hyper-median partial cubes are closed under convex subgraphs, retracts, Cartesian products and gated amalgamations.  We study two convex invariants: the Helly number of a Peano partial cube, and the depth of a hyper-median partial cube that contains no isometric rays.  Finally, for a finite Peano partial cube $G$, we prove an Euler-type formula, and a similar formula giving the isometric dimension of $G$.\\

\emph{Keywords:} Bipartite graph; Partial cube; Peano partial cube; Ph-homogeneous partial cube; Hyper-median partial cube; Netlike partial cube; Median graph; Hypercube; Hypertorus; Prism; Cartesian product; Gated amalgam; Retract; Geodesic convexity; Weak geodesic topology; Helly number; Depth; Euler-type property; Fixed subgraph property.\\

\emph{ 2010 Mathematics Subject Classification:} Primary(05C75, 05C63, 05C12, 05C30, 05C76);Secondary(52A37)
\end{abstract}

\section{Introduction}\label{S:introd.}

Median graphs are certainly the most extensively studied and
characterized partial cubes (i.e., isometric subgraphs of hypercubes).
Several classes of partial cubes containing median graphs as special
instances have already been defined and studied, often with emphasis
on properties generalizing some well-known properties of median
graphs.  In particular the class of \emph{netlike partial cubes} was introduced and studied in a series of papers \cite{P05-1,P05-2,P07-3,P08-4,P07-5} as a class of partial cubes containing, in addition to median graphs, even cycles, \emph{cellular bipartite graphs}~\cite{BC96} and benzenoid graphs.  Median graphs and netlike partial cubes have the common property that any of their finite convex subgraphs has a pre-hull number~\cite{PS09} which is at most $1$.  A graph that has this property is said to be \emph{ph-homogeneous}.

These graphs were introduced in~\cite{P16} as the bipartite graphs whose geodesic interval spaces are (closed) join spaces, i.e., which share a number of geometrical properties with Euclidean spaces.  It was proved, that ph-homogeneous partial cubes are the Pash-Peano partial cubes, i.e., the partial cubes that satisfy the Pash and the Peano Properties.  In fact, for partial cubes, the Peano Property is stronger than the Pash Property.  This is why we call them \emph{Peano partial cubes}.

The class of Peano partial cubes is closed under convex subgraphs, Cartesian products,  gated amalgams and retracts, as is the class of netlike partial cubes.  It follows that the class of Peano partial cubes is a variety.  We recall that, in graph theory, varieties are classes of graphs closed under retracts and products; the choice of products is not unique: it could be either the strong products or the Cartesian products.

This paper is organized as follows.  After the introduction of the basic concepts of Peano partial cubes and characterizations by convexity properties, several different structural characterizations of Peano partial cubes are given in Section~\ref{S:charac.prop.}.  Particular properties follows from these characterizations.  A Peano partial cube $G$ is a median graph (resp. a netlike partial cube) if the convex hull of any isometric cycle of $G$ is a hypercube (resp. this cycle itself or a hypercube).  Any Peano partial cube $G$ also satisfies an analogous but more general property, namely the convex hull of any isometric cycle of $G$ is a gated quasi-hypertorus, that is, the Cartesian product of copies of $K_2$ and even cycles.  The finite quasi-hypertori, which turn out to be the finite regular Peano partial cubes, are also the Peano partial cubes that are antipodal (a graph $G$ is antipodal if for any vertex $x$ of $G$ there is a  vertex $\bar{x}$ such that the interval between $x$ and $\bar{x}$ is equal to the vertex set of $G$).

In Section~\ref{S:decompos.}, by analogy with the fact that median graphs are the particular Peano partial cubes for which any triple of vertices has a unique median, we define a \emph{hyper-median partial cube} as a Peano partial cube $G$ such that any triple of vertices of $G$ has either a median or a \emph{hyper-median}, that is, a quasi-median whose convex-hull induces a hypertorus.  Only the existence of a median or a hyper-median for each triple of vertices is necessary in the definition, the uniqueness comes as a consequence.  We show in particular that a partial cube is a hyper-median partial cube if and only if all its finite convex subgraphs are obtained by a sequence of gated amalgamations from finite quasi-hypertori.  This generalizes an analogous property of finite median graph \cite{Isb80,V83} stating that any finite median graph can be obtained by a sequence of gated amalgamations from finite hypercubes.

Fairly recently we found out the existence of the paper of Chepoi, Knauer and Marc \cite{CKM19} entitled \emph{Partial cubes without $Q_3^-$ minors}.  This paper and the first version of ours were independently and almost simultaneously written, and it turns out that the graphs studied in~\cite{CKM19}, also called \emph{hypercellular partial cubes}, coincide with our hyper-median partial cubes (Subsection~\ref{SS:hypercell.}).  It follows that some of our results or parts thereof are also proved in~\cite{CKM19}.  Several other properties of these graphs are given in~\cite{CKM19}.

Since Mulder~\cite{M78} and Chepoi~\cite{Che88} introduced the
expansion procedure for median graphs and partial cubes, different
kinds of finite partial cubes have already been constructed from
$K_{1}$ by sequences of special expansions (see~\cite{IK98}).  In
\cite[Subsection 6]{P07-3} we showed that, if one can obtain all
partial cubes by Chepoi's theorem, not all graphs in the middle of
expansion are netlike, and even ph-homogeneous.  More precisely, there exist infinitely many
finite netlike partial cubes, in particular some benzenoid graphs,
which are not the expansion of any Peano partial cubes.  In Subsection~\ref{SS:expans.proc.} we
proved that there exists a particular kind of expansion that enables
to construct all finite hyper-median partial cubes from
$K_{1}$.

In Section~\ref{S:FSP} we focus on fixed subgraph properties, 
generalizing some results on netlike partial cubes~\cite{P08-4}, which
themselves are generalizations of three results of Tardif~\cite{T96}
on median graphs.  For infinite partial cubes we use the topological
concepts and results of~\cite{P09-1,P09-2}.  We prove three main
results for any compact Peano partial cube $G$ (i.e., containing no isometric rays).  The first deals with the existence of a
finite gated quasi-hypertorus that is fixed by all
automorphisms of $G$.  The second deals with the fact that any
self-contraction (i.e., non-expansive self-map) of $G$ fixes
a finite quasi-hypertorus.  The third extends
the second result to any commuting family of
self-contractions of $G$, and is of the same kind as two well-known
theorems stating that commuting families of endomorphisms of certain
structures have a common fixed point: the Markov-Kakutani
Theorem~\cite{Ma36,Ka38} for compact convex sets of locally convex
linear topological spaces, and the Tarski's theorem~\cite{Ta95} for
complete lattices.

The proofs of these properties require some preliminary results, the most important of which also has another interesting consequence: the closure of the class of Peano partial cubes under retracts.  This property is one of the main results of Section~\ref{S:retr./convex} which is devoted to the study of retracts and hom-retracts and their relation with convex subgraphs and more generally with the so-called ‘‘strongly faithful’’ subgraphs of Peano partial cubes, that is isometric subgraphs that are stable under median and hyper-median.  We show that any convex subgraph is a retract, and that any retract is strongly faithful.  However the characterization of the strongly faithful subgraphs of a Peano partial cube that are retracts (or hom-retracts) of this graph is still an open problem.

Section~\ref{S:conv.inv.} deals with two convex invariants.  First, a classical invariant: the Helly number, that is, the smallest integer $h$, if it exists, such that any finite family of $h$-wise non-disjoint convex sets has a non-empty intersection.  We show that the Helly number of a Peano partial cube $G$ is at most $3$ with the equality if and only if $G$ is not a median graph.  Secondly, a less classical invariant: the depth, that is, the supremum of the lengths of chains of half-spaces.  This invariant was introduced in \cite{BV91} by Bandelt and van de Vel in order to study the structure of finite median graphs.  Following Bandelt and van de Vel, we prove a recursive description of the 
compact hyper-median partial cubes with finite depth.  This result generalizes  \cite[Theorem 2.4]{BV91} and \cite[Theorem 6.4]{P07-5} on finite median graphs and finite tricycle-free Peano partial cubes, respectively.  We also show that the depth of the gated hull of any finite set of vertices of a  compact hyper-median partial cube is finite.  This is an interesting piece of information on gated hulls because, contrary to the property that the convex hull of any finite set of vertices of an infinite partial cube is always finite, we generally have no such property for the gated hull of a finite set: it may be finite or infinite.

Several Euler-type properties generalizing the well-known equality for trees “$n-m=1$” have already been obtained for finite median graphs.  Some of them concern special median graphs such as those that are cube-free or $Q_{4}$-free (Klav\v{z}ar and \v{S}krekovski~\cite{KS00}).  More generally Soltan and Chepoi~\cite{SoChe87} and independently \v{S}krekovski~\cite{Sk01} proved that the Euler characteristic of a median graph $G$ is $1$, that is, $\sum_{n \in \mathbb{N}}(-1)^{n}\beta_{n}(G) = 1$, where $\beta_{n}(G)$ denotes the number of $n$-cubes of $G$.  Some Euler-type properties were also proved for some special partial cubes that are not necessarily median, such as cellular bipartite graphs by Bandelt and Chepoi \cite[Corollary 1]{BC96} and cube-free netlike partial cubes \cite[Proposition 7.2]{P07-3}.  Moreover, in the paper mentioned above~\cite{Sk01}, \v{S}krekovski proved a formula giving the isometric dimension of a finite median graph (i.e., the least $n$ such
that $G$ is an isometric subgraph of an $n$-cube).  In Section~\ref{S:Euler} we prove analogous formulas for finite Peano partial cubes.

\section{Preliminaries} \label{S:prelim.}

\subsection{Graphs}\label{SS:gr.}

The graphs we consider are undirected, without loops or multiple
edges, and may be finite or infinite.  Let $G$ be a graph.  If $x \in
V(G)$, the set $N_{G}(x) := \{y \in V(G) : xy \in E(G)\}$ is the
\emph{neighborhood} of $x$ in $G$, $N_{G}[x] := \{x \} \cup N_{G}(x)$
is the \emph{closed neighborhood} of $x$ in $G$ and $\delta_{G}(x) :=
 |N_{G}(x)|$ is the \emph{degree} of $x$ in $G$.  For a set $X$ of vertices of  
$G$
we put $N_{G}[X] := \bigcup_{x \in X}N_{G}[x]$ and $N_{G}(X) :=
N_{G}[X] - X$, and we denote by $\partial_{G}(X)$ the
\emph{edge-boundary} of $X$ in $G$, that is, the set of all edges of
$G$ having exactly one endvertex in $X$.  Moreover, we denote by $G[X]$
the subgraph of $G$ induced by $X$, and we set $G-X:= G[V(G)-X]$.  A subgraph of $G$ is said to be \emph{non-trivial} if it has at least two vertices.

A \emph{path} $P = \langle x_{0},\dots,x_{n}\rangle$ is a graph with
$V(P) = \{x_{0},\dots,x_{n}\}$, $x_{i} \neq x_{j}$ if $i \neq j$, and
$ E(P) = \{x_{i}x_{i+1} : 0 \leq i < n\}$.  A path $P = \langle
x_{0},\dots,x_{n}\rangle$ is called an \emph{$(x_{0},x_{n})$-path},
$x_{0}$ and $ x_{n}$ are its \emph{endvertices}, while the other
vertices are called its \emph{inner} vertices, $n = |E(P)|$ is the
\emph{length} of $P$.  If $x$ and $y$ are two vertices of a path $P$,
then we denote by $P[x,y]$ the subpath of $P$ whose endvertices are
$x$ and $y$.  A \emph{ray} is a one-way
infinite path, and a graph is \emph{rayless} if it
contains no rays.

A \emph{cycle} $C$ with $V(C) = \{x_{1},\dots,x_{n}\}$, $x_{i} \neq
x_{j}$ if $i \neq j$, and $ E(C) = \{x_{i}x_{i+1} : 1 \leq i < n\}
\cup \{x_{n}x_{1} \}$, will be denoted by $\langle
x_{1},\dots,x_{n},x_{1}\rangle$.  The non-negative integer $n =
|E(C)|$ is the \emph{length} of $C$, and a cycle of length $n$ is
called a \emph{$n$-cycle} and is often denoted by $C_{n}$.

Let $G$ be a connected graph.  The usual \emph{distance} between two
vertices $x$ and $y$, that is, the length of any
\emph{$(x,y)$-geodesic} (= shortest $(x,y)$-path) in $G$, is denoted
by $d_{G}(x,y)$.  A connected subgraph $H$ of $ G $ is
\emph{isometric} in $G$ if $d_{H}(x,y) = d_{G}(x,y)$ for all vertices
$x$ and $y$ of $H$.  The \emph{(geodesic) interval} $I_{G}(x,y)$
between two vertices $x$ and $y$ of $G$ is the set of vertices of all
$(x,y)$-geodesics in $G$.

\subsection{Convexity}\label{SS:convex.}

A \emph{convexity} on a set $X$ is an algebraic closure system
$\mathcal{C}$ on $X$.  The elements of $\mathcal{C}$ are the
\emph{convex sets} and the pair $(X,\mathcal{C})$ is called a
\emph{convex structure}.  See van de Vel~\cite{V93} for a detailed
study of abstract convex structures.  Several kinds of graph
convexities, that is, convexities on the vertex set of a graph $G$,
have already been investigated.  We will principally work with the
\emph{geodesic convexity}, that is, the convexity on $V(G)$ which is
induced by the geodesic interval operator $I_{G}$.  In this convexity,
a subset $C$ of $V(G)$ is convex provided it contains the geodesic
interval $I_{G}(x,y)$ for all $x, y \in C$.  The \emph{convex hull}
$co_{G}(A)$ of a subset $A$ of $V(G)$ is the smallest convex
set that contains $A$.  The convex hull of a finite set is called a
\emph{polytope}.  A subset $A$ of $V(G)$ is a \emph{half-space} if $A$
and $V(G)-A$ are convex.

A \emph{copoint} at a point $x \in V(G)$ is a
convex set $C$ which is maximal with respect to the property that $x
\notin C$; $x$ is an \emph{attaching point} of $C$.  Note that
$co_G(C \cup \{x \}) = co_G(C \cup
\{y \})$ for any two attaching points $x, y$ of $C$.  We denote by $\mathrm{Att}(C)$ the set of all attaching
points of $C$, i.e., 
\begin{equation*}\label{E:Att}
\mathrm{Att}(C) := co_G(C \cup \{x \})-C.
\end{equation*}

We denote by $\mathcal{I}_{G}$ the pre-hull operator of
the geodesic convex structure of $G$, i.e., the self-map of
$\mathcal{P}(V(G))$ such that $$\mathcal{I}_{G}(A) := \bigcup_{x, y \in
A}I_{G}(x,y)$$ for each $A \subseteq V(G)$.  The convex hull of a set
$A \subseteq V(G)$ is then $co_{G}(A) = \bigcup_{n \in
\mathbb{N}}\mathcal{I}_{G}^{n}(A)$.  Furthermore we will say that a
subgraph of a graph $G$ is \emph{convex} if its vertex set is convex,
and by the \emph{convex hull} $co_{G}(H)$ of a subgraph $H$ of $G$ we
will mean the smallest convex subgraph of $G$ containing
$H$ as a subgraph, that is, $$co_{G}(H) := G[co_{G}(V(H))].$$  A graph is said to be \emph{interval monotone} if all its intervals are convex.

\subsection{Cartesian product}\label{SS:cart.prod.}

The \emph{Cartesian product} of a family of graphs
$\left(G_{i}\right)_{i \in I}$ is the graph denoted by $\cartbig_{i
\in I}G_{i}$ (or simply by $G_{1} \Box G_{2}$ if $\vert I \vert = 2$)
with $\prod_{i \in I}V\left(G_{i}\right)$ as vertex set and such that,
for every vertices $u$ and $v$, $uv$ is an edge whenever there exists
a unique $j \in I$ with $pr_{j}(u)\pr_{j}(v) \in E\left(G_{j}\right)$
and $\pr_{i}(u)=\pr_{i}(v)$ for every $i \in I-\{j\}$, where $\pr_{i}$
is the \emph{$i$-th projection} of $\prod_{i \in I}V(G_{i})$ onto
V($G_{i})$.

If $I$ is infinite, then the connected components of a Cartesian product of connected
graphs are called \emph{weak Cartesian products} (see \cite{HIK11}).  More precisely, the component of $\cartbig_{i
\in I}G_{i}$ containing some vertex $a$ is called the \emph{weak Cartesian product at $a$}, and is denoted by $$\cartbig_{i
\in I}^{a}G_i.$$  
Clearly, the Cartesian product coincides with the weak Cartesian
product provided that $I$ is finite and the factors are connected.

In
particular, \emph{hypercubes} are the weak Cartesian powers of
$K_{2}$.  For every non-negative integer $n$, a $n$-cube, that is, a hypercube of dimension $n$, is often denoted by $Q_{n}$.  In particular $Q_{0} = K_{1}$, $Q_{2} = C_{4}$, and a $3$-cube is generally called a \emph{cube}.  Given a vertex $a$ of a Cartesian product $G = \cartbig_{i \in I}G_{i}$, the \emph{$G_i$-layer through $a$} is the subgraph of $G$ induced by the set $\{x \in V(G) : p_j(x) = p_j(a) \text{\;for}\; j \neq i\}$.  We list below the properties, in part well-known, of the
Cartesian product that we will use in this paper (see \cite{IKR08} for
the main properties of the Cartesian product).  

\begin{pro}\label{P:pro.Cartes.prod.}
Let $G = G_{0} \Box G_{1}$ be a Cartesian product of two connected 
graphs.  We have the following properties:

\textnormal{(i)}\; $d_{G}(x,y) = d_{G_{0}}(pr_{0}(x),pr_{0}(y)) + 
d_{G_{1}}(pr_{1}(x),pr_{1}(y))$ for any $x, y \in V(G)$ 
(\textbf{Distance Property}).

\textnormal{(ii)}\; $I_{G}(x,y) = I_{G_{0}}(pr_{0}(x),pr_{0}(y))
\times I_{G_{1}}(pr_{1}(x),pr_{1}(y))$ for any $x, y \in V(G)$ 
(\textbf{Interval Property}).

\textnormal{(iii)}\; $pr_{i}(I_{G}(x,y)) =
I_{G_{i}}(pr_{i}(x),pr_{i}(y))$ for any $x, y \in V(G)$ and $i = 0,
1$.

\textnormal{(iv)}\; Let $e$ and $f$ be two adjacent edges of $G$ which are in
different layers.  Then there exists exactly one convex $4$-cycle in
$G$ that contains both $e$ and $f$ (\textbf{$\mathbf{4}$-Cycle
Property}).

\textnormal{(v)}\; A subgraph $F$ of $G$ is convex if and 
only if $F = pr_{0}(F) \Box pr_{1}(F)$, where both $pr_{0}(F)$ and 
$pr_{1}(F)$ are convex (\textbf{Convex Subgraph Property}).

\textnormal{(vi)}\; $pr_{i}(\mathcal{I}_{G_{i}}^{n}(A)) = \mathcal{I}_{G}^{n}(pr_{i}(A))$ for each $A \subseteq V(G)$, $i = 0, 1$ and any 
non-negative integer $n$.

\textnormal{(vii)}\; $pr_{i}(co_{G}(A)) = co_{G_{i}}(pr_{i}(A))$ for each $A \subseteq V(G)$ and $i = 0, 1$.
\end{pro}

\subsection{Partial cubes}\label{SS:part.cub.}

First we will recall some properties of \emph{partial cubes}, that is,
of isometric subgraphs of hypercubes.  Partial cubes are particular
connected bipartite graphs.

For an edge $ab$ of a graph $G$, let
\begin{align*}
	W_{ab}^{G}& := \{x \in V(G): d_{G}(a,x) < d_{G}(b,x) \},\\
	U_{ab}^{G}& := \{x \in W_{ab}: x \text{\;has a neighbor in}\; W_{ba}\}.
\end{align*}

If no confusion is likely, we will simply denote $W_{ab}^{G}$ and
$U_{ab}^{G}$ by $W_{ab}$ and $U_{ab}$, respectively.  Note that the
sets $W_{ab}$ and $W_{ba}$ are disjoint and that $V(G) = W_{ab} \cup
W_{ab}$ if $G$ is bipartite and connected.

Two edges $xy$ and $uv$ are in the Djokovi\'{c}-Winkler
relation $\Theta$ if

$$d_{G}(x,u)+d_{G}(y,v) \neq d_{G}(x,v)+d_{G}(y,u).$$

If $G$ is bipartite, the edges $xy$ and $uv$ are in
relation $\Theta$ if and only if $d_{G}(x,u) = d_{G}(y,v)$ and
$d_{G}(x,v) = d_{G}(y,u)$.  The relation $\Theta$ is clearly reflexive 
and symmetric.

\begin{lem}\label{L:bip.gr./convex}
Let $C$ be a convex set of a bipartite graph $G$.  Then $C \subseteq W_{ab}$ for any edge $ab \in \partial_{G}(C)$ with $a \in C$.
\end{lem}

\begin{proof}
Let $x \in C$ and $ab \in \partial_{G}(C)$ with $a \in C$.  Suppose that $x \notin W_{ab}$.  Then $b \in I_{G}(x,a)$, and thus $b \in C$ by the convexity of $C$, contrary to the fact that $ab \in \partial_{G}(C)$.
\end{proof}

\begin{rem}\label{R:Theta/notation}
If $G$ is bipartite, then, by~\cite[Lemma 11.2]{HIK11}, \emph{the notation can be chosen so that the edges $xy$ and $uv$ are in
relation $\Theta$ if and only if $$d_{G}(x,u) = d_{G}(y,v) = d_{G}(x,v)-1 = d_{G}(y,u)-1,$$ or equivalently if and only if $$y \in I_G(x,v) \text{\; and}\quad x \in I_G(y,u).$$}
From now on, we will always use this way of defining the relation $\Theta$.  Note that, in this way, the edges $xy$ and $yx$ are not in relation $\Theta$ because $y \notin I_G(x,x)$ and $x \notin I_G(y,y)$.  In other word, each time the relation $\Theta$ is used, the notation of an edge induces an orientation of this edge.
\end{rem}

\begin{thm}\label{T:Djokovic-Winkler}\textnormal{(Djokovi\'{c}~\cite[Theorem
1]{D73} and Winkler~\cite{W84})} A connected bipartite graph $G$ is a
partial cube if and only if it has one of the following properties:

\textnormal{(i)}\; For every edge $ab$ of $G$, the sets
$W_{ab}$ and $W_{ba}$ are convex.

\textnormal{(ii)}\;  The relation $\Theta$ is transitive.
\end{thm}

It follows in particular that \emph{the half-spaces of a partial cube
$G$ are the sets $W_{ab}$, $ab \in E(G)$}.  Furthermore we can easily 
prove that the copoints of a partial cube are its half-spaces.

We recall that the geodesic convexity of a partial cube $G$ has the
separation property $\mathrm{S}_{3}$: if a vertex $x$ does not
belong to a convex set $C \subseteq V(G)$, then there is a half-space
$H$ which separates $x$ form $C$, that is, $x \notin H$ and $C
\subseteq H$.  As a matter of fact, the geodesic convexity of a 
bipartite graph $G$ has property $\mathrm{S}_{3}$ if and only if 
$G$ is a partial cube (see~\cite[Proposition 2.2]{B89}).

We also recall that, if $u_{0}, u_{1}, u_{2}$ are three vertices of a graph
$G$, then a \emph{median} of the triple $(u_{0},u_{1},u_{2})$ is any
element of the intersection $I_{G}(u_{0},u_{1}) \cap
I_{G}(u_{1},u_{2}) \cap I_{G}(u_{2},u_{0})$.  Moreover a graph $G$ is a \emph{median graph} if any triple of its vertices has a unique median.  Median graphs are particular partial cubes.  Actually they are the retracts of hypercubes (see Bandelt~\cite{B84}).

We say that a subgraph $H$ of a partial cube $G$ is \emph{median-stable} if, for any
triple $(x,y,z)$ of vertices of $H$, if $(x,y,z)$ has a median $m$ in
$G$, then $m \in V(H)$.  Note that, if $H$ is isometric in $G$, then
$m$ is the median of $(x,y,z)$ in $H$.  A median-stable isometric
subgraph of $G$ is called a \emph{faithful} subgraph of $G$, or is
said to be \emph{faithful} in $G$.  Clearly any faithful subgraph of a
faithful subgraph of $G$ is itself a faithful subgraph of $G$, and
moreover any convex subgraph of $G$ is faithful.

In the following lemma we list some well-known properties of partial 
cubes.

\begin{lem}\label{L:gen.propert.}
Let $G$ be a partial cube.  We have the following properties:

\textnormal{(i)}\; If a triple of vertices of $G$ has a median, then
this median is unique.

\textnormal{(ii)}\; Each interval of $G$ is finite and convex.

\textnormal{(iii)}\; Each polytope of $G$ is finite.

\textnormal{(iv)}\; Let $x,y$ be two vertices of $G$, $P$ an
$(x,y)$-geodesic and $W$ an $(x,y)$-path of $G$.  Then each edge of
$P$ is $\Theta$-equivalent to some edge of $W$.

\textnormal{(v)}\; A path $P$ in $G$ is a geodesic if and only if no two distinct edges of $P$ are $\Theta$-equivalent.

\textnormal{(vi)}\; Any edge of a cycle $C$ is $\Theta$-equivalent to another edge of $C$.

\textnormal{(vii)}\; A cycle $C$ of $G$ is isometric if and only if all pairs of antipodal edges in $C$ are the only pairs of distinct edges of $C$ which are $\Theta$-equivalent.

\textnormal{(viii)}\; Any shortest cycle of $G$ is convex.

\textnormal{(ix)}\; If $F$ is a convex subgraph of $G$, then no edge of $\partial_G(F)$ is $\Theta$-equivalent to an edge of $F$.
\end{lem}

\begin{lem}\label{L:inf.p.c.}\textnormal{(Polat~\cite[Lemma 3.12]{P12})}
A bipartite graph $G$ is a partial cube if and only if every polytope of $G$ induces a partial cube.
\end{lem}

\begin{lem}\label{L:edge/co(H)}
Let $H$ be a subgraph of a partial cube $G$.  We have the following two properties:

\textnormal{(i)}\; If $H$ is connected, then any edge 
of $co_{G}(H)$ is $\Theta$-equivalent to an edge of $H$.

\textnormal{(ii)}\;  If any edge of $G$ is $\Theta$-equivalent to an edge of $H$, then $co_G(H) = G$.
\end{lem}

\begin{proof}
(i)\; We recall that $co_{G}(H) = G[\bigcup_{N \in
\mathbb{N}}\mathcal{I}_{G}^{n}(V(H))]$.  We prove by induction on $n$
that any edge of $G[\mathcal{I}_{G}^{n}(V(H))]$ is $\Theta$-equivalent
to an edge of $H$.  This is trivial if $n = 0$.  Suppose that this is
true for some $n \geq 0$.  Let $e$ be an edge of
$G[\mathcal{I}_{G}^{n+1}(V(H))]$ that is not an edge of
$G[\mathcal{I}_{G}^{n}(V(H))]$.  Then $e$ is an edge of an
$(x,y)$-geodesic $P$, for some $x, y \in \mathcal{I}_{G}^{n}(V(H))$.
Let $W$ be an $(x,y)$-path of $G[\mathcal{I}_{G}^{n}(V(H))]$, note
that this graph is connected.  Then $e$ is $\Theta$-equivalent to some
edge $e'$ of $W$ by Lemma~\ref{L:gen.propert.}(iv).  By the induction
hypothesis, $e'$ is $\Theta$-equivalent to an edge $e''$ of $H$.
Hence $e$ and $e''$ are $\Theta$-equivalent by transitivity of the
relation $\Theta$.

(ii)\; Suppose that $H' := co_G(H) \neq G$.  Because $G$ is connected, there is a vertex $x$ of $G-H'$ which is adjacent to some vertex $y$ of $H'$.  By the properties of $H$, the edge $xy$ is $\Theta$-equivalent to some edge $ab$ of $H$, and thus of $H'$, contrary to Lemma~\ref{L:gen.propert.}(ix).  Therefore $co_G(H) = G$.
\end{proof}

\begin{lem}\label{L:E(G[Wab])}
Let $ab$ be an edge of a partial cube $G$.  Then an edge of
$G[W_{ab}]$ is $\Theta$-equivalent to an edge of $G[W_{ba}]$ if and
only if it is $\Theta$-equivalent to an edge of
$G[\mathcal{I}_{G}(U_{ab})]$.
\end{lem}

\begin{proof}
Let $uv$ be an edge of $G[W_{ab}]$ which is $\Theta$-equivalent to an edge $u'v'$ of $G[W_{ba}]$.  Let $P_u$ and $P_v$ be a $(u,u')$-geodesic and a $(v,v')$-geodesic, respectively.  Let $u_0$ and $v_0$ be the only vertices of $P_u$ and $P_v$ in $U_{ab}$, respectively, and let $W$ be a $(u_0,v_0)$-geodesic.  Because $P_u \cup \langle u,v\rangle$ and $P_v \cup \langle u,v\rangle$ are $(u',v)$-geodesic and  $(v',u)$-geodesic, respectively, it follows that the edge $uv$ is $\Theta$-equivalent to no edge of $P_u$ and $P_v$.  Therefore, by Lemma~\ref{L:gen.propert.}(vi), $uv$ is $\Theta$-equivalent to an edge of $W$, and thus of $G[\mathcal{I}_{G}(U_{ab})]$.

Conversely, suppose that an edge $cd$ of $G[W_{ab}]$ is
$\Theta$-equivalent to an edge of $G[\mathcal{I}_{G}(U_{ab})]$, and thus to an edge $uv$ of an $(x,y)$-geodesic $P$
for some $x, y \in U_{ab}$.  Let $x'$ and $y'$ be the neighbors of $x$
and $y$ in $U_{ba}$, respectively, and let $P'$ be an
$(x',y')$-geodesic.  Then $C =\langle x',x\rangle \cup P \cup \langle
y,y'\rangle \cup P'$ is a cycle of $G$.  Hence, $uv$ is
$\Theta$-equivalent to an edge $u'v'$ of $C$.  Clearly $u'v'$ is an
edge of $P'$ because both $xx'$ and $yy'$ are $\Theta$-equivalent to
$ab$, while $uv$ is an edge of $G[W_{ab}]$.  It follows, by
transitivity, that the edge $cd$ is $\Theta$-equivalent to $u'v'$,
which is an edge of $G[\mathcal{I}_{G}(U_{ab})]$, and thus of $G[W_{ba}]$.
\end{proof}

\emph{Throughout this paper we will use the following notations.  Let $ab$ be an edge of a partial cube $G$:
\begin{align*}
G_{ab}& := G[co_G(U_{ab})]\\
G_{ \overrightarrow{ab}}& := G[co_G(U_{ab}) \cup W_{ba}]\\
G_{\overline{ab}}& := G_{ \overrightarrow{ab}} \cap G_{ \overrightarrow{ba}} = G[co_G(U_{ab}) \cup co_G(U_{ba})] = G_{\overline{ba}}.
\end{align*}}

\subsection{Weak geodesic topology of a partial cube}\label{SS:wg.top./part.cube}

We recall (see~\cite{P09-1}) that the \emph{weak geodesic topology} of
a graph $G$ is the finest weak topology on $V(G)$ endowed with the
geodesic convexity, that is, the topology (in terms of closes sets)
generated by all convex subsets of $V(G)$ as a subbase.

\emph{Throughout this paper, unless stated otherwise, we will always suppose
that that the vertex set of any partial cube $G$ is endowed with the
weak geodesic topology}.  Because the geodesic convexity on the vertex
set of a partial cube $G$ has property $\mathrm{S}_{3}$, it
follows that each convex set of $G$ is an intersection of half-spaces,
hence the family of half-spaces of $G$ is a subbase of the weak
geodesic topology on $V(G)$.

\begin{pro}\label{P:compactness}\textnormal{(Polat~\cite[Theorem 
    3.9]{P09-1})}
Let $G$ be a partial cube.  The following assertions are equivalent:

\textnormal{(i)}\; $V(G)$ is compact.

\textnormal{(ii)}\; $V(G)$ is weakly countably compact (i.e., every 
infinite subset of $V(G)$ has a limit point).

\textnormal{(iii)}\; The vertex set of any ray of $G$ has a limit point.

\textnormal{(iv)}\; The vertex set of any ray of $G$ has a 
finite positive number of limit points.
\end{pro}

From now on, by a \emph{compact partial cube}, we will mean a partial
cube whose vertex space is compact.  According to the above proposition it follows that
a compact partial cube contains no isometric rays \cite[Corollary 3.15]{P09-1}, and moreover that any rayless partial cube is compact.

\subsection{Geodesically consistent partial cubes}\label{SS:geod.consist.p.c.}

A vertex $x$ of a connected graph $G$ \emph{geodesically dominates} a
subset $A$ of $V(G)$ if, for every finite $S \subseteq V(G-x)$, there
exists an $a \in (A - \{x \})$ such that $S \cap I_{G}(x,a) =
\emptyset$.  The \emph{geodesic topology} on the vertex set of a graph
$G$ is the topology for which a subset $A$ of $V(G)$ is closed if and
only if every vertex that geodesically dominates $A$ belongs to $A$.
From \cite[Theorem 3.9]{P98} we have:

\begin{pro}\label{P:geod.dom/isom.ray}\textnormal{(Polat~\cite[Theorem 
3.9]{P98})}
Let $G$ be a graph.  The following assertions are equivalent:

\textnormal{(i)}\; The geodesic space $V(G)$ is compact.

\textnormal{(ii)}\;  $G$ contains no isometric rays.

\textnormal{(iii)}\; The vertex set of every ray of $G$ is
geodesically dominated.
\end{pro}

Any limit point of a set $A$ of vertices of a partial cube $G$
geodesically dominates $A$.  On the other hand, a vertex $x$ which
geodesically dominates a set $A$ is not necessarily a limit point of
$A$.  The geodesic topology is compatible with the geodesic convexity,
that is, all polytopes are geodesically closed (i.e., closed for the
geodesic topology). 

\begin{defn}\label{D:geod.consistent}
A graph $G$ is said to be \emph{geodesically consistent} if the
geodesic topology on $V(G)$ coincides with the weak geodesic topology.
\end{defn}

In other words, $G$ is geodesically consistent if the limit points of
any set $A \subseteq V(G)$ are the vertices of $G$ that geodesically dominate
$A$.
 
We say that a set $A$ of vertices of a graph $G$ is \emph{finitely geodesically dominated} if the set of vertices which geodesically dominate $A$ is finite and non-empty.  From Propositions~\ref{P:compactness} and \ref{P:geod.dom/isom.ray}, we infer the following result.

\begin{cor}\label{C:geod.cons./comp.}
A geodesically consistent partial cube $G$ is compact (i.e., contains no isometric rays) if and only if the vertex set of every ray of $G$ is finitely geodesically dominated.
\end{cor}

From~\cite[Theorem 4.8]{P09-1} we have:

\begin{pro}\label{P:charact.geod.consistent.p.c.}\textnormal{(Polat~\cite[Theorem 
4.8]{P09-1})}
Let $G$ be a partial cube.  $G$ is geodesically consistent if and only
if, for every edge $ab$ of $G$, each vertex in $co_{G}(U_{ab})$ which
geodesically dominates $U_{ab}$ belongs to $U_{ab}$.
\end{pro}

Median graphs and more generally netlike partial cubes are
geodesically consistent (cf. \cite[Proposition 4.15]{P09-1}).  Moreover any convex subgraph of a geodesically consistent partial cube is geodesically consistent.

\subsection{Gated sets}\label{SS:gated}

A set $A$ of vertices of a graph $G$ is said to be \emph{gated} if,
for each $x \in V(G)$, there exists a vertex $y$ (the \emph{gate} of
$x$) in $A$ such that $y \in I_{G}(x,z)$ for every $z \in A$.  Any
gated set is convex.  Note that conversely, any convex set of a median graph is gated.
Moreover the set of gated sets of a graph with the addition of the empty set is a
convexity, and thus is closed under any intersections.   We say that a
subgraph $F$ of a graph $G$ is \emph{gated} if its vertex set is gated, and the function of $G$ onto $F$ which assigns to each vertex of $G$ its gate in $F$ is called the \emph{gate map} of $F$.

\begin{pro}\label{P:gated.sets/Helly}
The gated sets of a compact partial cube $G$ have the strong Helly
property, that is, any family of gated sets of $G$ that pairwise
intersect have a non-empty intersection, and this intersection is gated.
\end{pro}

\begin{proof}
This is a consequence of the compactness of $V(G)$ and of the fact
that, by \cite[Proposition 2.4]{B89}, the gated sets of $G$ have
the Helly property, that is, any finite family $(A_i)_{i \in I}$ of gated sets of
$G$ that pairwise intersect have a non-empty intersection $A$.  The set $A$ is gated if $I$ is finite (see~\cite[Corollary 16.3]{HIK11}.  Assume that $I$ is infinite, and let $x \in V(G)$ and $a \in A$.  The gate of $x$ in $A$, if it exists is an element of $I_G(x,a)$.  Then, because, by Lemma~\ref{L:gen.propert.}(ii), every interval in a partial cube is finite, there exists a finite $J \subseteq I$ and a $y \in I_G(x,a)$ which is the gate of $x$ in $\bigcap_{i \in J'}A_i$ for every finite subset $J'$ of $I$ that contains $J$.  Therefore $y$ is the gate of $x$ in $A$.  Hence $A$ is gated.
\end{proof}

It follows that, for any set $A$ of vertices of a compact partial cube $G$, there exists a smallest gated set that contains $A$.  This set is called the \emph{gated hull} of $A$ in $G$.

The following lemma is clear.

\begin{lem}\label{L:gated.inter.convex}
Let $G$ be a partial cube.  Then we have the following properties:

\textnormal{(i)}\; If $H$ is a gated subgraph of $G$, and $F$ a convex 
subgraph of $G$, then $H \cap F$ is gated in $F$.

\textnormal{(ii)}\; If a finite subgraph $H$ of $G$ is gated in any 
finite convex subgraph of $G$ that contains $H$, then $H$ is gated 
in $G$.
\end{lem}

A graph $G$ is the \emph{gated amalgam} of two graphs
$G_{0}$ and $G_{1}$ if $G_{0}$ and $G_{1}$ are isomorphic to two
intersecting gated subgraphs $G'_{0}$ and $G'_{1}$ of $G$ whose union
is $G$.  More precisely we also say that $G$ is the gated amalgam of
$G_{0}$ and $G_{1}$ \emph{along} $G'_{0} \cap G'_{1}$.  The gated amalgam of two partial cubes is clearly a partial cube.

\subsection{Semi-peripheries}\label{SS:semi-periph.}

We recall that if $ab$ is an edge of a partial cube $G$ such that
$W_{ab} = U_{ab}$, then $W_{ab}$ is called a \emph{periphery} of $G$.
A partial cube has generally no periphery.  We generalize this concept
as follows.

\begin{defn}\label{D:semi-periphery}
Let $G$ be a partial cube.  If $W_{ab} = co_{G}(U_{ab})$ for some edge
$ab$ of $G$, then we say that $W_{ab}$ is a \emph{semi-periphery} of
$G$.
\end{defn}

Clearly any finite partial cube has a semi-periphery.  We will see that this is also true for any compact partial cube.

\begin{pro}\label{P:min.half-space/semi-periph.}
A set $A$ of vertices of a partial cube $G$ is a semi-periphery of $G$
if and only if it is a minimal non-empty half-space of $G$.
\end{pro}

\begin{proof}
Let $ab$ be an edge of $G$ such that $W_{ab}$ is not a 
semi-periphery.  Let $xy \in \partial_{G}(W_{ba} \cup co_{G}(U_{ab}))$ with
$x \in co_{G}(U_{ab})$.  The set $W_{ba} \cup co_{G}(U_{ab})$ is convex.  Hence the edge $xy$ is not 
$\Theta$-equivalent to any edge of 
$G[W_{ba} \cup co_{G}(U_{ab})]$ by Lemma~\ref{L:gen.propert.}(ix).  Hence 
$W_{yx} \subset W_{ab}$, and thus $W_{ab}$ is not a minimal 
half-space.

Conversely, suppose that the half-space $W_{ab}$ is not minimal for 
some edge $ab$ of $G$.  Then there is an edge $cd$ of $G$ such that 
$W_{cd} \subset W_{ab}$.  It follows that $cd$ is not 
$\Theta$-equivalent to an edge of $G[W_{ba}]$, and thus to an edge of 
$G[co_{G}(U_{ab})]$ by Lemma~\ref{L:E(G[Wab])}.  Hence $cd$ is an 
edge of $G[W_{ab}]$ which is not an edge of $G[co_{G}(U_{ab})]$.  
Therefore $W_{ab}$ is not a semi-periphery of $G$.
\end{proof}

\begin{pro}\label{P:half-spaces/compact}
Let $G$ be a compact partial cube.  Then any chain of half-spaces of
$G$ is finite.
\end{pro}

\begin{proof}
Suppose that there is an infinite chain of half-spaces of $G$.  Then, by the definition of a half-space,
there exists an infinite sequence $(C_{n})_{n \in \mathbb{N}}$ of
half-spaces of $G$ such that $C_{n} \supset C_{n+1}$ for every $n \in
\mathbb{N}$.  Therefore $(C_{n})_{n \in \mathbb{N}}$ is a sequence of non-empty half-spaces, and thus of non-empty closed sets, of $V(G)$ whose
intersection is empty, contrary to the fact that $G$ is compact.
Consequently any chain of half-spaces of $G$ is finite.
\end{proof}

We obtain immediately:

\begin{cor}\label{C:compact=>min.half-spaces}
Let $G$ be a compact partial cube.  Then there exists a minimal
non-empty half-space in $G$, and moreover any non-empty half-space of
$G$ contains a minimal non-empty half-space.
\end{cor}

From Proposition~\ref{P:min.half-space/semi-periph.} and
Corollary~\ref{C:compact=>min.half-spaces} we have:

\begin{pro}\label{P:compact=>semi-periphery}
Any compact partial cube has a semi-periphery.
\end{pro}

\begin{pro}\label{P:med.gr./loc.periph.}    
Let $G$ be a partial cube.  The following assertions are equivalent:

\textnormal{(i)}\; $G$ is a median graph.

\textnormal{(ii)}\; $U_{ab}$ is convex for any edge $ab$ of $G$.

\textnormal{(iii)}\; For every convex subgraph $F$ of $G$, any
semi-periphery of $F$ is a periphery.
\end{pro}

\begin{proof}
Conditions (i) and (ii) are equivalent by a result of Bandelt~\cite{B82}.

(i) $\Rightarrow$ (iii):\; Assume that $G$ is a median graph, and let
$F$ be a convex subgraph of $G$.  Then $F$ is also a median graph, and
thus, by (ii), any semi-periphery of a
median graph is a periphery.

(iii) $\Rightarrow$ (ii):\; Assume Condition (iii).  Let $ab \in E(G)$.
The subgraph $G_{ \overrightarrow{ab}}$ of $G$ induced by the set $W_{ba} \cup
co_{G}(U_{ab})$ is convex, and $co_{G}(U_{ab})$ is a semi-periphery of
$G_{ \overrightarrow{ab}}$.  Hence it is a periphery of this subgraph by (iii).  It
follows that $U_{ab}$ is convex.
\end{proof}

\begin{defn}\label{D:semi-periph./periph.}
A partial cube $G$ is said to be \emph{semi-peripheral} (resp. \emph{peripheral}) if each vertex of $G$ belongs to a semi-periphery (resp. periphery).
\end{defn}

We will characterize the peripheral partial cubes.  Note that the
Cartesian product of any partial cube with $K_{2}$ is clearly
peripheral.  We will see that the converse of this property is true
for finite partial cubes.

We need the following fact: if, for some edge $ab$ of a partial cube
$G$, the set $W_{ab}$ is a periphery, then $W_{ba}$ is gated.  Indeed,
for any $x \in W_{ab}$, the neighbor of $x$ in $W_{ba}$ belongs to
$I_{G}(x,y)$ for every $y \in W_{ba}$.

\begin{pro}\label{P:periph.p.c.}
Let $G$ be a finite partial cube with more than one vertex.  The
following assertions are equivalent:

\textnormal{(i)}\; $G$ is peripheral.

\textnormal{(ii)}\; $W_{ab}$ and $W_{ba}$ are peripheries for some 
edge $ab$ of $G$.

\textnormal{(iii)}\; $G$ is isomorphic to the Cartesian product of a
partial cube with $K_{2}$.
\end{pro}

\begin{proof}
(ii) and (iii) are clearly equivalent, and (ii) $\Rightarrow$ (i) is 
obvious.  It remains to prove (i) $\Rightarrow$ (ii).  

Assume that (ii) is not true, that is, for any $ab \in E(G)$, at most
one of the sets $W_{ab}$ and $W_{ba}$ is a periphery.  In other words,
the complement of any periphery of $G$ is not a periphery, and moreover
it is gated by the above remark.  Hence the complements of the
peripheries of $G$ are pairwise non-disjoint gated half-spaces.
Therefore, by Proposition~\ref{P:gated.sets/Helly}, they have a
non-empty intersection, say $A$.  It follows that the union of all
peripheries of $G$ is equal to $V(G)-A$, which is then distinct from
$V(G)$.  Hence $G$ is not peripheral.
\end{proof}

\begin{defn}\label{D:strong.periph.}
A partial cube $G$ is said to be \emph{strongly semi-peripheral} (resp. \emph{strongly peripheral}) if $W_{ab}$ and $W_{ba}$ are semi-peripheries (resp. peripheries) for every edge $ab$ of $G$.
\end{defn}

\begin{pro}\label{P:strong.-periph./hypercube}
A finite partial cube is strongly peripheral if and only if it is a hypercube.
\end{pro}

\begin{proof}
A finite hypercube is clearly strongly peripheral.  Conversely we will prove that any finite strongly peripheral partial cube $G$ is a hypercube by induction on its order.  This is obvious if $G = K_{1}$.  Suppose that this holds if the order of $G$ is $n$ for some positive integer $n$.  Let $G$ be a finite strongly peripheral partial cube of order $n+1$, 
 and let $ab \in E(G)$.  Then $G = G' \Box K_{2}$ since $W_{ab}$ and $W_{ba}$ are peripheries.  Let $uv \in E(G')$.  Then $W_{uv}$ and $W_{vu}$ are peripheries of $G$ since $G$ is strongly peripheral.  It clearly follows that $W_{uv}^{G'}$ and $W_{vu}^{G'}$ are peripheries of $G'$.  Hence $G'$ is a strongly peripheral partial cube of order $n$.  By the induction hypothesis, $G'$ is a hypercube.  Therefore so is $G = G' \Box K_{2}$.
 \end{proof}

We conclude these preliminaries by pointing out that, 
\emph{throughout this paper, we do not distinguish between isomorphic graphs}.  Thus, 
when we say that a graph $G$
is equal to some Cartesian product or to some special graph, such as $K_{2}$ for
example, we usually mean that $G$ is isomorphic to this Cartesian product or to
this special graph.

\section{Fundamental properties}\label{S:fund.prop.}

\subsection{Partial cubes with pre-hull number at most $1$}\label{SS:p.c.ph.at most1}

We begin by recalling some definitions and results from~\cite{PS09}.  In that paper we introduced and studied the concept of pre-hull
number of a convexity.  The (geodesic) pre-hull number $ph(G)$ of a graph $G$ is a parameter which measures the intrinsic non-convexity of $V(G)$ in terms of the number of iterations of the pre-hull operator associated with the interval operator $I_G$ which are necessary, in the worst case, to reach the canonical minimal convex extension of copoints of $V(G)$ when they are extended by the adjunction of an attaching point.

\begin{defn}\label{D:ph(G)}
Let $G$ be a graph.  The least non-negative integer $n$ (if it exists) such that
$co_{G}(C \cup \{x \}) = \mathcal{I}_{G}^{n}(C \cup \{x \})$ for each
vertex $x$ of $G$ and each copoint $C$ at $x$, is called the \emph{pre-hull number} of $G$ and is denoted by
$ph(G)$.  If there is no such
$n$ we put $ph(G) := \infty$.
\end{defn}

\begin{pro}\label{P:ph(bipart.gr.)=0/trees}\textnormal{(Polat and Sabidussi~\cite[Corollary 3.8]{PS09})}
The pre-hull number of a connected bipartite graph $G$ is zero if and
only if $G$ is a tree.
\end{pro}

\begin{defn}\label{D:ph-stable}\textnormal{(Polat and Sabidussi~\cite[Definition 7.1]{PS09})}
Call a set $A$ of vertices of a graph $G$ \emph{ph-stable} if any two
vertices $u,v \in \mathcal{I}_{G}(A)$ lie on a geodesic joining two
vertices in $A$.
\end{defn}

The condition of Definition~\ref{D:ph-stable}, which is
symmetric in $u$ and $v$, can be replaced by the formally
\textquotedblleft one-sided\textquotedblright condition: \emph{for any
two vertices $u,v \in \mathcal{I}_{G}(A)$ there is a $w \in A$
such that $v \subseteq I_{G}(u,w)$}.

\begin{lem}\label{L:ph-stable}\textnormal{(Polat~\cite[Proposition
2.4]{P05-1})}
If a set $A$ of vertices of a graph $G$ is ph-stable, then, for all
$u, v \in \mathcal{I}_{G}(A)$, $I_{G}(u,v) \subseteq I_{G}(a,b)$ for
some $a, b \in A$.  In particular, each edge of
$G[\mathcal{I}_{G}(A)]$ belongs to an $(a,b)$-geodesic for some $a, b
\in A$, and moreover $co_{G}(A) = \mathcal{I}_{G}(A)$.
\end{lem}

\begin{cor}\label{C:Wab=IGab}
Let $W_{ab}$ be a semi-periphery of some partial cube $G$.  If $U_{ab}$ is ph-stable, then $W_{ab} = \mathcal{I}_{G}(U_{ab})$, and more precisely,
$G[W_{ab}]$ is the union of the $(u,v)$-geodesics for all $u, v \in
U_{ab}$.
\end{cor}

\begin{pro}\label{P:bip.gr./ph=1}\textnormal{(Polat and Sabidussi~\cite[Theorem 7.4]{PS09})}
Let $G$ be a bipartite graph.  Then $ph(G) \leq 1$ if and only if, for
every copoint $K$ of $G$, the set $\mathrm{Att}(K)$ is convex and
$N_G(K) \cap \mathrm{Att}(K)$ is ph-stable.
\end{pro}

The following result follows immediately from the above proposition.

\begin{pro}\label{P:charact.p.c.ph}\textnormal{(Polat and Sabidussi~\cite[Theorem 7.5]{PS09})}
Let $G$ be a partial cube.  Then $ph(G) \leq 1$ if and only if
$U_{ab}$ and $U_{ba}$ are ph-stable for every edge $ab$ of $G$.
\end{pro}

\begin{thm}\label{T:bip.gr.ph1=>p.c.}\textnormal{(Polat~\cite[Theorem 4.6]{P12})}
Any connected bipartite graph $G$ such that $ph(G) \leq 1$ is a 
partial cube.
\end{thm}

Note that a bipartite graph whose pre-hull number is greater than $1$ may or may not be a partial cube.  For example, $2$ is the pre-hull number of both $K_{2,3}$, which is the smallest connected bipartite graph which is not a partial cube, and of the partial cube $Q_3^-$ (i.e., the $3$-cube $Q_3$ minus a vertex).

A lot of well-known partial cubes have a pre-hull number which is at most equal to $1$: median graphs, benzenoid graphs, cellular bipartite graphs and more generally netlike partial cubes, and also antipodal partial cubes.  We recall that a connected graph $G$ is called \emph{antipodal} if for any vertex $x 
\in V(G)$ there is a (necessarily unique) vertex $\bar{x}$ (the 
\emph{antipode} of $x$) such that $I_{G}(x,\bar{x}) = V(G)$.\footnote{Bipartite antipodal graphs were introduced by Kotzig~\cite{Ko68} under the name of \emph{S}-graph.  This 
concept of antipodality is a 
special case of the general concept of antipodality commonly used in 
algebraic graph theory.}  In such a graph one obviously has that $$d_{G}(x,y)+d_{G}(y,\bar{x}) = d_{G}(x,\bar{x})= \mathrm{diam}(G) \textnormal{\; for 
any}\; x, y \in V(G),$$ where $\mathrm{diam}(G)$ is the diameter of $G$.  Next result gives two characterizations of antipodal partial cubes.

\begin{thm}\label{T:antipod.p.c.}
Let $G$ be a bipartite antipodal graph.  The following assertions are equivalent:

\textnormal{(i)}\; $G$ is a partial cube.

\textnormal{(ii)}\; $ph(G) \leq 1$.

\textnormal{(iii)}\; $G$ contains no subdivision of $K_{3,3}$.

\textnormal{(iv)}\; $G$ is interval monotone.
\end{thm}

\begin{proof}
(i) $\Rightarrow$ (ii) is proved in \cite[Subsection 8]{PS09}.  The equivalence of (i) and (iii) is proved in~\cite{T91}, whereas the equivalence of (i) and (iv) is \cite[Theorem 4.3]{P18}.
\end{proof}

It follows that the bipartite antipodal graph graph $K_{n,n}-M$, where $n > 4$ and $M$ is a perfect matching of $K_{n,n}$, is not a partial cube since it clearly contains a subdivision of $K_{3,3}$.

We will now study some properties of partial cubes whose
pre-hull number is at most~$1$, with in particular the closure of the class of these graphs  under usual operations of partial cubes.

\begin{pro}\label{P:fin.conv./ph1}
Let $G$ be a partial cube such that any finite subgraph of $G$ 
is contained in a finite convex subgraph of $G$ whose pre-hull number 
is at most $1$.  Then $ph(G) \leq 1$.
\end{pro}

\begin{proof}
Let $ab \in E(G)$ and $u, v \in \mathcal{I}_{G}(U_{ab})$.  Let $P_u$ and $P_v$ be geodesics joining vertices in $U_{ab}$ on which lie $u$ and $v$, respectively.  Then $\langle a,b\rangle \cup P_u \cup P_v$ is contained in a finite convex subgraph $F$ of $G$ such that $ph(G) \leq 1$.  The set $U_{ab}^F$ is ph-stable since $ph(F) \leq 1$, and thus $u, v$ lie on an $(x,y)$-geodesic $R$ for some $x, y \in U_{ab}^F$.  Because $F$ is convex in $G$, it follows that $R$ is a geodesic in $G$, and also that $x, y \in U_{ab}$ since $U_{ab}^{F} = U_{ab} \cap V(F)$ by Lemma~\ref{L:prop.isom subgr.}.  Therefore $U_{ab}$, and analogously $U_{ba}$, are ph-stable.  Hence $ph(G) \leq 1$ by Proposition~\ref{P:charact.p.c.ph}.
\end{proof}

\begin{rem}\label{R:fin.subgr.fin.conv.subgr.ph1}
We now prove that the converse of Proposition~\ref{P:fin.conv./ph1} is not true.

1.\; We first construct a sequence $G_{0}, G_{1}, \ldots$ of graphs
satisfying the following conditions:

(1)\; $G_{n}$ is a finite partial cube.

(2)\; $G_{n}$ is a convex subgraph of $G_{n+1}$.

(3)\; A cycle of $G_{n}$ is isometric if and only if it is a
$6$-cycle, and thus any isometric cycle of $G_{n}$ is convex.

(4)\; For each edge $ab$ of $G_{n}$ and all $u, v \in U_{ab}^{G_{n}}$,
$d_{G_{n}}(u,v)$ is even and each $(u,v)$-geodesic $\langle
u_{0},\ldots,u_{2p} \rangle$ with $u_{0} = u$ and $u_{2p} = v$ is such
that $u_{2i} \in U_{ab}^{G_{n}}$ for $0 \leq i \leq p$.

(5)\; For each edge $ab$ of $G_{n}$, if $u, v \in
\mathcal{I}_{G_{n}}(U_{ab}^{G_{n}})$ are such that $v \notin
I_{G_{n}}(u,w)$ for any $w \in U_{ab}^{G_{n}}$, then $v \in
I_{G_{n+1}}(u,w)$ for some $w \in U_{ab}^{G_{n+1}}$.

(6)\; For each $n \geq 1$, there exist $ab \in E(G_n)$ and two vertices $u, v \in
\mathcal{I}_{G_{n}}(U_{ab}^{G_{n}})$ such that $v \notin
I_{G_{n}}(u,w)$ for any $w \in U_{ab}^{G_{n}}$, i.e., $U_{ab}^{G_{n}}$ is not ph-stable.

\begin{figure}[!h]
    \centering
    {\tt    \setlength{\unitlength}{0.92pt}
\begin{picture}(329,227)
\thinlines    \put(259,22){$G_2$}
              \put(57,22){$G_1 = \overline{G_0}$}
              \put(279,99){\circle*{5}}
              \put(266,89){\circle*{5}}
              \put(290,109){\circle*{5}}
              \put(290,143){\circle*{5}}
              \put(230,150){\circle*{5}}
              \put(254,153){\circle*{5}}
              \put(242,128){\circle*{5}}
              \put(254,99){\circle*{5}}
              \put(266,71){\circle*{5}}
              \put(302,100){\circle*{5}}
              \put(290,127){\circle*{5}}
              \put(302,148){\circle*{5}}
              \put(84,105){$x_4^C$}
              \put(56,105){$x_2^C$}
              \put(82,162){$x_0^C$}
              \put(82,135){$y^C$}
              \put(77,200){\circle*{5}}
              \put(77,132){\circle*{5}}
              \put(28,159){\circle*{5}}
              \put(28,93){\circle*{5}}
              \put(77,52){\circle*{5}}
              \put(126,93){\circle*{5}}
              \put(126,159){\circle*{5}}
              \put(77,163){\circle*{5}}
              \put(53,113){\circle*{5}}
              \put(101,112){\circle*{5}}
              \put(266,198){\circle*{5}}
              \put(217,157){\circle*{5}}
              \put(217,90){\circle*{5}}
              \put(266,50){\circle*{5}}
              \put(315,90){\circle*{5}}
              \put(315,156){\circle*{5}}
              \put(266,130){\circle*{5}}
              \put(266,163){\circle*{5}}
              \put(242,144){\circle*{5}}
              \put(242,110){\circle*{5}}
              \put(279,153){\circle*{5}}
              \put(133,158){$x_5$}
              \put(133,91){$x_4$}
              \put(72,41){$x_3$}
              \put(12,91){$x_2$}
              \put(12,158){$x_1$}
              \put(72,207){$x_0$}
              \put(266,130){\line(-6,-5){49}}
              \put(77,133){\line(-6,-5){49}}
              \put(315,155){\line(-2,-1){25}}
              \put(290,143){\line(0,-1){33}}
              \put(266,164){\line(6,-5){24}}
              \put(266,89){\line(6,5){24}}
              \put(266,89){\line(0,-1){39}}
              \put(242,109){\line(6,-5){25}}
              \put(242,143){\line(6,5){24}}
              \put(242,144){\line(0,-1){34}}
              \put(218,157){\line(2,-1){25}}
              \put(266,198){\line(-6,-5){49}}
              \put(266,197){\line(6,-5){47}}
              \put(266,130){\line(6,-5){49}}
              \put(218,91){\line(6,-5){49}}
              \put(266,49){\line(6,5){49}}
              \put(217,157){\line(0,-1){65}}
              \put(314,156){\line(1,0){1}}
              \put(315,157){\line(0,-1){65}}
              \put(217,157){\line(0,-1){66}}
              \put(266,196){\line(0,-1){68}}
              \put(77,198){\line(0,-1){68}}
              \put(28,159){\line(0,-1){66}}
              \put(126,159){\line(0,-1){65}}
              \put(125,158){\line(1,0){1}}
              \put(28,159){\line(0,-1){65}}
              \put(77,51){\line(6,5){49}}
              \put(29,93){\line(6,-5){49}}
              \put(77,133){\line(6,-5){49}}
              \put(77,199){\line(6,-5){47}}
              \put(77,200){\line(-6,-5){49}}
\end{picture}}
\caption{Construction of $G$.}
\label{F:Rem. 2.24}
\end{figure}
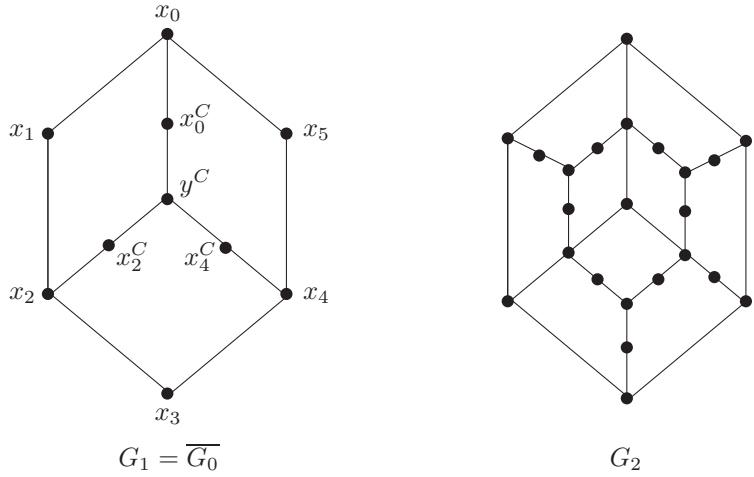

Let $G_{0}$ be a $6$-cycle.  Suppose that $G_{n}$ has already been
constructed for some $n \geq 0$.  We construct $G_{n+1}$ as follows.
Denote by $\mathcal{C}_{n}$ the set of all $6$-cycles of $G_{n}$ which
are not cycles of $G_{n-1}$ if $n > 1$.  Suppose that each cycle $C \in \mathcal{C}$ has exactly two vertices whose degree in $G_n$ is $2$, and that the distance between these two vertices is $2$.

For $C = \langle
x_{0},\ldots,x_{5},x_{0} \rangle \in \mathcal{C}_{n}$ whose vertices of degree $2$ are $x_2$ and $x_4$, let $y^{C}, 
x_{0}^{C}, \ldots, x_{5}^{C}$ be seven vertices which do
not belong to $G_{n}$, and let $$\overline{C} := C \cup \bigcup_{i = 0, 2, 4} \langle y^{C},x_i^{C},x_i\rangle$$ (see Figure~\ref{F:Rem. 2.24}).  Moreover, if $C$ and $C'$ are distinct
elements of $\mathcal{C}_n$, then the new vertices are chosen so that the sets 
$\{y^{C},x_{0}^{C},\ldots,x_{5}^{C} \}$ and
$\{y^{C'},x_{0}^{C'},\ldots,x_{5}^{C'} \}$ are
disjoint.  Now let $$G_{n+1} := G_{n} \cup \bigcup_{C \in
\mathcal{C}_{n}} \overline{C}.$$  For any $C \in \mathcal{C}_{n}$ the following facts are clear:

(a)\; The isometric cycles of $G_{n+1}$ are the $6$-cycles of this graph, and each of these cycles has exactly two vertices whose degree in $G_{n+1}$ is $2$, and the distance between these two vertices is $2$.

(b)\; The set of edges $\{x_{2i}x_{2i}^{C}: 0 \leq i \leq 2 \}$ is a new
$\Theta$-class in $G_{n+1}$.

(c)\; For $0 \leq i \leq 2$, the edges $y^{C}x_{2i}^{C},\;
x_{2i+2}x_{2i+1},\; x_{2i+4}x_{2i+5}$ (the
indices being modulo $6$) are $\Theta$-equivalent in $G_{n+1}$.

(d)\; $U_{x_0x_{0}^{C}}^{G_{n+1}} = \{x_0,x_{2},x_{4}\}$ and $\mathcal{I}_{G_{n+1}}(U_{y_{0}^Cx_{0}^{C}}^{G_{n+1}}) -
U_{y_{0}^Cx_{0}^{C}}^{G_{n+1}} \{x_1,x_3,x_5\}$.

(e)\; For any edge $ab$ of $G_{n}$,
$$\mathcal{I}_{G_{n}}(U_{ab}^{G_{n}}) -
U_{ab}^{G_{n}} \subseteq \mathcal{I}_{G_{n+1}}(U_{ab}^{G_{n+1}}) -
U_{ab}^{G_{n+1}}.$$

(f)\; The vertex $y^C$ has no gate in $C$ since there are three vertices of $C$, namely $x_0, x_2, x_4$, whose distance to $y^C$ is minimal.\\

Now we can check that $G_{n+1}$ satisfies the conditions (1)-(6).

(1):\; The relation $\Theta$ on $E(G_{n+1})$ is clearly transitive, and thus
$G_{n+1}$ is a partial cube.

(2), (3) and (4) follow immediately from the induction hypothesis,
the construction and the preceding facts.

(5):\; Let $ab$ be an edge of $G_{n}$ and two vertices $u, v \in
\mathcal{I}_{G_{n}}(U_{ab}^{G_{n}})$ such that $v \notin
I_{G_{n}}(u,w)$ for any $w \in U_{ab}^{G_{n}}$.  Then, in particular,
$v \notin U_{ab}^{G_{n}}$, and $n \geq 1$.  By (4), we can suppose
without loss of generality that $u \in U_{ab}^{G_{n}}$.  Let $\langle
u_{0},\ldots,u_{2p},v \rangle$ be a $(u,v)$-geodesic with $u_{0} = u$.
Then $u_{2i} \in U_{ab}^{G_{n}}$ for $0 \leq i \leq p$.  Because $v
\in \mathcal{I}_{G_{n}}(U_{ab}^{G_{n}})$, it follows that there is a
vertex $u_{2i+2} \in U_{ab}^{G_{n}}$ such that $\langle
u_{2p},v,u_{2p+2} \rangle$ is a geodesic.  Because $v \notin
I_{G_{n}}(u,u_{2p+2})$, it follows that $d_{G_{n}}(u,u_{2p}) =
d_{G_{n}}(u,u_{2p+2})$.  Then, if $u'_{2p}$ and $u'_{2p+2}$ are the
neighbors in $U_{ba}^{G_{n}}$ of $u_{2p}$ and $u_{2p+2}$,
respectively, and if $P'$ is a $(u'_{2p},u'_{2p+2})$-geodesic, then
$C = \langle u'_{2p},u_{2p},v,u_{2p+2},u'_{2p+2} \rangle \cup P'$ is a
$6$-cycle which is not contained in $G_{n-1}$ by the construction.
Denote $C$ as the cycle $\langle x_{0},\ldots,x_{5},x_{0} \rangle$
with in particular $x_{0} = v$, $x_{5} = u_{2p}$ and $x_{1} =
u_{2p+2}$.  Then, by (c), the edge $x_{0}^{C}y^C$ is
$\Theta$-equivalent to the edge $x_{1}x_{2}$, that is $u_{2p+2}u'_{2p+2}$.
Hence, by transitivity, $x_{0}^{C}y^C$ is $\Theta$-equivalent
to $ab$.  Therefore $v \in I_{G_{n+1}}(u,x_{0}^{C})$ since $G_{n}$
is a convex subgraph of $G_{n+1}$.

(6)\; Let $n \geq 1$ and $C = \langle x_{0},\ldots,x_{5},x_{0} \rangle \in \mathcal{C}_{n-1}$. Then, by (d), $U_{x_0x_{0}^{C}}^{G_n} = \{x_0,x_{2},x_{4}\}$ and $x_3 \in \mathcal{I}_{G_{n}}(U_{x_0x_{0}^{C}}^{G_{n}}) -
U_{x_0x_{0}^{C}}^{G_{n}}$, but $x_3 \notin I_{G_n}(x_0,x_2) \cup I_{G_n}(x_0,x_4)$ because $d_{G_n}(x_0,x_3) = 3$ whereas $d_{G_n}(x_0,x_2) = d_{G_n}(x_0,x_4) = 2$.\\

Now let $G := \bigcup_{n \in \mathbf{N}}G_{n}$.  Then, clearly:

\textbullet\; $G$ is a partial cube by (1) and (2).

\textbullet\; $ph(G) \leq 1$ by (5).

\textbullet\; $G_n$ is a convex subgraph of $G$ by (2) such that $ph(G_n) > 1$ if $n \geq 1$ by (6).

\textbullet\; For each edge $ab$ of $G$, any vertex $u \in \mathcal{I}_G(U_{ab})-U_{ab}$ (resp. $u \in \mathcal{I}_G(U_{ba})-U_{ba}$) lies on a $6$-cycle which belongs to $\mathbf{C}(G,ab)$, by the construction.

\textbullet\; The isometric cycles of $G$ are its $6$-cycles, and thus these cycles are convex by (3), but none of them is gated by (f).\\

Some of these properties are in fact not necessary for our present purpose, but they will be useful later (see Remark~\ref{R:gated/convex}).\\

2.\; Now we show that this graph $G$ is a counterexample of the converse of Proposition~\ref{P:fin.conv./ph1}, i.e., \emph{$ph(G) = 1$ but there exists a finite subgraph of $G$ which is not contained in a finite convex subgraph of $G$ whose pre-hull number is at most $1$}.

\begin{proof}
We already know that $ph(G) = 1$.  Suppose that there exists a finite convex subgraph $F$ of $G$ with $ph(F) = 1$ which contains $G_n$ for some $n \geq 1$.  Let $C = \langle
x_{0},\ldots,x_{5},x_{0} \rangle \in \mathcal{C}_{n-1}$ whose vertices of degree $2$ are $x_2$ and $x_4$, and let $y^{C}, 
x_{0}^{C}, \ldots, x_{5}^{C}$ be defined as in the preceding remark.  Then $\overline{C}$ is a subgraph of $G_n$.  Let $i = 1, 3, 5$.  Denote by $C_i$ the $6$-cycle of $\overline{C}$ which passes through $x_i$.  We know, by (b) and (5) above, that $x_i^{C_i}$ is the only vertex in $U_{x_0x_0^C}^{G_{n+1}}$ which is such that $x_i \in I_{G_{n+1}}(x_0,x_i^{C_i})$.  Moreover, by the construction of $G$, $x_i^{C_i}$ lies on any $(x_i,u)$-geodesic for every vertex $u \in U_{x_0x_0^C}^G$ such that $x_i \in I_{G}(x_0,u)$.  It follows that $x_i^{C_i} \in V(F)$ because $ph(F) = 1$, and consequently $\overline{C_i}$ is a subgraph of $F$ because this subgraph is convex in $G$.  This being true for every $C \in \mathcal{C}_{n-1}$, it follows that every cycle in $\mathcal{C}_n$ is contained in $F$, and hence $G_{n+1}$ is a subgraph of $F$.  We infer by induction that $F$ contains $G_i$ for every $i \geq n$, and thus $F = G$ contrary to the fact that $F$ is finite by hypothesis.
\end{proof}

\end{rem}

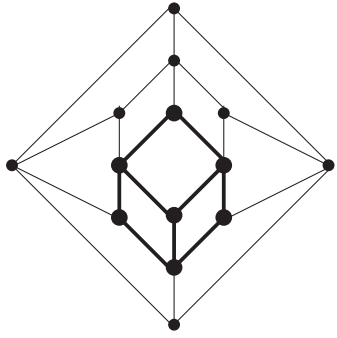
\begin{figure}[!h]
    \centering
   {\tt    \setlength{\unitlength}{0.90pt}
\begin{picture}(233,180)
\thinlines    

              \put(140,66){\line(2,1){43}}
              \put(140,111){\line(2,-1){43}}
              \put(97,66){\line(-2,1){45}}
              \put(96,111){\line(-2,-1){45}}
              \put(51,89){\line(1,-1){68}}
              \put(184,90){\line(-1,1){65}}
              \put(51,89){\line(1,1){66}}
              \put(184,87){\line(-1,-1){65}}
              \put(119,46){\line(0,-1){23}}
              \put(96,88){\line(0,1){26}}
              \put(140,90){\line(0,1){24}}
              \put(119,110){\line(0,1){26}}
              \put(119,134){\line(-1,-1){24}}
              \put(119,134){\line(1,-1){22}}
              \put(119,136){\line(0,1){20}}
              \put(119,133){\circle*{5}}
              \put(184,89){\circle*{5}}
              \put(51,89){\circle*{5}}
              \put(119,155){\circle*{5}}
              \put(119,22){\circle*{5}}
              \put(140,111){\circle*{5}}
              \put(96,111){\circle*{5}}
              \put(140,89){\circle*{7}}
              \put(96,89){\circle*{7}}
              \put(119,111){\circle*{7}}
              \put(119,68){\circle*{7}}
              \put(140,67){\circle*{7}}
              \put(96,67){\circle*{7}}
              \put(119,46){\circle*{7}}
 \linethickness{0,5mm}
              \put(140,90){\line(0,-1){23}}
              \put(96,89){\line(0,-1){22}}
              \put(119,44){\line(0,1){26}}
              \put(120,113){\line(-1,-1){24}}
              \put(119,112){\line(1,-1){22}}
              \put(97,66){\line(1,-1){22}}
              \put(141,68){\line(-1,-1){24}}
              \put(141,90){\line(-1,-1){24}}
              \put(97,88){\line(1,-1){22}}
\end{picture}}    
\caption{$M_{4,1}$ with a copy of $Q_3^-$ as a convex subgraph.}
\label{F:conv.subgr.}
\end{figure}

As was shown in \cite[Remark 8.1]{PS09}, the class of all partial cubes whose
pre-hull number is at most~$1$ is not closed under convex subgraphs.  The graph $M_{n,1}$,\; $n \geq 4$, i.e., the cube $Q_n$ from which a pair of antipodal vertices has been removed, has a pre-hull number equal to $1$ (note that this graph is  an antipodal partial cube).  On the other hand $M_{n,1}$ contains copies of $Q_{n-1}^{-}$ (the cube $Q_{n-1}$ with only one vertex deleted) as convex subgraphs (see Figure~\ref{F:conv.subgr.} for $n = 4$, where $Q_3^-$ is depicted by the big points and the thick lines), and $ph(Q_{n-1}) = 2$ by \cite[Theorem 5.8]{PS09}.   It was also shown in \cite[Remark 3.3]{P05-2} that $Q_3^-$ is a retract of $M_{4,1}$, which proves that the class of all partial cubes whose pre-hull number is at most~$1$ is not closed under retracts.  However, we will see that it is closed under gated subgraphs.

\begin{lem}\label{L:prop.isom subgr.}
Let $G$ be a partial cube, $F$ an isometric subgraph of $G$, and $ab$ an edge of $F$.  Then

\textbullet\; $W_{ab}^{F} = W_{ab}^{G} \cap V(F)$\; and\; $W_{ba}^{F} = W_{ba}^{G} \cap V(F)$

\textbullet\; $U_{ab}^{F} \subseteq U_{ab}^{G} \cap V(F)$.

If moreover $F$ is convex in $G$, then

\textbullet\; $U_{ab}^{F} = U_{ab}^{G} \cap V(F)$.
\end{lem}

\begin{proof}
The first assertions are immediate consequences of the definitions of $W_{ab}$ and $U_{ab}$, and of the fact that $F$ is isometric in $G$.  Assume now that $F$ is convex in $G$.  Let $x \in U_{ab}^{G} \cap V(F)$, and let $y$ be the neighbor of $x$ in $U_{ba}^{G}$.  Then $y \in I_{G}(x,b) = I_{F}(x,b)$ since $F$ is convex.  Hence $x \in U_{ab}^{F}$.  Therefore $U_{ab}^{F} \supseteq U_{ab}^{G} \cap V(F)$, and we are done by the above converse inclusion.
\end{proof}

\begin{lem}\label{L:gat.subgr./Uab}
Let $G$ be a partial cube, $F$ a gated subgraph of $G$, and $ab$ an edge of $F$.  Then the gate in $F$ of any $x \in U_{ab}^G$ belongs to $U_{ab}^F$.
\end{lem}

\begin{proof}
This is trivial if $x \in V(F)$.  Assume that $x \in V(G-F)$, and let $y$ be the neighbor of $x$ in $U_{ba}^G$.  Clearly, by Lemma~\ref{L:prop.isom subgr.},
\begin{gather*}
W_{ab}^F \subseteq W_{ab}^G \textnormal{\, and}\; W_{ba}^F \subseteq W_{ba}^G\\
U_{ab}^F \subseteq U_{ab}^G \textnormal{\, and}\; U_{ba}^F \subseteq U_{ba}^G
\end{gather*}
since $F$ is convex in $G$.

Denote by $g(x)$ and $g(y)$ the gates in $F$ of $x$ and $y$, respectively.  Then $g(x) \in I_G(x,a)$ and $g(y) \in I_G(y,b)$.  Hence $g(x) \in W_{ab}^F$ and $g(y) \in W_{ba}^F$.  On the other hand $y, g(x) \in I_G(x,g(y))$ and $x, g(y) \in I_G(y,g(x))$.  It easily follows that the vertices $g(x)$ and $g(y)$ are adjacent.  Therefore $g(x) \in U_{ab}^F$ and $g(y) \in U_{ba}^F$.
\end{proof}

\begin{thm}\label{T:gat.subg./ph}
Let $F$ be a gated subgraph of a partial cube $G$ such that $ph(G) \leq~1$.  Then $ph(F) \leq 1$.
\end{thm}

\begin{proof}
Let $ab$ be an edge of $F$.  By Lemma~\ref{L:prop.isom subgr.}, we have $U_{ab}^F \subseteq U_{ab}^G$ and $U_{ba}^F \subseteq U_{ba}^G$ since $F$ is convex in $G$.  We will show that $U_{ab}^{F}$ is ph-stable.

Let $x, y \in \mathcal{I}_{F}(U_{ab}^{F})$.  Because $\mathcal{I}_{F}(U_{ab}^{F}) \subseteq \mathcal{I}_{G}(U_{ab}^{G})$, and since $U_{ab}^G$ is ph-stable by Proposition~\ref{P:charact.p.c.ph}, it follows that $y \in I_G(x,z)$ for some $z \in U_{ab}^G$.  By Lemma~\ref{L:gat.subgr./Uab}, the gate $g(z)$ of $z$ in $F$ belongs to $U_{ab}^F$.  Moreover $y \in I_F(x,g(z))$ since $g(z) \in I_G(y,z)$.  Consequently $U_{ab}^F$ is ph-stable.  

In the same way we can prove that $U_{ba}^F$ is ph-stable.  We infer that $ph(F) \leq 1$ from Proposition~\ref{P:charact.p.c.ph}.
\end{proof}

\begin{thm}\label{T:gat.amalgam}
Let $G$ be the gated amalgam of two partial cubes $G_0$ and $G_1$.  Then $ph(G) \leq 1$ if and only if $ph(G_i) \leq 1$ for $i = 0, 1$.
\end{thm}

\begin{proof}
The necessity is clear by Theorem~\ref{T:gat.subg./ph} since $G_0$ and $G_1$ are isomorphic to two gated subgraphs of $G$.  Conversely, assume that $G = G_{0} \cup G_{1}$ where, for $i = 0, 1$,\; $G_{i}$ is a gated subgraph of $G$ such that $ph(G_i) \leq 1$.  The subgraph $G_{01} := G_{0} \cap G_{1}$ is also gated in $G$ as an intersection of gated subgraphs.  Let $ab$ be an edge of $G$.  We will show that $U_{ab}^G$ is ph-stable.  We distinguish two cases.

\emph{Case 1.}\; $U_{ab}^{G} = U_{ab}^{G_{i}}$ for some $i = 0$ or $1$.

Then $U_{ab}^G$ is ph-stable since so is $U_{ab}^{G_{i}}$ by Proposition~\ref{P:charact.p.c.ph}.

\emph{Case 2.}\; $U_{ab}^{G} \neq U_{ab}^{G_{i}}$ for $i = 0, 1$.

Then, for $i = 0, 1$, $G_{i}$ has an edge which is $\Theta$-equivalent
to $ab$.  Hence $G_{01}$, which is gated in $G$, also has an edge
$\Theta$-equivalent to $ab$.  Then, without loss of generality we can
suppose that $ab \in E(G_{01})$.  For any $x \in V(G)$ and $i = 0, 1$,
we denote by $g_{i}(x)$ the gate of $x$ in $G_{i}$.  Clearly 
\begin{gather}
\label{E:Wab}
W_{ab}^{G} =
W_{ab}^{G_{0}} \cup W_{ab}^{G_{1}} \textnormal{\quad and}\quad W_{ba}^{G} =
W_{ba}^{G_{0}} \cup W_{ba}^{G_{1}}\\
\label{E:Uab}
U_{ab}^{G} =
U_{ab}^{G_{0}} \cup U_{ab}^{G_{1}} \textnormal{\quad and}\quad U_{ba}^{G} = U_{ba}^{G_{0}} \cup U_{ba}^{G_{1}}\\
 \label{E:Iab}
 \mathcal{I}_{G_0}(U_{ab}^{G_0}) \cup \mathcal{I}_{G_1}(U_{ab}^{G_1}) \subseteq \mathcal{I}_{G}(U_{ab}^{G}).
\end{gather}

Let $u, v \in \mathcal{I}_{G}(U_{ab}^{G})$.  If $u, v \in \mathcal{I}_{G}(U_{ab}^{G_i})$ for some $i = 0$ or $1$, then $v \in I_{G_i}(u,w)$ for some $w \in U_{G_i}(ab)$.  Hence we are done because $v \in I_{G}(u,w)$ by (\ref{E:Iab}) and $w \in U_{G}(ab)$ by (\ref{E:Uab}).

Suppose that $u \in V(G_0)-V(G_1)$ and $v \in V(G_1)-V(G_0)$.  We first show that $u \in \mathcal{I}_{G_0}(U_{ab}^{G_0})$.  Because $u \in V(G_0)-V(G_1)$, we can suppose that $u \in I_G(x,y)$ for some $x \in U_{ab}^{G_{0}}$ and $y \in U_{ab}^{G_{1}}$.  Then $g_0(y) \in U_{ab}^{G_{0}}$ by Lemma~\ref{L:gat.subgr./Uab}, and thus $u \in I_{G_0}(x,g_0(y))$ since $g_0(y) \in I_{G_0}(u,y)$.  It follows that $g_1(u) \in I_{G_{01}}(g_1(x),g_0(y)) \subseteq \mathcal{I}_{G_1}(U_{ab}^{G_1})$.  Analogously $v \in \mathcal{I}_{G_1}(U_{ab}^{G_1})$.  Hence $v \in I_{G_1}(g_1(u),w)$ for some $w \in U_{ab}^{G_1}$ because $U_{ab}^{G_1}$ is ph-stable by Proposition~\ref{P:charact.p.c.ph}.  We infer that $v \in I_G(u,w)$, which proves that $U_{ab}^G$ is ph-stable.

In the same way we can prove that $U_{ba}^G$ is ph-stable.  Consequently $ph(G) \leq 1$ by Proposition~\ref{P:charact.p.c.ph}.
\end{proof}

\begin{thm}\label{T:cart.prod.}
Let $G = G_0 \Box G_1$ be the Cartesian product of two partial cubes $G_0$ and $G_1$.  Then $ph(G) \leq 1$ if and only if $ph(G_i) \leq 1$ for $i = 0, 1$.
\end{thm}

\begin{proof}
Assume that $ph(G) \leq 1$.  Let $F_i$ be a $G_i$-layer of $G$ for some $i = 0$ or $1$.  Then $F_i$ is a gated subgraph of $G$.  Indeed, by the Distance Property of the Cartesian product, the projection onto $F_i$ of any vertex $x$ of $G$ is the gate of $x$ in $F_i$.  Therefore, by Theorem~\ref{T:gat.subg./ph}, $F_i$, and thus $G_i$, has a pre-hull number which is at most $1$.

Conversely, assume that $ph(G_i) \leq 1$ for $i = 0, 1$.  For any $x \in V(G)$, we denote by $x_0$ and $x_1$ the projections of $x$ onto $G_0$ and $G_1$, respectively, i.e., $x = (x_0,x_1)$.  Let $ab \in E(G)$.  Then $a_i = b_i$ for exactly one $i$, say $i = 1$.  We will show that $U_{ab}^G$ is ph-stable.

Clearly, any $cd$ of $G$ is $\Theta$-equivalent to $ab$ if and only if $c_1 = d_1$ and $c_0d_0$ is $\Theta$-equivalent to $a_0b_0$.  Hence
\begin{equation}
\label{E:U'ab}
U_{ab}^G = U_{a_0b_0}^{G_0} \times V(G_1).
\end{equation}

Let $u, v \in \mathcal{I}_G(U_{ab}^G)$.  By the Interval Property of the Cartesian product, $u_0, v_0 \in \mathcal{I}_{G_0}(U_{a_0b_0}^{G_0})$.  Then, because $U_{a_0b_0}^{G_0}$ is ph-stable by Proposition~\ref{P:charact.p.c.ph}, it follows that $v_0 \in I_{G_0}(u_0,w_0)$ for some $w_0 \in U_{a_0b_0}^{G_0}$.  In the case where $u_0 = v_0$, we can choose $w_0$ as any element of $U_{a_0b_0}^{G_0}$.  Let $w := (w_0,v_1)$.  Then $w \in U_{ab}^G$ by (\ref{E:U'ab}), and $v \in I_G(u,w)$ by the Distance Property of the Cartesian product.  This proves that $U_{ab}^G$ is ph-stable.

In the same way we can prove that $U_{ba}^G$ is ph-stable.  Consequently $ph(G) \leq 1$ by Proposition~\ref{P:charact.p.c.ph}.
\end{proof}

According to the above theorems we infer the following result:

\begin{cor}\label{C:closure}
The class of all partial cubes whose pre-hull number is at most~$1$ is closed under gated  subgraphs, gated amalgams and Cartesian products.
\end{cor}

\subsection{Ph-homogeneous partial cubes}\label{SS:ph-homog.}

In this subsection we introduce the main concept of this study.

\begin{defn}\label{D:ph-homog.}
A graph $G$ is said to be \emph{ph-homogeneous} if any finite convex subgraph of $G$ has a pre-hull number which is at most $1$.
\end{defn}

Because, by~\cite[Proposition 4.4]{P05-2}, any convex subgraph of a netlike partial cube is also a netlike partial cube, it follows that netlike partial cubes, and thus median graphs, trees, cellular bipartite graphs, benzenoid graphs are ph-homogeneous partial cubes.

\begin{pro}\label{P:ph-homog.bip.gr.=p.c.}
Any ph-homogeneous bipartite graph is a partial cube.
\end{pro}

\begin{proof}
Let $G$ be ph-homogeneous bipartite graph. By Theorem~\ref{T:bip.gr.ph1=>p.c.}, any finite convex subgraph of $G$ is a partial cube.  Hence $G$ is itself a partial cube by Lemma~\ref{L:inf.p.c.}.
\end{proof}

Note that, as an immediate consequence of Propositions~\ref{P:charact.p.c.ph} and ~\ref{P:fin.conv./ph1}, we have:

\begin{pro}\label{P:hypernet./ph1}
Let $G$ be a ph-homogeneous partial cube.  Then $ph(G) \leq 1$, and thus $U_{ab}$ and $U_{ba}$ are ph-stable for every edge $ab$ of $G$.
\end{pro}

According to the fact that $ph(G) \leq 1$ and by Lemma~\ref{L:ph-stable}, we deduce that $co_{G}(U_{ab}) = \mathcal{I}_{G}(U_{ab})$ for every edge $ab$ of a ph-homogeneous partial cube $G$.

\begin{thm}\label{T:hypernet./gat.amalg.+cart.prod.}
The class of ph-homogeneous partial cubes is closed under convex subgraphs, gated amalgams and finite or infinite Cartesian products.
\end{thm}

\begin{proof}
\emph{Convex subgraph}:\; Any convex subgraph of a ph-homogeneous partial cube is obviously a ph-homogeneous partial cube, whence the result.

\emph{Gated amalgam}:\; Assume that $G = G_0 \cup G_1$, where $G_0$ and $G_1$ are ph-homogeneous partial cubes that are gated subgraphs of $G$.  Let $F$ be a finite convex subgraph of $G$. If $F$ is a subgraph of $G_i$ for some $i = 0$ or $1$, then $ph(F) \leq 1$ since $G_i$ is ph-homogeneous.  Suppose that $F_i := F \cap G_i \neq \emptyset$ for $i = 0, 1$.  Then, for any $i = 0$ or $1$, $F_i$ is convex in $G_i$, and moreover the gate $g_{1-i}(x)$ of each $x \in V(F_i)$ in $G_{1-i}$ belongs to $F$, because $g_{1-i}(x) \in I_G(x,y)$ for every $y \in V(F_{1-i})$ and $F$ is convex.  It follows that $F$ is the gated amalgam of $F_0$ and $F_1$.  Therefore $ph(F) \leq 1$ by Theorem~\ref{T:gat.amalgam}, since $ph(F_i) \leq 1$ for $i = 0, 1$ by what we saw above.  Consequently $G$ is a ph-homogeneous partial cube.

\emph{Cartesian product}:\; Let $G = \cartbig_{ \in I}G_i$ be the (weak) Cartesian product of a family of ph-homogeneous partial cubes, and $F$ a finite convex subgraph of $G$.  Because $F$ is finite, the set $J := \{j \in I: \vert pr_i(F)\vert > 1\}$ is finite, $pr_i(F)$ is convex for all $i \in I$, and $F = \cartbig_{i \in I}pr_i(F)$.  It follows that $F$ is isomorphic to $F' := \cartbig_{j \in J}pr_j(F)$.  For all $j \in J$,\; $ph(pr_j(F)) \leq 1$ since $G_j$ is a ph-homogeneous partial cube and $pr_j(F)$ is a finite convex subgraph of $G_j$.  Therefore $ph(F') \leq 1$ by Theorem~\ref{T:cart.prod.}, and thus $ph(F) \leq 1$.  Consequently $G$ is a ph-homogeneous partial cube.
\end{proof}

\begin{cor}\label{C:ph-hom./cart.prod.}
The Cartesian product $G_0 \Box G_1$ of two partial cubes $G_0$ and $G_1$ is ph-homogeneous if and only if so are $G_0$ and $G_1$.
\end{cor}

\begin{proof}
The sufficiency is a consequence of the above theorem.  Conversely suppose that $G := G_0 \Box G_1$ is ph-homogeneous.  Let $F_i$ be a $G_i$-layer of $G$ for some $i = 0$ or $1$.Then $F_i$ is a convex subgraph of $G$.  Hence $F_i$, and thus $G_i$, is ph-homogeneous by Theorem~\ref{T:hypernet./gat.amalg.+cart.prod.}.
\end{proof}

Recall that a partially ordered set $A$ is down-directed if any pair of elements of $A$ has a lower bound.

\begin{thm}\label{T:inters.inf.seq.}
If $\Gamma$ is a set of ph-homogeneous partial cubes which is down-directed for the subgraph relation, then $\bigcap\Gamma$ is ph-homogeneous.
\end{thm}

\begin{proof}
Let $F$ be a finite convex subgraph of $\bigcap\Gamma$.  Note that the set $\{co_G(F): G \in \Gamma\}$ is also down-directed for the subgraph relation, since $co_G(F) \subseteq co_{G'}(F)$ if $G \subseteq G'$.

\emph{Claim.  $\bigcap_{G \in \Gamma}co_G(F) = F$.}

This is clear if $\Gamma$ has a minimal element, say $G$, because a down-directed ordered set may have at most one minimal element, and thus $\bigcap\Gamma = G$.

Assume that $\Gamma$ has no minimal element, which in particular implies that $\Gamma$ is infinite.  It suffices to prove that $\bigcap_{G \in \Gamma}\mathcal{I}_G(F) = F$.  Let $x \in V(\bigcap_{G \in \Gamma}\mathcal{I}_G(F))$.  Suppose that $x \notin V(F)$, and let $G \in \Gamma$.  The set $L_G$ of geodesics in $G$ which pass through $x$ and join two vertices of $F$ is finite, because $F$ is finite and each interval of a partial cube is finite.  Let $P \in L_G$.  Because $\Gamma$ has no minimal element and $x \notin V(F)$, there exists $G' \in \Gamma$ such that $P \notin L_{G'}$.  Then, since $\Gamma$ is down-directed, there exists $G_P \in \Gamma$ such that $G_P$ is a subgraph of $G$ and $G'$.  It follows that $P \notin L_{G_P}$.  Because $L_G$ is finite, by repeating the same argument we can show that there exists $G_x \in \Gamma$ which is a subgraph of $G$ such that $L_{G_x}$ contains no element of $L_G$.  Therefore $x \notin \mathcal{I}_{G_x}(F)$, contrary to the hypothesis.  Hence $x \in V(F)$, which proves the claim.\\

According to the claim and the facts that the set $\{co_G(F): G \in \Gamma\}$ is down-directed and that $co_G(F)$ is finite by Lemma~\ref{L:gen.propert.}(iii) for every $G \in \Gamma$, we infer that $F = co_{G}(F)$ for some $G \in \Gamma$.  Hence $F$ is convex in $G$, and thus $ph(F) \leq 1$ since $G$ is ph-homogeneous.  Therefore $\bigcap\Gamma$ is ph-homogeneous.
\end{proof}

\subsection{Convexity properties}\label{SS:conv.prop.}

In order to state a characterization of ph-homogeneous partial cubes given in~\cite{P16}, we recall four specific properties of an abstract interval space $(X,I)$ and of a convex structure $(X,\mathcal{C})$ (see~\cite{V93}):

\emph{Peano Property}:\; For all $u,v,w \in X$, $x \in I(u,w)$ and
$y \in I(v,x)$, there exists a point $z \in I(v,w)$ such that $y \in
I(u,z)$.

\emph{Pash Property}:\; For all $u,v,w \in X$, $v' \in I(u,w)$ and
$w' \in I(u,v)$, the intervals $I(v,v')$ and $I(w,w')$ are 
non-disjoint.

\emph{Join-Hull Commutativity Property}:\; For any convex set $C
\subseteq X$ and any $u \in X$, the convex hull of $\{u\} \cup C$
equals the union of the convex hull of $\{u,v\}$ for all $v \in C$.

\emph{Kakutani Separation Property $\mathrm{S}_{4}$}:\; If $C, D \subseteq X$ are
disjoint convex sets, then there is a half-space $H$ which separates
$C$ from $D$, that is, $C \subseteq H$ and $D \subseteq
X-H$.

By \cite[Theorem 4.11]{V93}, a convex structure $(X,\mathcal{C})$ induced by an interval operator $I$ is join-hull commutative if and only if the interval space $(X,I)$ has the Peano Property.  Moreover, by \cite{Che86-1}, a convex structure of arity $2$ has the Pash Property if and only if it has the separation property $\mathrm{S}_{4}$.  Furthermore, according to Chepoi~\cite{Che86},
the geodesic interval space of a bipartite graph which is join-hull commutative and has convex intervals also have the separation property $\mathrm{S}_{4}$.  An interval space satisfying
the Pash and Peano properties is called a \emph{Pash-Peano space}.

From now on, we will say that a graph has one of the above properties if its geodesic interval space or its geodesic convex structure has this property.

If all intervals of some graph $G$ are convex, which is the case if $G$ is a partial cube, and if $G$ is join-hull commutative, then $co_G(\{u\} \cup C) = \mathcal{I}_G(\{u\} \cup C$ for all vertex $u$ and convex set $C$ of $G$.  This is in particular the case if $C$ is a copoint at $u$.  It follows that such a graph $G$ has a pre-hull number which is at most $1$.

We recall that an interval monotone bipartite graph is not necessarily a partial cube as is shown by the graph in Figure~\ref{F:mon.bip.gr.} (cf.~\cite{BK02}), but that a bipartite graph is a partial cube if its pre-hull number is at most $1$ (Theorem~\ref{T:bip.gr.ph1=>p.c.}).

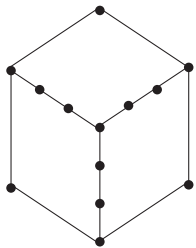
\begin{figure}[!h]
    \centering
    {\tt    \setlength{\unitlength}{1,20pt}
\begin{picture}(173,124)
\thinlines 
              \put(72,18){\circle*{3}}
              \put(72,91){\circle*{3}}
              \put(72,54){\circle*{3}}
              \put(72,30){\circle*{3}}
              \put(72,42){\circle*{3}}
              \put(72,55){\line(0,-1){36}}
              \put(44,71){\line(0,-1){35}}
              \put(44,72){\circle*{3}}
              \put(44,35){\circle*{3}}
              \put(73,91){\line(3,-2){27}}
              \put(72,92){\line(-3,-2){28}}
              \put(44,36){\line(3,-2){27}}
              \put(44,72){\line(3,-2){27}}
              \put(101,74){\line(-3,-2){29}}
              \put(100,72){\line(0,-1){35}}
              \put(100,73){\circle*{3}}
              \put(100,36){\circle*{3}}
              \put(81,61){\circle*{3}}
              \put(90,66){\circle*{3}}
              \put(53,66){\circle*{3}}
              \put(62,60){\circle*{3}}
              \put(100,37){\line(-3,-2){28}}	     
\end{picture}}    
\caption{A monotone bipartite graph which is not a partial cube.}
\label{F:mon.bip.gr.}
\end{figure}

Note that a partial cube may have the separation property $\mathrm{S}_{4}$ but not the Join-Hull Commutativity Property.  This is for example the case of $Q_3^-$, which has a pre-hull number equal to $2$ and clearly has the separation property $\mathrm{S}_{4}$.  On the other hand a partial cube with pre-hull number at most $1$ may not have the separation property $\mathrm{S}_{4}$, and thus not the Join-Hull Commutativity Property.  Take for example $M_{4,1}$ (see Figure~\ref{F:conv.subgr.}).  As we saw, $ph(M_{4,1}) = 1$.  Moreover its edge set contains exactly four $\Theta$-classes, and we can easily find two disjoint convex paths of length $2$ such that each $\Theta$-class has a representative edge in one of these paths.  This implies that the vertex sets of these paths cannot be separated by a half-space, and hence that $M_{4,1}$ has not the separation property $\mathrm{S}_{4}$.  However we have the following result.

\begin{thm}\label{T:Peano/ph}\textnormal{(Polat~\cite[Theorem 3.4]{P16})}
Let $G$ be an interval monotone bipartite graph.  The following assertions are equivalent:

\textnormal{(i)}\; $G$ has the Peano Property.

\textnormal{(ii)}\; $G$ is join-hull commutative.

\textnormal{(iii)}\; $G$ is a ph-homogeneous partial cube.

\textnormal{(iv)}\; $G$ is a partial cube all of whose intervals induce subgraphs with pre-hull number at most $1$.

\textnormal{(v)}\; $G$ has the separation property $\mathrm{S}_{4}$ and $ph(G) \leq 1$.

\textnormal{(vi)}\; $G$ has the Pash Property and $ph(G) \leq 1$.
\end{thm}

Due to the equivalence of (iii) and (iv) we have the following property:

\begin{cor}\label{C:ph-homog.p.c./interv.}
A partial cube is ph-homogeneous if an only if the pre-hull number of the subgraphs induced by each of its intervals is at most $1$.
\end{cor}

From the above theorem, we infer that the ph-homogeneous partial cubes are the Pash-Peano partial cubes, and moreover that the Peano Property is stronger than the Pash Property for partial cubes.  This is why we choose to call the partial cubes that are ph-homogeneous the \emph{Peano partial cubes}.

\subsection{Median graphs and netlike partial cubes}\label{SS:applications}

Theorem~\ref{T:Peano/ph} is useful to prove particular properties for some special partial cubes with pre-hull number at most $1$, such as median graphs and netlike partial cubes.  We first enlarge the long list of characterizations of median graphs by adding new ones to the result \cite[Theorem 7.7]{PS09}.  We recall that a graph $G$ is \emph{modular} if $\bigcap_{1 \leq i < j \leq 3}I_G(x_i,x_j) \neq \emptyset$ for every triple $(x_1,x_2,x_3)$ of vertices of $G$.  Each element of this intersection is called a \emph{median} of $G$.  A \emph{median graph} is a modular graph in which every triple of vertices has a unique median.

\begin{thm}\label{T:mod.gr./ph=1/Peano}
Let $G$ be a connected modular graph.  Then the following assertions 
are equivalent:

\textnormal{(i)}\; $G$ is a median graph.

\textnormal{(ii)}\; $ph(G) \leq 1$.

\textnormal{(iii)}\; The restriction of the relation $\Theta$ to the
edge-boundary of any convex set of $G$ is transitive.

\textnormal{(iv)}\; $G$ is join-hull commutative.

\textnormal{(v)}\; $G$ has the Peano Property.

\textnormal{(vi)}\; $G$ has the Pash Property.

\textnormal{(vii)}\; $G$ has the separation property $\mathrm{S}_{4}$.

\textnormal{(viii)}\; All intervals of $G$ are convex.
\end{thm}

\begin{proof}
The equivalences of the assertions (i) -- (iii) are those of \cite[Theorem 7.7]{PS09}, the equivalences of the assertions (v), (vi) and (viii) are consequences of \cite[Theorem 6.10]{V93},  the equivalence of (vi) and (vii) was proved in \cite{Che86-1}, and the equivalence of (iv) and (v) is \cite[Theorem 4.11]{V93}.

(i) $\Rightarrow$ (iv):\; If $G$ is a median graph, then so is any of its convex subgraphs.  Hence, by the equivalence of (iii) and (i), the pre-hull number of any convex subgraph of $G$ is at most $1$.  Therefore $G$ is join-hull commutative by Theorem~\ref{T:Peano/ph}.

(iv) $\Rightarrow$ (ii):\; Assume that $G$ is join-hull commutative.  By the equivalence of (iv) and (viii), all intervals of $G$ are convex.  Hence the Join-Hull Commutativity Property implies that $ph(G) \leq 1$.
\end{proof}

By~\cite{B82}, a partial cube $G$ is a median graph if, for every edge $ab$ of $G$, the sets $U_{ab}$ and $U_{ba}$ are convex.  More generally, a partial cube $G$ is a \emph{netlike partial cube}~\cite{P05-1} if, for every edge $ab$ of $G$, any cycle of $G[co_G(U_{ab})]$ (resp. $G[co_G(U_{ba})]$) is a cycle of $G[U_{ab}]$ (resp. $G[U_{ba}]$).

The class of netlike partial cubes contains median graphs, even cycles, cellular bipartite graphs and benzenoid graphs as particular instances.  We recall some characterizations of netlike partial cubes.

\begin{pro}\label{P:charact.(netlike)}\textnormal{(Polat~\cite[Theorems
3.8 and 3.10]{P05-1})} Let $G$ be a partial cube.  Then $G$ is netlike if and only if it satisfies any of
the pairs \textnormal{(i)(ii)} or \textnormal{(ii)(iii)} of the following three properties:

\textnormal{(i)}\; $ph(G) \leq 1$.

\textnormal{(ii)}\; For any edge $ab$ of $G$, each vertex in 
$co_{G}(U_{ab})-U_{ab}$ has degree $2$ in $G[co_{G}(U_{ab})]$.

\textnormal{(iii)}\; The convex hull of each non-convex isometric cycle
of $G$ is a hypercube.
\end{pro}

By \cite[Proposition 4.4]{P05-2}, any convex subgraph of a netlike partial cube is also netlike, and thus has a pre-hull number which is at most $1$ by Proposition~\ref{P:charact.(netlike)}(i).  Hence a netlike partial cube is a Peano partial cube.  Consequently we have:

\begin{thm}\label{T:netlike/JHC}
Any netlike partial cube is join-hull commutative, and thus has the Peano and Pash Properties, and the separation property $\mathrm{S}_{4}$.
\end{thm}

This theorem generalizes \cite[Theorem 4.1]{P07-3}, where this property was proved only for the particular netlike partial cubes that were said to have the \textquotedblleft Median Cycle Property\textquotedblright.  We will use this result to obtain the Helly number of a netlike partial cube.

\subsection{Elementary Peano partial cubes}\label{SS:elem.ph-homog.}

We denote by
$\mathbf{C}$ the class of even cycles of length greater than $4$.  The
Cartesian product of a finite family of even cycles such that
\emph{the length of at least one of them is greater than $4$} is
called a \emph{hypertorus}.  In particular an even cycle of length greater than $4$
is a hypertorus, and more precisely a $1$-torus, whereas a $4$-cycle
and more generally a $2n$-cube will not be considered as a hypertorus.
By the \emph{prism over a hypertorus} we mean
the Cartesian product of a hypertorus with $K_{2}$.  If $n$ is a positive integer, we call 
an hypertorus which is the Cartesian product of $n$ cycles an \emph{$n$-torus}, and
a prism over a $n$-torus an \emph{$n$-prism}.  Furthermore, by a \emph{quasi-hypertorus} we mean either a hypercube or a hypertorus or the prism over a hypertorus, that is, the Cartesian products of two-vertex complete graphs and even cycles.  We denote by $\mathbf{Tor}$ the class of all quasi-hypertori.  
Finally, the Cartesian product of an even cycle of length greater than $4$ with a path is called a \emph{cylinder}, and more generally the Cartesian product of an even cycle of length greater than $4$ with a Peano partial cube is called a \emph{hypercylinder}, and the class of all hypercylinders is denoted by $\mathbf{Cyl}$.

Quasi-hypertori are particular hypercylinders.  Because of Theorem~\ref{T:hypernet./gat.amalg.+cart.prod.} and the fact that even cycles are ph-homogeneous, we infer that \emph{quasi-hypercylinders are Peano partial cubes}.

\section{Characteristic and general properties}\label{S:charac.prop.}

Any properties of Theorem~\ref{T:Peano/ph} characterizes Peano partial cubes, but none of them gives some useful piece of information on the structure of these graphs.  In this section we give six structural characterizations of Peano partial cubes (Theorems~\ref{T:charact.}) which will be essential to obtain most of the properties of these graphs  in the subsequent sections.\\

\subsection{The characterization theorem}\label{SS:charact.thm.}

If $A$ is a set of vertices of a graph $G$, then, by an \emph{$A$-path} of $G$ we mean a path of $G$ of length at least $2$ joining two vertices in $A$ and with no inner vertex in $A$.

\begin{defn}\label{D:strong.ph-stable}
Let $ab$ be a edge of some partial cube $G$.  Then $U_{ab}$ is said to be \emph{strongly ph-stable} if, for any vertex $u \in \mathcal{I}_G(U_{ab})-U_{ab}$, there exists a convex $U_{ab}$-path $P_u$ which passes through $u$ and which satisfies the following two properties:

(SPS1)\; For every $x \in \mathcal{I}_{G}(U_{ab})$,\; $u \in I_G(x,v)$ for some endvertex $v$ of $P_u$.

(SPS2)\; For all vertices $x, y \in U_{ab}$ such that $u \in I_G(x,y)$,\; $P_u$ is a subpath of some $(x,y)$-geodesic.
\end{defn}

By (SPS1), $U_{ab}$ is ph-stable if it is strongly ph-stable.  The converse is clearly not true.  However we have the following results.

\begin{lem}\label{L:SPS2=>SPS1}
Let $G$ be a partial cube, $ab$ an edge of $G$ such that $U_{ab}$ is ph-stable, and $P_u$ a convex $U_{ab}$-geodesic which passes through a given vertex $u \in \mathcal{I}_G(U_{ab})-U_{ab}$.  If $P_u$ satisfies \textnormal{(SPS2)}, then it also satisfies \textnormal{(SPS1)}.
\end{lem}

\begin{proof}
Let $x \in \mathcal{I}_G(U_{ab})$.  Because $U_{ab}$ is ph-stable, $x \in I_G(u,y)$ for some $y \in U_{ab}$, and also $u \in I_G(y,z)$ for some $z \in U_{ab}$.  Then, by (SPS2), there exists a $(y,z)$-geodesic $Q$ which contains $P_u$ as a subpath.  Let $v$ be the endvertex of $P_u$ which lies in $Q[u,z]$, and let $R$ be a $(y,u)$-geodesic passing through $x$.  Then, by Lemma~\ref{L:gen.propert.}(iv,v), $R \cup Q[u,v]$ is an $(x,v)$-geodesic which passes through $u$, and thus $u \in I_G(x,v)$.
\end{proof}

We obtain immediately:

\begin{cor}\label{C:ph-stab.+SPS2=>strong.ph-stab.}
Let $G$ be a partial cube, $ab$ an edge of $G$ such that $U_{ab}$ is ph-stable.  Then $U_{ab}$ is strongly ph-stable if and only if, for any vertex $u \in \mathcal{I}_G(U_{ab})-U_{ab}$, there exists a convex $U_{ab}$-path $P_u$ which passes through $u$ and which satisfies \textnormal{(SPS2)}.
\end{cor}

Let $G$ be a partial cube, and $\mathbf{H}$ a class of graphs.  Then any convex subgraph of $G$ that belongs to $\mathbf{H}$
is called an \emph{$\mathbf{H}$-subgraph}
 of $G$, and the set of all $\mathbf{H}$-subgraphs of $G$ is denoted by $\mathbf{H}(G)$.  Furthermore, for
any edge $ab$ of $G$, we denote by $\mathbf{H}(G,ab)$ the set of all
$\mathbf{H}$-subgraphs of $G$ that have an edge $\Theta$-equivalent
to $ab$.

\begin{defn}\label{D:bulge}
Let $ab$ be an edge of a partial cube $G$, and $A$ a
component of $G[co_G(U_{ab})-U_{ab}]$.  Then the subgraph of $G$
induced by $N_{G[co_G(U_{ab})]}[V(A)]$ is called a \emph{bulge} of $co_G(U_{ab})$.
\end{defn}

\begin{thm}[\textbf{Characterization Theorem}]\label{T:charact.}
Let $G$ be a partial cube.  The following assertions are equivalent:

\textnormal{(i)}\; $G$ is ph-homogeneous.

\textnormal{(ii)}\; $U_{ab}$ and $U_{ba}$ are strongly ph-stable for every edge $ab$ of $G$.

\textnormal{(iii)}\; For each edge $ab$ of $G$ and any bulge $X$ of $co_{G}(U_{ab})$ (resp. $co_G(U_{ba})$), we have the following two properties:

\hspace{5pt} \textnormal{(HNB1)}\; There exists a convex $H \in \mathbf{Cyl}(G,ab)$ such that $X = H-W_{ba}$ (resp. $X = H-W_{ab}$).

\hspace{5pt} \textnormal{(HNB2)}\; $X-U_{ab}$ (resp. $X-U_{ba}$) is a separator of $G_{ab} := G[co_{G}(U_{ab})]$ (resp. $G_{ba}$).

\textnormal{(iv)}\; For each edge $ab$ of $G$ and any bulge $X$ of $co_{G}(U_{ab})$ (resp. $co_G(U_{ba})$), there exists a gated $H \in \mathbf{Cyl}(G,ab)$ such that $X =
H-W_{ba}$ (resp. $X = H-W_{ab}$).

\textnormal{(v)}\; For each edge $ab$ of $G$, any vertex $u \in \mathcal{I}_G(U_{ab})-U_{ab}$ (resp. $u \in \mathcal{I}_G(U_{ba})-U_{ba}$) lies on a gated cycle $C_u \in \mathbf{C}(G,ab)$.

\textnormal{(vi)}\; For each edge $ab$ of $G$, any vertex $u \in \mathcal{I}_G(U_{ab})-U_{ab}$ (resp. $u \in \mathcal{I}_G(U_{ba})-U_{ba}$) lies on an isometric cycle $C_u \in \mathbf{C}(G,ab)$, and the convex hull of any isometric cycle of a $G$ is a gated quasi-hypertorus.
\end{thm}

Note the analogy between Proposition~\ref{P:charact.p.c.ph} and the equivalence of the assertions (i) and (ii) of the above theorem.

To prove this theorem, we need a lot of secondary results.

\subsection{Expansion and $\Theta$-contraction}\label{SS:expansion}

To prove the implication (i)$\Rightarrow$(ii) of Theorem~\ref{T:charact.} we need some basic properties of an expansion and of a $\Theta$-contraction of a graph, a concept which was introduced by Mulder~\cite{M78} to
characterize median graphs and which was later generalized by
Chepoi~\cite{Che88}.

\begin{defn}\label{D:proper cover} 
A pair $(V_{0},V_{1})$ of sets of vertices of a graph $G$ is called a
\emph{proper cover} of $G$ if it satisfies the following conditions:

\textbullet\;  $V_{0} \cap V_{1} \neq \emptyset$ and $V_{0} \cup V_{1} = 
V(G)$;

\textbullet\;  there is no edge between a vertex in $V_{0}-V_{1}$ and 
a vertex in $V_{1}-V_{0}$;

\textbullet\;  $G[V_{0}]$ and $G[V_{1}]$ are isometric subgraphs of 
$G$.
\end{defn}

\begin{defn}\label{D:expansion}
An \emph{expansion} of a graph $G$ with respect to a proper cover 
$(V_{0},V_{1})$ of $G$ is the subgraph  of $G \Box K_{2}$ induced by
the vertex set $(V_{0} \times \{0 \}) \cup (V_{1} \times \{1 \})$ (where
$\{0,1 \}$ is the vertex set of $K_{2}$).
\end{defn}

An expansion of a partial cube is a partial cube (see~\cite{Che88}).  If $G'$ is an
expansion of a partial cube $G$, then we say that $G$ is
a \emph{$\Theta$-contraction} of $G'$, because, as we can easily
see, $G$ is obtained from $G'$ by contracting each element of
some $\Theta$-class of edges of $G'$.  More precisely, let $G$ be a
partial cube different from $K_{1}$ and let $uv$ be an edge of $G$.
Let $G/uv$ be the quotient graph of $G$ whose vertex set $V(G/uv)$ is
the partition of $V(G)$ such that $x$ and $y$ belong to the same block
of this partition if and only if $x = y$ or $xy$ is an edge which is
$\Theta$-equivalent to $uv$.  The natural surjection $\gamma_{uv}$
of $V(G)$ onto $V(G/uv)$ is a contraction (weak homomorphism
in~\cite{HIK11}) of $G$ onto $G/uv$, that is, an application which maps
any two adjacent vertices to adjacent vertices or to a single vertex.
Then clearly the graph $G/uv$ is a partial cube and
$(\gamma_{uv}(W_{uv}^{G}),\gamma_{uv}(W_{vu}^{G}))$ is a proper cover
of $G/uv$ with respect to which $G$ is an expansion of $G/uv$.  We
will say that $G/uv$ is the \emph{$\Theta$-contraction of $G$ with
respect to the $\Theta$-class of $uv$}.

Let $G'$ be an expansion of a graph $G$ with respect to a
proper cover $(V_{0},V_{1})$ of $G$.  We will use the following
notation.

\textbullet\; For $i = 0, 1$ denote by $\psi_{i} : V_{i} \to V(G')$
the natural injection $\psi_{i} : x \mapsto (x,i)$, $x \in V_{i}$, and
let $V'_{i} := \psi_{i}(V_{i})$.  Note that $V'_{0}$ and $V'_{1}$ are
complementary half-spaces of $G'$.  It follows in particular that
these sets are copoints of $G'$.

\textbullet\; For any vertex $x$ of $G$ (resp.  $G'$),
denote by $i(x)$ an element of $\{0,1 \}$ such that $x$ belongs to
$V_{i(x)}$ (resp.  $V'_{i(x)}$).  If $x \in V(G')$ and also if $x \in 
V(G) - (V_{0} \cap V_{1})$, then $i(x)$ is unique; if $x \in V_{0} \cap
V_{1}$ it may be $0$ or $1$.

\textbullet\; For $A \subseteq V(G)$ put $$\psi(A) := \psi_{0}(A \cap V_{0})
\cup \psi_{1}(A \cap V_{1}).$$  Note that in the opposite direction we
have that for any $A' \subseteq V(G')$, $$\textnormal{pr}(A') =
\psi_{0}^{-1}(A' \cap V'_{0}) \cup
\psi_{1}^{-1}(A' \cap V'_{1}),$$ where $\textnormal{pr} : G \Box K_{2}
\to G$ is the projection $(x,i) \mapsto x$.

The following lemma is a restatement with more precisions of
\cite[Lemma 4.5]{P05-1}.

\begin{lem}\label{L:interv.}
Let $G$ be a connected bipartite graph and $G'$ an expansion of
$G$ with respect to a proper cover $(V_{0},V_{1})$ of
$G$, and let $P = \langle x_{0},\ldots,x_{n}\rangle$ be a path in 
$G$.  We have the following properties:

\textnormal{(i)}\; If $x_{0}, x_{n} \in V_{i}$ for some $i = 0$ or $1$, then:

\textbullet\;  if $P$ is a geodesic in $G$, then there exists an $(x_{0},x_{n})$-geodesic $R$ in $G[V_i]$ such that $V(P) \cap V_i \subseteq V(R)$;

\textbullet\;  $P$ is a geodesic in $G[V_i]$ if and only if $P' = \langle \psi_{i}(x_{0}),\ldots,\psi_{i}(x_{n}) \rangle$ is
a geodesic in $G'$;  

\textbullet\;  $d_{G'}(\psi_{i}(x_{0}),\psi_{i}(x_{n}))
= d_{G}(x_{0},x_{n})$;

\textbullet\;  $I_{G'}(\psi_{i}(x_{0}),\psi_{i}(x_{n})) =
\psi_{i}(I_{G[V_i]}(x_{0},x_{n})) \subseteq \psi(I_{G}(x_{0},x_{n}))$.

\textnormal{(ii)}\; If $x_0 \in V_i$ and $x_1 \in V_{1-i}$ for some $i = 0$ or $1$, then:

\textbullet\;  if there exists $p$ such that $x_0,\ldots,x_p \in V_i$ and $x_p,\ldots,x_n \in V_{1-i}$, then $P$ is a geodesic in $G$ if and only if the path $$P' = \langle
\psi_{i}(x_{0}),\ldots,\psi_{i}(x_{p}),\psi_{1-i}(x_{p}),\ldots,\psi_{1-i}(x_{n})
\rangle$$ is a geodesic in $G'$;

\textbullet\;  $d_{G'}(\psi_{i}(x_{0}),\psi_{1-i}(x_{n}))
= d_{G}(x_{0},x_{n}) + 1$;

\textbullet\;  $I_{G'}(\psi_{i}(x_{0}),\psi_{1-i}(x_{n})) =
\psi(I_{G}(x_{0},x_{n}))$.
\end{lem}

The following result is an immediate consequence of Lemma~\ref{L:interv.}.

\begin{cor}\label{C:conv.G0/conv.G1}
Let $K$ be a convex set of connected bipartite graph $G$,  and $G'$ an expansion of $G$.  Then $\psi(K)$ is a convex set of 
$G'$.
\end{cor}

\begin{lem}\label{L:Theta-contraction/ph-stable}\textnormal{(Polat~\cite[Lemma
4.8]{P05-1})} Let $ab$ be an edge of a finite partial cube $G$ such
that $U_{ab}$ is ph-stable.  Let $cd$ be an edge which is $\Theta$-equivalent to an edge of $G[\mathcal{I}_G(U_{ab})]$, and let $G' := G/cd$ be the
$\Theta$-contraction of $G$ with respect to the $\Theta$-class of
$cd$, and $\gamma_{cd}$ the natural surjective contraction of $G$ onto
$G'$.  Then $\mathcal{I}_{G'}(U_{a'b'}^{G'}) =
\gamma_{cd}(\mathcal{I}_{G}(U_{ab}^{G}))$, where $a' :=
\gamma_{cd}(a)$ and $b' := \gamma_{cd}(b)$, and $U_{a'b'}^{G'}$ is
ph-stable.
\end{lem}

\begin{lem}\label{L:expans./theta}
Let $G$ be a partial cube and $G'$ an expansion of
$G$ with respect to a proper cover $(V_{0},V_{1})$ of
$G$.  Let $u_{0}u_{1}$ and $v_{0}v_{1}$ be two edges of $G$.  For $j = 0, 1$, let $i(u), i(v) \in \{0,1\}$ be such that $u_j \in V_{i(u)}$ and $v_j \in V_{i(v)}$, and let $u'_{j} :=
\psi_{i(u)}(u_{j})$ and $v'_{j} := \psi_{i(v)}(v_{j})$.  Then
$u_{0}u_{1}$ and $v_{0}v_{1}$ are $\Theta$-equivalent if and only if so are
$u'_{0}u'_{1}$ and $v'_{0}v'_{1}$.
\end{lem}

\begin{proof}
By Lemma~\ref{L:interv.}, for $j, k \in \{0,1 \}$, 
$d_{G}(u'_{j},v'_{k}) = d_{G}(u_{j},v_{k}) + \epsilon$ where 
$\epsilon$ is equal to $0$ or $1$ depending to whether $i(u)$ is or 
is not equal to $i(v)$.  Whence the result.    
\end{proof}

\begin{lem}\label{L:expans./conv.subgr.}
Let $G$ be a connected bipartite graph and $G'$ an expansion of
$G$ with respect to a proper cover $(V_{0},V_{1})$ of
$G$, and let $F$ be a convex subgraph of $G$.  Then $F' := G'[\psi(V(F))]$ is a convex subgraph of $G'$.  Moreover, if $V(F) \cap V_i \neq \emptyset$ for $i = 0, 1$, then $(V_0 \cap V(F),V_1 \cap V(F))$ is a proper cover of $F$, and $F'$ is the expansion of $F$ with respect to $(V_0 \cap V(F),V_1 \cap V(F))$.
\end{lem}

\begin{proof}
Let $P'$ be a $(u,v)$-geodesic for some $u', v' \in V(F')$, then $V(P') \subseteq \psi(V(P))$ for some $(u,v)$-geodesic $P$ of $G$ by Lemma~\ref{L:interv.}, where $u$ and $v$ are vertices of $F$ such that $u' = \psi(u)$ and $v' = \psi(v)$. Then $P$ is a path of $F$ since $F$ is convex in $G$, and thus $P'$ is a path of $F'$ by the definition of $F'$, which proves that $F'$ is convex in $G'$.

Assume now that $V(F) \cap V_i \neq \emptyset$ for $i = 0, 1$.  Let $V'_i := V_i \cap V(F)$ for $i = 0, 1$.  The first two properties of a proper cover are clearly satisfied by $(V'_0,V'_1)$ since they are satisfied by $(V_0,V_1)$.  Moreover $F[V'_0]$ and $F[V'_1]$ are isometric subgraphs of $F$ since $F$ is convex in $G$ and $G[V_0]$ and $G[V_1]$ are isometric in $G$.  Consequently $(V'_0,V'_1)$ is a proper cover of $F$.

By definition, $G'$ is the subgraph of $G \Box K_2$ induced by $(V_{0} \times \{0 \}) \cup (V_{1} \times \{1 \})$, where
$\{0,1 \}$ is the vertex set of $K_{2}$.  It follows that $F'$, which is equal to $G'[\psi(V(F))]$ by definition, is the subgraph of $F \Box K_2$ induced by $(V'_{0} \times \{0 \}) \cup (V'_{1} \times \{1 \})$ since $V'_i = V_i \cap V(F)$ for $i = 0, 1$.  Therefore $F'$ is the expansion of $F$ with respect to $(V'_0,V'_1)$.
\end{proof}

\begin{rem}\label{R:notation}
Let $G$ be a partial cube, $e = cd$ some edge of $G$,\; $G' := G/e$ the $\Theta$-contraction of $G$ with respect to the $\Theta$-class of $e$, and $\gamma_{e}$ the natural surjective contraction
of $G$ onto $G'$.  We use the following notation already introduced: $$V_0 := W_{cd}\quad V_1 := W_{dc}\quad V'_0 := \gamma_{e}(V_0)\quad V'_1 := \gamma_{e}(V_1).$$  For $i = 0, 1$, let $\psi_i : V'_i \to V_i$ be such that $\gamma_{e}(\psi_i(x)) = x$ for each $x \in V'_i$, and for any $A \subseteq V(G')$, let $$\psi(A) := \psi_{0}(A \cap V'_{0}) \cup
\psi_{1}(A \cap V'_{1}).$$  Furthermore we denote by $x'$ the vertex $\gamma_e(x)$ for all $x \in V(G)$.

We make two remarks.

1.\; Let $Q$ be an $(x,y)$-geodesic in $G$ which passes through some vertex $u$.  Then $Q$ clearly has at most one edge which is $\Theta$-equivalent to $e$.  Hence, by Lemma~\ref{L:interv.}, $\gamma_e(Q)$ is a $(x', y')$-geodesic in $G'$ which passes through $u'$.

2.\; Let $R := \langle x_0,\dots,x_p\rangle$ be a geodesic in $G'$.  Because $G'[V'_0]$ and $G'[V'_1]$ are isometric subgraphs of $G'$, there exists an $(x_0,x_p)$-geodesic $R' := \langle y_0,\dots,y_p\rangle$ in $G'$ such that:

\textbullet\; if $x_0, x_p \in V'_i$ for some $i = 0$ or $1$, then $y_j \in V'_i$ for all $j$ with $0 \leq j \leq p$;

\textbullet\; if $x_0 \in V'_i-V'_{1-i}$ and $x_p \in V'_{1-i}-V'_i$ for some $i = 0$ or $1$, and if $k$ is any non-negative integer such that $x_k \in V'_0 \cap V'_1$, then $y_j$ belongs to $V'_i$ or $V'_{1-i}$ with $y_j = x_j$ if $x_j$ belongs to $V'_i$ or $V'_{1-i}$ according as $j \leq k$ or $j \geq k$.

Then, by Lemma~\ref{L:interv.}, in the first case $\psi_i(R')$ is a $(\psi_i(x_0),\psi_i(x_p))$-geodesic, and in the second case $\psi_i(R'[y_0,y_k]) \cup \psi_{1-i}(R'[y_k,y_p])$ is a $(\psi_i(x_0),\psi_{1-i}(x_p))$-geodesic.  It follows in particular that $d_G(x_0,x_p) = d_{G'}(x'_0,x'_p) + \epsilon$ with $\epsilon = 0$ or $1$ according as $x_0, x_p \in V'_i$ or $x_0 \in V'_i-V'_{1-i}$ and $x_p \in V'_{1-i}-V'_i$ for some $i = 0$ or $1$.  In each case, this geodesic in $G$ will be denoted by $\Psi(R')$ in what follows.

  Note that, by the above remarks, if $ab$ is an edge of $G$ which is not $\Theta$-equivalent to $e$, then we clearly have $\mathcal{I}_{G'}(U_{a'b'}^{G'}) = \gamma_e(\mathcal{I}_G(U_{ab}))$.
\end{rem}

\subsection{Local ph-homogeneity}\label{SS:local.hypernet.}

We now introduce a property of partial cubes which is weaker than that of being ph-homogeneous.

\begin{defn}\label{L:local.hypernet.}
Let $G$ be a partial cube, and $ab$ an edge of $G$.  We say that $G$ is \emph{ph-homogeneous in $ab$} if $U_{cd}^F$ and $U_{dc}^F$ are ph-stable for any finite convex subgraph $F$ of $G$ that contains an edge $cd$ which is $\Theta$-equivalent to $ab$.
\end{defn}

By Proposition~\ref{P:charact.p.c.ph}, \emph{a partial cube is ph-homogeneous if and only if it is ph-homogeneous in each of its edges}.  Moreover, \emph{if $G$ is ph-homogeneous in $ab$, then any finite convex subgraph $F$ of $G$ that contains an edge $cd$ which is $\Theta$-equivalent to $ab$ is ph-homogeneous in $cd$}.

\begin{thm}\label{T:loc.hypernet./charact.}
Let $G$ be a partial cube, and $ab$ one of its edges.  The following assertions are equivalent:

\textnormal{(i)}\; $G$ is ph-homogeneous in $ab$.

\textnormal{(ii)}\; $U_{ab}$ and $U_{ba}$ are strongly ph-stable.

\textnormal{(iii)}\; For any bulge $X$ of $co_{G}(U_{ab})$ (resp. $co_G(U_{ba})$), we have the properties (HNB1) and (HNB2).
\end{thm}

The equivalences of the assertions (i),(ii) and (iii) of Theorem~\ref{T:charact.} are immediate consequences of the above theorem.

\subsection{Proof of the equivalence (i)$\Leftrightarrow$(ii) of Theorem~4.15}\label{SS:(i)<=>(ii)4.15}

We need several lemmas.

\begin{lem}\label{L:contract./local.hypernet.}
Let $G$ be a partial cube which is ph-homogeneous in one of its edges $ab$.  Let $e$ be an edge of $G[\mathcal{I}_G(U_{ab})]$,\; $G' := G/e$ the
$\Theta$-contraction of $G$ with respect to the $\Theta$-class of
$e$, and $\gamma_{e}$ the natural surjective contraction of $G$ onto
$G'$.  Then $G'$ is ph-homogeneous in the edge $a'b'$, where $a' :=
\gamma_{e}(a)$ and $b' := \gamma_{e}(b)$.
\end{lem}

\begin{proof}
$G'$ is partial cube by what we saw above.  We will use the notation of Remark~\ref{R:notation}.

Let $F'$ be a finite convex subgraph of $G'$ which contains an edge $\Theta$-equivalent to $a'b'$.  Without loss of generality we will suppose that $a'b' \in E(F')$.  We will show that $U_{a'b'}^{F'}$ and $U_{b'a'}^{F'}$ are ph-stable.  By Lemma~\ref{L:expans./conv.subgr.}, $F := G[\psi(V(F'))]$ is a finite convex subgraph of $G$, and thus $U_{ab}^F$ and $U_{ba}^F$ are ph-stable since $G$ is ph-homogeneous in $ab$.  We have two cases:

(a)\; $V(F') \cap V'_i = \emptyset$ for some $i = 0$ or $1$.

Say $i = 1$.  Then $F := \psi_0(F')$ is isomorphic to $F'$.  It follows that $U_{a'b'}^{F'} = U_{ab}^F$ and $U_{b'a'}^{F'} = U_{ba}^F$, and thus $U_{a'b'}^{F'}$ and $U_{b'a'}^{F'}$ are ph-stable.

(b)\; $V(F') \cap V'_i \neq \emptyset$ for $i = 0, 1$.

By Lemma~\ref{L:expans./conv.subgr.}, $F := G[\psi(V(F'))]$ is the expansion of $F'$ with respect to $(V'_0 \cap V(F'),V'_1 \cap V(F'))$, and thus $F'$ is the $\Theta$-contraction of $F$ with respect to the $\Theta$-class of the edge $e$, and the restriction $\gamma_{e}$ onto $V(F)$ is the natural surjective contraction
of $F$ onto $F'$.  Because $e$ is $\Theta$-equivalent to an edge of $F[\mathcal{I}_F(U_{ab}^F)]$ and thus of $F[\mathcal{I}_F(U_{ba}^F)]$, it follows that $U_{a'b'}^{F'}$ and $U_{b'a'}^{F'}$ are ph-stable by Lemma~\ref{L:Theta-contraction/ph-stable}.

Consequently $G'$ is ph-homogeneous in $a'b'$.
\end{proof}

\begin{lem}\label{L:tripod}
Let $G$ be a partial cube which is ph-homogeneous in one of its edges $ab$, and $u \in W_{ab}$ which has two neighbors $v$ and $w$ in $U_{ab}$.  Then $u \notin I_G(x,v) \cup I_G(x,w)$ for every vertex $x \in U_{ab}-\{ v,w\}$ such that $I_G(u,x) \cap U_{ab} = \{ x\}$.
\end{lem}

\begin{proof}
The proof will be by induction on $d_G(u,x)$.

(a)\; Suppose that $d_G(u,x) = 1$.  Then $u$ has three neighbors $x, v, w$ in $U_{ab}$.  Let $F$ be the subgraph of $G$ induced by $co_G(u,x,v, w,x')$, where $x'$ is the neighbor of $x$ in $U_{ba}$.  Then $F$ is a finite convex subgraph of $G$.  Let $x', v', w'$ be the neighbors in $U_{ba}^F$ of $x, v, w$, respectively.  Then $d_F(x',v') = d_F(v',w') = d_F(w',x') = 2$.  Denote by $u'$ the common neighbor of $x'$ and $w'$.  Suppose that $u'$ and $v'$ are not adjacent (see Figure~\ref{F:tripod}).  Then $d_F(u',v') = 3$, and thus any $(v',u')$-geodesic is a geodesic of maximal length because, by Lemma~\ref{L:edge/co(H)}, any edge of $F$ is $\Theta$-equivalent to one of the edges $ux, uv, uw, xx'$.  It follows that $u' \notin I_F(v',y)$ for some $y \in U_{ba}^F$, contrary to the fact that $U_{ba}^F$ is ph-stable since $G$ is ph-homogeneous in $ab$.  Therefore $u'$ is adjacent to $x', v', w'$.  Hence both the edges $u'x'$ and $u'w'$ are $\Theta$-equivalent to the edge $vu$, and thus they are $\Theta$-equivalent by transitivity, which is impossible.  We infer that $u$ has exactly two neighbors in $U_{ab}$.

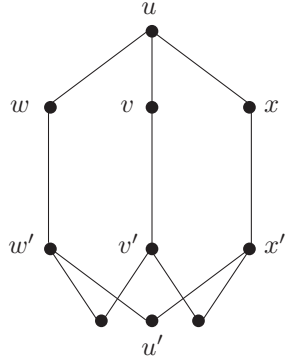
\begin{figure}[!h]
    \centering
{\tt    \setlength{\unitlength}{0.80pt}
\begin{picture}(123,199)
\thinlines    \put(62,83){\line(0,1){104}}
              \put(84,49){\circle*{6}}
              \put(38,49){\circle*{6}}
              \put(62,186){\circle*{6}}
              \put(14,150){\circle*{6}}
              \put(108,150){\circle*{6}}
              \put(14,83){\circle*{6}}
              \put(108,83){\circle*{6}}
              \put(62,83){\circle*{6}}
              \put(62,150){\circle*{6}}
              \put(62,49){\circle*{6}}
              \put(109,84){\line(-2,-3){23}}
              \put(13,84){\line(2,-3){24}}
              \put(61,83){\line(2,-3){24}}
              \put(13,151){\line(4,3){48}}
              \put(109,151){\line(-4,3){48}}
              \put(109,151){\line(0,-1){66}}
              \put(13,151){\line(0,-1){65}}
              \put(13,84){\line(4,-3){46}}
              \put(108,83){\line(-4,-3){45}}
              \put(62,84){\line(-2,-3){23}}
              \put(115,81){$x'$}
              \put(57,194){$u$}
              \put(-5,148){$w$}
              \put(47,148){$v$}
              \put(115,148){$x$}
              \put(-6,81){$w'$}
              \put(46,81){$v'$}
              \put(57,33){$u'$}
\end{picture}}    
\caption{Part (a) of the proof of Lemma~\ref{L:tripod}.}
\label{F:tripod}
\end{figure}

(b)\;  Suppose that $u \notin I_G(x,v) \cup I_G(x,w)$ for every vertex $x \in U_{ab}-\{ v,w\}$ such that $I_G(u,x) \cap U_{ab} = \{ x\}$ and $d_G(u,x) \leq n$ for some positive integer $n$.  Let $x \in U_{ab}-\{ v,w\}$ be such that $I_G(u,x) \cap U_{ab} = \{ x\}$ and $d_G(u,x) = n+1$.  Suppose that $u \in I_G(x,v) \cup I_G(x,w)$.  Let $e = uy$ where $y \in N_G(u) \cap I_G(u,x)$, and let $G' := G/e$ be the $\Theta$-contraction of $G$ with respect to the $\Theta$-class of $e$.  We will use the notations introduced in Remark~\ref{R:notation}.  $G'$ is ph-homogeneous in $a'b'$ by Lemma~\ref{L:contract./local.hypernet.}, and $v'$ and $w'$ are neighbors of $u'$ in $U_{a'b'}^{G'}$.  Suppose that some vertex $c \in I_G(u,x)$ is adjacent to some vertex $d \in U_{ab}$ such that the edge $dc$ is $\Theta$-equivalent to $e$.  Then clearly $c \neq u$ and $d \neq x$, and moreover $z \in I_G(u,x)$, contrary to the hypothesis.  Therefore $I_G(u',x') \cap U_{a'b'}^{G'} = \{ x'\}$.  Furthermore, by Lemma~\ref{L:interv.}, $u' \in I_G(x',v') \cup I_G(x',w')$ with $d_{G'}(u',x') = p$.  This yields to a contradiction with the induction hypothesis.  Consequently $u \notin I_G(x,v) \cup I_G(x,w)$.
\end{proof}

Recall that the
\emph{isometric dimension} of a finite partial cube $G$, i.e., the least non-negative integer $n$ such
that $G$ is an isometric subgraph of an $n$-cube, coincides with the
number of $\Theta$-classes of $E(G)$.  We denote it by  $\mathrm{idim}(G)$.

\begin{lem}\label{L:idim(G)=length(geod)=>G=cycle}
Let $G$ be a partial cube which is ph-homogeneous in one of its edges $ab$, and $P$ a $U_{ab}$-geodesic whose length is as small as possible, and such that any edge of $G$ is $\Theta$-equivalent either to $ab$ or to an edge of $P$.  Let $v$ and $w$ be the endvertices of $P$,\; $v'$ and $w'$ the neighbors in $U_{ba}$ of $v$ and $w$, respectively, and let $P'$ be a $(v',w')$-geodesic.  Then $G$ is equal to the cycle $C := P \cup \langle w,w'\rangle \cup P' \cup \langle v,v'\rangle$.
\end{lem}

\begin{proof}
The proof will be by induction on $\mathrm{idim}(G) = d_G(v,w)+1$.  This is clear if $d_G(v,w) = 2$.  Suppose that this holds if $d_G(v,w) \leq n$ for some $n \geq 2$, and let $G$,\; $P$ and $v, w$ be such that $\mathrm{idim}(G) = d_G(v,w)+1 = n+2$.  

(a)\; Let $e$ be any edge of $P$,\; $G' := G/e$ the $\Theta$-contraction of $G$ with respect to the $\Theta$-class of $e$, and $\gamma_{e}$ the natural surjective contraction
of $G$ onto $G'$.  Then $G'$ is ph-homogeneous in $a'b'$ by Lemma~\ref{L:contract./local.hypernet.}.

Denote by $x_v$ and $x_w$ the neighbors in $P$ of $v$ and $w$, respectively.  We will show that, if the edge $e$ is $\Theta$-equivalent to some edge $xy$ where $x$ is an inner vertex of $P$ and $y \in U_{ab}-\{v,w\}$, then $e = xy$ and moreover it is equal to $x_vv$ or $x_ww$.  Suppose that $e$ is distinct from $x_vv$ and $x_ww$.  Then, by Lemma~\ref{L:tripod}, $x$ is distinct from $x_v$ and $x_w$.  It follows that $P[v,x] \cup \langle x,y\rangle$ or $P[w,x] \cup \langle x,y\rangle$ is a $U_{ab}$-geodesic depending on whether $e$ is an edge of $P[w,x]$ or of $P[v,x]$.  This yields a contradiction with the fact that $P$ is a $U_{ab}$-geodesic whose length is as small as possible.  Note that, by the properties of $E(G)$ and the minimality of $l(P)$, any $U_{ab}$-geodesic $Q$ of $G$ has length $n+1$, and the above result also holds for $Q$.

We deduce that $\gamma_e(P)$ is a $U_{\gamma_{e}(a)\gamma_{e}(b)}^{G'}$-geodesic whose length is $n$, and thus is as small as possible, and moreover, by Lemma~\ref{L:expans./theta}, any edge of $G'$ is $\Theta$-equivalent either to $\gamma_{e}(a)\gamma_{e}(b)$ or to an edge of $\gamma_{e}(P)$.   It follows, by the induction hypothesis, that $$G' = \gamma_e(C) = \gamma_e(P) \cup \langle \gamma_e(w),\gamma_e(w')\rangle \cup \gamma_e(P') \cup \langle \gamma_e(v),\gamma_e(v')\rangle.$$  This proves in particular that $\gamma_e(P)$ is convex.

(b)\; Suppose that $P$ is not convex.  Then there exists another $(v,w)$-geodesic $Q$.  Because $n \geq 2$, we can choose the edge $e$ such that $\gamma_e(P) \neq \gamma_e(Q)$, contrary to the fact that $\gamma_e(P)$ is convex by (a).  Therefore $P$ is convex.  For the same reason $P'$ is convex.

(c)\; Suppose that $C$ is not convex.  Then there exists a geodesic joining a vertex of $P$ and a vertex of $P'$ which contains an edge $cd$ which is $\Theta$-equivalent to $ab$ and distinct from the edges $vv'$ and $ww'$.  Because $n \geq 2$, and thus $l(P) \geq 3$, the vertex $c$ is not adjacent to both $v$ and $w$.  Hence we can choose $e$ so that $\gamma_e(c)$ is distinct from $\gamma_e(v)$ and $\gamma_e(w)$.  It follows that $\gamma_e(c)\gamma_e(d)$ is an edge of $G'$ distinct from the edges $\gamma_e(v)\gamma_e(v')$ and $\gamma_e(w)\gamma_e(w')$, contrary to the fact that $G' = C'$ by (a).  Therefore the cycle $C$ is convex.

(d)\; Suppose now that $G \neq C$.  Then, because a partial cube is convex, there exists an edge $xy$ with $x \in V(C)$ and $y \notin V(C)$.  By the properties of $G$, $xy$ is $\Theta$-equivalent to some edge $cd$ of $G$.  Then $y \in I_G(d,x)$.  Hence $y \in V(C)$ since $d, x \in V(C)$ and $C$ is convex.  Consequently $G = C$.
\end{proof}

\begin{lem}\label{L:Pconv.=>l(P)min.}
Let $G$ be a partial cube, $ab$ an edge of $G$, and $P$ a $U_{ab}$-geodesic such that each edge of $G$ is $\Theta$-equivalent to some edge of $P$ or to $ab$.  If $P$ is convex, then $P$ is a $U_{ab}$-path of minimal length.
\end{lem}

\begin{proof}
Let $P = \langle v_0,\dots,v_n\rangle$.  Assume that $P$ is not a $U_{ab}$-path of minimal length.  It follows that there exists a vertex $x \in U_{ab}$ which is distinct from $v_0$ and $v_n$.  Let $Q = \langle x_0,\dots,x_p\rangle$ be a $(v_0,x)$-geodesic with $x_0 = v_0$ and $x_p = x$, and let $i$ be the smallest integer such that $x_{i+1} \notin V(P)$.  Then $x_i = v_i$, and the edge $x_ix_{i+1}$ is $\Theta$-equivalent to the edge $v_jv_{j+1}$ for some $j > i$.  It follows that $x_{i+1} \in I_G(v_i,v_{j+1})$.  Therefore there exist a $(v_0,v_n)$-geodesic which passes through $x_{i+1}$, and thus which is distinct from $P$, which proves that $P$ is not convex.
\end{proof}

Let $G$ be a partial cube, $ab$ one of its edges, and $u \in \mathcal{I}_G(U_{ab})-U_{ab}$.  Then any convex $U_{ab}$-path that passes through $u$ and that satisfies (SPS1) and (SPS2) will be said to be \emph{associated with $u$}.

\begin{lem}\label{L:min.Uab-path=associated}
Let $G$ be a partial cube, $ab$ one of its edges, and $P_u$ a $U_{ab}$-path which is associated with some vertex $u \in \mathcal{I}_G(U_{ab})-U_{ab}$.  Then any $U_{ab}$-geodesic that passes through $u$ of minimal length is equal to $P_u$, and thus this path is unique.
\end{lem}

\begin{proof}
Let $Q$ be a $U_{ab}$-geodesic which passes through $u$ whose length is minimal with respect to this property, and let $v$ and $w$ be its endvertices.  By (SPS2), $P_u$ is a subpath of some $(v,w)$-geodesic.  Because $l(Q) \leq l(P)$, it follows that $Q$ has the same endvertices as $P_u$.  Therefore $P_u = Q$ since $P_u$ is convex.
\end{proof}

Because of its uniqueness, $P_u$ is called \emph{the $U_{ab}$-geodesic associated with $u$}.

\begin{proof}[\textnormal{\textbf{Proof of the equivalence (i)$\Leftrightarrow$(ii) of 
Theorem~\ref{T:loc.hypernet./charact.}}}]
Let $G$ be a partial cube, and $ab$ one of its edges.

\emph{Necessity.}\; Assume that $G$ is ph-homogeneous in $ab$, and let $u \in \mathcal{I}_G(U_{ab})-U_{ab}$.  We distinguish two cases.

\emph{Case 1.}\; $G$ is finite.

We proceed by induction on the isometric dimension of $G$ to prove that, if $G$ is ph-homogeneous in $ab$, then $U_{ab}$ and $U_{ba}$ are strongly ph-stable.  We clearly have $\mathrm{idim}(G) \geq 3$.  If $\mathrm{idim}(G) = 3$, then we can easily prove that $G$ is a $6$-cycle, and thus we are done.  Suppose that the result is true for any partial cube which is ph-homogeneous in one of its edge $ab$ and whose isometric dimension is at most $n$ for some $n \geq 3$, and let $G$ be a partial cube of isometric dimension $n+1$ which is ph-homogeneous in $ab$.  Without loss of generality we can suppose that $V(G) = \mathcal{I}_G(U_{ab}) \cup \mathcal{I}_G(U_{ba})$, i.e., $G = G_{\overline{ab}}$.

  Let $P$ be a $U_{ab}$-geodesic passing through $u$ whose length $l(P)$ is as small as possible, and let $v$ and $w$ be its endvertices.  We have two subcases.\\
  
\emph{Subcase 1.1.}\; Assume that any edge of $G$ is $\Theta$-equivalent either to $ab$ or to an edge of $P$.

Suppose that $P$ is not a $U_{ab}$-path of minimal length.  Then there exists a vertex $x \in U_{ab}$ which is distinct from $v$ and $w$.  Because $U_{ab}$ is ph-stable, it follows that $u \in I_G(x,y)$ for some $y \in U_{ab}$.  Hence $d_G(x,y) = l(P)$ because $l(P)$ is minimal and since any edge of $G$ is $\Theta$-equivalent either to $ab$ or to an edge of $P$.  Let $Q$ be an $(x,y)$-geodesic which passes through $u$. By the Interval Property of the Cartesian product and since $l(P)$ is minimal, both $Q[u,x]$ and $Q[u,y]$ have edges which are $\Theta$-equivalent to edges of both $P[u,v]$ and $P[u,w]$.  We deduce that $l(P) \geq 4$.

We infer that there exists an edge $e$ of $P$ which is not $\Theta$-equivalent to an edge $uz$ for some $z \in U_{ab}$, and which is such that $\gamma_e(P) \neq \gamma_e(Q)$.  By Lemma~\ref{L:contract./local.hypernet.}, $G' := G/e$ is ph-homogeneous in $a'b'$.  Hence, by the induction hypothesis, since the isometric dimension of $G'$ is equal to $n$, there exists a $U_{a'b'}^{G'}$-geodesic associated with $u'$.  Hence $\gamma_e(P) = \gamma_e(Q)$, by Lemma~\ref{L:min.Uab-path=associated}, because $\gamma_e(P)$ and $\gamma_e(Q)$ are both $U_{a'b'}^{G'}$-paths passing through $u'$ of minimal length.  This yields a contradiction with the fact that $\gamma_e(P) \neq \gamma_e(Q)$.  Consequently $P$ is convex, and thus, by Lemma~\ref{L:Pconv.=>l(P)min.}, $P$ is a $U_{ab}$-path of minimal length.

It follows, by Lemma~\ref{L:idim(G)=length(geod)=>G=cycle}, that $G$ is the cycle $P \cup \langle w,w'\rangle \cup P' \cup \langle v,v'\rangle$, where $v'$ and $w'$ are the neighbors in $U_{ba}$ of $v$ and $w$, respectively, and $P'$ is a $(v',w')$-geodesic.  Therefore $U_{ab}$, and analogously $U_{ba}$, are strongly ph-stable.  Hence 
$G$ satisfies condition (ii) of Theorem~\ref{T:charact.}.\\

\emph{Subcase 1.2.}\; Assume now that some edge $e = cd$ of $G$ is not $\Theta$-equivalent to $ab$ or to an edge of $P$.

(a)\; We first show that $e$ is not $\Theta$-equivalent to an edge incident to an inner vertex of $P$ and to a vertex in $U_{ab}$.  If $l(P) = 2$, then we already know, by Lemma~\ref{L:tripod}, that $v$ and $w$ are the only neighbors of $u$ in $U_{ab}$.  Assume now that the length of $P$ is greater than $2$, and that some inner vertex $x$ of $P$ is adjacent to some vertex $y \in U_{ab}$.  By Lemma~\ref{L:tripod}, $y$ coincides with $v$ or $w$ if $x$ is a neighbor of $v$ or $w$, respectively.  Suppose that $x$ is not a neighbor of $v$ or of $w$, and, without loss of generality, that $u \in V(P[v,x])$.  Then $P[v,x] \cup \langle x,y\rangle$ is not a geodesic by the minimality of the length of $P$.  It follows that the edge $xy$ is $\Theta$-equivalent to some edge of $P[v,x]$, and thus of $P$.

We deduce that in any case the edge $e$ is not $\Theta$-equivalent to an edge incident to an inner vertex of $P$ and to an edge in $U_{ab}$.\\

(b)\; Let $G' := G/e$ be the $\Theta$-contraction of $G$ with respect to the $\Theta$-class of $e$.  We will use the notations introduced in Remark~\ref{R:notation}.  Because $G = G_{\overline{ab}}$ by hypothesis, the graph $G'$ is ph-homogeneous in $a'b'$.  Due to the properties of $e$, it follows that $\gamma_e(P)$ is a $U_{a'b'}^{G'}$-geodesic whose length is the same as that of $P$.

By the induction hypothesis, since the isometric dimension of $G'$ is equal to $n$, there exists a $U_{a'b'}^{G'}$-geodesic $P_{u'}$ which is associated with $u'$.  Because $\gamma_e(P)$ is a $(v',w')$-geodesic which passes through $u'$, it follows, by (SPS2), that there exists in $G'$ a $(v',w')$-geodesic $Q'$ which contains $P_{u'}$ as a subpath.  Let $Q$ be a $(v,w)$-geodesic in $G$ such that $\gamma_e(Q) = Q'$.  Because any edge of $Q$ is $\Theta$-equivalent to an edge of $P$ by Lemma~\ref{L:gen.propert.}(iv), it follows that $e$ is not $\Theta$-equivalent to an edge of $Q$.  Moreover, by (a), $e$ is not $\Theta$-equivalent to an edge incident to an inner vertex of $Q$ and to a vertex in $U_{ab}$.  Finally $Q$ is a $U_{ab}$-geodesic by the minimality of $l(P)$.  Hence $v'$ and $w'$ are the only vertices of $Q'$ that belongs to $U_{a'b'}^{G'}$.  Therefore $Q' = P_{u'}$, and thus $P_{u'} = \gamma_e(P)$ since $P_{u'}$ is convex.  Consequently $P$ is convex since any geodesic in $G$ between the endvertices of $P$ contains no edge $\Theta$-equivalent to $e$ by Lemma~\ref{L:gen.propert.}(iv).  It remains to prove that $P$ satisfies (SPS1) and (SPS2).  Note that, due to Lemma~\ref{L:SPS2=>SPS1}, it suffices to show that $P$ satisfies (SPS2) since $U_{ab}$ is ph-stable.  Without loss of generality we suppose that $u$, and thus any vertex of $P$, belongs to $V_0$ (recall that $V_0 := W_{cd}$ if $e = cd$ (see Remark~\ref{R:notation})).

Let $x, y \in U_{ab}$ such that $u \in I_G(x,y)$.  Then $u' \in I_{G'}(x',y')$ by Remark~\ref{R:notation}.1.  By (SPS2), there exists an $(x',y')$-geodesic in $G'$ which contains $P_{u'}$ as a subpath.  By Remark~\ref{R:notation}.2, we can choose $R$ such that, if $x', y' \in V'_0$, then $V(R) \subseteq V'_0$, or, if $x'$ or $y'$ belongs to $V'_1$, say $x' \in V'_1$ and thus $y' \in V'_0$, such that $R[v',y'] \subseteq V'_0$.  Then, by the conclusion of Remark~\ref{R:notation}.2, $\Psi(R)$ is an $(x,y)$-geodesic in $G$ which contains $P$ as a subpath.  Hence $P$ satisfies (SPS2).

Consequently $U_{ab}$, and analogously $U_{ba}$, are strongly ph-stable.\\

\emph{Case 2.}\; $G$ is infinite.

Let $c, d \in U_{ab}$ be such that $u \in I_G(c,d)$.  Then $F := G[co_G(a,b,c,d,u)]$ is a finite convex subgraph of $G$, and thus is a ph-homogeneous in $ab$.  By Case 1, $u$ lies on a convex $U_{ab}^F$-path $P$ of $F$ which satisfies (SPS1) and (SPS2) in $F$.

Let $F'$ be a finite convex subgraph of $G$ which contains $F$.  Then, as above, $F'$ contains a convex $U_{ab}^{F'}$-path $P'$ which satisfies (SPS1) and (SPS2) in $F'$.  By (SPS2), $P'$ is a subpath of $P$.  This implies that $P = P'$ since $P$ is also a $U_{ab}^{F'}$-path.  It clearly follows that $P$ is a $U_{ab}$-path which satisfies (SPS1) and (SPS2) in $G$.  Therefore $U_{ab}$, and analogously $U_{ba}$, are strongly ph-stable.\\

\emph{Sufficiency.}\; Assume that $U_{ab}$ and $U_{ba}$ are strongly ph-stable.  Let $F$ be a convex subgraph of $G$ that contains an edge $\Theta$-equivalent to $ab$.  Without loss of generality we suppose that $ab$ is an edge of $F$.  We will show that $U_{ab}^F$ is ph-stable.

Let $u \in \mathcal{I}_F(U_{ab}^F)$.  Then $u \in I_G(x,y)$ for some $x, y \in U_{ab}^F$.  Because $F$ is convex, and thus $U_{ab}^{F} = U_{ab}^{G} \cap V(F)$ by Lemma~\ref{L:prop.isom subgr.}, it follows, by (SPS2), that the $U_{ab}$-geodesic $P_u$ associated with $u$ is a subpath of some $(x,y)$-geodesic of $G$, and thus of $F$.  Therefore $P_u$ is a $U_{ab}^F$-geodesic.  It follows, by (SPS1), that, for every $z \in \mathcal{I}_{F}(U_{ab}^F)$, $u \in I_G(z,v)$ for some endvertex $v$ of $P_u$, and thus $u \in I_F(z,v)$ by the convexity of $F$.  Hence $U_{ab}^F$ is ph-stable.  Analogously $U_{ba}^F$ is ph-stable.

We infer that $G$ is ph-homogeneous in $ab$.
\end{proof}

\subsection{Proof of the equivalence (ii)$\Leftrightarrow$(iii) of Theorem~4.15}\label{SS:(ii)<=>(iii)4.15}

\emph{Throughout this subsection, $G$ is a partial cube such that $U_{ab}$ and $U_{ba}$ are strongly ph-stable for some edge $ab$ of $G$}, and thus, by the equivalence of conditions (i) and (ii) of this theorem that we have already proved, $G$ is ph-homogeneous in $ab$.  Hence, by Lemma~\ref{L:ph-stable}, we deduce that $co_{G}(U_{ab}) = \mathcal{I}_{G}(U_{ab})$.

\begin{lem}\label{L:Pu/conv.unique}
Let $ab$ be an edge of $G$, $u \in \mathcal{I}_{G}(U_{ab})-U_{ab}$,\; $P_u$ the $U_{ab}$-geodesic associated with $u$, and $P_v$ the $U_{ab}$-geodesic associated with some inner vertex $v$ of $P_u$.  Then $P_u = P_v$.
\end{lem}

\begin{proof}
Let $x$ and $y$ be the endvertices of $P_u$.  By (SPS2), $P_v$ is a subpath of some $(x,y)$-geodesic passing through $v$.  This $(x,y)$-geodesic is then $P_u$, because $P_u$ is convex.  It follows that $P_v = P_u$ since no inner vertex of $P_u$ belongs to $U_{ab}$.
\end{proof}

\begin{lem}\label{L:I(u,x)/v,w}
Let $ab$ be an edge of $G$, $u \in \mathcal{I}_{G}(U_{ab})-U_{ab}$, $P_u$ the $U_{ab}$-geodesic associated with $u$, and $v$ and $w$ the endvertices of $P_u$.  Then, for each vertex $x \in U_{ab}$, $v$ or $w$ belongs to $I_G(u,x)$.
\end{lem}

\begin{proof}
Let $x \in U_{ab}$.  By (SPS1), $u$ belongs to $I_G(x,v)$ or $I_G(x,w)$, say $I_G(x,v)$.  Then, by (SPS2), $P_u$ is a subpath of some $(x,v)$-geodesic.  It follows that $w \in I_G(x,u)$.
\end{proof}

\begin{lem}\label{L:ab-cycle}
Let $ab$ be an edge of $G$,\; $P$ the $U_{ab}$-geodesic associated with some vertex in $\mathcal{I}_{G}(U_{ab})-U_{ab}$, and $v$ and $w$ its endvertices.  Let $v'$ and $w'$ be the neighbors in $U_{ba}$ of $v$ and $w$, respectively, and $P'$ a $(v',w')$-geodesic.  Then $C := P \cup \langle w,w'\rangle \cup P' \cup \langle v',v\rangle$ is the unique convex cycle containing $P$ and an edge $\Theta$-equivalent to $ab$.
\end{lem}

\begin{proof}
By the uniqueness of $P$, it suffices to show that $C$ is convex.

(a)\; Let $R$ be an $(u,u')$-geodesic for some $u \in V(P)$ and $u' \in V(P')$.  Let $z$ and $z'$ be the vertices of $R$ in $U_{ab}$ and $U_{ba}$, respectively.  Because $P$ is associated with $u$ by Lemma~\ref{L:Pu/conv.unique}, it follows, by (SPS1), that $u \in I_G(z,v) \cup I_G(z,w)$, say $x \in I_G(z,w)$.  By (SPS2), $P$ is a subpath of some $(z,w)$-geodesic.  Hence $d_G(z,w) = d_G(z,v)+d_G(v,w)$.  It follows that $d_G(z',w') = d_G(z',v')+d_G(v',w')$.  Therefore, if $z \neq v$, and thus $z' \neq v'$, then $d_G(u,v) < d_G(u,z)$ and $d_G(u',v') < d_G(u',z')$, contrary to the hypothesis that $R$ is a geodesic.  Consequently $z = v$ and $z' = v'$.

(b)\; Suppose that $P'$ is not convex.  Then there exists a vertex $x$ of $P'$, and $y, z \in N_G(x) \cap I_G(x,w')$ such that only $y$ is a vertex of $P'$.  By (a), the edges $xy$ and $xz$ are both $\Theta$-equivalent to some edge of $P$.  Hence $xy$ and $xz$ are $\Theta$-equivalent by transitivity of $\Theta$.  Therefore $P'$ is convex.

We infer that $C$ is convex.
\end{proof}

This cycle $C$ will be called \emph{the $ab$-cycle associated with some given inner vertex of $P$}, and thus, by Lemma~\ref{L:Pu/conv.unique}, with \emph{any} inner vertex of $P$.  Note that $C \in \mathbf{C}(G,ab)$.

\begin{lem}\label{L:ab-cycle/cd-cycle}
Let $ab$ be an edge of $G$, 
and $C$ the $ab$-cycle associated with some vertex $u \in \mathcal{I}_{G}(U_{ab})-U_{ab}$.  Let $cd$ be an edge of $C$.  Then $C$ is the $cd$-cycle associated with any inner vertex of $C-W_{dc}$.
\end{lem}

\begin{proof}
Denote by $c'd'$ the other edge of $C$ which is $\Theta$-equivalent to $cd$.  
Let $x$ be an inner vertex of $C-W_{ba}$, and $P_x$ the $U_{cd}$-geodesic associated with $x$.  By (SPS2), $P_x$ is a subpath of some $(c,c')$-geodesic.  It follows that $P_x = C-W_{dc}$ by the convexity of $C$ and the fact that no inner vertex of $C-W_{dc}$ belong to $U_{cd}$.  Therefore $C$ is the $cd$-cycle associated with $x$.
\end{proof}

\begin{lem}\label{L:ab-cycle/gated}
Let $ab$ be an edge of $G$, $u \in \mathcal{I}_{G}(U_{ab})-U_{ab}$, and $C_u$ the $ab$-cycle associated with $u$.  Then $C_u$ is gated in $G_{\overline{ab}}$.
\end{lem}

\begin{proof}
(a)\; Let $x$ be a vertex of $G_{\overline{ab}}-C_u$.  Without loss of generality, we can suppose that $x \in V(G_{ab})$.  Note that $VG_{ab}) = \mathcal{I}_G(U_{ab})$ since $U_{ab}$ is ph-stable, because $ph(G) \leq 1$.  By (SPS1), $u \in I_G(x,v)$ for some endvertex $v$ of $P_u := C-W_{ba}$.  Then, by (SPS2), $P_u$ is a subpath of some $(x,v)$-geodesic.  If $w$ is the endvertex of $P_u$ distinct from $v$, then $w \in I_G(x,y)$ for every $y \in V(P_u)$.

(b)\; Denote by $v'$ and $w'$ the neighbors in $U_{ba}$ of $v$ and $w$, respectively.  Analogously, for any $x' \in W_{ba}$,\; $v'$ or $w'$ belong to $I_G(x',y')$ for every $y' \in V(C_u-P_u)$.  Then, because $w \in I_G(x,v') \cup I_G(x,w')$ since $w \in I_G(x,v)$ by (a), it follows that $w \in I_G(x,y')$ for every $y' \in V(C_u-P_u)$.

We infer, from (a) and (b), that  $w$ is the gate of $x$ in $C_u$.
\end{proof}

\begin{lem}\label{L:uv/CuCv}
Let $ab$ be an edge of $G$,\; $u$ and $v$ two adjacent vertices in $\mathcal{I}_{G}(U_{ab})-U_{ab}$, and $C_u$ and $C_v$ the $ab$-cycles associated with $u$ and $v$, respectively.  Then $C_u = C_v$ or the subgraph induced by $V(C_u \cup C_v)$ is isomorphic to the prism $C_u \Box K_2$ over $C_u$, and moreover this subgraph is gated in $G_{\overline{ab}}$.
\end{lem}

\begin{proof}
(a)\; Assume that $C_u \neq C_v$.  Put $P_u := C_u - W_{ba}$ and $P_v := C_v - W_{ba}$.  Then $P_u$ and $P_v$ are disjoint by Lemma~\ref{L:Pu/conv.unique}.  Denote by $x_u$ and $y_u$, and $x_v$ and $y_v$ the endvertices of $P_u$ and $P_v$, respectively.  By the convexity of $P_u$ and $P_v$, the paths $P_u[x_u,u] \cup \langle u,v\rangle$,\; $P_u[y_u,u] \cup \langle u,v\rangle$,\; $P_v[x_v,v] \cup \langle u,v\rangle$ and $P_v[y_v,v] \cup \langle u,v\rangle$ are geodesics.

Without loss of generality, we can suppose that $v \in I_G(x_u,y_v)$ and $u \in I_G(x_v,y_u)$.  By (SPS2), there exist an $(x_u,x_v)$-geodesic $R_x$, and a $(y_u,y_v)$-geodesic $R_y$ such that $P_u \cup R_x$ and $P_v \cup R_y$ are $(x_v,y_u)$-geodesics, and $P_v \cup R_x$ and $P_u \cup R_y$ are $(x_u,y_v)$-geodesics.  This straightforward implies that $x_u$ and $x_v$ are adjacent, and that $y_u$ and $y_v$ are adjacent.

By using Lemma~\ref{L:ab-cycle/cd-cycle}, and by repeating the above argument for different edges of $C$, we can prove that $H := G[V(C_u \cup C_v)] = C_u \Box K_2$.

(b)\; We will now show that $H$ is gated.  Let $x$ be a vertex of $G_{\overline{ab}}-H$.  Without loss of generality we can suppose that $d_G(x,C_u) \leq d_G(x,C_v)$.  We will first show that $d_G(x,C_u) < d_G(x,C_v)$ by induction on $k := d_G(x,C_u)$.  If $k = 1$ and $d_G(x,C_u) = d_G(x,C_v)$, then $x, u_{i+1}$ and $v_j$ are adjacent to $u_i$ and $v_{i+1}$, contrary to the fact that a partial cube contains no $K_{2,3}$.  Suppose that $d_G(x,C_u) < d_G(x,C_v)$ for any vertex $x$ such that $d_G(x,C_u) = k$ for some positive integer $k$.  Let $x$ be such that $d_G(x,C_u) = k+1$.

 Put $C_u = \langle u_1,\dots,u_{2n},u_1,\rangle$ and $C_v = \langle v_1,\dots,v_{2n},v_1,\rangle$.  By Lemma~\ref{L:ab-cycle/gated}, $C_u$ and $C_v$ are gated in $G_{\overline{ab}}$.  Let $u_i$ and $v_j$ be the gates of $x$ in $C_u$ and $C_v$, respectively.  If $d_G(x,C_u) = d_G(x,C_v)$, then $j = i+1$ or $i-1$, say $j = i+1$.  Denote by $y$ a neighbor of $x$ in $I_G(x,v_{i+1})$.  Suppose that $d_G(y,u_i) = k$.  Then $d_G(y,u_i) = d_g(y,v_{i+1})$.  Clearly, $u_i$ and $v_{i+1}$ are the gates of $y$ in $C_u$ and $C_v$, respectively.  Hence this would yield a contradiction to the induction hypothesis.  Therefore $d_G(y,u_i) = k+2$.  Then, because $d_G(x,u_i) = d_G(x,v_{i+1}) = k+1$, and $d_G(x,u_{i+1}) = k+2$ since $u_i$ is the gate of $x$ in $C_u$, it follows that the edges $xy$ and $u_iu_{i+1}$ are $\Theta$-equivalent.  Hence the edges $xy$ and $v_iv_{i+1}$ shall also be $\Theta$-equivalent, contrary to the fact that $d_G(x,v_i) = k+2$ since $v_i$ is the gate of $x$ in $C_v$.  Consequently $d_G(x,C_u) < d_G(x,C_v)$.
 
 With the above notation, we infer that the gate of $x$ in $C_v$ must be $v_i$.  Therefore $u_i$ is the gate of $x$ in $H$.
\end{proof}

\begin{lem}\label{L:bulge/ab-cycle}
Let $ab$ be an edge of $G$, and $X$ a bulge of $co_G(U_{ab})$.  Then we have the following properties:

\textnormal{(i)}\; There exists a unique $H \in \mathbf{Cyl}(G,ab)$ such that $X = H-W_{ba}$.

\textnormal{(ii)}\; $H = C \Box A$, where $C$ is the $ab$-cycle which is associated with some vertex of $X-U_{ab}$, and $A$ is a component of $X[U_{ab}]$.

\textnormal{(iii)}\; $H-W_{ab}$ is a bulge of $co_G(U_{ba})$.

\textnormal{(iv)}\; $H$ is a convex subgraph of $G$.
\end{lem}

This unique hypercylinder $H$ will be denoted $\mathbf{Cyl}[X]$.

\begin{proof}
Let $\delta := \sup_{x \in V(X-U_{ab})}\delta_X(x)$.  We distinguish two cases.

\emph{Case 1.}\; $\delta > 2$.

(a)\; We first show that the degree in $X$ of each vertex of
$X-U_{ab}$ is at least equal to $3$.  Suppose that some vertices of
$X-U_{ab}$ have degree $2$ in $X$.  Because $\delta \geq 3$, some of these
vertices, say $y$, is adjacent to a vertex $x$ of $X-U_{ab}$ whose
degree in $X$ is greater than $2$.  By Lemma~\ref{L:uv/CuCv}, $x$ and $y$ belongs to a convex $1$-prism $P_{xy} := C_x \Box K_2$, where $C_x$ is the $ab$-cycle associated with $x$. Because a $1$-prism
is $3$-regular, it follows that $\delta_{X}(y)
\geq \delta_{P_{xy}}(y) = 3$, contrary to the hypothesis.  Therefore the
degree in $X$ of each vertex of $X-U_{ab}$ is at least equal to $3$.\\

(b)\; We now show that any two vertices $x$ and $y$ of $X-U_{ab}$ are vertices of the Cartesian product of some path with the cycle $C_x$.  Let $\langle y_0,\dots,y_n\rangle$ be an $(x,y)$-geodesic in $X-U_{ab}$ with $y_0 = x$ and $y_n=y$.  We construct a
sequence $x_{0}, x_{1}, \ldots$ of vertices of $X-U_{ab}$, and a
sequence $P_{1}, P_{2}, \ldots$ of $1$-prisms such that, for every $i
\geq 1$, $x_{i-1}$ and $x_{i}$ are adjacent vertices
of $P_{i}$ which do not lie in the same cycle of $P_{i}$ containing
edges $\Theta$-equivalent to $ab$.  For each $i \geq 1$, we denote by $C_{i}$
and $C'_{i}$ the convex cycles of $P_{i}$ containing edges
$\Theta$-equivalent to $ab$ and such that $x_{i} \in V(C'_{i})$, i.e., $C_i = C_{x_{i-1}}$ and $C'_i = C_{x_i}$.

Let $x_{0}=y_0$, and let $P_{1}=P_{y_0y_1}$ (with the notation of (a)).  Suppose
that $x_{0}, \ldots, x_{i}$ and $P_{1}, \ldots, P_{i}$ have already
been constructed for some positive integer $i$.  If $y_n \in V(P_i)$, then we are done.  Suppose that $y_n \notin V(P_i)$.  Because $\langle y_0,\dots,y_n\rangle$ is a geodesic, there exists $p<n$ such that $y_p \in V(C'_i)$ and $y_j \notin \bigcup_{1
\leq k \leq i}P_{k}$ for $p+1 \leq j \leq n$.  Put $P_{i+1} := P_{y_p y_{p+1}}$, and let $x_{i+1}$ be the neighbor of $x_i$ which does not lie in the same cycle of $P_{i+1}$ containing edges $\Theta$-equivalent to $ab$.  Then $x_i$ lies in the cycles $C'_i$ by the induction hypothesis and $C_{i+1}$ by construction.  It follows that $C_{i+1} = C'_i = C_{x_i}$ because $C_{x_i}$ is unique by Lemma~\ref{L:ab-cycle}.  We
deduce from the construction that $C_{i+1}$ is isomorphic to $C_{1}$.  Because $n$ is finite, $y_n \in V(P_q)$ for some positive integer $q$.

Consequently $x$ and $y$ are vertices of the Cartesian product $C_{x_0} \Box \langle x_0,\dots,x_q\rangle$.\\

(c)\; We infer that, for any vertex $x$ of $X-U_{ab}$, the convex $ab$-cycle $C_x$ which is associated with $x$ is isomorphic to $C_{x_0}$.  Let $A_X$ be a component of $X[U_{ab}]$ (where, \emph{from now on, we will use $X[U_{ab}]$ as a short notation for $X[U_{ab} \cap V(X)]$}).  Let $u \in V(A_X)$.  By the definition of a bulge, $u$ is adjacent to some vertex $x$ of $X-U_{ab}$.  Then $u \in V(C_x)$ by the properties of $C_x$.

We deduce, by what we proved above, that $H := C_x \Box A_X$ is a subgraph of $G$ such that $X = H-W_{ba}$.  On the other side, $H-W_{ab}$ is contained in a bulge $Y$ of $co_G(U_{ba})$.  As we showed for $X$,\; $H' := C \Box A_{Y}$, where $A_{Y}$ is some component of $Y[U_{ba}]$, is a subgraph of $G$ such that $Y = H'-W_{ab}$.  Since $H$ is then a subgraph of $H'$, and because $X = H-W_{ba}$, it follows that $H=H'$, and thus $Y = H-W_{ab}$.  Note that this hypercylinder $H$ is clearly unique.

(d)\;  Suppose that $H$ is not convex.  Due to the facts that $H-W_{ba}$ and $H-W_{ab}$ are bulges of $co_G(U_{ab})$ and $co_G(U_{ba})$, respectively, it follows that, if $R$ is a $(V(H))$-geodesic which passes through a vertex $x$ of $G-H$ and whose length is as small as possible, then both its endvertices $u$ and $v$ belongs to $U_{ab}$ or $U_{ba}$, say to $U_{ab}$, and thus $R$ is a path of $G_{ab}$ since this graph is convex.  Denote by $u'$ and $v'$ the unique neighbors of $u$ and $v$ in $X-U_{ab}$.  Let $D$ be the unique $uu'$-cycle associated with $x$.  By (SPS2) there exists a $(u,v)$-geodesic which contains $D-W_{u'u}$.  It follows that there exists a $(u',v')$-geodesic which contains $D-W{uu'}$.  Then any vertex $y$ of $D-W{uu'}$ belongs to $co_G(U_{ab})$ since $u'$ and $v'$ belongs to $co_G(U_{ab})$.  It follows that $y$ is a vertex of the bulge $X$, and thus of $H$.  By Lemma~\ref{L:ab-cycle/cd-cycle}, $D$ is then equal to the $C$-layer of $H$ passing through $y$.  Hence $x \in V(H)$, contrary to the hypothesis.  Consequently $H$ is convex.\\

\emph{Case 2.}\; $\delta = 2$.

Then all vertices of $X-U_{ab}$ have degree $2$ in $G_{ab}$ since $X$ is connected.  Let $u$ be a vertex of $X-U_{ab}$.  Then all vertices of $X-U_{ab}$ lie to the $ab$-cycle $C$ associated with $u$.  Hence $X = C-W_{ba}$, and moreover $C$ is convex by Lemma~\ref{L:ab-cycle}.

On the other side, $C-W_{ab}$ is contained in a bulge $Y$ of $co_G(U_{ba})$.  Clearly  $Y = C-W_{ab}$ if $\sup_{x \in V(Y-U_{ba})}\delta_Y(x) = 2$.  Suppose that $\sup_{x \in V(Y-U_{ba})}\delta_Y(x) \geq 3$.  Then, by what we proved above, there exists some $H \in \mathbf{Cyl}(G,ab)$ such that $Y = H-W_{ab}$.  Since $C$ is then a subgraph of $H$, and because $X = C-W_{ba}$, it follows that $C=H$, and thus $Y = C-W_{ab}$.
\end{proof}

The following theorem shows that Peano partial cubes satisfy a property which generalizes the facts that the convex hull of any isometric cycle of a median graph (resp. a netlike partial cube) $G$ is a hypercube (resp. is either this cycle itself or a hypercube).

\begin{thm}\label{T:hypernet./conv.hull(isom.cycle)}
The convex hull of any isometric cycle of a Peano partial cube is a quasi-hypertorus.
\end{thm}

\begin{proof}
Let $D$ be an isometric cycle of a Peano partial cube $G$, and $F$ the convex hull of $D$ in $G$.  Without loss of generality we will suppose that $G = F$, and we will show that $G$ is a quasi-hypertorus.  By Lemma~\ref{L:gen.propert.}(iii), $G$ is finite, and thus it has a semi-periphery.

(a)\; We first show that $G$ is strongly semi-peripheral (see Definition~\ref{D:strong.periph.}).  Let $ab \in E(G)$.  Suppose that $W_{ab}$ is not a semi-periphery.  Then there exists an edge $xy \in \partial_G(W_{ba} \cup co_G(U_{ab}))$.  By Lemma~\ref{L:gen.propert.}(ix), $xy$ is not $\Theta$-equivalent to an edge of $G[W_{ba} \cup co_G(U_{ab})]$ since this subgraph is convex.  This yields a contradiction with the fact that, because $G = co_G(D)$, any edge of $G$ is $\Theta$-equivalent to an edge of $D$ by Lemma~\ref{L:edge/co(H)}(i), and that, since $D$ is an isometric cycle, any edge of $D$ is $\Theta$-equivalent to an edge of $D[W_{ba}]$ and thus to an edge of $G[W_{ba}]$.  Therefore $W_{ab}$ is a semi-periphery, and analogously, $W_{ba}$ is also a semi-periphery.

(b)\; We construct a sequence $H_0, H_1,\dots$ of non-empty convex subgraphs of $G$ such that, for each non-negative integer $n$,\; $H_n = C_{n+1} \Box H_{n+1}$, where $C_{n+1} \in \mathbf{C}(H_n)$.

Put $H_0 := G$.  Suppose that $H_0,\dots,H_n$ has already been constructed for some non-negative integer $n$.  Because $H_n$ is a non-empty convex subgraph of $G$, it follows that $H_n$ is also a strongly semi-peripheral Peano partial cube.  

Suppose that $W_{ab}^{H_n}$ is not a periphery for some $ab \in E(H_n)$.    Let $X$ be a bulge of $co_{H_n}(U_{ab}^{H_n})$, and $H$ the hypercylinder defined by Lemma~\ref{L:bulge/ab-cycle} (note that $H_n$ is ph-homogeneous in $ab$).  Suppose that $H_n \neq H$.  Then there exists an edge $uv \in \partial_{H_n}(V(H))$ with $u \in U_{ab}^{H_n} \cap V(H)$ and $v \in W_{ab}^{H_n}$.  By Lemma~\ref{L:gen.propert.}(ix), $uv$ is not $\Theta$-equivalent to an edge of $H$.  Let $C$ be the cycle of $H$ of length greater than $4$ which contains an edge $\Theta$-equivalent to $ab$ and which passes through $u$, and let $cd$ be an edge of $C$ which is not $\Theta$-equivalent to $ab$ and such that $u \in W_{cd}^{H_n}$.  Then $v \notin co_{G}(U_{cd}^{H_n})$ since $W_{cd}^{H}$ is a bulge of $W_{cd}^{H_n}$.  It follows that $W_{vu}^{H_n} \subset W_{cd}^{H_n}$, contrary to the assumption that $H_n$ is strongly semi-peripheral.  Therefore $H_n = H$.  It follows that $H_n = C_{n+1} \Box H_{n+1}$, where $C_{n+1} \in \mathbf{C}(H_n,ab)$ and $H_{n+1}$ is a component of $X[U_{ab}^{H_n}]$.  Then $H_{n+1}$ is a convex subgraph of $H_n$, and thus it is a strongly semi-peripheral Peano partial cube since so is $H_n$.

$H_0 \supset H_1 \supset \dots$ is a decreasing sequence of non-empty convex subgraphs of $G$.  Because $G$ is finite, $H_n$ is strongly peripheral for some non-negative integer $n$.  Then $H_n$ is a hypercube by Proposition~\ref{P:strong.-periph./hypercube}.  It follows that $G$ is a quasi-hypertorus because $G = H_0$ or $G = C_0 \Box \dots \Box C_{n} \Box H_n$ depending on whether $n$ is or is not equal to $0$.
\end{proof}

\begin{proof}[\textnormal{\textbf{Proof of
the equivalence (ii)$\Leftrightarrow$(iii) of Theorem~\ref{T:loc.hypernet./charact.}}}]

Let $G$ be a partial cube, and $ab$ an edge of $G$.

(ii)$\Rightarrow$(iii):\; Assume that $U_{ab}$ and $U_{ba}$ are strongly ph-stable.  Then $G$ is ph-homogeneous in $ab$ by the equivalence (i)$\Leftrightarrow$(ii) of Theorem~\ref{T:loc.hypernet./charact.}.  Property (HNB1) is an obvious consequence of Lemma~\ref{L:bulge/ab-cycle}.

We now prove that (HNB2) is satisfied.  Let $A_X$ and $B_X$ be the two components of $X[U_{ab}]$.  Suppose that $X-U_{ab}$ is not a separator of $G_{ab}$.  Then there exists in $G_{ab}$ a cycle $D$ of minimal length which passes through a vertex $u$ of $A_X$ and a vertex $v$ of $B_X$, such that $P := D \cap X$ is an $(u,v)$-geodesic in $X$, and $Q := D-P$ is a non-empty path of $G_{ab}-X$.  Because the length of $D$ is minimal, it follows that $D$ is isometric in $G$.

Put $P = \langle u_0,\dots,u_n\rangle$ with $u_0 = u$ and $u_n = v$.  Then, because $P$ is a geodesic in $X$, and thus in $H$, and since each $C$-layer of $H$ is a convex cycle of $G$, and no edge of $A_X$ is $\Theta$-equivalent to an edge of a $C$-layer of $H$, it follows, by Lemma~\ref{L:gen.propert.}(v and vii), that each edge $e$ of $P$ that is $\Theta$-equivalent to an edge of $C$ distinct from $ab$ is $\Theta$-equivalent to exactly one edge of $Q$, and that this edge is antipodal to $e$ in $P \cup Q$.  Because the length of $C$ is at least $6$, it follows that there exists $i$ with $0 \leq i < n$ such that the edge $u_iu_{i+1}$ of $P$ is an edge of a $C$-layer $C_i$ of $H$ which is $\Theta$-equivalent to an edge $cd$ of $Q$ with $d \neq v$.  Hence the path $Q[d,v]$ is a path of $G_{u_iu_{i+1}} := G[co_G(U_{u_iu_{i+1}})]$.

By Theorem~\ref{T:hypernet./conv.hull(isom.cycle)} and the fact that $G$ is ph-homogeneous, the convex hull of $D$ is a quasi-hypertorus.  By the Distance Property of the Cartesian product, there exists in $co_G(D)$ an $(u,v)$-geodesic in $X$ which is the union $R_0 \cup R_1$ of the path $R_0$ of some $C$-layer $C_0$ of $H$ joining $u$ to some vertex $v'$ of $B_X$, and a $(v',v)$-geodesic $R_1$ in $B_X$.  Let $R_0 := \langle r_1,\dots,r_p\rangle$ with $r_1 = u$,\; $r_p = v'$ and $p \geq 3$.  Let $j$ be such that the edge $r_jr_{j+1}$ is $\Theta$-equivalent to the edge $u_iu_{i+1}$ of $P$.  Because $p \geq 3$, there exists $k$ with $1 \leq k \leq p$ which is distinct from $j$ and $j+1$.  Hence the vertex $r_k$ cannot be incident to an edge $\Theta$-equivalent to the edge $u_iu_{i+1}$ since $C_0$ is a convex cycle of $G$.  It follows that $co_G(D)$ is neither a hypercube nor a prism $T \Box K_2$ one of whose $K_2$-layer is $\langle u_i,u_{i+1}\rangle$.  

Therefore $co_G(D)$ is the Cartesian product of a quasi-hypertorus with some convex cycle which passes through $r_k$ and which contains an edge $\Theta$-equivalent to $u_iu_{i+1}$.  By Remark~\ref{R:notation}.2, the only cycle which has this properties is $C_0$.  It follows that every vertex of the path $Q[d,v]$ belongs to a $C$-layer of $co_G(D)$.  Therefore, because $C_0$ is cycle of $H$, the path $Q[d,v]$ is a path of $H \cap G_{u_iu_{i+1}}$, and thus of $X$.  This implies in particular that $d \in V(X)$, contrary to the hypothesis and the fact that $d \neq v$.  Consequently $X-U_{ab}$ is a separator of $G_{ab}$.\\

(iii)$\Rightarrow$(ii):\; Conversely, assume that $G$ satisfies (iii).  Let $u \in co_G(U_{ab})-U_{ab}$.  Then $u$ is a vertex of some bulge $X$ of $co_G(U_{ab})$.  By (iii), there exists a convex $H \in \mathbf{Cyl}(G,ab)$ such that $X = H-W_{ba}$.  Then $H = C \Box F$, where $C \in \mathbf{C}(G,ab)$ and $F$ is some partial cube.  Let $C_u$ be the $C$-layer of $H$ which contains $u$, and $P_u := C_u-W_{ba}$.  We will show that $P_u$ has the properties (SPS1) and (SPS2), which will implies that $U_{ab}$ is strongly ph-stable, and consequently that $G$ satisfies condition (ii) of Theorem~\ref{T:charact.}.

(SPS1):\; Let $x \in \mathcal{I}_G(U_{ab})$, and $R$ an $(x,u)$-geodesic.  By the definition of a bulge, $R$ passes through a vertex of one of the two components $A_0$ and $A_1$ of $X[U_{ab}]$, say $A_0$.  Let $v$ be the endvertex of $P_u$ in $A_1$.  Then, by the properties of the Cartesian product and the fact that $X-U_{ab}$ is a separator of $G_{ab}$ by (HNB2), $R \cup P_u[u,v]$ is a geodesic, which proves that $P_u$ satisfies (SPS1).

(SPS2):\; Let $x, y \in U_{ab}$ such that there exists an $(x,y)$-geodesic $R$ which passes through $u$.  Then, by (HNB2) and the definition of a bulge, there exists $i = 0$ or $1$ such that $R[x,u]$ and $R[y,u]$ pass through some vertices $z_x \in A_i$ and $z_y \in A_{1-i}$, respectively.  By the properties of the Cartesian product, $P_u$ is a subpath of some $(z_x,z_y)$-geodesic $Q$.  It follows that $R[x,z_x) \cup Q \cup R[z_y,y]$ is an $(x,y)$-geodesic which contains $P_u$ as a subpath.  Hence $P_u$ satisfies (SPS2).
\end{proof}

As we already noticed, the equivalences of the assertions (i),(ii) and (iii) of Theorem~\ref{T:charact.} are immediate consequences of Theorem~\ref{T:charact.}.  In particular, it follows that the hypercylinder defined by Lemma~\ref{L:bulge/ab-cycle}, and denoted by $\mathbf{Cyl}[X]$, is the hypercylinder introduced in (HNB1).  In order to prove that assertion (iv) of Theorem~\ref{T:charact.} is also equivalent to the other assertions of this theorem, we will study the gated sets in a Peano partial cube.

\subsection{Gated sets in Peano partial cubes}\label{SS:gated/hypernet.}

We say that a
subgraph $H$ of a graph $G$ is \emph{$\Gamma$-closed} if every convex
cycle that has at least three vertices in common with $H$ is a cycle
of $H$.

\begin{pro}\label{P:gated/G-closed}
Any gated subgraph is $\Gamma$-closed.
\end{pro}

\begin{proof}
Let $F$ be a subgraph of a graph $G$.  Assume that $F$ is gated.  
Suppose that there
is a convex cycle $C = \langle x_{1},\ldots,x_{2n},x_{1} \rangle$ of
$G$ such that $|V(C \cap F)| \geq 3$ and $C \nsubseteq F$.  Without
loss of generality we can suppose that $C \cap F = \langle
x_{p},\ldots,x_{2n} \rangle$ with $2n-p \geq 2$.  Because $F$ is
convex, it follows that $n < p$.  Let $i$ be the largest integer less
than or equal to $(p-1)/2$.  Then, because $C$ is convex, $\langle
x_{2n},x_{1},\ldots,x_{i} \rangle$ and $\langle
x_{i},x_{i+1},\ldots,x_{p} \rangle$ are the only
$(x_{i},x_{2n})$-geodesic and $(x_{i},x_{p})$-geodesic, respectively.
It follows that $x_{i}$ has no gate in $F$, and thus that $F$ is not
gated, contrary to the assumption.
\end{proof}

The converse is true if $G$ is a netlike partial cube (see
\cite[Theorem 6.2]{P05-1}).  This is generally not true if $G$ is not netlike.  However we have the following result.

\begin{thm}\label{T:hypernet.=>G-closed+conv.=gated}
A convex subgraph of a Peano partial cube is gated
if and only if it is $\Gamma$-closed.
\end{thm}

\begin{proof}    
The necessity is Proposition~\ref{P:gated/G-closed}.  Conversely, let
$F$ be a convex subgraph of a Peano partial cube
$G$.  Assume that $F$ is not gated.  Then there exist a vertex $u$ of
$G-F$ and $x, a \in V(F)$ such that $d_{G}(u,x) = d_{G}(u,V(F)) =: k$,
$d_{G}(u,a) < d_{G}(u,x) + d_{G}(x,a)$, and $k$ is minimal with
respect to these properties (see Figure~\ref{F:Thm.3.29}).  Without loss of generality we can
suppose that $a$ is chosen so that $d_{G}(x,a)$ is minimal with
respect to the preceding properties.  Note that $d_{G}(x,a) \geq 2$
since $k = d_{G}(u,V(F))$ and $G$ is bipartite.  Also note that $d_{G}(u,x)$ and
$d_{G}(u,a)$ are greater than $1$, since otherwise $u$ would belong to
$V(F)$ by the convexity of $F$, contrary to the hypothesis $u \notin
V(F)$.

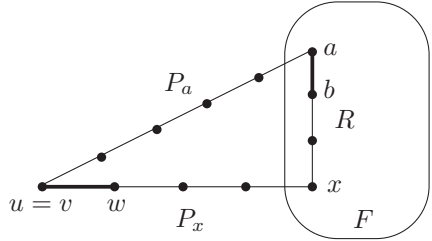
\begin{figure}[!h]
    \centering
{\tt    \setlength{\unitlength}{0.70pt}
\begin{picture}(247,148)
\thinlines    \put(196,16){$F$}
              \put(198,74){\oval(78,128)}
              \put(94,90){$P_{a}$}
              \put(186,70){$R$}
              \put(100,16){$P_{x}$}
              \put(63,25){$w$}
              \put(10,25){$u=v$}
              \put(182,35){$x$}
              \put(180,86){$b$}
              \put(180,109){$a$}
              \put(174,91){\line(0,-1){52}}
              \put(67,38){\line(1,0){106}}
\linethickness{0,5mm}
              \put(27,38){\line(1,0){39}}
              \put(174,110){\line(0,-1){22}}
\thinlines    \put(28,38){\circle*{5}}
              \put(174,38){\circle*{5}}
              \put(104,38){\circle*{5}}
              \put(67,38){\circle*{5}}
              \put(138,38){\circle*{5}}
              \put(174,111){\circle*{5}}
              \put(174,88){\circle*{5}}
              \put(174,63){\circle*{5}}
              \put(60,54){\circle*{5}}
              \put(90,69){\circle*{5}}
              \put(117,83){\circle*{5}}
              \put(145,97){\circle*{5}}
              \put(174,112){\line(-2,-1){145}}
\end{picture}}
\caption{Illustration of the first part of the proof of Theorem~\ref{T:hypernet.=>G-closed+conv.=gated}.}
\label{F:Thm.3.29}
\end{figure}

Let $P_{x}$, $P_{a}$ and $R$ be a $(u,x)$-geodesic, a $(u,a)$-geodesic
and an $(x,a)$-geodesic, respectively.  Then $R$ is a path of $F$
since $F$ is convex.  By the minimality of $k$ and of $d_{G}(x,a)$,\;
$C := P_{x} \cup P_{a} \cup R$ is a cycle.  If $C$ is convex then $F$
is not $\Gamma$-closed because $d_{G}(x,a) \geq 2$.

Assume that $C$ is not convex, and let $b$ be the neighbor of $a$ in
$R$.  By the minimality of $d_{G}(x,a)$ and the fact that $G$ is
bipartite, we have $d_{G}(u,b) = d_{G}(u,a)~+~1$.  Then, there
is an edge $vw$ of $P_{x} \cup P_{a}$ which is in relation $\Theta$
with the edge $ab$.  Because $a \in I_{G}(u,b)$, it follows that 
$vw \notin E(P_{a})$.  Therefore $vw \in E(P_{x})$.  Then  $$d_{G}(v,b) = d_{G}(w,a) = d_{G}(w,b)+1 = d_{G}(v,a)+1.$$  Hence $v = u$ by the minimality of $k$.

Let $y$ be an inner vertex of the $(w,b)$-geodesic $Q := P_{x}[w,x]
\cup R[x,b]$.  Suppose that $y \in U_{ba}$.  Let $y'$ be the neighbor
of $y$ in $U_{ab}$.  Then $y' \in I_{G}(u,y)$ because the edges $uw$
and $y'y$ are $\Theta$-equivalent.  We distinguish two cases:

\textbullet\; If $y \in V(P_{x}[w,x]-x)$, then $y' \notin F$ and $y'
\in I_{G}(u,x) \cap I_{G}(u,a)$, contrary to the minimality of $k$.

\textbullet\; If $y \in V(R[x,b])$, then $y' \in I_{G}(x,a)$, and thus
$y' \in F$ since $F$ is convex, and moreover $d_{G}(u,y') <
d_{G}(u,x)+d_{G}(x,y')$, contrary to the choice of $a$.

\noindent Therefore the only vertices of $Q$ that belong to $U_{ba}$
are its endvertices $w$ and $b$.

By what we proved above, the inner
vertices of $Q$ belong to $co_{G}(U_{ba})-U_{ba}^{G}$.  It follows that
$Q$ is a path of a bulge $X$ of $co_{G}(U_{ba})$.  Because $G$ is a Peano partial cube, there exists, by the Characterization Theorem, a convex $H := \mathbf{Cyl}[X]$.  It follows that
$Q$ is a geodesic of $H-W_{ab}$.  Hence the cycle $C$ is a subgraph of $H$.

On the other hand, $H = A \Box B$, where $A \in \mathbf{C}(G,ab)$ and $B$ is a component of $X[U_{ba}]$.
Because $b$ is the only vertex of $R[x,b]$ which belongs to $U_{ba}$,
it follows that $\langle a,b,c\rangle$, where $c$ is the neighbor of $b$ in
$R[x,b]$, is a path of some $A$-layer $A_a$ of $H$.

Suppose that $A_a$ is a cycle of $F$.  Then, because $F$ is convex,
it follows that $F \cap H$ is the Cartesian product of $A$ with some convex subgraph of $B$.  Let $A_x$ be the $A$-layer of $H$ containing $x$, and let $u'$ be the vertex of $A_x$ such that $d_G(u,u') = d_G(u,V(A_x))$.  Then $u' \in U_{ab}$ since $u \in U_{ab}$.  Because $A_x$ is a cycle of $F$, it follows that $k \leq d_G(u,u')$.  On the other hand, by the fact that $u' \neq x$ since $x \in U_{ba}$, by the Distance Property of Cartesian product and the fact that $d_G(u,x) = k$ by assumption, we have $$d_G(u,u') < d_G(u,u')+d_G(u',x) = d_G(u,x) = k,$$ a contradiction with the above inequality.

Therefore $A$ is not a cycle of $F$.  On the other hand, $A$
is a convex cycle of $G$ whose three vertices
$a, b, c$ belong to $F$.  It follows that $F$ is not $\Gamma$-closed.
\end{proof}

Note that a partial cube $G$ being an isometric subgraph of some
hypercube $Q$, any hypercube in $G$ is then a convex subgraph of $Q$,
and thus is gated in $G$.  We now generalize
\cite[Corollary~6.4]{P05-1} stating that any convex cycle of a netlike
partial cube is gated.

We first have to characterize the finite regular Peano partial cubes.  We recall that the finite regular median graphs are the hypercubes, 
and that the finite regular netlike partial cubes are the hypercubes 
and the even cycles.  More generally we have:

\begin{thm}\label{T:reg. hypernet./str.sem.-periph./quasi-hypert.}
Let $G$ be a compact Peano partial cube.  The following assertions are equivalent:

\textnormal{(i)}\; $G$ is a quasi-hypertorus.

\textnormal{(ii)}\; $G$ is finite and regular.

\textnormal{(iii)}\; $G$ is strongly semi-peripheral.
\end{thm}

\begin{proof}
The implications (i) $\Rightarrow$ (ii) and (iii) are obvious.

(ii) $\Rightarrow$ (i):\; Let $G$ be a finite $n$-regular Peano partial cube.  We proceed by induction on $|V(G)|$.  The result is obvious if $|V(G)|
= 1$.  Suppose it this is true if $|V(G)| \leq k$ for some positive
integer $k$.  Let $G$ be such that $|V(G)| = k+1$.  Because $G$ is finite, it
has a semi-periphery $W_{ab}$ for some $ab \in E(G)$, and this
semi-periphery may be a periphery.  So we distinguish two cases.

\emph{Case 1.}\; $W_{ab} = U_{ab}$.

By
Theorem~\ref{T:hypernet./gat.amalg.+cart.prod.}, $G[U_{ab}]$ is ph-homogeneous since it is a convex subgraph of $G$.
Moreover it is $(n-1)$-regular with $|V(G[U_{ab}])| \leq k$.  Therefore, by the induction
hypothesis, $G[U_{ab}]$ is a $(n-1)$-cube, or a $p$-prism or a
$p$-torus depending on whether $n-1 = 2p$ or $2p+1$.  By the properties
of partial cubes, $G = G[U_{ab}] \Box K_{2}$ because $W_{ab} =
U_{ab}$.  It follows that $G$ is $n$-cube, or a $p$-torus or a
$p$-prism depending on whether $n = 2p$ or $2p+1$.

\emph{Case 2.}\; $W_{ab} \neq U_{ab}$.

Let $X$ be a bulge of $W_{ab}$, and let $H := \mathbf{Cyl}[X]$.  Then $H = C \Box A$, where $C \in \mathbf{C}(G,ab)$ and $A$ is a component of $X[U_{ab}]$.  For each $x \in V(X)-U_{ab}$,\; $n = \delta_{G}(x) =
\delta_{H}(x) = \delta_{A_{x}}(x)+2$, where $A_{x}$ is the 
$A$-layer of $H$ containing $x$.  It follows that $A$ is 
$(n-2)$-regular, and thus $H$ is $n$-regular.  Consequently $G = H$ 
because $G$ is connected and also $n$-regular.  On the other hand $A$ is a convex subgraph of $G$, and thus is ph-homogeneous by Theorem~\ref{T:hypernet./gat.amalg.+cart.prod.}.  Moreover $|V(A)| \leq k$.  Hence, by the induction hypothesis, $A$ is a $(n-2)$-cube, or a $p$-prism or a
$p$-torus depending on whether $n-2 = 2p$ or $2p+1$.  It follows that $H$, being the Cartesian product of $C$ with $A$, is a $n$-cube, or a $p$-torus or a
$p$-prism depending on whether $n = 2p$ or $2p+1$.\\

(iii) $\Rightarrow$ (i):\; Let $G$ be a compact strongly semi-peripheral Peano partial cube.  Suppose that $G$ is not a quasi-hypertorus.  We construct a sequence $H_0, H_1,\dots$ of non-empty convex subgraphs of $G$ such that, for each non-negative integer $n$,\; $H_n = C_{n+1} \Box H_{n+1}$, where $C_{n+1} \in \mathbf{C}(H_n)$.

Put $H_0 := G$.  Suppose that $H_0,\dots,H_n$ has already been constructed for some non-negative integer $n$.  Because $H_n$ is a non-empty convex subgraph of $G$, it follows that $H_n$ is also a compact strongly semi-peripheral Peano partial cube.  

Suppose that $W_{ab}^{H_n}$ is a periphery for any $ab \in E(H_n)$, that is, that $H_n$ is strongly peripheral.  Then $H_n$ is a hypercube by Proposition~\ref{P:strong.-periph./hypercube}.  Moreover $H_n$ is a finite hypercube since  it contains no isometric rays by \cite[Corollary 3.15]{P09-1}.  It follows that $G$ is a quasi-hypertorus because $G = H_0$ or $G = C_0 \Box \dots \Box C_{n} \Box H_n$ depending on whether $n$ is or is not equal to $0$.  This yields a contradiction with the hypothesis that $G$ is not a quasi-hypertorus.

Hence $W_{ab}^{H_n}$ is not a periphery for some $ab \in E(H_n)$.    Let $X$ be a bulge of $co_{H_n}(U_{ab}^{H_n})$, and $H := \mathbf{Cyl}[X]$.  Suppose that $H_n \neq H$.  Then there exists an edge $uv \in \partial_{H_n}(V(H))$ with $u \in U_{ab}^{H_n} \cap V(H)$ and $v \in W_{ab}^{H_n}$.  By Lemma~\ref{L:gen.propert.}(ix), $uv$ is not $\Theta$-equivalent to an edge of $H$.  Let $C$ be the cycle of $H$ of length greater than $4$ which contains an edge $\Theta$-equivalent to $ab$ and which passes through $u$, and let $cd$ be an edge of $C$ which is not $\Theta$-equivalent to $ab$ and such that $u \in W_{cd}^{H_n}$.  Then $v \notin co_{G}(U_{cd}^{H_n})$ since $W_{cd}^{H}$ is a bulge of $W_{cd}^{H_n}$.  It follows that $W_{vu}^{H_n} \subset W_{cd}^{H_n}$, contrary to the assumption that $H_n$ is strongly semi-peripheral.  Therefore $H_n = H$.  It follows that $H_n = C_{n+1} \Box H_{n+1}$, where $C_{n+1} \in \mathbf{C}(H_n,ab)$ and $H_{n+1}$ is a component of $X[U_{ab}^{H_n}]$.  Then $H_{n+1}$ is a convex subgraph of $H_n$, and thus it is a compact strongly semi-peripheral Peano partial cube since so is $H_n$.

$H_0 \supset H_1 \supset \dots$ is a decreasing sequence of non-empty convex subgraphs of $G$.  Hence $H := \bigcap_{n \in \mathbb{N}}H_n \neq \emptyset$ since $G$ is compact.  Therefore $G = \cartbig_{n \in \mathbb{N}}C_n \Box H$.  It follows that $G$ contains the infinite convex hypertorus $\cartbig_{n \in \mathbb{N}}C_n$, and thus contains an isometric ray, contrary to \cite[Corollary 3.15]{P09-1}.

Consequently $G$ is a quasi-hypertorus.
\end{proof}

We will give another equivalent condition later. According to the above proof we immediately infer the following corollary.

\begin{cor}\label{C:reg.hypernet.}
A finite Peano partial cube is
$n$-regular for some positive integer $n$ if and only if it is a
$n$-cube, or a $p$-torus or a $p$-prism depending on whether $n = 2p$
or $2p+1$.
\end{cor}

\begin{thm}\label{T:conv.reg.H-subgr.=>gated}
Let $G$ be a Peano partial cube.  Then any finite convex regular
subgraph of $G$ is gated.
\end{thm}

\begin{proof}
\begin{sloppypar}
Let $F$ be a finite convex regular subgraph of $G$.  Then $F$ is ph-homogeneous since it is convex.  By Theorem~\ref{T:reg. hypernet./str.sem.-periph./quasi-hypert.}, $F$ is a quasi-hypertorus.  Let $C$ be a convex cycle of $G$ which has at least three vertices in
$F$.  If $C$ is a $4$-cycle, then $C$ is clearly a cycle of $F$ since
$F$ is convex.  Suppose that the length of $C$ is greater than $4$.
\end{sloppypar}

Because $C$ and $F$ are convex, $C \cup F$ has at least two adjacent edges $ab$ and $ac$, and, 
since $F$ is a quasi-hypertorus, at least one of them, say $ab$, is an edge of some convex cycle $D$ of $F$ of length greater than $4$.  Clearly $F$ is the Cartesian product of $D$ with some regular partial cube.  By the Characterization Theorem, $F- W_{ba}$ is contained in a bulge $X$ of $co_G(U_{ab})$, and moreover $\mathbf{Cyl}[X] = D \Box A$, where $A$ is a component of $X[U_{ab}]$.

On the other hand, $F$ also has the vertex $c$ in common with $C$.  Hence $c$ also belongs to a $D$-layer $D_c$ of $H$.  This vertex cannot belongs to $U_{ab}$, since otherwise it would belong to a $4$-cycle, contrary to the fact that $C$ is convex and of length greater than $4$.  Then $c \notin U_{ab}$.  It follows that $C = D_c$ by Remark~2 at the beginning of the proof of Theorem~\ref{T:charact.} stating that each vertex of $X-U_{ab}$ lies in at most one convex cycle containing an edge $\Theta$-equivalent to $ab$.  Hence $C$ is a cycle of $F$.  Therefore $F$ is $\Gamma$-closed, and thus gated by Theorem~\ref{T:hypernet.=>G-closed+conv.=gated}.
\end{proof}

\begin{rem}\label{R:conv./gated}
Because of Theorems~\ref{T:reg. hypernet./str.sem.-periph./quasi-hypert.} and~\ref{T:conv.reg.H-subgr.=>gated}, any convex quasi-hypertorus of a Peano partial cube is always gated.  Hence, in several results of this paper, the term \textquotedblleft gated\textquotedblright will often be implicit, and thus not plainly expressed.
\end{rem}

\begin{thm}\label{T:faithf.hypertor.=>gated}
Any faithful quasi-hypertorus of a Peano partial cube is gated.
\end{thm}

\begin{proof}
Let $F$ be a faithful quasi-hypertorus of a Peano partial cube $G$.  By Theorem~\ref{T:conv.reg.H-subgr.=>gated}, it suffices to prove that $F$ is convex, and for that, it clearly suffices to prove that any cycle
of $F$ of length greater than $4$ that is convex in $F$ is also
convex in $G$.  Let $C$ be such a cycle of $F$.  Then $C$ is faithful in $G$ since so is $F$.  
Suppose that $C$ is
not convex in $G$, and let $C'$ be its convex hull in $G$.  Then, by Theorem~\ref{T:hypernet./conv.hull(isom.cycle)}, $C'$
is a quasi-hypertorus.
Because $C \neq C'$, there exists a path $\langle u,v,w \rangle$ of
$C$ such that $uv$ and $vw$ are edges of different layers of $C'$.
Hence, by the $4$-Cycle Property of Cartesian product since $C'$ is a Cartesian product,
$u, v, w$ are vertices of exactly one $4$-cycle $\langle u,v,w,x,u
\rangle$ of $C'$.  Because $C$ is an isometric cycle, if $u', v', w'$
are the vertices of $C$ which are antipodal to $u, v, w$,
respectively, then the edges $v'u'$ and $w'v'$ are $\Theta$-equivalent
to $uv$ and $vw$, respectively, and thus to $xw$ and $ux$,
respectively.  It follows that $x \in I_{G}(v',u) \cap I_{G}(v',w)
\cap I_{G}(u,w)$, and thus $x$ is the median of the triple $(u,w,v')$
of vertices of $C$, contrary to the facts that $C$ is median-stable
and $x \notin V(C)$.  Therefore $C = C'$.
\end{proof}

From now on we will use the following notation.  Let $G$ be a Peano partial cube, and $ab$ an edge of $G$.  We denote:
\begin{align*}
\mathbf{Cyl}[G,ab] &:= \{\mathbf{Cyl}[X]: X \text{ bulge of}\; co_G(U_{ab})\}\\
\mathbf{Cyl}[G] &:= \bigcup_{ab \in E(G)}\mathbf{Cyl}[G,ab].
\end{align*}

\begin{thm}\label{T:C-subgr./gated}
Let $G$ be a Peano partial cube.  Then any convex hypercylinder of $G$ is gated in $G$.
\end{thm}

\begin{proof}
Let $H$ be a convex hypercylinder of $G$.  Then $H = C \Box A$, where $C \in \mathbf{C}(G,ab)$ and $A$ is a partial cube.  By Theorem~\ref{T:hypernet.=>G-closed+conv.=gated}, it suffices to prove that $H$ is $\Gamma$-closed in $G$.  Let $D$ be a convex cycle of $G$ which has at least three vertices in common with $H$.  If $D$ is a $4$-cycle, then $D$ is a cycle of $H$ since $H$ is convex, and then we are done.  Assume that the length of $D$ is at least $6$.  Because $D$ is convex it follows, by the Distance property of the Cartesian product, that $D \cap H$ is a geodesic of a $C$-layer or a $A$-layer of $H$.

\emph{Case 1.}\; $D \cap H$ is a geodesic of a $C$-layer $C_0$ of $H$.

Because $C_0$ is a convex cycle of $G$, it follows by Theorem~\ref{T:conv.reg.H-subgr.=>gated} that $C_0$ is gated, and thus $\Gamma$-closed.  Hence $D = C_0$, and thus $D$ is a cycle of $H$.

\emph{Case 2.}\; $D \cap H$ is a geodesic of a $A$-layer $A_0$ of $H$.

Then there exists an edge $ab$ of $C$ such that $V(D) \subseteq W_{ab}$, and $A_0$ is one of the two components of the bulge $X$ of $co_G(U_{ab})$ such that $H = \mathbf{Cyl}[X]$.  Suppose that $V(D) \nsubseteq U_{ab}$.  Then there exists a bulge $Y$ of $co_G(U_{ab})$ which contains a geodesic of $D-X$.  It follows that $Y$ is not a separator of $G_{ab}$, contrary to (HNB2).  Therefore $V(D) \subseteq U_{ab}$.

Because $D \cap A_0$ has at least three vertices, it contains a geodesic $\langle x_0,x_1,x_2\rangle$ of length $2$.  Let $Y$ be the bulge of $co_G(U_{x_1x_0})$ which contains $x_2$.  Then $F := \mathbf{Cyl}[Y]$ is the Cartesian product of $D$ with a component $B$ of $Y[U_{x_1x_0}]$, and $B$ contains the $C$-layer $C_{x_2}$ of $H$ passing through $x_2$.  It follows that $F$ contains a subgraph $K$ isomorphic to $C \Box D$ which contains $C_{x_2}$ and $D$.  Consequently $H$ also contains $K$, and thus $D$ is a cycle of $H$.
\end{proof}

\subsection{Completion of the proof of Theorem 4.5}\label{SS:compl.Thm4.5}

\begin{proof}[\textnormal{\textbf{Proof of the implications (iii)$\Rightarrow$(iv)$\Rightarrow$(v)$\Rightarrow$(ii) of Theorem~\ref{T:charact.}}}]

\begin{sloppypar}
Let $G$ be a partial cube that satisfies (iii), $ab$ an edge of $G$, and $X$ a bulge of $co_G(U_{ab})$.  By (iii), there exist a convex $H \in \mathbf{Cyl}(G,ab)$ such that $X = H-W_{ba}$.  By the equivalence of conditions (i) and (iii) of Theorem~\ref{T:charact.} that we already proved, $G$ is ph-homogeneous.  It follows that $H$ is gated by Theorem~\ref{T:C-subgr./gated}.  Therefore $G$ satisfies (iv).
\end{sloppypar}

(iv)$\Rightarrow$(v):\; Let $ab$ be an edge of a Peano partial cube $G$,\; $u \in \mathcal{I}_G(U_{ab})-U_{ab}$, and $X$ the bulge of $co_G(U_{ab})$ that contains $u$.  By (iv), there exists a gated $H \in \mathbf{Cyl}(G,ab)$ such that $X = H-W_{ba}$.  Then $H$ is the Cartesian product of an even cycle $C$ of length greater than $4$ with  some partial cube.  Denote by $C_u$ the $C$-layer of $H$ passing through $u$.  Then $C_u$ is convex in $H$, and thus in $G$ since $H$ is convex.  Hence $C_u$ is gated by Theorem~\ref{T:conv.reg.H-subgr.=>gated}.

(v)$\Rightarrow$(ii):\; Let $G$ be a partial cube which satisfies (v).  Let $ab \in E(G)$ and $u \in \mathcal{I}_G(U_{ab})-U_{ab}$.  Then $u$ lies on a gated cycle $C_u \in \mathbf{C}(G,ab)$.  The $U_{ab}$-path $P_u := C_u-W_{ba}$ is convex since so are $C_u$ and $W_{ab}$. We will show that $P_u$ is the $U_{ab}$-geodesic associated wit $u$.

In the following, for each $x \in V(G)$, we denote by $g(x)$ its gate in $C_u$.  Clearly $g(x)$ belongs to $W_{ab}$ or $W_{ba}$ depending on whether $x$ belongs to $W_{ab}$ or $W_{ba}$, because these sets are convex.  It follows that, if $x \in U_{ab}$, and if $x'$ is the neighbor of $x$ in $U_{ba}$, then we can easily prove that $g(x)$ and $g(x')$ are adjacent, which implies that $g(x) \in U_{ab}$ and $g(x') \in U_{ba}$. 

(SPS1):\; Let $x \in \mathcal{I}_G(U_{ab})$.  Then $g(x) \in V(P_u)$, and $u \in I_G(g(x),v)$ for some endvertex $v$ of $P_u$.  It follows that $u \in I_G(x,v)$.

(SPS2):\; Let $x, y \in U_{ab}$ be such that $u \in I_G(x,y)$.  Clearly $u \in I_G(g(x),g(y))$.  Moreover $I_G(g(x),g(y)) = V(P_u)$ since $P_u$ is convex and $g(x)$ and $g(y)$ belongs to $U_{ab}$, and thus are the endvertices of $P_u$.  It follows that $P_u$ is a subpath of some $(x,y)$-geodesic.

Consequently $U_{ab}$ is strongly ph-stable.
\end{proof}

\begin{proof}[\textnormal{\textbf{Proof of the equivalence (i)$\Leftrightarrow$(vi) of Theorem~\ref{T:charact.}}}]

\emph{Necessity.}\; Let $G$ be a Peano partial cube.  The first part of condition (vi) is obvious by Theorem~\ref{T:charact.}(v).  The second part of this condition is an immediate consequence of Theorems~\ref{T:hypernet./conv.hull(isom.cycle)} and \ref{T:conv.reg.H-subgr.=>gated}.

\emph{Sufficiency.}\; Let $G$ be a partial cube which satisfies condition (vi).  Let $ab$ be an edge of $G$, and $u \in \mathcal{I}_G(U_{ab}) - U_{ab}$.  By the first part of (vi), $u$ lies on an isometric cycle $C \in \mathbf{C}(G,ab)$.  Then, by the second part of (vi), the convex hull $F$ of $C$ is a gated quasi-hypertorus of $G$.  Because a quasi-hypertorus is a Peano partial cube, we infer from condition (v) of Theorem~\ref{T:charact.} that $u$ lies on a gated cycle $C_u \in \mathbf{C}(F,ab)$.  This cycle $C_u$ which is gated in $F$ is then gated in $G$ since $F$ is itself gated in $G$.

We deduce that $G$ is a Peano partial cube by the equivalence (i)$\Leftrightarrow$(v) of Theorem~\ref{T:charact.}.
\end{proof}

\subsection{Consequences of Theorem 4.5 and special Peano partial cubes}\label{SS:consequences}

\subsubsection{Quasi-hypertori and antipodal partial cubes}\label{SSS:quasi-hyp.antip.p.c.}

We first give two new characterizations of quasi-hypertori.  We recall that the \emph{eccentricity} of a
vertex $x$ of a graph $G$ is $e_{G}(x) := \max_{y \in V(G)}d_{G}(x,y)$, and that a
\emph{central vertex} of $G$ is a vertex of minimum eccentricity.  We say that the graph $G$ is \emph{self-centered} if all vertices of $G$ are central.  We need a preliminary result that will be useful in the subsequent sections.

\begin{lem}\label{L:phi_ab}
For each edge $ab$ of a Peano partial cube $G$, there exists a unique isomorphism $\phi_{ab}$ of $G[\mathcal{I}_G(U_{ab})]$ onto $G[\mathcal{I}_G(U_{ba})]$ such that $\phi_{ab}(x)$ is the neighbor of $x$ in $U_{ba}$ for every $x \in U_{ab}$.
\end{lem}

\begin{proof}
We construct $\phi_{ab}$ as follows.  For each $x \in U_{ab}$, define $\phi_{ab}(x)$ as the neighbor of $x$ in $U_{ba}$.  Let $x \in \mathcal{I}_G(U_{ab})-U_{ab}$, and denote by $C_x$ the $ab$-cycle which is associated with $x$.  Recall that $C_x$ is a convex subgraph of $G$, and that it is also associated with each of its vertices.  Then $C = \langle x_{1},\dots,x_{2n},x_{1}\rangle$ with $x_{1}, x_{n} \in U_{ab}$ and $x = x_{i}$ for some $i$ with $1 < i < n$.  Put $\phi_{ab}(x_{i}) = x_{2n+1-i}$.  This map is clearly a bijection which preserves the edges, with $\phi_{ab}^{-1} = \phi_{ba}$.  Whence the result.
\end{proof}

We can easily prove that $\phi_{ab}$ is distance-preserving, and thus is an isometry of $G[\mathcal{I}_G(U_{ab})]$ onto $G[\mathcal{I}_G(U_{ba})]$.

We recall the following result which was first proved by Kotzig and Laufer~\cite[Theorem 2]{KoL78}, then, independently, by G\"{o}bel and Veldman~\cite[Proposition 19]{GV86}.

\begin{pro}\label{P:antip./cart.prod.}
The Cartesian product $G \Box H$ of two graphs $G$ and $H$ is antipodal if and only if both $G$ and $H$ are antipodal.
\end{pro}

\begin{thm}\label{T:quasi-hypert./antipod./self-center.}
Let $F$ be a finite convex subgraph of a Peano partial cube $G$.  The following assertions are equivalent:

\textnormal{(i)}\; $F$ is a quasi-hypertorus.

\textnormal{(ii)}\; $F$ is antipodal.

\textnormal{(iii)}\; $F$ is self-centered.
\end{thm}

\begin{proof}
By Theorem~\ref{T:hypernet./gat.amalg.+cart.prod.}, $F$ is a Peano partial cube.  Then, without loss of generality, we can assume that $G$ is finite and that $F = G$.

(i) $\Rightarrow$ (ii):\; Because $K_2$ and any even cycle are bipartite  antipodal partial cubes, it follows, by Proposition~\ref{P:antip./cart.prod.}, that any quasi-hypertorus is an  antipodal partial cube.

(ii) $\Rightarrow$ (iii):\; Assume that $G$ is antipodal.  Then $G$ is self-centered, because $I_G(x,\overline{x}) = V(G)$ with $e_G(x) = d(x,\overline{x}) = \mathrm{diam}(G)$ for every vertex $x$ of $G$.

(iii) $\Rightarrow$ (i):\; Assume that $G$ is self-centered.  Suppose that $G$ is not strongly semi-peripheral.  Then, because $G$ is finite, there exist an edge $ab$ of $G$ such that $W_{ab} \neq \mathcal{I}_G(U_{ab})$, and such that, for each edge $uv \in \partial_{G[W_{ab}]}(\mathcal{I}_G(U_{ab}))$ with $u \in \mathcal{I}_G(U_{ab})$, $W_{vu}$ is a semi-periphery, i.e., $W_{vu} = \mathcal{I}_G(U_{vu})$.  Let $uv$ be an edge as above.  Let $w$ be a vertex of $G$ such that $d_G(u,w) = e_G(u)$.

Suppose that $w \in W_{ba} \cup \mathcal{I}_G(U_{ab})$.  Then, clearly, $w \in W_{ba}$ with $e_G(u) = d_G(u,w) \geq d_G(u,\phi_{ab}(u))$.  We cannot have $v \in I_G(u,w)$, since otherwise $v \in W_{ba} \cup \mathcal{I}_G(U_{ab})$ because this set is convex in $G$, contrary to the choice of $v$.  It follows that $$e_G(v) \geq d_G(v,w) = d_G(u,w)+1 = e_G(u)+1,$$ contrary to the fact that $G$ is self-centered.

Therefore $w \in W_{ab}-\mathcal{I}_G(u_{ab})$.  Hence there is an edge $xy \in \partial_{G[W_{ab}]}(\mathcal{I}_G(U_{ab}))$ with $x \in \mathcal{I}_G(U_{ab})$ (note that the subgraph $G[W_{ab}-\mathcal{I}_G(u_{ab})]$ may have several components).  Because $U_{yx}$ is ph-stable since $ph(G) \leq 1$, it follows that $w \in U_{yx}$.  Let $w' := \phi_{yx}(w)$.  Then $w' \in I_G(w,u)$.  Denote by $X$ the bulge of $\mathcal{I}_G(u_{ab})$ which contains $u$.  By Theorem~\ref{T:charact.}(iv), $\mathbf{Cyl}[X]$ is gated.  Moreover the gate of $w'$ in $\mathbf{Cyl}[X]$ belongs to $\mathcal{I}_G(U_{ab})$ since $W_{yx} \cup \mathcal{I}_G(U_{xy}) \subseteq W_{ab}$.  It follows, because $G$ is self-centered, that $$e_G(w) = e_G(u) = d_G(w,u) = d_G(w',u)+1 < d_G(w',\phi_{ab}(u))+1 \leq e_G(w'),$$ contrary to the fact that $G$ is self-centered.

Consequently $G$ is strongly semi-peripheral, and thus it is a quasi-hypertorus by Theorem~\ref{T:reg. hypernet./str.sem.-periph./quasi-hypert.}.
\end{proof}

In other words we have:

\begin{cor}\label{C:antip.Peano.p.c.}
The partial cube that are both antipodal and Peano are the quasi-hypertori.
\end{cor}

\begin{figure}[!h]
    \centering
{\tt    \setlength{\unitlength}{1.30pt}
\begin{picture}(128,134)
\thinlines   
              \put(115,94){\line(0,-1){55}}
              \put(115,95){\line(-2,1){51}}
              \put(115,38){\line(-4,3){21}}
              \put(115,94){\line(-4,-3){21}}              
              \put(64,13){\line(2,1){53}}
              \put(62,14){\line(-2,1){51}}
              \put(64,39){\line(0,-1){26}}

               \put(12,40){\line(3,2){22}}
              \put(12,94){\line(4,-3){21}}
              \put(64,39){\line(2,1){30}}
              \put(64,39){\line(-2,1){31}}
              \put(64,94){\line(2,-1){30}}
              \put(64,94){\line(-2,-1){31}}
              \put(64,94){\line(0,1){25}}
              \put(94,79){\line(0,-1){24}}
              \put(33,77){\line(0,-1){23}}
              \put(12,94){\line(0,-1){55}}
              \put(12,95){\line(2,1){53}}             
              
              \put(64,13){\circle*{3}}
              \put(115,39){\circle*{3}}
              \put(94,53){\circle*{3}}
              \put(64,39){\circle*{3}}
              \put(33,55){\circle*{3}}
              \put(12,40){\circle*{3}}
              \put(94,78){\circle*{3}}
              \put(64,93){\circle*{3}}
              \put(33,77){\circle*{3}}
              \put(115,94){\circle*{3}}
              \put(64,121){\circle*{3}}
              \put(12,94){\circle*{3}}
              
              \put(64,112){\circle*{3}}
              \put(64,103){\circle*{3}}
              \put(64,29){\circle*{3}}
              \put(64,21){\circle*{3}}
              \put(19,89){\circle*{3}}
              \put(26,83){\circle*{3}} 
              \put(19,45){\circle*{3}}
              \put(26,50){\circle*{3}} 
              \put(101,83){\circle*{3}}
              \put(108,88){\circle*{3}} 
              \put(101,49){\circle*{3}}
              \put(108,44){\circle*{3}}            
             
              \put(19,89){\line(2,1){45}}
              \put(19,45){\line(2,-1){45}}
              \put(108,44){\line(0,1){45}} 
              \put(26,83){\line(0,-1){35}}
              \put(64,29){\line(2,1){38}}
              \put(64,103){\line(2,-1){38}}
              
              \put(119,38){$1$}
              \put(37,71){$\overline{1}$}
              \put(119,94){$2$}
              \put(37,56){$\overline{2}$}
              \put(62,126){$3$}
              \put(62,45){$\overline{3}$}
              \put(4,93){$4$}
              \put(87,56){$\overline{4}$}
              \put(4,38){$5$}
              \put(87,73){$\overline{5}$}
              \put(62,3){$6$}
              \put(62,82){$\overline{6}$}
              \put(102,39){$7$}
              \put(29,84){$\overline{7}$}
              \put(103,91){$8$}
              \put(29,44){$\overline{8}$}
              \put(67,111){$9$}
              \put(57,30){$\overline{9}$}
              \put(15,80){$10$}
              \put(98,53){$\overline{10}$}
              \put(15,50){$11$}
              \put(98,73){$\overline{11}$}
              \put(66,21){$12$}
              \put(53,99){$\overline{12}$}
               
\end{picture}}
\caption{The partial cube $B_1$.}
\label{F:B1}
\end{figure}
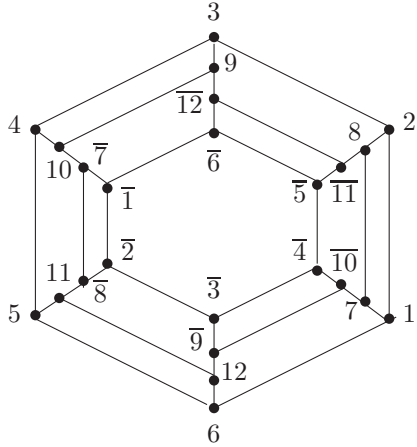

If the regular finite Peano partial cubes are the quasi-hypertori, this not so for regular antipodal partial cubes, as is shown in next theorem.

\begin{thm}\label{T:reg.antip.p.c.}
For any integer $n \geq 3$ there exits a $n$-regular antipodal partial cube which is not a quasi-hypertorus.
\end{thm}

\begin{proof}
For $n = 3$, take the cubic partial cube $B_1$ of~\cite{BoKL}.  It is clearly antipodal (see Figure~\ref{F:B1}), but it is not a quasi-hypertorus.

Suppose that there is a $n$-regular antipodal partial cube $G$ that is not a quasi-hypertorus.  By Proposition~\ref{P:antip./cart.prod.}, the prism $K_2 \Box G$ is a $(n+1)$-regular antipodal partial cube since so are both $K_2$ and $G$ by the induction hypothesis, but it is not a quasi-hypertorus because so is not $G$ by the same hypothesis.
\end{proof}

\subsubsection{Median graphs and netlike partial cubes}\label{SSS:net.p.c.}

Among the numerous characterizations of median graphs we recall the following one.

\begin{thm}\label{T:isom.cycl./med.gr.}\textnormal{(Bandelt~\cite{B82})}
A connected graph $G$ is a median graph if and only if the convex hull
of any isometric cycle of $G$ is a hypercube.
\end{thm}

\begin{pro}\label{P:hypernet./med.gr.}
Let $G$ be a Peano partial cube.  The following assertions are equivalent:

\textnormal{(i)}\; $G$ is a median graph.

\textnormal{(ii)}\; The convex hull of any isometric cycle of $G$ is a hypercube.

\textnormal{(iii)}\; Any convex cycle of $G$ is a $4$-cycle.
\end{pro}

The Characterization Theorem enables us to obtain the following five characterizations of netlike partial cubes.

\begin{pro}\label{P:charact.netl.p.c./bulge}
A Peano partial cube $G$ is netlike if and only if, for each edge $ab$ of $G$ and any bulge $X$ of $co_{G}(U_{ab})$, there exists a gated $C \in \mathbf{C}(G,ab)$ such that $X = C-W_{ba}$.
\end{pro}

\begin{proof}
\emph{Necessity.}\; Let $G$ be a Peano partial cube, $ab$ an edge of $G$, and $X$ a bulge of $co_{G}(U_{ab})$.  Then $G$ is ph-stable as we already saw and, by Proposition~\ref{P:charact.(netlike)}(ii), $\mathbf{Cyl}[X] \in \mathbf{C}(G,ab)$.

\emph{Sufficiency.}\; Assume that, for each edge $ab$ of $G$ and any bulge $X$ of $co_{G}(U_{ab})$, there exists a gated $C \in \mathbf{C}(G,ab)$ such that $X = C-W_{ba}$.  Then $ph(G) \leq 1$ since $G$ is a Peano partial cube, and moreover each vertex in $co_G(U_{ab})$ has degree $2$ in $G[co_G(U_{ab})$.  Therefore $G$ is a netlike partial cube by Proposition~\ref{P:charact.(netlike)}.
\end{proof}

\begin{pro}\label{P:hypernet./netl.}
A partial cube $G$ is netlike if and only if it is a Peano partial cube such that the convex hull of any isometric cycle of $G$ is either this cycle itself or a finite hypercube.
\end{pro}

\begin{proof}
The necessity is clear by what we saw in Subsection~\ref{SS:applications}.  Conversely, assume that $G$ is a Peano partial cube such that the convex hull of any isometric cycle of $G$ is either this cycle itself or a finite hypercube.  Let $ab$ be an edge of $G$, and $X$ a bulge of $co_{G}(U_{ab})$.  By Theorem~\ref{T:charact.}(iv), there exists a gated $H \in \mathbf{Cyl}(G,ab)$ such that $X = H-W_{ba}$.  If $H$ is not a cycle, then it contains a convex prism $P$ over a cycle which belongs to $\mathbf{C}(G,ab)$.  
Clearly $P$ contains an isometric cycle $C$ such that $co_P(C) = P$ (this a particular case of Proposition~\ref{P:TPQ/prop.isom.cycle} that we will prove later).  Because $P$ is convex in $G$ since so is $H$, it follows that $co_P(C) = co_G(C)$, and that $C$ is isometric in $G$.  Hence $co_G(C)$ is either $C$ itself or a hypercube by the assumption, contrary to the fact that $P$ is the prism over a cycle of length greater than $4$.  Therefore $H \in \mathbf{Cyl}(G,ab)$, and thus $G$ is a netlike partial cube by Proposition~\ref{P:charact.netl.p.c./bulge}.
\end{proof}

By the assertion (vi) of the Characterization Theorem and Proposition~\ref{P:hypernet./netl.} we have:

\begin{cor}\label{P:netl./conv.hull}
A partial cube $G$ is netlike if and only if, for each edge $ab$ of $G$, any vertex $u \in \mathcal{I}_G(U_{ab})-U_{ab}$ (resp. $u \in \mathcal{I}_G(U_{ba})-U_{ba}$) lies on an isometric cycle $C_u \in \mathbf{C}(G,ab)$, and the convex hull of any isometric cycle of $G$ is gated and is either this cycle itself or a hypercube.
\end{cor}

The next result has the same formulation as Definition~\ref{D:ph-homog.} of ph-homogeneous partial cubes.

\begin{pro}\label{P:netl./fin.faithf.}
A partial cube $G$ is netlike if and only if $ph(F) \leq 1$ for every finite faithful subgraph $F$ of $G$.
\end{pro}

\begin{proof}
\emph{Necessity.}\; If $G$ is netlike then, by \cite[Proposition 4.4]{P05-2}, any faithful subgraph of a netlike partial cube is also netlike, and thus has a pre-hull number which is at most $1$.

\emph{Sufficiency.}\; Assume that the pre-hull number of any finite faithful subgraph of $G$ is at most equal to $1$.  Then $G$ is a Peano partial cube since any convex subgraph of $G$ is faithful.  Suppose that the convex hull $F$ of some isometric cycle of $G$ is not this cycle nor a hypercube.  Then $F$, which is a quasi-hypertorus by the Characterization Theorem, contains a convex prism $C \Box K_2$ over a cycle $C = \langle c_1,\dots,c_{2n},c_1\rangle$ of length $2n \geq 6$.  Let $V(K_2) = \{0,1\}$, and let $H$ be the subgraph of $C \Box K_2$ induced by the set of vertices $\{(c_i,0): 1 \leq i \leq 2n\} \cup \{(c_1,1),(c_2,1),(c_3,1)\}$ (see Figure~\ref{F:faithful}, where $n = 3$ and $H$ is the subgraph depicted by the thick edges and the big vertices).  Then $H$ is clearly a faithful subgraph of $C \Box K_2$, and thus of $G$, because $C \Box K_2$ is convex in $F$, which is itself convex in $G$.  In the other hand, we can easily check that $ph(H) = 2$, contrary to the assumption.  Therefore the convex hull of any isometric cycle of $G$ is either this cycle itself or a hypercube.  Consequently $G$ is netlike by Proposition~\ref{P:hypernet./netl.}.
\end{proof}

\begin{figure}[!h]
    \centering
{\tt    \setlength{\unitlength}{0.90pt}
\begin{picture}(128,134)
\thinlines   
              \put(115,94){\line(0,-1){55}}
              \put(115,95){\line(-2,1){51}}
              \put(115,38){\line(-4,3){21}}
              \put(115,94){\line(-4,-3){21}}              
              \put(64,13){\line(2,1){53}}
              \put(62,14){\line(-2,1){51}}
              \put(64,39){\line(0,-1){26}}

\linethickness{0,5mm}
               \put(12,40){\line(3,2){22}}
              \put(12,94){\line(4,-3){21}}
              \put(64,39){\line(2,1){30}}
              \put(64,39){\line(-2,1){31}}
              \put(64,94){\line(2,-1){30}}
              \put(64,94){\line(-2,-1){31}}
              \put(64,94){\line(0,1){25}}
              \put(94,79){\line(0,-1){24}}
              \put(33,77){\line(0,-1){23}}
              \put(12,94){\line(0,-1){55}}
              \put(12,95){\line(2,1){53}}             
              
              \put(64,13){\circle*{4}}
              \put(115,39){\circle*{4}}
              \put(94,53){\circle*{6}}
              \put(64,39){\circle*{6}}
              \put(33,55){\circle*{6}}
              \put(12,40){\circle*{6}}
              \put(94,78){\circle*{6}}
              \put(64,93){\circle*{6}}
              \put(33,77){\circle*{6}}
              \put(115,94){\circle*{4}}
              \put(64,121){\circle*{6}}
              \put(12,94){\circle*{6}}
\end{picture}}
\caption{A faithful subgraph of $C_{6} \Box K_{2}$ whose pre-hull number is $2$.}
\label{F:faithful}
\end{figure}
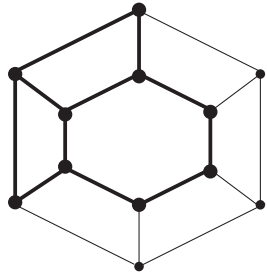

\begin{pro}\label{P:netlike/fin.conv.subgr.}
A partial cube is netlike if and only if so are all its finite convex subgraphs.
\end{pro}

\begin{proof}
Let $G$ be a netlike partial cube, $F$ one of its finite convex subgraph, and $H$ a faithful subgraph of $F$.  Clearly $H$ is a faithful subgraph of $G$, and thus $ph(H) \leq 1$, by Proposition~\ref{P:netl./fin.faithf.}.  Hence $F$ is a netlike partial cube by this same proposition.

Conversely, assume that all finite convex subgraphs of $G$ are netlike.  Let $F$ be a finite faithful subgraph of $G$.  Then $F$ is a faithful subgraph of $co_G(F)$, which is a finite convex subgraph of $G$.  Hence $co_G(F)$ is a netlike partial cube by the assumption, and thus $ph(F) \leq 1$ by Proposition~\ref{P:netl./fin.faithf.}.  Therefore $G$ is netlike by the same proposition.
\end{proof}

\begin{pro}\label{P:netlike/faithful}
Every faithful subgraph of a netlike partial cube is a netlike partial cube.
\end{pro}

\begin{proof}
Let $H$ be a faithful subgraph of some netlike partial cube $G$, and $F$ a finite faithful subgraph of $H$.  Clearly $F$ is faithful in $G$, and thus $ph(F) \leq 1$ by Proposition~\ref{P:netl./fin.faithf.}.  Therefore $H$ is netlike by the same theorem.
\end{proof}

\subsubsection{Cube-free netlike partial cubes}\label{SSS:cube-free}

\begin{thm}\label{T:cube-free netlike}\textnormal{(\cite[Theorem 7.4]{P05-1})}
Let $G$ be a partial cube.  The following assertions are equivalent:

\textnormal{(i)}\; $\mathcal{I}_G(U_{ab})$ and $\mathcal{I}_G(U_{ba})$ induce trees for every edge $ab$ of $G$.

\textnormal{(ii)}\; $G$ is a cube-free netlike partial cube.

\textnormal{(iii)}\; $G$ is a netlike partial cube whose isometric cycles are convex.
\end{thm}

By Proposition~\ref{P:hypernet./netl.} and Theorem~\ref{T:cube-free netlike} we have the following property:

\begin{pro}\label{P:hypernet./cube-free netl.}
A partial cube is a cube-free netlike partial cube if and only if it is a Peano partial cube whose isometric cycles are convex.
\end{pro}

\begin{pro}\label{P:isom.=faithf.=>cube-free}
A partial cube $G$ is cube-free if every isometric subgraph of $G$ is median-stable.
\end{pro}

\begin{proof}
Assume that every isometric subgraph of $G$ is median-stable.  Then $G$ contains no cube $Q_3$ since otherwise $Q_3^-$ would be an isometric subgraph of $Q_3$, and thus of $G$, but $Q_3^-$ is not median-stable in $Q_3$, and thus not in $G$, contrary to the assumption.
\end{proof}

\begin{thm}\label{T:isom.=faithf.<=cube-free}
A netlike partial cube $G$ is cube-free if and only if every isometric subgraph of $G$ is median-stable.
\end{thm}

\begin{proof}
By Proposition~\ref{P:isom.=faithf.=>cube-free}, we only have to prove the necessity.  Assume that $G$ is a cube-free netlike partial cube.  Let $F$ be an isometric subgraph of $G$.  We show that $F$ is median-stable.  Suppose that there is a triplet $(u,v,w)$ of vertices of $F$ which has a median $a$ in $G$ that does not belong to $V(G)$.  Let $b$ be the neighbor of $a$ in some $(a,w)$-geodesic.

Suppose that $u \notin W_{ab}$.  Let $P$ be a $(u,a)$-geodesic.  Then $P$ has an edge $xy$  which is $\Theta$-equivalent to $ab$.  It follows that $d_G(u,b) = d_G(u,a)-1$, and thus $d_G(u,w) = d_G(u,a)+d_G(a,w)-1$, contrary to the fact that $a \in I_G(u,w)$.  Therefore $u \in W_{ab}$, and likely $v \in W_{ab}$.

Because $F$ is isometric in $G$, there exits in $F$ a $(u,w)$-geodesic $P_u$ and a $(v,w)$-geodesic $P_v$ that do not pass through $a$.  Then there are edges $a_ub_u$ and $a_vb_v$ of $P_u$ and $P_v$, respectively, which are $\Theta$-equivalent to $ab$.  Because $F$ is isometric in $G$, it follows that there exist an $(a_u,a_v)$-geodesic in $G$ that does not pass through $a$.  Hence the subgraph of $G$ induced by $I_G(a_u,a) \cup I_G(a,a_v) \cup I_G(a_v,a_u)$ contains a cycle.  A fortiori the subgraph $G[I_G(U_{ab})]$ contains a cycle, contrary to the fact that this subgraph is a tree, by Theorem~\ref{T:cube-free netlike}, since $G$ is a cube-free netlike partial cube.  It follows that $a \in V(F)$.  Hence any isometric subgraph $F$ of $G$ is median-stable.
\end{proof}

As Proposition~\ref{P:netl./fin.faithf.}, the following characterization has the same formulation as Definition~\ref{D:ph-homog.} of ph-homogeneous partial cubes.

\begin{pro}\label{P:cube-free netlike/isom.subgr.}
A partial cube $G$ is a cube-free netlike partial cube if and only if $ph(F) \leq 1$ for every finite isometric subgraph $F$ of $G$.
\end{pro}

\begin{proof}
\emph{Sufficiency.}\; Assume that the pre-hull number of every finite isometric subgraph of a $G$ is at most $1$.  Then $G$ is a netlike partial cube by Proposition~\ref{P:netl./fin.faithf.} since a faithful subgraph is isometric.  Suppose that $G$ contains a $3$-cube $Q$.  Then $Q$ is convex in $G$.  Moreover $Q$ minus any vertex is a $Q_3^-$, that is isometric in $Q$, and thus in $G$.  But $ph(Q_3^-) = 2$, contrary to the fact that the pre-hull number of every isometric subgraph of a partial cube $G$ is at most $1$ by assumption.  Therefore $G$ is cube-free.

\emph{Necessity.}\; Assume that $G$ is a cube-free netlike partial cube.  Let $F$ be a finite isometric subgraph of $G$.  By~\ref{T:isom.=faithf.<=cube-free}, $F$ is a faithful subgraph of $G$.  Therefore $ph(F) \leq 1$, by Proposition~\ref{P:netl./fin.faithf.}. 
\end{proof}

\begin{thm}\label{T:cube-free.net./gat.amal.}
The gated amalgam of two cube-free netlike partial cubes is a cube-free netlike partial cube.
\end{thm}

\begin{proof}
Let $G_0$ and $G_1$ be two cube-free netlike partial cubes.  By \cite[Theorem 6.5]{P07-3}, the gated amalgam of two netlike partial cubes is a netlike partial cube.  Hence  $G$ is a netlike partial cube.  Let $C$ be an isometric cycle of $G$.  Clearly $C$ is an isometric cycle of $G_0$ or $G_1$, say $G_0$, because these graphs are convex in $G$.  Hence, by Theorem~\ref{T:cube-free netlike}, $C$ is convex in $G_0$ since $G_0$ is a cube-free netlike partial cube.  It follows that $C$ is convex in $G$ since $G_0$ is gated in $G$.  Therefore $G$ is cube-free by Theorem~\ref{T:cube-free netlike}.
\end{proof}

As an extension of Proposition~\ref{P:netlike/fin.conv.subgr.} we clearly have:

\begin{pro}\label{P:cube-free netlike/fin.conv.subgr.}
A partial cube is a cube-free netlike partial cube if and only if so are all its finite convex subgraphs.
\end{pro}

\subsection{Remarks}\label{SS:rem.}

\begin{rem}\label{R:gated/convex}
In conditions (iv), (v) and (vi) of the Characterization Theorem, the term ‘‘gated’’ is essential and cannot be replaced by ‘‘convex’’, as is shown by the graph $G$ defined in Remark~\ref{R:fin.subgr.fin.conv.subgr.ph1}.  Recall that this graph has the following properties:

\textbullet\; $G$ is a partial cube.

\textbullet\; $ph(G) \leq 1$.

\textbullet\; $G_n$ is a convex subgraph of $G$ such that $ph(G_n) > 1$ if $n \geq 1$, and thus $G$ is not ph-homogeneous.

\textbullet\; For each edge $ab$ of $G$, any vertex $u \in \mathcal{I}_G(U_{ab})-U_{ab}$ (resp. $u \in \mathcal{I}_G(U_{ba})-U_{ba}$) lies on a $6$-cycle which belongs to $\mathbf{C}(G,ab)$.

\textbullet\; The isometric cycles of $G$ are its $6$-cycles, and these cycles are convex but none of them is gated.

This graph $G$ is not ph-homogeneous.  However, we can note that the above last two properties are analogous to assertion (vi) of the Characterization Theorem with the substitution of convex for gated.

\end{rem}

\begin{rem}\label{R:ph/isom.cycle}
If a partial cube $G$ has a pre-hull number less than or equal to $1$, then, for any edge $ab$ of $G$ and any vertex $u \in \mathcal{I}_G(U_{ab})-U_{ab}$, there is not necessarily an isometric cycle in $\mathbf{C}(G,ab)$ which passes through $u$.  See for example the graph in Figure~\ref{F:Rem. 4.38}.  Note that this graph contains an isometric cycle whose convex hull is not a quasi-hypertorus.
\end{rem}

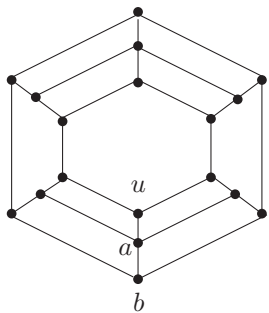
\begin{figure}[!h]
    \centering
{\tt    \setlength{\unitlength}{0.92pt}
\begin{picture}(128,148)
\thinlines    \put(61,62){$u$}
              \put(62,13){$b$}
              \put(56,36){$a$}
              \put(64,41){\circle*{4}}
              \put(104,61){\circle*{4}}
              \put(24,61){\circle*{4}}
              \put(22,101){\circle*{4}}
              \put(64,122){\circle*{4}}
              \put(105,100){\circle*{4}}
              \put(64,41){\line(2,1){41}}
              \put(63,42){\line(-2,1){39}}
              \put(63,122){\line(2,-1){43}}
              \put(21,101){\line(2,1){42}}
              \put(12,108){\circle*{4}}
              \put(64,136){\circle*{4}}
              \put(115,108){\circle*{4}}
              \put(33,91){\circle*{4}}
              \put(64,107){\circle*{4}}
              \put(94,92){\circle*{4}}
              \put(12,53){\circle*{4}}
              \put(33,68){\circle*{4}}
              \put(64,53){\circle*{4}}
              \put(94,68){\circle*{4}}
              \put(115,53){\circle*{4}}
              \put(64,26){\circle*{4}}
              \put(12,109){\line(2,1){53}}
              \put(64,26){\line(2,1){53}}
              \put(116,109){\line(-2,1){51}}
              \put(62,27){\line(-2,1){51}}
              \put(12,108){\line(0,-1){55}}
              \put(115,108){\line(0,-1){55}}
              \put(33,91){\line(0,-1){23}}
              \put(94,93){\line(0,-1){24}}
              \put(64,108){\line(0,1){25}}
              \put(64,53){\line(0,-1){26}}
              \put(64,108){\line(-2,-1){31}}
              \put(64,108){\line(2,-1){30}}
              \put(64,53){\line(-2,1){31}}
              \put(64,53){\line(2,1){30}}
              \put(12,108){\line(4,-3){21}}
              \put(115,108){\line(-4,-3){21}}
              \put(115,53){\line(-4,3){21}}
              \put(12,53){\line(3,2){22}}
\end{picture}}
\caption{Illustration of Remarque~\ref{R:ph/isom.cycle}.}
\label{F:Rem. 4.38}
\end{figure}

\begin{rem}\label{R:ph>1}
If the pre-hull number of a partial cube $G$ is greater than $1$, then the convex hull of any isometric cycle of $G$ may still be a gated quasi-hypertorus, as is shown by the following example.  Let $G$ be the expansion of $Q_3^-$ consisting of two $6$-cycles having one edge $uv$ in common, and of a vertex $w$ adjacent to both neighbors of $v$ distinct from $u$ in the union of these two cycles.  Then $ph(G) = 2$, but each isometric cycle of $G$ is gated.
\end{rem}

These three remarks give rise to the following question:

\begin{que}\label{Q:charact.ph-homog.}
Let $G$ be a partial cube such that $ph(G) \leq1$ and such that the convex hull of each of its isometric cycles is a gated quasi-hypertorus.  Is $G$ a Peano partial cube?  Otherwise what kind of graph is $G$?
\end{que}

\subsection{Infinite quasi-hypertori}\label{SS:inf.hypert.}

We can extend the finite concept of hypertorus in order to define infinite hypertori and more generally infinite quasi-hypertori by considering Cartesian products of infinite families of even cycles instead of finite ones.  Infinite quasi-hypertori are Peano partial cubes since the class of these graphs is closed under infinite Cartesian products.  Note that, \emph{throughout this paper, by a hypertorus and a quasi-hypertorus, we always mean a finite hypertorus and a finite quasi-hypertorus, respectively}.  In this subsection we will only generalize Theorem~\ref{T:faithf.hypertor.=>gated}.

\begin{thm}\label{T:faithf.inf.hypertor.=>gated}
Let $G$ be a Peano partial cube.  Then any faithful finite or infinite quasi-hypertorus of $G$ is gated.
\end{thm}

\begin{proof}
By Theorem~\ref{T:faithf.hypertor.=>gated}, we only have to consider the infinite case.  Let $F$ be an infinite faithful quasi-hypertorus of $G$.  Then $F$ is convex in $G$ by the proof of Theorem~\ref{T:faithf.hypertor.=>gated} where we did not use the fact that $F$ is finite.

Let $C$ be a convex cycle of $G$ which has at least three vertices in $F = \cartbig_{i \in I}^{a}F_i$, where $\cartbig_{i \in I}^{a}F_i$ is the component of $\cartbig_{i \in I}F_i$ containing some vertex $a$ by the definition of the weak Cartesian product.  Because $C$ is finite, the set $J := \{j \in I: \vert pr_j(C \cap F)\vert > 1\}$  is finite.  Then $C \cap F = C \cap \cartbig_{i \in I}^{a}F'_i$, where $F'_i = F_i$ if $i \in J$ and $F'_i = pr_i(C \cap F)$ otherwise.  Moreover $\cartbig_{i \in I}^{a}F'_i$ is isomorphic to the finite quasi-hypertorus $\cartbig_{j \in J}F_j$.  Hence $\cartbig_{j \in J}^{a}F'_i$ is gated by Theorem~\ref{T:faithf.hypertor.=>gated}, and thus $\Gamma$-closed.  It follows that $C$ is a cycle of $\cartbig_{i \in I}^{a}F'_i$, and thus of $F$.

Therefore $F$ is $\Gamma$-closed, and thus gated by Theorem~\ref{T:hypernet.=>G-closed+conv.=gated}.
\end{proof}

Infinite quasi-hypertori are obviously regular infinite Peano partial cubes.  However, there are not the only ones.  Moreover there are regular infinite Peano partial cubes which are not even Cartesian products, such as
 double rays (i.e., two-way infinite paths) and the hexagonal grid.

\subsection{Compact Peano partial cubes}\label{SS:compact}

We complete this section with a result which will be useful in several subsequent sections.  We recall that a partial cube is said to
be geodesically consistent if the geodesic topology on $V(G)$
coincides with the weak geodesic topology
(Subsection~\ref{SS:geod.consist.p.c.}).  Geodesically consistent
partial cubes are interesting in view of the property that they are
compact if and only if they contains no isometric rays.  By
\cite[Proposition 4.15]{P09-1}, any netlike partial cube, and thus any
median graph, is geodesically consistent.  The following result shows that this property also holds for
all Peano partial cubes.

\begin{thm}\label{T:loc.H-periph./geod.consist.}
Any Peano partial cube is geodesically consistent.
\end{thm}

\begin{proof}
Let $ab$ be an edge
of some Peano partial cube $G$, and $u \in
co_{G}(U_{ab})-U_{ab}$.  Let $P_u$ be the $U_{ab}$-geodesic associated with $u$, and let $v$ and $w$ be its endvertices.  Then, by Lemma~\ref{L:I(u,x)/v,w}, for any vertex $x \in U_{ab}$ $v$ or $w$ belongs to $I_G(u,x)$.  Therefore $u$ cannot geodesically dominates $U_{ab}$.  It
follows that any vertex in $co_{G}(U_{ab})$ which geodesically
dominates $U_{ab}$ must belong to $U_{ab}$.  Consequently $G$ is
geodesically consistent by
Proposition~\ref{P:charact.geod.consistent.p.c.}.
\end{proof}

Consequently we have:

\begin{cor}\label{C:comp.hyp.=no.isom.rays}
A Peano partial cube is compact if and only if it contains no isometric rays.
\end{cor}

\section{Decomposability and hyper-median partial\\
 cubes}\label{S:decompos.}

In this section, by generalizing the relation between median graphs and Peano partial cubes, we introduce hyper-median partial cubes, and we deal with the process of decomposition of hyper-median partial cubes into simpler partial cubes.  As we will see, some special triples of convex cycles, that we call tricycles, are the cornerstones of these concepts.

\subsection{Tricycles and gated amalgams}\label{SS:tric.gat.amalg.}

We define a concept of tricycle that is slightly different from the
one introduced in~\cite{P07-3}.

\begin{defn}\label{D:tricycle}
Let $G$ be a Peano partial cube.  
A triple of convex cycles of $G$ of length greater than $4$ (resp. such that at most two of them are $4$-cycles), which intersect in one
vertex and have pairwise exactly one edge in common, is called a  \emph{tricycle} (resp. \emph{quasi-tricycle}) of $G$.
\end{defn}

See Figure~\ref{F:tricycle} for an example of a quasi-tricycle.  Note that a 
quasi-hypertorus $H$ contains no tricycle because any two convex 
cycles of length greater than $4$ in $H$ may have at most one vertex 
in common.

\begin{lem}\label{L:quasi-tricycle}
Let $G$ be a Peano partial cube, $T = (C_{0},C_{1},C_{2})$ a quasi-tricycle of $G$ such that $C_{2}$ is the only cycle in $T$ of length greater than $4$, and $uv_{i}$ the common edge of $C_{0}$ with $C_{i}$ for $i = 1, 2$.  Then $T$ is a quasi-tricycle of some maximal convex hypercylinder of $G$.  Moreover $T$ is a quasi-tricycle of the prism $C_{2} \Box \langle u,v_{1} \rangle$.
\end{lem}

\begin{proof}
Let $X$ be the bulge of $co_{G}(U_{uv_{2}})$ in the subgraph $G[co_{G}(U_{uv_{2}}) \cup W_{v_{2}u}]$.  Then $T$ is clearly a quasi-tricycle of $H := \mathbf{Cyl}[X]$.  Moreover $H$ contains the prism $C_{2} \Box \langle u,v_{1} \rangle$, and $T$ is clearly a quasi-tricycle of this prism.
\end{proof}

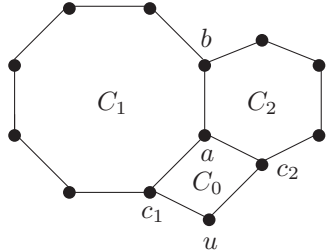
\begin{figure}[!h]
    \centering
    {\tt    \setlength{\unitlength}{0.80pt}
\begin{picture}(194,144)
\thinlines    \put(103,16){$u$}
              \put(138,51){$c_{2}$}
              \put(74,31){$c_{1}$}
              \put(102,112){$b$}
              \put(102,58){$a$}
              \put(98,44){$C_{0}$}
              \put(124,82){$C_{2}$}
              \put(53,82){$C_{1}$}
              \put(40,130){\circle*{6}}
              \put(78,130){\circle*{6}}
              \put(104,103){\circle*{6}}
              \put(131,115){\circle*{6}}
              \put(158,103){\circle*{6}}
              \put(158,70){\circle*{6}}
              \put(131,56){\circle*{6}}
              \put(106,30){\circle*{6}}
              \put(78,43){\circle*{6}}
              \put(104,70){\circle*{6}}
              \put(40,43){\circle*{6}}
              \put(14,70){\circle*{6}}
              \put(14,103){\circle*{6}}
              \put(79,44){\line(2,-1){28}}
              \put(106,31){\line(1,1){25}}
              \put(159,103){\line(0,-1){32}}
              \put(158,71){\line(-2,-1){26}}
              \put(104,71){\line(2,-1){28}}
              \put(132,115){\line(5,-2){26}}
              \put(104,104){\line(5,2){28}}
              \put(4670,-2590){\oval(0,5545)}
              \put(4670,-2591){\oval(0,5544)}
              \put(40,43){\line(1,0){38}}
              \put(40,131){\line(1,0){38}}
              \put(78,131){\line(1,-1){25}}
              \put(40,132){\line(-1,-1){26}}
              \put(104,103){\line(0,-1){32}}
              \put(14,103){\line(0,-1){32}}
              \put(14,71){\line(0,1){9}}
              \put(14,70){\line(0,1){1}}
              \put(78,44){\line(1,1){25}}
              \put(14,71){\line(1,-1){26}}
\end{picture}}
\caption{A quasi-tricycle.}
\label{F:tricycle}
\end{figure}

\begin{lem}\label{L:tricycle/4-cycle}
A Peano partial cube
contains no quasi-tricycle with exactly one $4$-cycle.
\end{lem}

\begin{proof}
Let $(C_{0},C_{1},C_{2})$ be a quasi-tricycle of a Peano partial cube $G$ such that $C_{0}$ is the only
$4$-cycle.  Let $ab$ be the common edge of $C_{1}$ and $C_{2}$ and,
for $i = 1, 2$, let $ac_{i}$ be the common edge of $C_{i}$ and $C_{0}$
(see Figure~\ref{F:tricycle}).  Let $u$ be the antipode of $a$ in $C_{0}$.

If $u \notin U_{ab}$, then $C_{0} \cup C_{1} \cup C_{2}$ is contained 
in a bulge $X$ of $co_{G}(U_{ab})$, and $(C_{0},C_{1},C_{2})$ is a 
tricycle of $\mathbf{Cyl}[X]$, which is impossible.

If $u \in U_{ab}$, then $C_{1} \cup \langle c_{1},u \rangle$ is
contained in a bulge $X$ of $co_{G}(U_{ab})$, and thus in $\mathbf{Cyl}[X]$, which is impossible as well.    
\end{proof}

Those two results clearly imply \cite[Lemma 2.15]{P07-3} stating that a netlike partial cube contains no quasi-tricycle.

We will say that a graph is \emph{tricycle-free} if it contains no tricycle.  We now state two of the main results of this subsection:

\begin{thm}\label{T:tricycle/decomp.}
Let $G$ be a compact Peano partial
cube that is not a quasi-hypertorus.  If
$G$ is tricycle-free, then $G$ is the gated amalgam of two of its proper
subgraphs.
\end{thm}

\begin{thm}\label{T:tricycle/decomp.fin.hypernet.}
A finite partial cube is a tricycle-free Peano partial cube if and only if it is obtained by a sequence of gated amalgamations from finite quasi-hypertori.
\end{thm}

We recall that, for any edge $ab$ of a partial cube $G$, we denote:
\begin{align*}
G_{ \overrightarrow{ab}}& := G[co_G(U_{ab}) \cup W_{ba}]\\
G_{\overline{ab}}& := G_{ \overrightarrow{ab}} \cap G_{ \overrightarrow{ba}}.
\end{align*}

\begin{lem}\label{L:tricycle/Gab}
If a Peano partial cube $G$
is tricycle-free, then the subgraphs $G_{ \overrightarrow{ab}}$,\; $G_{ \overrightarrow{ba}}$ and $G_{ \overline{ab}}$ are gated for any edge $ab$ of $G$.
\end{lem}

\begin{proof}
We will prove that $G_{ \overrightarrow{ab}}$ is gated. The proof that $G_{ \overrightarrow{ba}}$ is gated would be analogous, and then $G_{ \overline{ab}}$ will be gated as the intersection of two gated subgraphs.

$G_{ \overrightarrow{ab}}$ is clearly convex.  We will prove that it is also
$\Gamma$-closed.  Let $C$ be a convex cycle of $G[W_{ab}]$ whose intersection
with $G_{ \overrightarrow{ab}}$ has at least three vertices.  Because $C$ and $G_{ \overrightarrow{ab}}$
are convex it follows that $C \cap G_{ \overrightarrow{ab}}$ has two adjacent edges
$uv_{1}$ and $uv_{2}$.  If $C$ is a $4$-cycle, then $C$ is a cycle of
$G_{ \overrightarrow{ab}}$ since $G_{ \overrightarrow{ab}}$ is convex.  Assume that the length of $C$ is
greater than $4$.  The edges $uv_{1}$ and $uv_{2}$ do not belong to a
$4$-cycle since $C$ is convex.  Moreover they also do not belong to a
convex cycle $C' \neq C$ of length greater than $4$, because otherwise
$C$ and $C'$ would be convex cycles that pass through $v_2$ and contain edges $\Theta$-equivalent to $uv_1$, contrary to Remark~\ref{R:notation}.2.  In the following, we will denote by $x'$ the neighbor in $U_{ba}$ of any $x \in U_{ab}$.  We distinguish two cases.

\emph{Case 1.}\; At least one of the edges $uv_{1}$ or $uv_{2}$ is an 
edge of $G[U_{ab}]$.

Assume that $uv_{1} \in E(G[U_{ab}])$.  Suppose that $uv_{2} \notin
E(G[U_{ab}])$.  Then $uv_{2}$ is an edge of a bulge of
$\mathcal{I}_{G}(U_{ab})$, and thus an edge of a convex cycle $C_{2}$
of length greater than $4$ having an edge $\Theta$-equivalent to $ab$.
Let $C_{1} := \langle u,v_{1},v'_{1},u',u \rangle$.  Then
$(C,C_{1},C_{2})$ is a quasi-tricycle of $G$ such that $C_{1}$ is a
$4$-cycle, which is impossible by Lemma~\ref{L:tricycle/4-cycle}.

Then $uv_{2} \in E(G[U_{ab}])$.  Let $C_{i} := \langle
u,v_{i},v'_{i},u',u \rangle$ for $i = 0, 1$.  Then $(C,C_{1},C_{2})$ is a quasi-tricycle of $G$ with two $4$-cycles.  Hence, by Lemma~\ref{L:quasi-tricycle}, it is a quasi-tricycle of the prism $C \Box \langle u,u' \rangle$.  Therefore $V(C) \subseteq U_{ab}$, and thus $C$ 
is a cycle of $G_{ \overrightarrow{ab}}$.

\emph{Case 2.}\; $uv_{1}$ and $uv_{2}$ are not edges of $G[U_{ab}]$.

If $u \in U_{ab}$, then there are two $H_{1}, H_{2} \in
\mathbf{Cyl}[G,ab]$ such that $uv_{i}$ is an edge
of $H_{i}$ for $i = 0, 1$.  By the properties of $H_{i}$, for $i = 0,
1$, there exists a convex cycle $C_{i}$ of $H_{i}$ of length greater
than $4$ containing edges $\Theta$-equivalent to $ab$, such that
$uv_{i}$ is an edge of $C_{i}$.  It follows that $C_{1}$ and
$C_{2}$ contain the edge $uu'$.  Therefore $(C,C_{1},C_{2})$ is a
tricycle of $G$, contrary to the hypothesis.

Then $u \notin U_{ab}$.  Hence $uv_{1}$ and $uv_{2}$ are edges
of some element $H$ of $\mathbf{Cyl}[G,ab]$, and
thus $C$ is a cycle of $H$, and thus of $G_{ab}$, because $H$
is gated by Theorem~\ref{T:C-subgr./gated}.\\

Consequently $G_{ \overrightarrow{ab}}$ is $\Gamma$-closed, and thus gated by
Theorem~\ref{T:hypernet.=>G-closed+conv.=gated}.
\end{proof}

We recall that every median graph with more than two vertices is either a Cartesian product or a gated amalgam of proper median subgraphs.  In particular every finite median graph can be obtained by a sequence of gated amalgamations from hypercubes.

\begin{proof}[\textnormal{\textbf{Proof of
Theorem~\ref{T:tricycle/decomp.}}}]
Assume that $G$ is $2$-connected.  Otherwise, $G$ could be decomposed
via gated amalgamation along a single vertex.  We denote by $\mathbf{Cyl}^+[G]$ the set of all subgraphs of $G$ that are either elements of $\mathbf{Cyl}[G]$ or maximal hypercubes of $G$.  Then $G =
\bigcup\mathbf{Cyl}^+[G]$ since $G$ is $2$-connected.  We recall that, by Theorem~\ref{T:C-subgr./gated}, all elements of $\mathbf{Cyl}^+[G]$ are gated in $G$.  
If $G$ is a median
graph but not a hypercube, then $G$ is the gated amalgam of two of its proper median subgraphs.  Suppose that $G$ is
not a median graph.  It follows that $G$ contains an element of $\mathbf{Cyl}[G]$, say $H_0$.  We distinguish two cases.

\emph{Case 1.}\; There exists an edge $ab$ of $G$ such that $W_{ab}$ 
and $W_{ba}$ are not semi-peripheries.

Then, since $G_{ \overrightarrow{ab}}$ and $G_{ \overrightarrow{ba}}$ are gated by Lemma~\ref{L:tricycle/Gab}, it follows that $G$ is the gated amalgam of these
two subgraphs.

\emph{Case 2.}\; For every edge $ab$ of $G$, $W_{ab}$ or $W_{ba}$ is a
semi-periphery.

Then the elements of $\mathbf{Cyl}^+[G]$ are
pairwise non-disjoint.  Indeed, if $H$ and $H'$ are disjoint
elements of $\mathbf{Cyl}^+[G]$, and if $ab$ is
an edge of a smallest path between $H$ and
$H'$, then neither $W_{ab}$ nor $W_{ba}$ is a semi-periphery,
contrary to the assumption.  Then, by Proposition~\ref{P:gated.sets/Helly} and because $G$ is compact, the elements of $\mathbf{Cyl}^+[G]$ have a non-empty intersection $Q$.

\emph{Subcase 2.1.}\; $Q = H_0 \cap H$ for every $H \in \mathbf{Cyl}^+[G]-\{H_0\}$.

  Then $G$ is
the gated amalgam along $Q$ of $H_0$ and the subgraph induced by the union of
all the other elements of $\mathbf{Cyl}^+[G]$.

\emph{Subcase 2.2.}\; $Q \neq H_0 \cap H$ for some $H \in \mathbf{Cyl}^+[G]-\{H_0\}$.

Let $ab$ be an edge of $H_0$ such that $a \in V(Q)$ and $b \in
V(H-Q)$.  Then $W_{ab}$ is not a semi-periphery because $G[W_{ab}]$
contains all the elements of $\mathbf{Cyl}^+[G]$
that do not contain $b$.  It follows that $W_{ba}$ is a
semi-periphery by the assumption, and thus $G_{ \overrightarrow{ab}} = G_{ \overline{ab}}$.  Then $G$ is the union of $G_{ \overline{ab}}$
and $G[W_{ab}]$.  Moreover
$G[W_{ab}]$ is convex.  We will show that it is gated as well.

Suppose that there is a convex cycle $C$ of $G$ that has at least
three vertices in common with $G[W_{ab}]$ and that is not a cycle of this
subgraph.  Then the length of $C$ is greater than $4$ since $G[W_{ab}]$ is
convex, and thus $C$ contains edges $\Theta$-equivalent to $ab$.
Hence $C$ is a convex cycle of some element $H'$ of $\mathbf{Cyl}[G,ab]$.  Then $H'$ contains $Q$ and thus the
vertex $a$.  Let $C'$ be the convex cycle of $H'$ containing $ab$ and
isomorphic to $C$.  Because $G[W_{ab}]$ is
convex, it follows that $C' \cap G[W_{ab}]$ is a path of length at least $2$,
one of whose endvertices is $a$.  Because $a \in V(Q)$, and since each
edge of $G[W_{ab}]$ is an edge of some element of $\mathbf{Cyl}^+(G[W_{ab}])$, it follows that $C' \cap G[W_{ab}] = C' \cap
F$ for some $F \in \mathbf{Cyl}^+(G[W_{ab}])$.  Then $C'$ is a cycle of $F$ since $F$ is
gated, and thus $\Gamma$-closed by
Theorem~\ref{T:hypernet.=>G-closed+conv.=gated}, contrary to the fact that $F$ is a subgraph of $G[W_{ab}]$, and thus cannot contain the edge $ab$.

Therefore $G[W_{ab}]$ is gated by Theorem~\ref{T:hypernet.=>G-closed+conv.=gated}.
Furthermore $G_{ \overline{ab}}$ is gated by Lemma~\ref{L:tricycle/Gab} since
$G$ contains no tricycle.  Consequently
$G$ is the gated amalgam of $G[W_{ab}]$ and $G_{ \overline{ab}}$.
\end{proof}

\begin{pro}\label{P:tricycle/gat.amalg.}
Let $G$ be the gated amalgam of two partial cubes $G_{0}$ and $G_{1}$.
Then $G$ is tricycle-free if and only if so are $G_{0}$ and $G_{1}$.
\end{pro}

\begin{proof}
We only have to prove the sufficiency.  Assume that $G = G_{0} \cup
G_{1}$ and that $G_{0}$ and $G_{1}$ are tricycle-free.  Suppose
that $G$ contains a tricycle $(C_{0},C_{1},C_{2})$.  Because $G_{0}$
and $G_{1}$ are gated, and thus $\Gamma$-closed by
Proposition~\ref{P:gated/G-closed}, it follows that each cycle $C_{i}$
is a subgraph of $G_{0}$ or $G_{1}$.

Suppose that two of these cycles, say $C_{0}$ and $C_{1}$ are cycles
of $G_{0}$.  Then, by the definition of a tricycle, $C_{2}$ has two
edges in common with $G_{0}$.  Hence $C_{2}$ is also a cycle of
$G_{0}$ because $G_{0}$ is gated, and thus $\Gamma$-closed.  This
yields a contradiction with the assumption that $G_{0}$ contains no
tricycle.
\end{proof}

\begin{proof}[\textnormal{\textbf{Proof of
Theorem~\ref{T:tricycle/decomp.fin.hypernet.}}}] The necessity is a
consequence of Theorem~\ref{T:tricycle/decomp.}, and the converse
follows immediately from Proposition~\ref{P:tricycle/gat.amalg.}, the fact that quasi-hypertori are tricycle-free, and the property that Cartesian multiplication distributes over gated amalgamation, viz., the Cartesian product of a graph $G_0$ with the gated amalgams of two graphs $G_1$ and $G_2$ is equal to the gated amalgam of $G_0 \Box G_1$ and $G_0 \Box G_2$.
\end{proof}

\subsection{Hyper-median partial cubes}\label{SS:hyper-median}

We recall several definitions.  Let $(u,v,w)$ be a triple of vertices of a graph $G$.  Then:

\textbullet\; a \emph{median} of $(u,v,w)$ is any element of the intersection $I_{G}(u,v) \cap I_{G}(v,w) \cap I_{G}(w,u)$;

\textbullet\; $(u,v,w)$ is a \emph{metric triangle} if the intervals $I_G(u,v), I_G(v,w), I_G(w,u)$ pairwise intersect in their common endvertices;

\textbullet\; a \emph{quasi-median}\footnote{This definition is more general than the specific notion used in the context of quasi-median graphs where a quasi-median $(x,y,z)$ of a triple $(u,v,w)$ must satisfies the additional conditions: $d_G(x,y) = d_G(y,z) = d_G(z,x) =: k$ and $k$, the size of the quasi-median, is minimum.} of $(u,v,w)$ is a metric triangle $(x,y,z)$ such that
\begin{align*}
d_G(u,v) = d_G(u,x)+d_G(x,y)+d_G(y,v),\\
d_G(v,w) = d_G(v,y)+d_G(y,z)+d_G(z,w),\\
d_G(w,u) = d_G(w,z)+d_G(z,x)+d_G(x,u).
\end{align*}

\begin{lem}\label{L:met.tri./equiv.}
Let $(u,v,w)$ be a metric triangle of a Peano partial cube $G$.  The following assertions are equivalent:

\textnormal{(i)}\; There exists an isometric cycle of $G$ that passes through the three vertices $u, v, w$.

\textnormal{(ii)}\; $co_G(u,v,w)$ induces a hypertorus.
\end{lem}

\begin{proof}
(i) $\Rightarrow$ (ii)\; Let $C$ be an isometric cycle of $G$ that passes through $u, v, w$.  By Theorem~\ref{T:hypernet./conv.hull(isom.cycle)}, $F := G[co_G(u,v,w)] = co_G(C)$ is a quasi-hypertorus.  $F$ is not a hypercube, since otherwise $(u,v,w)$ would have a median in $F$, and thus in $G$, contrary to the fact that $(u,v,w)$ is a metric triangle.  Suppose now that $F = T \Box K_2$ is the prism over a hypertorus $T$.  Then each $T$-layer of $F$ contains at least one of the vertices $u, v, w$.  Without loss of generality, we can suppose that $u, v$ belong to the $T$-layer $T_0$, and $w$ belong to the other $T$-layer $T_1$.  Then $I_G(u,w)$ and $I_G(v,w)$ contains the projection of $w$ onto $T_0$, contrary to the fact that $(u,v,w)$ is a metric triangle.  Therefore $F$ is a hypertorus.

(ii) $\Rightarrow$ (i)\; Conversely, assume that $co_G(u,v,w)$ induces a hypertorus $T$.  Then $T$ is the Cartesian product of some finite family of even cycles.  We will prove the existence of an isometric cycle of $G$ that passes through $u, v, w$ by induction on the number of these cycles.

This is trivial if $T$ is a cycle.  Suppose that the result holds if $T$ is any $n$-torus, for some positive integer $n$.  Let $T$ be an $(n+1)$-torus.  Then $T = T_0 \Box T_1$, where $T_0$ is an $n$-torus and $T_1$ is an even cycle.  Clearly $pr_{T_i}(u,v,w)$ is a metric triangle of $T_i$ for $i = 0, 1$.  Then, by the induction hypothesis, $T_0$ contains an isometric cycle $C$ that passes through the projections on $T_0$ of $(u,v,w)$.  Put $C = \langle x_1,\dots,x_{2p},x_1 \rangle$ and $T_1 = \langle c_1,\dots,c_{2q},c_1 \rangle$.  Without loss of generality we can suppose that $u = (x_1,c_1),\; v = (x_i,c_k)$ and $w = (x_j,c_l)$ for some $i, j, k, l$ with $0 < i < j < 2p$ and $0 < k < l < 2q$.  By the Distance Property of Cartesian product, $C \Box T_1$ is an isometric subgraph of $F$ such that $\mathrm{idim}(C \Box T_1) = \mathrm{idim}(F)$.  Then
\begin{multline*}
D := \langle (x_0,c_0),\dots,(x_i,c_0),(x_i,c_1),\dots,(x_i,c_k),(x_{i+1},c_k),\\
\dots,(x_j,c_k),(x_j,c_{k+1}),\dots,(x_j,c_l),(x_{j+1},c_l),\\
\dots,(x_{2p},c_l),(x_{2p},c_{l+1}),\dots,(x_{2p},c_{2q}),(x_1,c_1) \rangle
\end{multline*}
is an isometric cycle of $F$, and thus of $G$, that passes through $u, v, w$.
\end{proof}

\begin{defn}\label{D:hyper-median}
Let $G$ be a Peano partial cube.  A \emph{hyper-median} of a triple $(u,v,w)$ of vertices of $G$ is a quasi-median $(x,y,z)$ of $(u,v,w)$ that satisfies the equivalent properties (i) and (ii) of Lemma~\ref{L:met.tri./equiv.}.
\end{defn}

Recall that a convex hypertorus is gated by Theorem~\ref{T:conv.reg.H-subgr.=>gated}.

\begin{pro}\label{P:hypermed./gates}
Let $(x,y,z)$ be a hyper-median of a triple $(u,v,w)$ of vertices of a Peano partial cube $G$, and $H := \cartbig_{i \in I}C_i$ the hypertorus induced by the convex hull of $(x,y,z)$.  Then:

\textnormal{(i)}\; For each $i \in I$,\; $(pr_i(x),pr_i(y),pr_i(z))$ is a metric triangle of $C_i$, and thus $C_i$ is an even cycle of length greater than $4$.

\textnormal{(ii)}\; $x, y, z$ are the gates in $H$ of $u, v, w$, respectively.

\textnormal{(iii)}\; $\mathrm{idim}(H) = (d_G(x,y)+d_G(y,z)+d_G(z,x))/2$.
\end{pro}

\begin{proof}
(i)\; For every $i \in I$, either $(pr_i(x),pr_i(y),pr_i(z))$ is a metric triangle of $C_i$ or  $pr_i(x) = pr_i(y) = pr_i(z)$.  On the other hand, by Proposition~\ref{P:pro.Cartes.prod.}(vii), $V(C_i) = co_{C_i}(pr_i(x),pr_i(y),pr_i(z))$.  It follows that $(pr_i(x),pr_i(y),pr_i(z))$ must be a metric triangle of $C_i$, which implies that $C_i$ cannot be a $4$-cycle.

(ii)\; Let $(x',y',z')$ be the gates in $H$ of $u, v, w$, respectively.  Then $(x,y,z)$ is a quasi-median of $(x',y',z')$.  It follows that, for every $i \in I$, because $C_i$ is a cycle,   $(pr_i(x),pr_i(y),pr_i(z))$ must be a quasi-median of $(pr_i(x'),pr_i(y'),pr_i(z'))$, and thus $pr_i(x') = pr_i(x)$,\; $pr_i(y') = pr_i(y)$ and $pr_i(z') = pr_i(z)$.  Therefore $x' = x$,\; $y' = y$ and $z' = z$.

(iii)\; By Lemma~\ref{L:edge/co(H)}(i), each edge of $H$ is $\Theta$-equivalent to an edge of $I_G(x,y) \cup I_G(y,z) \cup I_G(z,x)$.  Moreover, each geodesic between any two vertices has the same $\Theta$-classes of edges.  Finally, for each edge $ab$ of $H$, there are exactly two pairs of vertices among $x, y, z$ such that any geodesic between the vertices of each of these two pairs contains an edge that is $\Theta$-equivalent to $ab$.  Whence the result.
\end{proof}

Clearly median graphs are the particular Peano partial cubes for which any triple of vertices has a median.  By generalizing this property we obtain the new concept of hyper-median partial cubes.

\begin{defn}\label{D:hypermed.p.c.}
A Peano partial cube all of whose triples of vertices admit a median or a hyper-median is called a \emph{hyper-median partial cube}.
\end{defn}

Median graphs are then hyper-median partial cubes.  Cellular bipartite graphs
are other examples of hyper-median partial cubes. The \emph{cellular
bipartite graphs} are the graphs that can be obtained from single
edges and even cycles by a sequence of gated amalgamations.  These graphs
were defined and studied by Bandelt and Chepoi~\cite{BC96}.  They
showed in particular~\cite[Proposition 3]{BC96} that the cellular
bipartite graphs are hyper-median partial cubes.  For a cellular bipartite graph, and more generally for a netlike partial cube, a hyper-median is a gated cycle.

As was shown in~\cite{P07-3}, there are netlike partial cubes that are not hyper-median. This is the case of the benzenoid graph in Figure~\ref{F:noMCP}
where the triple $(u,v,w)$ of vertices has neither a median nor a
hyper-median.

\begin{figure}[!h]
    \centering
    {\tt    \setlength{\unitlength}{1,2pt}
\begin{picture}(162,131)
\thinlines    \put(107,26){$w$}
              \put(23,26){$v$}
              \put(66,100){$u$}
              \put(51,82){\line(3,2){16}}
              \put(86,79){\line(-4,3){17}}
              \put(50,80){\line(0,-1){19}}
              \put(85,79){\line(0,-1){19}}
              \put(68,49){\line(3,2){16}}
              \put(68,47){\line(-4,3){17}}
              \put(101,48){\line(-4,3){17}}
              \put(51,16){\line(-4,3){17}}
              \put(51,18){\line(3,2){16}}
              \put(68,48){\line(0,-1){19}}
              \put(33,49){\line(0,-1){19}}
              \put(33,48){\line(3,2){16}}
              \put(102,47){\line(0,-1){19}}
              \put(84,18){\line(3,2){16}}
              \put(85,16){\line(-4,3){17}}
              \put(68,92){\circle*{4}}
              \put(50,80){\circle*{4}}
              \put(85,80){\circle*{4}}
              \put(84,17){\circle*{4}}
              \put(68,29){\circle*{4}}
              \put(50,17){\circle*{4}}
              \put(33,29){\circle*{4}}
              \put(102,47){\circle*{4}}
              \put(68,47){\circle*{4}}
              \put(33,47){\circle*{4}}
              \put(85,59){\circle*{4}}
              \put(50,59){\circle*{4}}
              \put(102,29){\circle*{4}}
\end{picture}}
\caption{A benzenoid graph that is not hyper-median.}
\label{F:noMCP}
\end{figure}
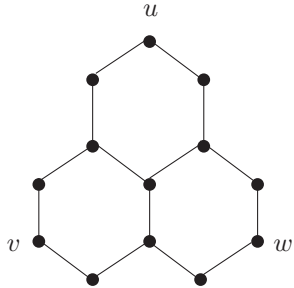

We first prove the uniqueness of the median or of the hyper-median of a triple of vertices of a hyper-median partial cube.

\begin{lem}\label{L:BC/pro.3}
Let $G$ be a Peano partial cube such that the convex hull of any metric triangle induces a hypertorus.  Then, for any three vertices $u, v, w$ of $G$, there exists a (necessarily unique) vertex $x$ such that $$I_G(u,v) \cap I_G(u,w) = I_G(u,x).$$
\end{lem}

\begin{proof}
Suppose the contrary, and select a triple $(u,v,w)$ of distinct vertices of $G$ violating the assertion of the proposition such that $d_G(u,v)+d_G(u,w)$ is as small as possible.  Then, by the proof of \cite[Proposition 2]{BC96}, there exist two distinct vertices $x$ and $y$ such that both $I_G(u,x)$ and $I_G(u,y)$ are properly contained in $I_G(u,v) \cap I_G(u,w)$, and such that the triples $(v,x,y)$ and $(w,x,y)$ are metric triangles.

By the property of $G$,\; $H_v := G[co_G(v,x,y)]$ and $H_w := G[co_G(w,x,y)]$ are hypertori, and more precisely gated hypertori.  Let $w'$ be the gate of $w$ in $H_v$.  Then $(u,v,w')$ is also a triple of vertices of $G$ that violates the assertion of the proposition.  By the minimality of $d_G(u,v)+d_G(u,w)$, it follows that $w' = w$.  Hence $w \in V(H_v)$, and analogously $v \in V(H_w)$. By the definition of a Peano partial cube, $G$ is join-hull commutative.  Hence $w \in I_G(v,z)$ and $v \in I_G(w,z')$ for some $z, z' \in I_G(x,y)$.  On the other hand $I_G(x,y) \subseteq I_G(u,v) \cap I_G(u,w)$ since any interval of a partial cube is convex, and thus $z, z' \in I_G(u,v) \cap I_G(u,w)$.  
Therefore $$I_G(u,w) \subseteq I_G(u,v) \subseteq I_G(u,w),$$ and thus $v = w$, contrary to the hypothesis.
\end{proof}

\begin{pro}\label{P:hypermed.p.c./met.triangle}
A Peano partial cube $G$ is a hyper-median partial cube if and only if the convex hull of any metric triangle of $G$ induces a hypertorus.
\end{pro}

\begin{proof}
(a)\; Let $(u,v,w)$ be a metric triangle of a hyper-median partial cube $G$.  then $(u,v,w)$ is clearly its own hyper-median, and thus $co_G(u,v,w)$ induces a hypertorus.

(b)\; Assume that the convex hull of any metric triangle of $G$ induces a hypertorus.  Let $(u,v,w)$ be a triple of vertices of $G$, and let $x, y,z$ be the vertices successively determined by
\begin{align*}
I_G(u,v) \cap I_G(u,w) = I_G(u,x),\\
I_G(v,x) \cap I_G(v,w) = I_G(v,y),\\
I_G(w,x) \cap I_G(w,y) = I_G(w,z),
\end{align*}
according to Lemma~\ref{L:BC/pro.3}.  If $x = y = z$, then this vertex is the unique median vertex.  Otherwise, the three intervals $I_G(x,y), I_G(y,z), I_G(z,x)$ pairwise intersect in their common endvertices.  Hence $(x,y,z)$ is metric triangle, and thus the convex hull of $\{x,y,z\}$ induces a hypertorus $H$.  Therefore $(x,y,z)$ is a hyper-median of $(u,v,w)$.  Hence $G$ is a hyper-median partial cube.
\end{proof}

\begin{pro}\label{P:quasi-med./uniqueness}
Let $G$ be a Peano partial cube.  If a triple of vertices of $G$ has a quasi-median, then this quasi-median is unique.
\end{pro}

\begin{proof}
Assume that a triple $(u,v,w)$ of vertices of $G$ has a quasi-median $(x,y,z)$.  Then $(x,y,z)$ satisfies the following properties:
\begin{align*}
I_G(u,v) \cap I_G(u,w) = I_G(u,x),\\
I_G(v,x) \cap I_G(v,w) = I_G(v,y),\\
I_G(w,x) \cap I_G(w,y) = I_G(w,z).
\end{align*}

In order to prove the uniqueness of this quasi-median, we first note that, by the proof of \cite[Proposition 3]{BC96}, we have $$I_G(u,v) \cap I_G(v,w) = I_G(v,y) \text{\quad and}\quad I_G(u,w) \cap I_G(v,w) = I_G(w,z).$$  Let $(x',y',z')$ be any hyper-median of $(u,v,w)$.  Then $$x' \in I_G(u,x),\quad y' \in I_G(v,y)  \text{\quad and}\quad  z' \in I_G(w,z).$$  Then, because $(x',y',z')$ is a metric triangle, it follows that $x' = x,\; y' = y$ and $z' = z$.
\end{proof}

Because any triple of vertices of a partial cube may have at most one median, we immediately deduce the following result from the above proposition.

\begin{cor}\label{C:hypermed./uniqueness}
Let $G$ be a Peano partial cube.  If a triple of vertices of $G$ has a median or a hyper-median, then this median or hyper-median is unique.
\end{cor}

In the following, if a triple $(u,v,w)$ of vertices of a Peano partial cube $G$ has a median (resp. a hyper-median), then this median (resp. hyper-median) will be denoted by $m_G(u,v,w)$ (resp. $h_G(u,v,w)$).

\begin{cor}\label{C:quasi-hypertori/hyper-median}
Any quasi-hypertorus is a hyper-median partial cube.
\end{cor}

\begin{proof}
Let $G$ be a quasi-hypertorus.  The result is clear if $G$ is a hypercube.  Assume that $G$ is not a hypercube.  Let $(u,v,w)$ be a metric triangle of $G$.  If $G$ is a prism $F \Box K_2$ over a hypertorus $F$, then, as we saw in the proof of Proposition~\ref{P:hypermed.p.c./met.triangle}, the vertices $u, v, w$ must belong to the same $F$-layer of $G$.  Hence, without loss of generality we can suppose that $G$ is a hypertorus, say $G = \cartbig_{i \in I}C_i$, where $C_i$ is an even cycle for every $i \in I$, the  length of $C_i$ being greater than $4$ for at least one $i \in I$.  Because the intervals $I_G(u,v), I_G(v,w),I_G(w,u)$ pairwise intersect in their common endvertices, it follows that, for every $i \in I$, either $pr_i(u) = pr_i(v) = pr_i(w)$, which is necessarily  the case if $C_i$ is a $4$-cycle, or the intervals $I_{C_i}(pr_i(u),pr_i(v)), I_{C_i}(pr_i(v),pr_i(w)), I_{C_i}(pr_i(w),pr_i(u))$ pairwise intersect in their common endvertices, and this is the case for at least one $i$.  Therefore the subgraph of $C_i$ induced by $co_{C_i}(pr_i(u),pr_i(v),pr_i(w))$ is either a $K_1$ or the cycle $C_i$.  It follows that the subgraph of $G$ induced by $co_G(u,v,w)$ is a hypertorus.  Consequently $G$ is a hyper-median partial cube by Proposition~\ref{P:hypermed.p.c./met.triangle}.
\end{proof}

By \cite[Theorem 3.5]{P07-3}, a netlike partial cube is hyper-median if and only if it is tricycle-free.  We will extend this result to all Peano partial cubes.

\begin{thm}\label{T:charact.hyper-median}
Let $G$ be a Peano partial cube.  Then $G$ is a hyper-median partial cube if and only $G$ is tricycle-free.
\end{thm}

\begin{proof}
(a)\; Assume that $G$ is a hyper-median partial cube. Suppose that
$G$ contains a tricycle $(C_{1}, C_{2}, C_{3})$.  Then the length of each of these cycles is
greater than $4$.  Let $a$ be the common vertex of these three cycles
and, for every triple $(i,j,k)$ of elements of $\{1,2,3\}$, let
$ab_{i}$ be the common edge of $C_{j}$ and $C_{k}$.  For $i = 1, 2,
3$, let $u_{i}$ be the antipode of $a$ in $C_{i}$.  Because of the
convexity of each of these cycles we have $ V(C_{i}) \cap U_{b_{j}a} =
\{u_{i},b_{j} \}$ for any $i \neq j$.

For $i \neq j$, let $C_{i}[u_{i},b_{j}]$ be the
$(u_{i},b_{j})$-geodesic in $C_{i}$.  We will show that $P_{ij} :=
C_{i}[u_{i},b_{k}] \cup C_{j}[u_{j},b_{k}]$ is the only
$(u_{i},u_{j})$-geodesic.  Without loss of generality we will suppose
that $i = 1,\; j = 2$ and $k = 3$, and moreover that the length of
$C_{1}$ is less than or equal to that of $C_{2}$, and hence that
$d_{G_{1}}(u_{1},b_{3}) \leq d_{G_{2}}(u_{2},b_{3})$.

For $i = 1, 2$, because $C_{i}$ is a convex cycle of length greater
than $4$, there exists a bulge $X_{i}$ of $co_{G}(U_{b_{3}a})$ such
that $C_{i}$ is a convex cycle of $H_{i} := \mathbf{Cyl}[X_{i}]$.  Clearly $H_{1} \neq H_{2}$.  Hence, by
the definition of bulges, $u_{1} \notin V(H_{2})$.  It
follows that the gate $x$ of $u_{1}$ in $H_{2}$ belongs to
$U_{b_{3}a}$.  By the Distance Property of the Cartesian product
$H_{2}$,\; $C_{2}$ is gated in $H_{2}$ and thus in $G$, and moreover the
gate of $x$, and thus of $u_{1}$, in $C_{2}$ belongs to $U_{b_{3}a}$
as well.  Therefore the gate of $u_{1}$ in $C_{2}$ is $b_{3}$ because
$d_{G_{1}}(u_{1},b_{3}) \leq d_{G_{2}}(u_{2},b_{3})$ by hypothesis.
Then $d_{G_{1}}(u_{1},u_{2}) = d_{G_{1}}(u_{1},b_{3}) +
d_{G_{1}}(u_{2},b_{3})$.  Therefore $P_{12}$ is a
$(u_{1},u_{2})$-geodesic.

Suppose that there exists another $(u_{1},u_{2})$-geodesic $R$ 
distinct from $P_{12}$.  By the properties of $H_{1}$ and $H_{2}$ and 
the fact that $C_{1}$ and $C_{2}$ are convex, it follows that 
$P_{12}$ and $R$ are internally disjoint.  Then, clearly, the 
neighbor of $u_{1}$ in $R$ has no gate in $C_{2}$, contrary to the 
fact that $C_{2}$ is gated.  Consequently $P_{12}$ is the only 
$(u_{1},u_{2})$-geodesic in $G$.

It follows that the intervals $I_{G}(u_{1},u_{2}),\; I_{G}(u_{2},u_{3}),\;
I_{G}(u_{3},u_{1})$ pairwise intersect in their common endvertices.  On the other hand the subgraph of $G$ induced by $co_G(u_1,u_2,u_3)$ contains the tricycle $(C_{1}, C_{2}, C_{3})$, and thus cannot be a hypertorus, contrary to Proposition~\ref{P:hypermed.p.c./met.triangle} and the assumption that $G$ is a hyper-median partial cube.

(b)\; Conversely assume that $G$ contains no tricycle.
Let $(u,v,w)$ be a metric triangle of $G$.  Then $F := G[co_G(u,v,w)]$ is a finite convex subgraph of $G$, and thus is a Peano partial cube, that is tricycle-free since so is $G$ by assumption.

Suppose that $F$ is not a hypertorus.  Then it is not a quasi-hypertorus by the property of the triple $(u,v,w)$ and the fact that quasi-hypertori are hyper-median partial cubes  by Corollary~\ref{C:quasi-hypertori/hyper-median}.  It follows, by
Theorem~\ref{T:tricycle/decomp.}, that $F$ is the gated amalgam of two
gated subgraphs $F_{0}$ and $F_{1}$.
Because $F$ is induced by the convex hull of $\{u,v,w\}$, it follows that the vertices $u,
v, w$ cannot belong to $V(F_{i})$ for some $i \in \{0,1 \}$.  Hence,
without loss of generality, we can suppose that $u, v \in
V(F_{0}-F_{1})$ and $w \in V(F_{1}-F_{0})$.  Then the gate of $w$ in
$F_{0}$ must belong to $I_{F}(w,u) \cap I_{F}(w,v)$, contrary to the
hypothesis that $I_{G}(w,u) \cap I_{G}(w,v) = \{w \}$.  Therefore $F$
is a hypertorus, and consequently $G$ is a hyper-median partial cube by
Proposition~\ref{P:hypermed.p.c./met.triangle}.
\end{proof}

The following result is an immediate consequence of Theorems~\ref{T:tricycle/decomp.fin.hypernet.} and \ref{T:charact.hyper-median} and Corollary~\ref{C:quasi-hypertori/hyper-median}.

\begin{thm}\label{T:charact.finite.hyper-median}
A finite partial cube is a hyper-median partial cube if and only if it is obtained by a sequence of gated amalgamations from finite quasi-hypertori.
\end{thm}

\begin{pro}\label{P:hyper-med./fin.conv.subgr.}
A partial cube is hyper-median if and only if so are all its finite convex subgraphs.
\end{pro}

\begin{proof}
Let $G$ be a partial cube.  Assume that $G$ is hyper-median, and let $H$ be one of its finite convex subgraphs.  Then $H$ is ph-homogeneous by Theorem~\ref{T:hypernet./gat.amalg.+cart.prod.}, and moreover $H$ is tricycle-free since so is $G$.  Hence $H$ is a hyper-median partial-cube by Theorem~\ref{T:charact.hyper-median}.

Conversely, assume that all finite convex subgraphs of $G$ are hyper-median.  Then $G$ is ph-homogeneous because the pre-hull number of any of its finite convex subgraphs is at most $1$.  Moreover $G$ is tricycle-free since otherwise any tricycle $(C_1,C_2,C_3)$ of $G$ would be a tricycle of the finite convex subgraph $co_G(C_1 \cup C_2 \cup C_3)$, contrary to the fact that this subgraph is tricycle-free by Theorem~\ref{T:charact.hyper-median}.  Therefore $G$ is a hyper-median partial cube by this theorem.
\end{proof}

We have the following consequence of the above theorem and proposition.

\begin{cor}\label{C:charact.hyper-median}
A partial cube is a hyper-median partial cube if and only if all its finite convex subgraphs are obtained by a sequence of gated amalgamations from finite quasi-hypertori.
\end{cor}

\subsection{Subgraphs and operations}\label{SS:subgr.+operat.}

By Proposition~\ref{P:hypermed.p.c./met.triangle} we clearly have:

\begin{pro}\label{P:conv.subgr./hypermed.}
Any convex subgraph of a hyper-median partial cube is hyper-median.
\end{pro}

\begin{thm}\label{T:gat.amalgam/hypermed.}
Let $G$ be the gated amalgam of two partial cubes $G_0$ and $G_1$.  Then $G$ is a hyper-median partial cube if and only if so are $G_0$ and $G_1$.
\end{thm}

\begin{proof}
The necessity is clear by Proposition~\ref{P:conv.subgr./hypermed.} since $G_0$ and $G_1$ are isomorphic to two gated subgraphs of $G$.  Conversely, assume that $G = G_{0} \cup G_{1}$ where $G_{0}$ and $G_{1}$ are hyper-median partial cubes that are gated subgraphs of $G$.  Then $G$ is a Peano partial cube by Theorem~\ref{T:hypernet./gat.amalg.+cart.prod.}.  Let $(u,v,w)$ be a metric triangle of $G$.  Suppose that  $(u,v,w)$ is not a triple of vertices of $G_i$ for some $i = 0$ or $1$.  Without loss of generality suppose that $u, v \in V(V_0-V_1)$ and $w \in V(V_1-V_0)$.  Let $w'$ be the gate of $w$ in $G_0$.  Then $w' \in I_G(u,w) \cap I_G(v,w)$, contrary to the fact that $(u,v,w)$ is a metric triangle.

Therefore $(u,v,w)$ is a metric triangle of $G_i$ for some $i = 0$ or $1$.  Then its convex hull in $G_i$, and thus in $G$, induces a hypertorus because $G_i$ is hyper-median.  Hence $G$ is a hyper-median partial cube by Proposition~\ref{P:hypermed.p.c./met.triangle}.
\end{proof}

\begin{thm}\label{T:cart.prod./hyper-median}
Let $G = G_0 \Box G_1$ be the Cartesian product of two partial cubes $G_0$ and $G_1$.  Then $G$ is a hyper-median partial cube if and only if so are $G_0$ and $G_1$.
\end{thm}

\begin{proof}
Assume that $G$ is a hyper-median partial cube.  Let $F_i$ be a $G_i$-layer of $G$ for some $i = 0$ or $1$.  Then $F_i$ is a convex subgraph of $G$.  Therefore, by Proposition~\ref{P:conv.subgr./hypermed.}, $F_i$, and thus $G_i$, is a hyper-median partial cube.

Conversely, assume that $G_0$ and $G_1$ are hyper-median partial cubes.  Let $(u,v,w)$ be a metric triangle of $G$.  Clearly, for $i = 0, 1$, either $pr_i(u) = pr_i(v) = pr_i(w)$ or $(pr_i(u),pr_i(v),pr_i(w))$ is a metric triangle of $G_i$, and this occurs at least for one value of $i$.  Because $G_i$ is a hyper-median partial cube, it follows that the convex hull of $\{pr_i(u),pr_i(v),pr_i(w)\}$ induces either a $K_1$ or a hypertorus.  Consequently, by the Convex Subgraph Property of the Cartesian product, $co_G(u,v,w)$ induces a hypertorus, and thus $G$ is a hyper-median partial cube by Proposition~\ref{P:hypermed.p.c./met.triangle}.
\end{proof}

\subsection{Special hyper-median partial cubes}\label{SS:spec.hyp-med.}

\subsubsection{Hyper-median netlike partial cubes}\label{SSS:hyp-med.netl.}

A hyper-median of a triple $(u,v,w)$ of vertices of a netlike partial cube $G$ is a quasi-median $(x,y,z)$ of $(u,v,w)$ such that there exists a gated cycle $C$ of $G$ that passes through $x, y, z$.  Moreover, $x, y$ and $z$ are the gates in $C$ of $u, v$ and $w$, respectively.  A netlike partial cube that is hyper-median was said in~\cite{P08-4} to have the \emph{Median Cycle Property} (MCP for short).

By~\cite[Theorem 3.5]{P07-3}, \ref{P:hyper-med./fin.conv.subgr.} and~\ref{P:netlike/fin.conv.subgr.}, we obtain the following particular case of Theorem~\ref{T:charact.hyper-median}.

\begin{thm}\label{T:charact.hyper-med.net.}
A partial cube is a hyper-median netlike partial cube if and only if all its finite convex subgraphs are obtained from even cycles and finite hypercubes by a sequence of gated amalgamations.
\end{thm}

\subsubsection{Median graphs}\label{SSS:med.gr.}

A median graph is trivially a hyper-median partial cubes, and we clearly have:

\begin{pro}\label{P:med./fin.conv.subgr.}
A partial cube is a median graph if and only if so are all its finite convex subgraphs.
\end{pro}

From this result and the result of Bandelt and van 
de Vel~\cite{BV91}, we have:

\begin{thm}\label{T:med.gr./gated.amalg.} A
graph is a median graph if and only if all its finite convex subgraphs are obtained from finite hypercubes by a sequence of gated amalgamations.
\end{thm}

\subsubsection{Cellular bipartite graphs}\label{SSS:cell.bip.gr.}

Finite cellular bipartite graphs were defined and studied by Bandelt and Chepoi \cite{BC96}.  Their definition can be extended to infinite graphs by using condition (3) of~\cite[Theorem~1]{BC96} as follows.

\begin{defn}\label{D:cell.bip.gr.}
A bipartite graph is a \emph{cellular bipartite graph} if it is tricycle-free, and each of its isometric cycles is gated.
\end{defn}

Cellular bipartite graphs are hyper-median Peano partial cubes.  As the two parts of the above definition are properties of finite subgraphs, we clearly have:

\begin{pro}\label{P:cell./fin.conv.subgr.}
A partial cube is a cellular bipartite graph if and only if so are all its finite convex subgraphs.
\end{pro}

We infer from the above proposition and condition (4) of~\cite[Theorem~1]{BC96} the following result:

\begin{thm}\label{T:cell.bip.gr./gated.amalg.} A
graph is a cellular bipartite graph if and only if all its finite convex subgraphs are obtained from  copies of $K_2$ and even cycles by a sequence of gated amalgamations.
\end{thm}

 By Theorems~\ref{T:charact.hyper-med.net.} and~\ref{T:cell.bip.gr./gated.amalg.}, cellular bipartite graphs are hyper-median netlike partial cubes.  The following proposition is then an immediate consequence of Definition~\ref{D:cell.bip.gr.}, Theorem~\ref{T:cube-free netlike}, Theorems~\ref{T:reg. hypernet./str.sem.-periph./quasi-hypert.} and~\ref{T:conv.reg.H-subgr.=>gated}.
 
 \begin{pro}\label{P:cell.=hyp-med.cub-free.net.p.c.}
 A graph is a cellular bipartite graph if and only it is a hyper-median cube-free netlike partial cube.
 \end{pro}

\subsection{Hypercellular partial cubes}\label{SS:hypercell.}

We recall some concepts, results or consequences thereof of Chepoi, Knauer and Marc~\cite{CKM19}.

\begin{enumerate}
\item[(1)]  A \emph{pc-minor} of a partial cube $G$ is any graph that can be obtained from a convex subgraph of $G$ by a sequence of $\Theta$-contactions.
\item[(2)]  A partial cube $G$ is said to be \emph{hypercellular} if $Q_3^-$ is not a pc-minor of $G$.
\item[(3)]  \emph{Any pc-minor of a hypercellular graph is hypercellular.}
\item[(4)]  \emph{Any pc-minor of some pc-minor of a partial cube $G$ is a pc-minor of $G$.}
\item[(5)]  \emph{A partial cube $G$ is hypercellular if and only if each finite convex subgraph of $G$ can be obtained by successive gated amalgamations from Cartesian products of copies of $K_2$ and even cycles.}
\end{enumerate}

From (5) and Corollary~\ref{C:charact.hyper-median}, it follows that \emph{a partial cube is hypercellular if and only if it is hyper-median.}\\

\begin{lem}\label{L:ph(G)>1}
Let $G$ be a partial cube such that $ph(G) > 1$.  Then $Q_3^-$ is a pc-minor of $G$.
\end{lem}

\begin{proof}
If $ph(G) > 1$, then $G$ is not a Peano partial cube by Proposition~\ref{P:hypernet./ph1}, and a fortiori not a hyper-median partial cube.  Hence $Q_3^-$ is a pc-minor of $G$ by (2).
\end{proof}

Note that the converse is false.  Indeed any Peano partial cube that is not hyper-median has a pc-minor isomorphic to $Q_3^-$.

\begin{pro}\label{P:hyp-med./pc-min.}
A partial cube $G$ is hyper-median if and only if $ph(F) \leq 1$ for every pc-minor $F$ of $G$.
\end{pro}

\begin{proof}
Assume that $ph(F) \leq 1$ for every pc-minor $F$ of $G$.  Then $Q_3^-$ is not a pc-minor of $G$, since $ph(Q_3^-) = 2$.  Hence $G$ is hyper-median by (2).

Conversely suppose that $ph(F) \geq 2$ for some pc-minor $F$ of $G$.  Then $Q_3^-$ is a pc-minor of $F$ by Lemma~\ref{L:ph(G)>1}.  Because a pc-minor of $F$ is also a pc-minor of $G$ by (4), it follows that $G$ is not hyper-median by (2).
\end{proof}

From Proposition~\ref{P:cell.=hyp-med.cub-free.net.p.c.}, \ref{P:cube-free netlike/isom.subgr.} and~\ref{P:hyp-med./pc-min.} we infer the following corollary.

\begin{cor}\label{C:cell.}
A partial cube $G$ is a cellular bipartite graph if and only if it satifies the following two properties:

\textnormal{(i)}\; $ph(F) \leq 1$ for every finite isometric subgraph $F$ of $G$;

\textnormal{(ii)}\; $ph(F) \leq 1$ for every pc-minor $F$ of $G$.
\end{cor}

\subsection{Prime Peano partial cubes}\label{SS:prime}

A Peano partial cube is said to be \emph{prime} if it is neither the Cartesian product nor the gated amalgam of smaller Peano partial cubes.  By Theorem~\ref{T:tricycle/decomp.}, copies of $K_2$ and even cycles of length greater than $4$ are the finite prime hyper-median partial cubes, and thus are prime Peano partial cubes.  A benzenoid graph that is the union of three $6$-cycles forming a tricycle is also a prime Peano partial cube.  All these examples are netlike partial cubes, but some finite prime Peano partial cubes are not netlike, as is shown by the following example.  Let $F$ be the union of two $6$-cycles that have exactly one edge in common, and let $H := F \Box K_2$ be the prism over $F$.  Finally let $G$ (see Figure~\ref{F:prime}) be the graph that is obtained by adding a new $6$-cycle $C$ to $H$ so that $C$ and $H$ have exactly two edges in common in such a way that, if $A$ and $B$ are the two $6$-cycles of one of the two $F$-layers of $H$, then $(A,B,C)$ is a tricycle of $G$.  This graph $G$ is clearly a Peano partial cube that is not netlike, and we can easily check that it is prime.

\begin{figure}[!h]
    \centering
    {\tt    \setlength{\unitlength}{0.70pt}
\begin{picture}(230,196)
\thinlines    \put(105,13){$C$}
              \put(196,125){$B$}
              \put(10,125){$A$}
              \put(77,70){\line(0,-1){24}}
              \put(77,70){\line(0,-1){24}}
              \put(109,31){\line(2,1){30}}
              \put(109,31){\line(-2,1){31}}
              \put(139,72){\line(0,-1){24}}
              \put(139,46){\circle*{6}}
              \put(110,31){\circle*{6}}
              \put(77,46){\circle*{6}}
              \put(139,73){\line(-2,1){31}}
              \put(139,183){\line(-2,-1){30}}
              \put(77,73){\line(2,1){31}}
              \put(77,181){\line(2,-1){29}}
              \put(107,167){\circle*{6}}
              \put(107,88){\circle*{6}}
              \put(139,182){\circle*{6}}
              \put(190,155){\circle*{6}}
              \put(139,154){\circle*{6}}
              \put(169,139){\circle*{6}}
             
              \put(140,100){\circle*{6}}
              \put(169,115){\circle*{6}}
              \put(190,99){\circle*{6}}
              \put(139,73){\circle*{6}}
              \put(139,73){\line(2,1){53}}
              \put(190,156){\line(-2,1){51}}
              \put(190,155){\line(0,-1){55}}
              \put(169,140){\line(0,-1){24}}
              \put(139,155){\line(0,1){25}}
              \put(139,100){\line(0,-1){26}}
              \put(139,155){\line(-2,-1){31}}
              \put(139,155){\line(2,-1){30}}
              \put(139,100){\line(-2,1){31}}
              \put(139,100){\line(2,1){30}}
              \put(190,155){\line(-4,-3){21}}
              \put(190,100){\line(-4,3){21}}
              \put(25,99){\line(3,2){22}}
              \put(25,154){\line(4,-3){21}}
              \put(77,99){\line(-2,1){31}}
              \put(77,154){\line(2,-1){30}}
              \put(77,154){\line(-2,-1){31}}
              \put(77,99){\line(0,-1){26}}
              \put(77,154){\line(0,1){25}}
              \put(107,139){\line(0,-1){24}}
              \put(46,137){\line(0,-1){23}}
              \put(25,154){\line(0,-1){55}}
              \put(75,73){\line(-2,1){51}}
              \put(25,155){\line(2,1){53}}
              \put(78,99){\circle*{6}}
              \put(47,115){\circle*{6}}
              \put(77,153){\circle*{6}}
              \put(47,138){\circle*{6}}
              \put(77,182){\circle*{6}}
              \put(25,99){\line(3,2){22}}
              \put(25,154){\line(4,-3){21}}
              \put(77,99){\line(2,1){30}}
              \put(77,99){\line(-2,1){31}}
              \put(77,154){\line(2,-1){30}}
              \put(77,154){\line(-2,-1){31}}
              \put(77,99){\line(0,-1){26}}
              \put(77,154){\line(0,1){25}}
              \put(107,139){\line(0,-1){24}}
              \put(46,137){\line(0,-1){23}}
              \put(25,154){\line(0,-1){55}}
              \put(75,73){\line(-2,1){51}}
              \put(25,155){\line(2,1){53}}
              \put(77,72){\circle*{6}}
              \put(107,115){\circle*{6}}
              \put(78,99){\circle*{6}}
              \put(47,115){\circle*{6}}
              \put(25,99){\circle*{6}}
              \put(107,138){\circle*{6}}
              \put(77,153){\circle*{6}}
              \put(77,182){\circle*{6}}
              \put(25,154){\circle*{6}}
              \qbezier(107,167)(95,127)(107,88)
              \put(107,88){\line(0,1){25}}
              \put(107,167){\line(0,-1){25}}
\end{picture}}
\caption{A prime Peano partial cube that is not netlike.}
\label{F:prime}
\end{figure}
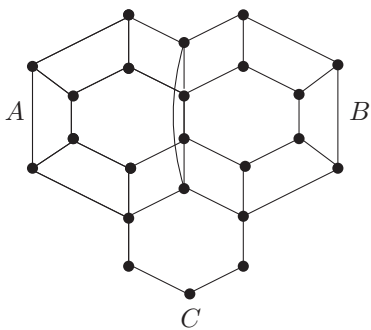

\subsection{The expansion procedure}\label{SS:expans.proc.}

In this subsection we extend to finite hyper-median partial cubes the expansion procedure that we introduced in \cite{P07-3} for netlike partial cubes.  We recall a result of Chepoi:

\begin{pro}\label{P:exp.fin.part.cub.}\textnormal{(Chepoi~\cite{Che88,Che94})}
A finite graph is a partial cube if and only if it can be obtained
from $K_{1}$ by a sequence of expansions.
\end{pro}

Several theorems of this kind have been stated for different
subclasses of partial cubes, see~\cite{IK98}.  The first one is the
following theorem of Mulder for median graphs.  An expansion of a
partial cube with respect to a proper cover $(V_{0},V_{1})$ is said to
be \emph{convex} if $V_{0} \cap V_{1}$ is convex.

\begin{pro}\label{P:exp.fin.med.gr.}\textnormal{(Mulder \cite{M78})} A
finite graph is a median graph if and only if it can be obtained from
$K_{1}$ by a sequence of convex expansions.
\end{pro}

For netlike partial cubes and Peano partial cubes such a result is impossible.  As we showed in \cite[Subsection 6]{P07-3}, there exist
netlike partial cubes such that none of their $\Theta$-contractions is
netlike nor ph-homogeneous (see Figure~\ref{F:ph-prime}).

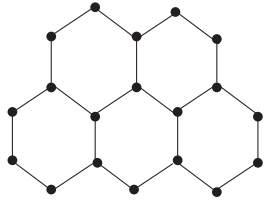
\begin{figure}[!h]
    \centering
   {\tt    \setlength{\unitlength}{0.92pt}
\begin{picture}(126,117)
\thinlines   
              \put(81,31){\circle*{4}}
              \put(97,18){\circle*{4}}
              \put(114,30){\circle*{4}}
              \put(29,61){\circle*{4}}
              \put(64,61){\circle*{4}}
              \put(97,61){\circle*{4}}
              \put(13,51){\circle*{4}}
              \put(47,49){\circle*{4}}
              \put(81,51){\circle*{4}}
              \put(114,49){\circle*{4}}
              \put(13,31){\circle*{4}}
              \put(29,19){\circle*{4}}
              \put(48,30){\circle*{4}}
              \put(63,19){\circle*{4}}
              \put(64,82){\circle*{4}}
              \put(97,81){\circle*{4}}
              \put(29,82){\circle*{4}}
              \put(47,94){\circle*{4}}
              \put(80,92){\circle*{4}}
              \put(97,18){\line(-4,3){17}}
              \put(97,19){\line(3,2){16}}
              \put(114,50){\line(0,-1){19}}
              \put(64,18){\line(-4,3){17}}
              \put(63,19){\line(3,2){16}}
              \put(81,50){\line(0,-1){19}}
              \put(12,50){\line(3,2){16}}
              \put(13,51){\line(0,-1){19}}
              \put(47,50){\line(0,-1){19}}
              \put(30,19){\line(3,2){16}}
              \put(30,18){\line(-4,3){17}}
              \put(80,50){\line(-4,3){17}}
              \put(80,51){\line(3,2){16}}
              \put(97,82){\line(0,-1){19}}
              \put(47,49){\line(-4,3){17}}
              \put(47,50){\line(3,2){16}}
              \put(64,81){\line(0,-1){19}}
              \put(29,82){\line(0,-1){19}}
              \put(115,49){\line(-4,3){17}}
              \put(98,79){\line(-4,3){17}}
              \put(64,81){\line(-4,3){17}}
              \put(64,81){\line(3,2){16}}
              \put(30,83){\line(3,2){16}}
\end{picture}}
\caption{A benzenoid graph none of whose $\Theta$-contractions is ph-homogeneous.}
\label{F:ph-prime}
\end{figure}

We first extend the definition of strong ph-stability.

\begin{defn}\label{D:gen.strong.ph-stab.}
Let $A$ be a set of vertices of some partial cube $G$.  Then $A$ is said to be \emph{strongly ph-stable} in $G$ if, for any vertex $u \in \mathcal{I}_G(A)-A$, there exists a convex $A$-path $P_u$ which passes through $u$ and which satisfies the following two properties:

(SPS1')\; For every $x \in \mathcal{I}_{G}(A)$,\; $u \in I_G(x,v)$ for some endvertex $v$ of $P_u$.

(SPS2')\; For all vertices $x, y \in A$ such that $u \in I_G(x,y)$,\; $P_u$ is a subpath of some $(x,y)$-geodesic.
\end{defn}

\begin{defn}\label{D:ph-respectful.prop.cov.}
A proper cover $(V_0,V_1)$ of a partial cube $G$ is said to be \emph{ph-respectful} if it has the following properties:

(PHR1)\; $V_0 \cap V_1$ is strongly ph-stable in $G[V_i]$ for $i = 0, 1$;

(PHR2)\; $\mathcal{I}_G(V_0,\cap V_1)$ is gated.
\end{defn}

\begin{defn}\label{D:ph-resp.expans.}
An expansion $G_1$ of a partial cube $G$ with respect to a ph-respectful proper cover of $G$ is called a \emph{ph-respectful expansion} of $G$.
\end{defn}

\begin{thm}\label{T:ph-resp.expans.hypomed.p.c.}
Any ph-respectful expansion of a hyper-median partial cube is a hyper-median partial cube.
\end{thm}

\begin{proof}
We use the notations about expansions introduced in~\ref{SS:expansion}.  Let
$G_{1}$ be the expansion of some Peano partial cube $G$ with respect to a proper cover
$(V_{0},V_{1})$.

(a)\; Assume that $(V_{0},V_{1})$ is ph-respectful.  Let $ab$
be an edge of $G_{1}$.  We will distinguish three cases.

\emph{Case 1.}\; $ab = \psi_{0}(x)\psi_{1}(x)$ for some $x \in V_{0}
\cap V_{1}$.

Then $U_{ab}^{G_{1}} = \psi_{0}(V_{0} \cap V_{1})$.  Hence
$U_{ab}^{G_{1}}$ is strongly ph-stable by (PHR1).\\

\emph{Case 2.}\; There is $i \in \{0,1 \}$ such that $ab =
\psi_{i}(a_{i})\psi_{i}(b_{i})$ for some edge $a_{i}b_{i} \in
E(G[V_{i}])$ which is not $\Theta$-equivalent to any edge of
$G[\mathcal{I}_{G[V_{i}]}(V_{0} \cap V_{1})]$.

Because the set $\psi(V_{i})$ is convex, it follows that
$\mathcal{I}_{G_{1}}(U_{ab}^{G_{1}}) =
\psi_{i}(\mathcal{I}_{G}(U_{a_{i}b_{i}}^{G}))$.  Let $u_i \in \mathcal{I}_{G}(U_{a_{i}b_{i}}^{G})$ and $u := \psi_i(u_i)$.  Then, clearly, a $U_{ab}^{G_1}$-geodesic $P$ passing through $u$ satisfies the properties (SPS1) and (SPS2) if and only if $\psi_i(P)$ is a $U_{a_ib_i}^{G_i}$-geodesic which satisfies (SPS1) and (SPS2).

Therefore $U_{ab}^{G_1}$ is strongly ph-stable since so is $U_{a_ib_i}^{G_i}$.\\

\emph{Case 3.}\;  There is $i \in \{0,1 \}$ such that $ab =
\psi_{i}(a_{i})\psi_{i}(b_{i})$ for some edge $a_{i}b_{i} \in
E(G[V_{i}])$ which is $\Theta$-equivalent to some edge of
$G[\mathcal{I}_{G[V_{i}]}(V_{0} \cap V_{1})]$.

Without loss of generality we will suppose that $i = 0$ and that
$a_{0}b_{0}$ is an edge of $G[\mathcal{I}_{G[V_{0}]}( V_{0} \cap
V_{1})]$.  Let $u \in \mathcal{I}_{G_{1}}(U_{ab}^{G_{1}})$.  Let $i = 0$ or $1$ be such that $u = \psi_i(u_i)$ for some $u_i \in \mathcal{I}_{G}(U_{a_{0}b_{0}}^{G})$.  Denote by $P$ the $U_{a_0b_0}^G$-geodesic that is associated with $u_i$.

We have two subcases.

\emph{Subcase 3.1}\; $P$ is a path of $G[V_i]$.

Then, by Lemma~\ref{L:interv.}, $\psi_i(P)$ is the $U_{ab}^{G_1}$-geodesic that is associated with $u$.\\

\emph{Subcase 3.2}\; $P$ meets $V_j$ for $j = 0, 1$.

According to the fact that $P$ is convex by definition, it follows that $P_i := P \cap G[V_i]$ is a geodesic for $i = 0, 1$, and that $P = P_0 \cup P_1$.  Then, by Lemma~\ref{L:interv.}, $\psi_0(P_0) \cup \langle \psi_0(x),\psi_1(x)\rangle \cup \psi_1(P_1)$, where $x$ is the only vertex of $P_0 \cap P_1$, is the $U_{ab}^{G_1}$-geodesic that is associated with $u$.  Note that it is convex by Corollary~\ref{C:conv.G0/conv.G1}.

Therefore $U_{ab}^{G_1}$ is strongly ph-stable.

It follows from these three cases that $G_{1}$ is a Peano partial cube.\\

(b)\; Assume now that $G$ is a hyper-median partial cube and that $(V_{0},V_{1})$ is a
gated ph-respectful proper cover of $G$.  By (a), $G_{1}$
is a Peano partial cube.  Suppose that $G_{1}$ contains a tricycle
$(C_{1},C_{2},C_{3})$.  Let $j \in \{1,2,3 \}$.  The length of $C_{j}$
is greater than $4$ by definition.  Therefore
$C_{j} = \psi(C'_{j})$, where $C'_{j}$ is a convex cycle of $G$, and
thus whose vertex set cannot be contained in $V_{0} \cap V_{1}$.
Because $G$ contains no tricycle by Theorem~\ref{T:charact.hyper-median},
at least two of the cycles $C'_{1}, C'_{2}, C'_{3}$, say $C'_{2},
C'_{3}$, have a unique vertex in common.  Then this vertex, say $x$,
belongs to $V_{0} \cap V_{1}$ and is such that
$\psi_{0}(x)\psi_{1}(x)$ is the common edge of $C_{2}$ and $C_{3}$.
Therefore $C'_{1}$ intersects $C'_{2}$ and $C'_{3}$ in two distinct
edges.  It follows that $V(C'_{1}) \subseteq V_{i}$ for some $i \in
\{0,1 \}$, and moreover that $C'_{1}$ has at least three vertices in
$\mathcal{I}_{G}(V_{0} \cap V_{1})$.  Whence $V(C'_{1}) \subseteq
V_{i} \cap \mathcal{I}_{G}(V_{0} \cap V_{1})$ since
$\mathcal{I}_{G}(V_{0} \cap V_{1})$ is gated by (PHR2).

Let $u$ be the antipode of $x$ in $C'_1$.  Because $V_0 \cap V_1$ is strongly ph-stable in $G[V_i]$ by (PHR1), there exists a $(V_0,V_1)$-geodesic $P$ in $G[V_i]$ that passes through $u$ and satisfies the properties (SPS1') and (SPS2').  Let $v_2$ and $v_3$ be the antipodes of $x$ in $C'_2$ and $C'_3$, respectively.  By (SPS1'), $u \in I_G(v_2,w)$ for some endvertex $w$ of $P$.  By (SPS2'), $P$ is a subpath of any $(v_2,w)$-geodesic which passes through $u$.  Because any $(v_2,u)$-geodesic meets $V_0 \cap V_1$ in $v_2$ only, it follows that $v_2$ is an endvertex of $P$.  Likely $v_3$ is an endvertex of $P$, contrary to the fact that $u \notin I_G(v_2,v_3)$ since the length of $C'_1$ is at least $6$.

Consequently $G_{1}$ contains no tricycle, and hence it is a hyper-median partial cube by
Theorem~\ref{T:charact.hyper-median}.
\end{proof}

\begin{defn}\label{D:periph.prop.cov.}
A proper cover $(V_0,V_1)$ of a partial cube $G$ is said to be \emph{peripheral} if $V_i = \mathcal{I}_{G[V_i]}(V_0 \cap V_1)$ for some $i = 0$ or $1$.  An expansion of a partial cube $G$ with respect to a peripheral proper cover of $G$ is called a \emph{peripheral expansion} of $G$.
\end{defn}

We can now prove the main theorem of this subsection, which is of the
same kind of results as Chepoi's~\cite{Che88},
Mulder's~\cite{M78} (see Propositions~\ref{P:exp.fin.part.cub.}
and~\ref{P:exp.fin.med.gr.}) and Polat~\cite[Theorem 6.15]{P07-3}.

\begin{thm}\label{T:hypom.p.c./expans.}
A finite graph is a hyper-median partial cube if and only
if it can be obtained from $K_{1}$ by a sequence of
peripheral ph-respectful expansions.
\end{thm}

\begin{proof}
By Theorem~\ref{T:ph-resp.expans.hypomed.p.c.}, we only have to prove
the necessity.  The proof will be done by induction on the number of
vertices of finite hyper-median partial cubes. This is
trivial if such a graph has only one vertex.  Let $n \geq 1$.  Suppose
that any finite hyper-median partial cube with at most
$n$ vertices can be obtained from $K_{1}$ by a sequence of peripheral ph-respectful expansions.  Let $G$ be a hyper-median partial cube having $n+1$ vertices.  Because $G$ is finite and $ph(G) \leq 1$, there exists an edge $ab$ of $G$ such
that $W_{ab} = co_G(U_{ab}^G) = \mathcal{I}_{G}(U_{ab}^G)$.

Let $H := G/ab$ be the $\Theta$-contraction of $G$ with respect to the
$\Theta$-class of $ab$.  Put $V_{0} := \gamma_{ab}(W_{ab}^{G})$ and
$V_{1} := \gamma_{ab}(W_{ba}^{G})$.  Then $(V_{0},V_{1})$ is a proper
cover of $H$.  For simplification we will identify $H$ with the graph
obtained from $G$ by collapsing each edge $xy$ which is
$\Theta$-equivalent to $ab$ onto its endvertex $y \in U_{ba}^{G}$.
Whence $V_{1} = W_{ba}^{G}$,\; $V_{0} \cap V_{1} = U_{ba}^{G}$ and $\mathcal{I}_G(U_{ba}^G) = \mathcal{I}_{H[V_1]}(V_0 \cap V_1)$.

First note that, if $H$ is ph-homogeneous, then the proper cover $(V_0,V_1)$ is ph-respectful.  Indeed, (PHR1) is a consequence of the facts that $U_{ab}^G$ and $U_{ba}^G$ are strongly ph-stable since $G$ is ph-homogeneous, and (PHR2) is a consequence of the fact that $G_{\overline{ab}}$ is gated by Lemma~\ref{L:tricycle/Gab} since $G$ is tricycle-free by Theorem~\ref{T:charact.hyper-median}.  Moreover $(V_0,V_1)$ is peripheral because $V_0 = \gamma_{ab}(\mathcal{I}_{G}(U_{ab}^G)) = \mathcal{I}_{H[V_0]}(V_0 \cap V_1)$.

Now we show that $H$ is a hyper-median partial cube. Let
$uv$ be an edge of $H$.  Without loss of generality we can suppose
that $uv$ is an edge of $H[V_{1}]$, i.e., of $G[W_{ba}^{G}]$, because
each edge of $H[V_{0}]$ is $\Theta$-equivalent to an edge of
$H[V_{1}]$ since $W_{ab}^{G} = \mathcal{I}_{G}(U_{ab}^{G})$ by the
choice of $ab$.\\

\emph{Case 1.}\; $uv$ is $\Theta$-equivalent to an edge of 
$H[\mathcal{I}_{H}(V_0 \cap V_1)]$.

Clearly $\mathcal{I}_H(U_{uv}^H) \subseteq \mathcal{I}_{H}(V_0 \cap V_1)$.  Let $x \in \mathcal{I}_H(U_{uv}^H)$, and let $P$ be the $U_{uv}^G$-geodesic which is associated with $\psi_i(x)$ for some $i$ such that $x \in V_i$.  Because $P$ is convex, it has at most one edge which is $\Theta$-equivalent with $ab$.  Hence, because of the convexity of $P$, it follows that $\gamma_{ab}(P)$ is also convex and satisfies (SPS1) and (SPS2), and thus is the $U_{uv}^H$-geodesic which is associated with $x$.  Therefore $U_{uv}^H$ is strongly ph-stable.\\

\emph{Case 2.}\; $uv$ is $\Theta$-equivalent to no edge of 
$H[\mathcal{I}_{H}(V_0 \cap V_1)]$.

We will show that $\mathcal{I}_{H}(U_{uv}^{H}) \subseteq \mathcal{I}_{G}(U_{ba}^{G})$.
Suppose that this is not true.  Then there exist two non-adjacent
vertices $x, y \in U_{ba}^{G} \cap \mathcal{I}_{H}(U_{uv}^{H})$ and an $(x,y)$-geodesic
$P$ having only its endvertices in $U_{ba}^{G}$.  Let $C$ be the $ab$-cycle of $G$ which is associated to some inner vertex of $P$.  Without loss of generality, we can suppose that $P =C-W_{ab}$.  Recall that $C$ is a convex cycle of $G$, and that its length is greater than $4$.  Because $P$ is a geodesic in $H[\mathcal{I}_H(U_{uv}^H)]$, we can distinguish two subcases.\\

\emph{Subcase 2.1.}\; $P$ is a path of some bulge $X$ of $\mathcal{I}_G(U_{uv}^G)$.

Let $H := \mathbf{Cyl}[X]$.  Then $H$ is the Cartesian product of some $uv$-cycle $A$ and a component $B$ of $X[U_{uv}^G)]$.  Every vertex of $P$ lie on some $A$-layer of $H$, and $P$ is a path of some $B$-layer of $H$.  Let $P = \langle x_0,\dots,x_n\rangle$ with $n \geq 2$.  For $0 \geq i \geq n$, denote by $A_i$ the $A$-layer of $H$ which passes through $x_i$.  Then $\mathcal{I}_G(U_{x_1x_0}^G) = \bigcup_{1 \leq i \leq n}V(A_i) \cup \{x'_n\}$, where $x'_n$ is the neighbor of $x_n$ in $U_{ba}^G$.  Clearly the set $U_{x_1x_0}^G$ is not ph-stable because, for instance, if $y_n$ is a neighbor of $x_n$ in $A_n$, then $y \notin I_G(x_1,z)$ for every $z \in U_{x_1x_0}^G$, contrary to the fact that $ph(G) \leq 1$ since $G$ is ph-homogeneous.\\

\emph{Subcase 2.2.}\; $P$ is not a path of some bulge of $\mathcal{I}_G(U_{uv}^G)$.

Then there exist three vertices $y_1, y_2, y_3$ of $P$ such that $y_2 \in U_{uv}^G$, $\langle y_1,y_2,y_3\rangle$ is a subpath of $P$, and the edges $y_1y_2$ and $y_2y_3$ belong to distinct convex $uv$-cycles $C_1$ and $C_3$ of $G$, respectively.  Because, by Theorem~\ref{T:conv.reg.H-subgr.=>gated}, two distinct convex, and thus gated,
cycles of $G$ cannot have more than one edge in common, it follows that the intersections of $P$ with $C_1$ and $C_3$ are $\langle y_1,y_2\rangle$ and $\langle y_2,y_3\rangle$, respectively.

Suppose that $C_1$ and $C_3$ are $4$-cycles, i.e., that $y_1, y_3 \in U_{uv}^G$, then, by Lemma~\ref{L:quasi-tricycle}, the triple $(C,C_1,C_3)$ is a quasi-tricycle of the prism $C \Box \langle y_2,y'_2\rangle$, where $y'_2$ is the neighbor of $y_2$ in $U_{vu}^G$, contrary to the hypothesis that $uv$ is not $\Theta$-equivalent to an edge of 
$H[\mathcal{I}_{H}(V_0 \cap V_1)]$.  On the other hand, by Lemma~\ref{L:tricycle/4-cycle}, exactly one of the cycles $C_1$ and $C_3$ cannot be a $4$-cycle.  Therefore, the length of both $C_1$ and $C_3$ is greater than $4$.  It follows that he triple $(C,C_1,C_3)$ is a tricycle of $G$, contrary to the fact that $G$ is tricycle-free by Theorem~\ref{T:charact.hyper-median}.

Consequently, $\mathcal{I}_{H}(U_{uv}^{H}) = 
\mathcal{I}_{G}(U_{uv}^{G})$ and moreover $U_{uv}^{H} = U_{uv}^{G}$.  
Hence $U_{uv}^{H}$ is strongly ph-stable since so is $U_{uv}^{G}$.

Therefore $H$ is a Peano partial cube which clearly contains no
tricycle since $G$ tricycle-free.  Hence $H$ is a hyper-median partial cube, and $G$ is a peripheral ph-respectful expansion of $H$.

This completes the proof since, by the induction hypothesis, $H$,
which has at most $n$ vertices, can be obtained from $K_{1}$ by a
sequence of peripheral ph-respectful expansions.    
\end{proof}

\subsection{Bulge-regular Peano partial cubes}\label{SS:bulge-regul.}

Among Peano partial cubes are those all of whose bulges are finite and in a certain sense regular, such as netlike partial cubes.  In this subsection we study the decomposition of these graphs.

\begin{defn}\label{D:b-regular}
A Peano partial cube $G$ is said to be \emph{bulge-regular} if, for each edge $ab$ of $G$ and any bulge $X$ of $co_G(U_{ab})$, $X$ is finite and all vertices of $X-U_{ab}$ have the same degree in $X$.
\end{defn}

Clearly, \emph{a Peano partial cube $G$ is bulge-regular if and only if any element of $\mathbf{Cyl}[G]$ is finite and regular}, and thus is a finite hypertorus or the prism over a finite hypertorus.

For example, any netlike partial cube $G$ is bulge-regular because any convex hypercylinder of $G$ is an even cycle.  The class of all bulge-regular Peano partial cubes is not closed under:

\textbullet\;  convex subgraphs:  for example the hypertorus $G_0 := C_6 \Box C_6$ contains the Cartesian product $G_1$ of a $6$-cycle with a path of length $2$ as a convex subgraph;

\textbullet\; retracts: for example $G_1$ is a retract of $G_0$;
  
\textbullet\; Cartesian products, as is shown by $G_1$;

\textbullet\; gated amalgams: for example, let $G_{2}$ be the prism over a $6$-cycle, and $G$ 
the gated amalgam of $G_{0}$ and $G_{2}$ along a $6$-cycle.  Then, for any edge $ab$ of the common 
$6$-cycle of $G_{0}$ and $G_{2}$, the subgraph $G_{ab}$ is the only bulge of $co_G(U_{ab})$.  Then $G = \mathbf{Cyl}[G_{ab}]$, but $G$ is not regular.

However we have the following two theorems.

\begin{thm}\label{T:gated.amalg.(cyl.-reg.hypern.along.med.gr.)}
The gated amalgam
of two bulge-regular Peano partial cubes along a median
graph is a bulge-regular Peano partial cube.
\end{thm}

\begin{proof}
Let $G = G_{0} \cup G_{1}$ be the gated amalgam of two of its gated
subgraphs $G_{0}$ and $G_{1}$ that are bulge-regular ph-homogeneous and such that $G_{01} := G_{0} \cap G_{1}$ is
a median graph.  Then $G$ is a Peano partial cube by Theorem~\ref{T:hypernet./gat.amalg.+cart.prod.}.  We have to show that it is bulge-regular.  
Let $ab \in E(G)$,\; $G_{\overrightarrow{ab}} := G[co_{G}(U_{ab}^{G})
\cup W_{ba}^{G}]$ and $(G_i)_{\overrightarrow{ab}} := G_{i}[co_{G_{i}}(U_{ab}^{G_{i}}) \cup
W_{ba}^{G_{i}}]$ for $i = 0, 1$.  We will show that any element of $\mathbf{Cyl}[G,ab]$ are finite and regular.

\emph{Case 1.}\; $U_{ab}^{G} = U_{ab}^{G_{i}}$ for some $i = 0$ or $1$.

Then $G_{\overrightarrow{ab}} = (G_i)_{\overrightarrow{ab}} \cup G_{1-i}$ and $co_{G}(U_{ab}^{G}) = 
co_{G_{i}}(U_{ab}^{G_{i}})$ since $G_{i}$ is gated in $G$.  Hence 
$\mathbf{Cyl}[G,ab] = \mathbf{Cyl}[G_i,ab]$, and we are done.

\emph{Case 2.}\; $U_{ab}^{G} \neq U_{ab}^{G_{i}}$ for $i = 0, 1$.

Then, for $i = 0, 1$, $G_{i}$ has an edge that is $\Theta$-equivalent
to $ab$.  Hence $G_{01}$, which is gated in $G$, also has an edge
$\Theta$-equivalent to $ab$.  Then, without loss of generality we can
suppose that $ab \in E(G_{01})$.  For any $x \in V(G)$ and $i = 0, 1$,
we denote by $g_{i}(x)$ the gate of $x$ in $G_{i}$.

\emph{Claim.\; For every $ab \in E(G)$,\; $x \in
V(\mathcal{I}_{G_{i}}(U_{ab}^{G_{i}}))$ and $i = 0, 1$,\;
$co_{G}(U_{ab}^{G}) = \mathcal{I}_{G}(U_{ab}^{G}) =
\mathcal{I}_{G_{0}}(U_{ab}^{G_{0}}) \cup
\mathcal{I}_{G_{1}}(U_{ab}^{G_{1}})$ and $g_{1-i}(x) \in
U_{ab}^{G_{01}}$.}

Clearly $$W_{ab}^{G} =
W_{ab}^{G_{0}} \cup W_{ab}^{G_{1}} \textnormal{\quad and}\quad
U_{ab}^{G} = U_{ab}^{G_{0}}
\cup U_{ab}^{G_{1}}.$$

Let $i = 0$ or $1$.  By Corollary~\ref{C:Wab=IGab},
$co_{G_{i}}(U_{ab}^{G_{i}}) = \mathcal{I}_{G_{i}}(U_{ab}^{G_{i}})$.
Let $u_{i} \in \mathcal{I}_{G_{i}}(U_{ab}^{G_{i}})$ and $w \in
U_{ab}^{G_{01}}$.  By Lemma~\ref{L:ph-stable} and because $U_{ab}^{G_{i}}$ is ph-stable, there exists $v_{i} \in
U_{ab}^{G_{i}}$ such that $u_{i} \in I_{G_{i}}(w,v_{i})$.  Let
$v'_{i}$ be the neighbor of $v_{i}$ in $U_{ba}^{G_{i}}$.  Then
$v'_{i}, g_{1-i}(v_{i}) \in I_{G_{i}}(v_{i},g_{1-i}(v'_{i}))$ and
$v_{i}, g_{1-i}(v'_{i}) \in I_{G_{i}}(v'_{i},g_{1-i}(v_{i}))$.  Hence
$d_{G_{i}}(v_{i},g_{1-i}(v'_{i})) =
d_{G_{i}}(v_{i},g_{1-i}(v_{i}))+1$ and
$d_{G_{i}}(v'_{i},g_{1-i}(v_{i})) =
d_{G_{i}}(v'_{i},g_{1-i}(v'_{i}))+1$.  It follows that $g_{1-i}(v_{i}) \in
U_{ab}^{G_{01}}$ and $g_{1-i}(v'_{i}) \in U_{ba}^{G_{01}}$.

Because $g_{1-i}(u_{i}) \in I_{G_{i}}(v_{i},w)$, it follows that
$g_{1-i}(u_{i}) \in I_{G_{i}}(g_{1-i}(v_{i}),w)$.  Hence
$g_{1-i}(u_{i}) \in U_{ab}^{G_{01}}$, since $U_{ab}^{G_{01}}$ is
convex by Proposition~\ref{P:med.gr./loc.periph.} because $G_{01}$ is a
median graph.  It follows that
\begin{align*}
I_{G}(u_{0},u_{1}) &=
I_{G_{0}}(u_{0},g_{1}(u_{0})) \cup
I_{G_{01}}(g_{1}(u_{0}),g_{0}(u_{1})) \cup
I_{G_{01}}(g_{0}(u_{1}),u_{1})\\
&\subseteq I_{G_{0}}(v_{0},g_{1}(u_{0})) \cup
I_{G_{01}}(g_{1}(u_{0}),g_{0}(u_{1})) \cup
I_{G_{01}}(g_{0}(u_{1}),v_{1})\\
&\subseteq \mathcal{I}_{G_{0}}(U_{ab}^{G_{0}}) \cup U_{ab}^{G_{01}}
\cup \mathcal{I}_{G_{1}}(U_{ab}^{G_{1}}) \subseteq
\mathcal{I}_{G_{0}}(U_{ab}^{G_{0}}) \cup
\mathcal{I}_{G_{1}}(U_{ab}^{G_{1}}).
\end{align*}
This proves the claim.\\

By the above claim, $(G_i)_{\overrightarrow{ab}} = G_{\overrightarrow{ab}} \cap G_{i}$.  Note that $(G_i)_{\overrightarrow{ab}}$ is
gated in $G_{\overrightarrow{ab}}$ because $G_{i}$ is gated in $G$, and $(G_i)_{\overrightarrow{ab}}$ and $G_{\overrightarrow{ab}}$ are
convex subgraphs of $G_{i}$ and $G$, respectively.  Let $X$ be a bulge
of $co_{G}(U_{ab}^{G})$.  By the claim, $X \subseteq G_{i}$ for some
$i = 0$ or $1$.  It follows that $\mathbf{Cyl}[G,ab] = \mathbf{Cyl}[G_0,ab] \cup \mathbf{Cyl}[G_1,ab]$, and thus we are done.
\end{proof}

Let $G = G_{0} \cup G_{1}$.  We will say that the intersection $G_{0}
\cap G_{1}$ is \emph{relatively thick} (in $G$) if, for any $i = 0, 1$ and $x_{i}
\in V(G_{i}-G_{1-i})$,\; $d_{G}(x_{0},x_{1}) \geq 3$, that is, if any
$(x_{0},x_{1})$-geodesic contains an edge of $G_{0} \cap G_{1}$.  In
other words, $G_{0} \cap G_{1}$ is relatively thick if and only if
$N_{G}(G_{0}-G_{1}) \cap N_{G}(G_{1}-G_{0}) = \emptyset$.  If $G$ is
the gated amalgam of $G_{0}$ and $G_{1}$, and if $G_{0} \cap G_{1}$ is
relatively thick, then we will say that $G$ is the gated amalgam of $G_{0}$ and
$G_{1}$ \emph{along a relatively thick graph}.

\begin{thm}\label{T:gated.amalg.(cyl.-reg.hypern.along.relatively thick.gr.)}
The gated amalgam of two bulge-regular Peano partial cubes along a relatively thick graph is a bulge-regular Peano partial cube.
\end{thm}

\begin{proof}
Let $G = G_{0} \cup G_{1}$ be the gated amalgam of two of its gated
subgraphs $G_{0}$ and $G_{1}$ that are bulge-regular ph-homogeneous and such that $G_{01} := G_{0} \cap G_{1}$ is
relatively thick.  Let $ab \in E(G)$.  We will show that any element of $\mathbf{Cyl}[G,ab]$ are finite and regular.

\emph{Case 1.}\; $U_{ab}^{G} = U_{ab}^{G_{i}}$ for some $i = 0$ or $1$.

This case is analogous to Case~1 in the proof of
Theorem~\ref{T:gated.amalg.(cyl.-reg.hypern.along.med.gr.)}.

\emph{Case 2.}\; $U_{ab}^{G} \neq U_{ab}^{G_{i}}$ for $i = 0, 1$.

As in Case~2 in the proof of
Theorem~\ref{T:gated.amalg.(cyl.-reg.hypern.along.med.gr.)}, we can
suppose that $ab \in E(G_{01})$.  For any $x \in V(G)$ and $i = 0, 1$,
we denote by $g_{i}(x)$ the gate of $x$ in $G_{i}$.

\emph{Claim.}\; $\mathcal{I}_{G}(U_{ab}^{G}) =
\mathcal{I}_{G_{0}}(U_{ab}^{G_{0}}) \cup
\mathcal{I}_{G_{1}}(U_{ab}^{G_{1}}).$

Clearly $$W_{ab}^{G} =
W_{ab}^{G_{0}} \cup W_{ab}^{G_{1}} \textnormal{\quad and}\quad
U_{ab}^{G} = U_{ab}^{G_{0}}
\cup U_{ab}^{G_{1}}.$$

Let $i = 0$ or $1$.  By Corollary~\ref{C:Wab=IGab},
$co_{G_{i}}(U_{ab}^{G_{i}}) = \mathcal{I}_{G_{i}}(U_{ab}^{G_{i}})$.
Let $u_{i} \in \mathcal{I}_{G_{i}}(U_{ab}^{G_{i}}) -
\mathcal{I}_{G_{01}}(U_{ab}^{G_{01}})$.  Then $$I_{G}(u_{0},u_{1}) =
I_{G_{0}}(u_{0},g_{1}(u_{0})) \cup
I_{G_{01}}(g_{1}(u_{0}),g_{0}(u_{1})) \cup
I_{G_{01}}(g_{0}(u_{1}),u_{1}).$$ Moreover $g_{1-i}(u_{i}) \in
N_{G}(G_{i}-G_{1-i})$.  Suppose that $g_{1-i}(u_{i}) \notin
U_{ab}^{G_{01}}$.  Let $w \in U_{ab}^{G_{01}}$.  By
Lemma~\ref{L:ph-stable} and Proposition~\ref{P:hypernet./ph1},
there exists $v_{i} \in U_{ab}^{G_{i}}$ such that $u_{i} \in
I_{G_{i}}(w,v_{i})$.  It follows that $g_{1-i}(u_{i})$ is a vertex of
a bulge $X_{i}$ of $co_{G_{i}}(U_{ab}^{G_{i}})$, and then a vertex of
$H_{i} := \mathbf{Cyl}[X_{i}]$, whose intersection with
$G_{i}-G_{1-i}$ is non-empty.  Then $g_{1-i}(u_{i})$ is a vertex of a
convex cycle $C_{i}$ of $H_{i}$ of length greater than $4$ and that
contains edges $\Theta$-equivalent to $ab$.  Because $G_{01}$ is
gated, and thus $\Gamma$-closed by
Theorem~\ref{T:hypernet.=>G-closed+conv.=gated}, $C_{i}$ must be a
cycle of $G_{01}$.  Moreover, for the same reason, $V(H_{i}) \subseteq
N_{G}[G_{i}-G_{1-i}]$.

Suppose that $g_{1}(u_{0})$ and $g_{0}(u_{1})$ do not belong to 
$U_{ab}^{G_{01}}$, and that there exists a 
$(g_{1}(u_{0}),g_{0}(u_{1}))$-geodesic that does not meet 
$U_{ab}^{G_{01}}$.  Then $C_{0}$ and $C_{1}$ would be cycles of a 
same bulge $X$ of $G_{01}$.  Hence, for $i = 0, 1$,\; $C_{1-i}$ would 
be a cycle of $H_{i}$, contrary to the fact that $V(H_{i}) \cap 
N_{G}(G_{1-i}-G_{i}) = \emptyset$ by what we proved above.  
Consequently, any $(u_{0},u_{1})$-geodesic meets $U_{ab}^{G_{01}}$, 
which proves the claim.\\

It follows that $\mathbf{Cyl}[G,ab] = \mathbf{Cyl}[G_0,ab] \cup \mathbf{Cyl}[G_1,ab]$, as in the proof of the above theorem, and thus we are done.
\end{proof}

By the proofs of the above two theorems, we have:

\begin{pro}\label{P:phi(G)}
Let $G_0$ and $G_1$ be two bulge-regular Peano partial cubes, and $G$ their gated amalgam along a median or a relatively thick graph.  Then $\mathbf{Cyl}[G] = \mathbf{Cyl}[G_0] \cup \mathbf{Cyl}[G_1]$.
\end{pro}

We will now prove a result which is analogous to Theorem~\ref{T:tricycle/decomp.}.

\begin{thm}\label{T:cyl.reg.hypermed./decomp.}
Let $G$ be a compact bulge-regular hyper-median partial
cube that is not a quasi-hypertorus.  Then $G$ is the gated amalgam along a median or a relative thick graph of two of its proper
subgraphs.
\end{thm}

\begin{lem}\label{L:inters.hypercyl.}
Let $G$ be a bulge-regular Peano partial cube.  Then $H_0 \cap H_1$ is a hypercube for all $H_0, H_1 \in \mathbf{Cyl}^+[G]$.
\end{lem}

\begin{proof}
This is obvious if $H_0$ and $H_1$ are hypercubes.  Assume that $H_0$ is not a hypercube, and let $C$ be a convex cycle of $H_0$ of length greater than $4$.  If $H_1$ contains two edges of $C$, then $C \subseteq H_1$ since $H_1$ is gated by Theorem~\ref{T:C-subgr./gated}.  Hence $H_1 = \mathbf{Cyl}[G,ab] = H_0$ for any edge $ab$ of $C$.  Therefore, if $H_1 \neq H_0$, then $H_1$ contains at most one edge of each convex cycle of $H_0$ of length greater than $4$.  It clearly follows that $H_0 \cap H_1$ is a hypercube.
\end{proof}

\begin{proof}[\textnormal{\textbf{Proof of
Theorem~\ref{T:cyl.reg.hypermed./decomp.}}}]
By Theorem~\ref{T:charact.hyper-median} and the fact that $G$ is compact, $G$ is a tricycle-free Peano partial cube.  We will then only complete the proof of Theorem~\ref{T:tricycle/decomp.} without recalling it.  

\emph{Case 1}:\; In this case, $G$ is the gated amalgam of $G_{ \overrightarrow{ab}}$ and $G_{ \overrightarrow{ba}}$ over $G_{ \overline{ab}}$.  Then we are done because the graph $G_{ab}$ is relatively thick 
since any path between a vertex in $W_{ab}$ and a vertex in $W_{ba}$ 
contains an edge $\Theta$-equivalent to $ab$, and thus which belongs 
to $E(G_{ab})$.

\emph{Subcase 2.1}:\; In this case, $Q$ is a hypercube by Lemma~\ref{L:inters.hypercyl.}, and thus a median graph.  Hence we are done.

\emph{Subcase 2.2}:\; In this case, $G$ is the gated amalgam of $G[W_{ab}]$ and $G_{\overline{ab}}$.  It remains to prove that $G_{\overline{ab}} \cap G[W_{ab}]$ is a median graph.  By Proposition~\ref{P:conv.subgr./hypermed.}, as a gated, and thus convex, subgraph of $G$, $G_{\overline{ab}} \cap G[W_{ab}]$ is a hyper-median partial cube.  Then, by
Proposition~\ref{P:hypernet./med.gr.}, it is sufficient to
show that any convex cycle of $G_{\overline{ab}} \cap G[W_{ab}]$ is a $4$-cycle.  Suppose
that $G_{\overline{ab}} \cap G[W_{ab}]$ contains a convex cycle $C$ of length greater
than $4$.  Then $C$ is a convex cycle of two distinct elements of
$\mathbf{Cyl}[G]$, which is impossible by
Lemma~\ref{L:inters.hypercyl.}.  Therefore $G_{\overline{ab}} \cap G[W_{ab}]$ is a median
graph.\\

Consequently, in any case, $G$ is  decomposable along a median or a
relatively thick graph.
\end{proof}

The following result is an immediate consequence of Theorems~\ref{T:charact.finite.hyper-median} and \ref{T:cyl.reg.hypermed./decomp.}.

\begin{thm}\label{T:charact.finite.cyl.reg.hyper-median}
A finite partial cube is a bulge-regular hyper-median partial cube if and only if it is is obtained from finite quasi-hypertori by a sequence of gated amalgamations along median or relative thick graphs.
\end{thm}

\section{Retracts and convex subgraphs}\label{S:retr./convex}

Retracts have been one the basic
topics of metric graph theory, for example in the study of
absolute retracts, and of varieties of graphs -- that is, classes of
graphs closed under retracts and products -- and also to obtain fixed
subgraph theorems in several classes of metric graphs.  In this section, we first prove that the class of Peano partial cubes is closed under retracts.  Any retract of a partial cube $G$ is a faithful subgraph of $G$, but the converse is generally not true, except for some special partial cubes, such as netlike ones.  If, moreover, $G$ is a Peano partial cube, then any retract of $G$ preserves not only the medians, but also the hyper-medians, and thus is what we call a ‘‘strongly faithful’’ partial cube.  However the converse is still not true.

\subsection{Principal cycles of an antipodal partial cube}\label{SS:isom.cycle}

We first recall some results about bipartite antipodal graphs (cf. Subsections~\ref{SS:p.c.ph.at most1} and \ref{SS:consequences}).

\begin{defn}\label{D:princip.cycle}
An isometric cycle $C$ of a bipartite antipodal graph $G$ is called \emph{principal} if $\alpha(C) = C$, where $\alpha : x \mapsto \overline{x}$, $x \in V(G)$, is the antipodal map of $G$.
\end{defn}

Principal cycles always exist; indeed, any geodesic in $G$ lies on a principal cycle.  Moreover, we have:

\begin{pro}\label{P:princ.cycl.}\textnormal{(Glivjak, Kotzig and Plesnik~\cite[Theorem 3]{GKP70})}
An isometric cycle $C$ of a bipartite antipodal graph $G$ is principal if and only if $\mathrm{diam}(C) = \mathrm{diam}(G)$.
\end{pro}

We call the smallest median-stable subgraph $F$ of a partial cube $G$ which contains a subgraph $H$ of $G$ the
 \emph{median-closure} of $H$.  Such 
a subgraph $F$ always exists, and moreover $F$ is finite if so is 
$H$, because $F$ is a subgraph of the convex hull of $H$, which is 
finite by Lemma~\ref{L:gen.propert.}(iii).

\begin{pro}\label{P:TPQ/prop.isom.cycle}\textnormal{(Polat~\cite[Theorem 5.5]{P18})}
Let $C$ be an isometric cycle of  an antipodal partial cube $G$.  The following assertions are equivalent:

\textnormal{(i)}\; $C$ is a principal cycle of $G$. 

\textnormal{(ii)}\; $\mathrm{diam}(C) = \mathrm{diam}(G)$.

\textnormal{(iii)}\; $\mathrm{idim}(C) = \mathrm{idim}(G)$.

\textnormal{(iv)}\; $\mathcal{I}_{G}(C) = G$.

\textnormal{(v)}\; The convex hull of $C$ is $G$.

\textnormal{(vi)}\; The median-closure of $C$ is $G$.
\end{pro}

We will need this proposition in the proofs of several results dealing with quasi-hypertori.

\subsection{Retracts of Peano partial cubes}\label{SS:retr.hypernet.}

We recall that, if $G$ and
$H$ are two graphs, then a map $f : V(G) \to V(H)$ is a
\emph{contraction} (\emph{weak homomorphism} in~\cite{HIK11}) if $f$ is
a non-expansive map between the metric spaces $(V(G),d_{G})$ and
$(V(H),d_{H})$, i.e., $d_{H}(f(x),f(y)) \leq d_{G}(x,y)$ for all $x, y
\in V(G)$.  A contraction $f$ of $G$ onto one of its induced subgraphs
$H$ of $G$ is a \emph{retraction}, and $H$ is a \emph{retract}
(\emph{weak retract} in~\cite{HIK11}) of $G$, if its restriction to
$V(H)$ is the identity.

We know that the class of median graphs and that of netlike partial cubes \cite[Theorem 3.1]{P05-2} are closed under retracts.  The aim of this subsection is the proof of the following extension of these properties.

\begin{thm}\label{T:retract(hypernet.)}
The class of Peano
partial cubes is closed under retracts.
\end{thm}

Actually we will prove a stronger result (Proposition~\ref{P:Gf.hypernet.}) which will be essential in Section~\ref{S:FSP}.  For a self-contraction $f$ of a graph $G$ and $x \in V(G)$ we set:
\begin{align*}
[x]_{f}& := \{f^{n}(x): n \geq 0 \}\\
V(G)_{f}& := \{x \in V(G): f^{n}(x) = x \mbox{\; for some}\; n > 0
\}\\
G_{f}& := G[V(G)_{f}].
\end{align*}

By \cite[Propositions 3.9, 3.11 and 3.12]{P09-2}, we have:

\begin{lem}\label{L:Gf}
Let $f$ be a self-contraction of a partial cube $G$.  Then $G_{f}$ is 
a non-empty faithful subgraph of $G$ whose vertex set is a closed 
subset of $V(G)$.
\end{lem}

\begin{pro}\label{P:Gf.hypernet.}
Let $f$ be a self-contraction of a Peano partial cube $G$.  Then $G_{f}$ is 
a Peano Partial cube.
\end{pro}

\begin{proof}
We already know that $G_f$  is faithful in $G$.  We show that $G_f$ has the Peano Property.  Let $(u,v,w,v',x)$ be a $5$-tuple of vertices of $G_f$ such that $v' \in I_{G_f}(u,w)$ and $x \in I_{G_f}(v,v')$.  Because $G_f$ is an isometric subgraph of $G$, it follows that $v' \in I_{G}(u,w)$ and $x \in I_{G}(v,v')$.  Hence $x \in I_G(u,u')$ for some vertex $u' \in I_G(v,w)$, since $G$ has the Peano Property by Theorem~\ref{T:Peano/ph}.  Because the interval $I_G(v,w)$ is finite, there exist two positive integers $n < m$ such that $f^n(u') = f^m(u')$.  Put $u'' := f^n(u')$ and $p := m-n$.  Then $f^p(u'') = u''$.  Therefore $u'' \in V(G_f)$, and more precisely $u'' \in I_{G_f}(f^p(v),f^p(w))$, and $f^p(x) \in I_{G_f}(f^p(u),u'')$.  Because the intervals of $G$ and thus of $G_f$ are finite, there exists a positive integer $q$ such that $f^q(y) = y$ for all $y \in I_{G_f}(u,v) \cup I_{G_f}(v,w) \cup I_{G_f}(w,u) \cup I_{G_f}(v,v')$.  It follows that $f^{pq}(y) = y$ for all $y \in I_{G_f}(u,v) \cup I_{G_f}(v,w) \cup I_{G_f}(w,u) \cup I_{G_f}(v,v') \cup I_{G_f}(u,u'')$.  This implies that $u'' \in I_{G_f}(v,w)$ and $x \in I_{G_f}(u,u'')$, which proves that $G_f$ has the Peano Property, and thus is a Peano partial cube.

\end{proof}

\begin{proof}[\textnormal{\textbf{Proof of Theorem~\ref{T:retract(hypernet.)}}}]
Let $f$ be a retraction of a Peano
partial cube $G$ onto some subgraph $F$.  Clearly $F = G_{f}$.  Hence $F$ is 
a Peano partial cube by Proposition~\ref{P:Gf.hypernet.}.
\end{proof}

We complete this subsection by three particular results.

\begin{pro}\label{P:TuPuQ/isom.cycle/retract}
Let $C$ be an isometric cycle of some quasi-hypertorus $G$ such that $\mathrm{idim}(C) = \mathrm{idim}(G)$.  Then any 
retract of $G$ that contains $C$ is equal to $G$.
\end{pro}

\begin{proof}
Let $F$ be a retract of $G$ which contains $C$.  Then $F$ is clearly
median-stable.  It follows that $F$ contains the median-closure of $C$
since $C$ is a subgraph of $F$.  Hence $F = G$ by
Proposition~\ref{P:TPQ/prop.isom.cycle}.
\end{proof}

\begin{pro}\label{P:retract(T.K2)}
Let $G = T \Box K_{2}$ be the prism over a hypertorus $T$, and let $T_{0}$ and $T_1$ be the two $T$-layers of $G$, and $u$ a vertex of $T_1$.  Then:

\textnormal{(i)}\; $\mathcal{I}_G(T_0 \cup \langle u\rangle) = G$.

\textnormal{(ii)}\; Any retract of
$G$ that contains $T_0 \cup \langle u\rangle$ is equal to $G$.
\end{pro}

\begin{proof}
(i) is a simple consequence of the Interval Property of the Cartesian product.

(ii)\; Let $F$ be a retract of $G$ that contains $T_0 \cup \langle u\rangle$.  
Suppose that $\mathrm{idim}(T) = d$.  Then $\mathrm{idim}(G) = d+1$.  Let $\langle x_{1},\dots,x_{2d},x_{1}\rangle$ be an isometric cycle of $T$, and $V(K_{2}) = \{ 0,1\}$.  Suppose that $F \neq G$.  Because $F \cap T_1 \neq \emptyset$, there exist $1 \leq i < j \leq 2d+1$ such that $x_{2d+1} := x_1$ and $(x_k,1) \in F$ if and only if $k = i$ or $j$.  Then $f(x_k,1) = (x_{k-1},0)$ for $i+1 \leq k<j$, which is impossible since $(x_j,1) = f(x_j,1)$ is adjacent to $(x_{j-1},1)$ but not to $(x_{j-2},0) = f(x_{j-1},1)$.  Therefore $F = G$.
\end{proof}

\begin{pro}\label{P:retract(C.P)}
Let $G = C \Box P$, where $C \in \mathbf{C}$ and $P$ is a path.  Let $ab$ be an edge of some $C$-layer of $G$,\; $u, v \in U_{ab}$ such that $d_G(u,v) = \mathrm{idim}(G)-1$,\; $R$ a $(u,v)$-geodesic and $R'$ a $(u',v')$-geodesic, where $u'$ and $v'$ are the neighbors in $U_{ba}$ of $u$ and $v$, respectively, and $D := R \cup \langle v,v'\rangle  \cup R' \cup \langle u',u\rangle$.  Then:

\textnormal{(i)}\; $\mathcal{I}_G(D) = G$.

\textnormal{(ii)}\; Any retract of $G$ that contains $D$ is equal to $G$.
\end{pro}

\begin{proof}
$D$ is a cycle of $G$ with $\mathrm{idim}(D) = \mathrm{idim}(G)$.  
The results are then consequences of Propositions~\ref{P:TPQ/prop.isom.cycle} and~\ref{P:TuPuQ/isom.cycle/retract}, respectively, if $D$ is isometric in $G$, which is possible only if the length of $P$ is at most $1$.  Assume that $D$ is not an isometric cycle of $G$.  Put $P = \langle 0,\dots,p\rangle$ for some non-negative integer $p$, and $C = \langle c_1,\dots,c_{2n}\rangle$ for some positive integer $n$.  Then $\mathrm{idim}(G) = \mathrm{idim}(D) = p+n$.  For $0 \leq i \leq n$, denote by $C_i$ the $C$-layer of $G$ whose projection on $P$ is $i$.  Note that, by the properties of $R$ and $R'$, for all $1 \leq j \leq 2n$ and $0 \leq i \leq p$, the sets $V(D) \cap W_{(c_j,i)(c_{j+1},i)}^{C_i}$ and $V(D) \cap W_{(c_{j+1},i)(c_j,i)}^{C_i}$ (with $c_{2n+1} := c_1$) are non-empty.

(i):\; Without loss of generality, we can suppose that $u = (c_1,0)$.  Then $v = (x_n,p)$ since $V(R) \subseteq W_{ab}$ and $d_G(u,v) = \mathrm{idim}(G)-1$, and thus $v' = (x_{n+1},p)$ and $u' = (x_{2n},0)$.  For $j$ with $1 \leq j \leq 2n$.  If $j \leq n$ (resp. $n < j \leq 2n$),  the path $R$ (resp. $R'$) contains exactly one edge that is $\Theta$-equivalent to the edge $x_{j-1}x_j$.  Hence, by the Interval Property of the Cartesian product, $(x_1,0), \dots, (x_j,0) \in \mathcal{I}_G(D)$ or $(x_{j-1},0), \dots, (x_{2n},0), (x_1,0) \in \mathcal{I}_G(D)$ depending on whether $j \leq n$ or $i > n$.  Therefore $V(C_0) \subseteq \mathcal{I}_G(D)$, and for analogous reasons $V(C_i) \subseteq \mathcal{I}_G(D)$ for $0 \leq i \leq p$ since $D$ meets each $C_i$.  Consequently $\mathcal{I}_G(D) = G$.

(ii):\; Let $F$ be a retract of $G$ that contains $D$, and $f$ a retraction of $G$ onto $F$.  Suppose that $F \neq G$.  Then the restriction of $f$ to $C_i$ is not the identity for some $i$ with $0 \leq i \leq p$.

Suppose that $f(C_i) \nsubseteq C_i$.  Then there exists $1 \leq j < k \leq 2n$ with $k-j \geq 2$ such that, $f(c_l,i) = (c_l,i)$ for $l = j, k$ and $f(c_l,i) \notin V(C_i)$ for $j<l<k$.  This is clearly impossible since $f$ would map the path $\langle (c_j,i),\dots,(c_k,i)\rangle$ of length $k-j$ onto a path of length at least $k-j+2$.

It follows that $f(C_i) \subseteq C_i$.  Because $f(C_i) \neq C_i$ by hypothesis, there exist $j, k$ with $1 < j+1 < k \leq 2n$ such that $f(c_j,i) = (c_k,i)$.  Because the paths $\langle (c_j,0),\dots,(c_j,n)\rangle$ and $\langle (c_k,0),\dots,(c_k,n)\rangle$ are $P$-layers of $G$, it follows that $f(c_j,q) = (c_k,q)$ for every $q$ with $0 \leq q \leq p$, contrary to the fact that $(c_j,q) \in V(D)$ for at least one $q$, and thus $f(c_j,q) = (c_j,q)$ since $D$ is a cycle of $F$.

Therefore the restriction of $f$ to $C_i$ is the identity for every $i$, and thus $F = G$.
\end{proof}

\begin{thm}\label{T:retract/hyper-median}
Any retract of a hyper-median partial cube is a hyper-median partial cube.
\end{thm}

\begin{proof}
Let $G$ be a hyper-median partial cube, and $f$ a retraction of $G$ onto one of its proper subgraphs $F$.  Then, by Theorem~\ref{T:retract(hypernet.)}, $F$ is a Peano partial cube, and moreover an isometric subgraph of $G$.  Suppose that $F$ contains a tricycle $(C_0,C_1,C_2)$.  Let $i = 0, 1, 2$.  Then $C_i$ is convex in $F$, and thus isometric in $G$.  Because $G$ is a Peano partial cube, it follows that the convex hull $C'_i$ of $C_i$ in $G$ is a quasi-hypertorus such that $f(C'_i) = C_i$.  Therefore $C'_i = C_i$ by Proposition~\ref{P:TuPuQ/isom.cycle/retract}.  It follows that $(C_0,C_1,C_2)$ is a tricycle of $G$, contrary, by Theorem~\ref{T:charact.hyper-median}, to the fact that a hyper-median partial cube is tricycle-free.  Consequently $F$ contains no tricycle.  Therefore $F$ is hyper-median by Theorem~\ref{T:charact.hyper-median}.
\end{proof}

According to Theorems~\ref{T:hypernet./gat.amalg.+cart.prod.} and \ref{T:retract(hypernet.)} and Theorems~\ref{T:cart.prod./hyper-median} and \ref{T:retract/hyper-median}, respectively, we infer the following result.

\begin{thm}\label{T:varieties}
The class of Peano partial cubes and that of hyper-median partial cubes are varieties.
\end{thm}

\subsection{Cycle-representative subgraphs and moorings}\label{SS:cycle-represent./mooring}

In this subsection we prove a result (Theorem~\ref{T:conv.=>mooring}) which will be useful to obtain the main result of Subsection~\ref{SS:ret.hom-ret./convex}.  The following concept, essentially due to Tardif~\cite{T96}, was initially defined for median graphs.

\begin{defn}\label{D:mooring}
Let $F$ be a gated subgraph of a Peano partial cube $G$.  A self-contraction $f$ of $G$ is a \emph{mooring} of $G$ on $F$ if, for all $u \notin V(F)$, $uf(u)$ is an edge of $G[\mathcal{I}_G(u,g_F(u))]$, where $g_F(u)$ is the gate of $u$ in $F$.
\end{defn}

We recall the following two results.

\begin{pro}\label{P:mooring/med.gr.}\textnormal{(Tardif~\cite[Corollary 3.2.5]{T96})}
Let $F$ be a convex subgraph of a median graph $G$.  Then there exists a mooring of $G$ on $F$.
\end{pro}

\begin{pro}\label{P:mooring/net.p.c.}\textnormal{(Polat~\cite[Proposition 6.2]{P09-2})}
If a netlike partial cube $G$ contains a unique convex cycle $C$ of length greater than $4$, then there exists a mooring of $G$ on $C$.
\end{pro}

We first extend the relation of $\Theta$-equivalence to cycles.

\begin{defn}
Let $G$ be a Peano partial cube.

\textnormal{(i)}\; If $C = \langle x_1,\dots,x_{2n},x_1\rangle$ is an isometric cycle of $G$, then a cycle $C' = \langle x'_1,\dots,x'_{2n},x'_1\rangle$ of $G$ is said to be \emph{$\Theta$-equivalent} to $C$, or is called a \emph{$\Theta$-copy} of $C$, if the edges $x_ix_{i+1}$ and $x'_ix'_{i+1}$ are $\Theta$-equivalent for every $i$ with $1 \leq i \leq 2n$, and such that $x_{2n+1} := x_1$ and $x'_{2n+1} := x'_1$.

\textnormal{(ii)}\; A convex subgraph $F$ of $G$ is said to be \emph{cycle-representative} if $F$ contains a $\Theta$-copy of each isometric cycle of $G$ of length greater than $4$.
\end{defn}

We clearly infer from (i) and Lemma~\ref{L:gen.propert.}(vii) that $C'$ is an isometric cycle of $G$, and that $d_G(x_i,x'_i) = d_G(x_j,x'_j)$ for all $i, j$ with $1 \leq i \leq j \leq 2n$.

\begin{lem}\label{L:cycle-rep./cycle}
Let $G$ be a Peano partial cube $G$, and $F$ a cycle-representative convex subgraph of $G$.  Then any isometric cycle of $G$ that has at least one vertex in $F$ is a cycle of $F$.
\end{lem}

\begin{proof}
Let $C = \langle x_1,\dots,x_{2n},x_1\rangle$ be an isometric cycle of $G$ that has at least one vertex, say $x_1$, in $F$.  Because $F$ is cycle-representative, there exists a $\Theta$-copy $C_F = \langle x_1^F,\dots,x_{2n}^F,x_1^F\rangle$ of $C$ in $F$.  Suppose that $x_1, \dots, x_i \in V(F)$ for some $i$ with $1 \leq i < 2n$.  Because the edges $x_ix_{i+1}$ and $x_i^Fx_{i+1}^F$ are $\Theta$-equivalent, it follows that $x_{i+1}\in I_G(x_i,x_{i+1}^F)$, and thus $x_{i+1} \in V(F)$ since $F$ is convex.  Therefore $x_1, \dots, x_{2n} \in V(F)$, and thus $C$ is a cycle of $F$.
\end{proof}

Because, by Theorem~\ref{T:hypernet.=>G-closed+conv.=gated}, a convex subgraph of a Peano partial cube is gated if and only if it is $\Gamma$-closed, we infer immediately the following corollary from the above lemma.

\begin{cor}\label{C:cycle-represent./gated}
Let $G$ be a Peano partial cube.  Then any cycle-representative convex subgraph of $G$ is gated in $G$.
\end{cor}

\begin{lem}\label{L:mooring/cycle}
Let $F$ be a convex subgraph of a graph $G$, $\mu$ a mooring of $G$ on $F$, and $C$ an isometric cycle of $G-F$.  Then $\mu(C)$ is a $\Theta$-copy of $C$.
\end{lem}

\begin{proof}
Let $C = \langle x_1,\dots,x_{2n},x_1\rangle$.  We first show that $\mu(x_i) \neq \mu(x_j)$ if $i \neq j$.  Let $i, j$ with $1 \leq i < j \leq 2n$.  Suppose that $\mu(x_i) = \mu(x_j) =: y$.  Then $j = i+2$ since $C$ is isometric in $G$.  It follows that the three vertices $x_i, x_{i+2}, \mu(x_{i+1})$ are adjacent to $x_{i+1}$ and $y$, contrary to the fact that a partial cube contains no $K_{2,3}$.  Therefore $\mu(C)$ is a cycle of length $2n$.

On the other hand, $\mu(e)$ and $e$ are clearly $\Theta$-equivalent for each edge $e$ of $C$.  Hence $\mu(C)$ is a $\Theta$-copy of $C$.
\end{proof}

\begin{thm}\label{T:conv.=>mooring}
Let $G$ be a Peano partial cube, and $F$ a cycle-representative convex subgraph of $G$.  Then there exists a mooring of $G$ on $F$.
\end{thm}

This theorem, which will we proved latter, generalizes Proposition~\ref{P:mooring/med.gr.} since any median graph contains no convex cycles of length greater than $4$.  The converse is clearly false.  For example, there exits a mooring of a cube $Q$ on any of its $4$-cycles, even if they are not cycle-representative since they contains no $\Theta$-copy of the isometric $6$-cycles of $Q$.  The following result, which generalizes Proposition~\ref{P:mooring/net.p.c.}, is an immediate consequence of Theorem~\ref{T:conv.=>mooring} because the given cycle $C$ is a convex cycle-representative subgraph of $G$.

\begin{cor}\label{C:conv.cycle/mooring}
Let $C$ be a convex cycle of a Peano partial cube $G$ whose length is greater than $4$.  If $C$ is a $\Theta$-copy of each convex cycle of $G$ of length greater than $4$, then there exists a mooring of $G$ on $C$.
\end{cor}

We need a few preliminary results.

\begin{lem}\label{L:Wba.retract}
Let $W_{ab}$ be a periphery of a Peano partial cube $G$.  Then $G[W_{ba}]$ is a retract of $G$.
\end{lem}

\begin{proof}
Because $W_{ab}$ is a periphery and $G$ is ph-homogeneous, we have $W_{ab} = U_{ab}$, and $\phi_{ab}$ is an isomorphism of $G[U_{ab}]$ onto $G[U_{ba}]$.  Therefore the map $\phi : V(G) \to W_{ba}$ such that $\phi(x) = x$ if $x \in W_{ba}$ and $\phi(x) = \phi_{ab}(x)$ if $x \in W_{ab}$, is a retraction of $G$ onto $G[W_{ba}]$.
\end{proof}

Let $F$ be a convex subgraph of a Peano partial cube
$G$.  We will see that there is always a convex subgraph of $G$ which properly contains $F$ as a subgraph and which is minimal with respect to
the subgraph relation.  This graph (resp. its vertex set) is called a \emph{minimal convex
extension} of $F$ (resp. $V(F)$).  We need the following lemma:

\begin{lem}\label{L:min.conv.ext.}
Let $G$ be a Peano partial cube, and $A$ a non-empty convex set of vertices of
$G$.  We have the following properties:

\textnormal{(i)}\; A set $A'$ is a minimal convex extension of $A$ if
and only if $A' = \mathcal{I}_{G}(\{u \} \cup A)$ for some vertex $u
\in N_{G}(A)$.

\textnormal{(ii)}\; If $A'$ is a minimal convex extension of $A$, then $A =
W_{vu}^{G[A']}$ and $A'-A = W_{uv}^{G[A']}$ 
for any edge $uv$ of $G$ with $v \in A$ and $u \in A'-A$.
\end{lem}

\begin{proof}
Let $u  \in N_{G}(A)$ and let $v$ be the neighbor of $u$ in $A$.  This 
neighbor is unique because $A$ is convex and $G$ is bipartite.

\emph{Claim~1.\;  $\mathcal{I}_{G}(U_{uv} \cap N_{G}(A)) \cap 
U_{uv} \subseteq N_{G}(A)$.}

Let $x \in \mathcal{I}_{G}(U_{uv} \cap N_{G}(A)) \cap U_{uv}$.  Then
$x \in I_{G}(a,b)$ for some vertices $a, b \in U_{uv} \cap N_{G}(A)$. 
Let $a', b'$ and $x'$ be the neighbors in $U_{vu}$ of $a, b$ and $x$, 
respectively.  Then $a', b' \in A$ and $x' \in I_{G}(a',b')$ since $G$
is a partial cube.  Therefore $x' \in A$ by convexity.  Hence $x \in
N_{G}(A)$.

\emph{Claim~2.\; The set $U_{uv} \cap N_{G}(A)$ is ph-stable.}

Let $x, y \in \mathcal{I}_{G}(U_{uv} \cap N_{G}(A))$.  We have to prove that $y \in I_{G}(x,z)$ for some $z \in U_{uv} \cap N_{G}(A)$.  We are done if $y \in U_{uv}$ by Claim~1.  Assume that $y \notin U_{uv}$.  Then $y \in I_{G}(a,b)$ for some vertices $a, b \in U_{uv} \cap N_{G}(A)$ that we choose so that $d_{G}(a,b)$ is as small as possible.  Let $P$ be the $U_{uv}$-geodesic associated with $y$, and let $a'$ and $b'$ be its endvertices.  By (SPS2), $P$ is a subpath of some $(a,b)$-geodesic.  It follows that $a', b' \in N_G(A)$ by Claim~1.  Therefore $P$ is an $(a,b)$-geodesic by the minimality of $d_G(a,b)$.  We infer, by (SPS1), that $y \in I_{G}(x,a) \cup I_{G}(x,b)$.  Consequently $U_{uv} \cap N_{G}(A)$ is ph-stable.

\emph{Claim~3.\; $\mathcal{I}_{G}(\{u \} \cup A) =
\mathcal{I}_{G}(U_{uv} \cap N_{G}(A)) \cup A$.}

Clearly $\mathcal{I}_{G}(U_{uv} \cap N_{G}(A)) \cup A \subseteq
\mathcal{I}_{G}(\{u \} \cup A)$.  Conversely let $x \in A$, and let $P
= \langle u_{0},\ldots,u_{n} \rangle$ be a $(u,x)$-geodesic with
$u_{0} = u$ and $u_{n} = x$.  Without loss of generality we can
suppose that $u_{n-1} \notin A$.  Suppose that there is a vertex $w \in A \cap W_{uv}$.  Then any $(v,w)$-geodesic contains an edge $\Theta$-equivalent to $uv$.  Such an edge, say $u'v'$, is then an edge of $G[A]$ since $A$ is convex, which implies that $u \in I_{G}(v,u') \subseteq A$ by the convexity of $A$, contrary to the hypothesis.  Therefore $A \subseteq W_{vu}$.  It follows that $P$ contains an edge $u_{i}u_{i+1}$ which is $\Theta$-equivalent to $uv$.  Then $v \in I_{G}(u_{0},u_{i+1})$.  Hence $u_{i+1} \in I_{G}(v,x)$, and thus $u_{i+1} \in A$ since $A$ is convex.  It follows that $i = n-1$.  Hence $u_{n-1} \in U_{uv} \cap N_{G}(A)$.  Consequently $V(P) \in \mathcal{I}_{G}(U_{uv} \cap
N_{G}(A)) \cup A$, and thus $\mathcal{I}_{G}(\{u \} \cup A)
\subseteq \mathcal{I}_{G}(U_{uv} \cap N_{G}(A)) \cup A$.

\emph{Claim~4.\; $\mathcal{I}_{G}(\{u \} \cup A)$ is convex
and is equal to $\mathcal{I}_{G}(\{a \} \cup A)$ for each vertex $a \in
\mathcal{I}_{G}(\{u \} \cup A) - A$.}

By Claim~3, $\mathcal{I}_{G}(\{u \} \cup A) - A =
\mathcal{I}_{G}(U_{uv} \cap N_{G}(A))$.  By Claim~2, the set $U_{uv}
\cap N_{G}(A)$ is ph-stable, and thus $\mathcal{I}_{G}(U_{uv} \cap
N_{G}(A))$ is convex by Lemma~\ref{L:ph-stable}.  Therefore
$\mathcal{I}_{G}(\{u \} \cup A)$ is convex.  Now let $a \in
\mathcal{I}_{G}(\{u \} \cup A) - A = \mathcal{I}_{G}(U_{uv} \cap
N_{G}(A))$.  Clearly $\mathcal{I}_{G}(\{a \} \cup A) \subseteq
\mathcal{I}_{G}(\{u \} \cup A)$.  On the other hand, $\mathcal{I}_{G}(U_{uv}
\cap N_{G}(A)) \cup A \subseteq \mathcal{I}_{G}(\{a \} \cup A)$ since $U_{uv}
\cap N_{G}(A)$ is ph-stable.
Therefore $\mathcal{I}_{G}(\{a \} \cup A) = \mathcal{I}_{G}(\{u \}
\cup A)$.\\

From Claim~4, it follows immediately that $\mathcal{I}_{G}(\{u \} \cup
A)$ is a minimal convex extension of $A$.

Conversely, let $A'$ be a
minimal convex extension of $A$.  Let $u \in N_{G}(A) \cap A'$.  By what we proved above, $\mathcal{I}_{G}(\{u \} \cup A)$ is a minimal convex
extension of $A$, which moreover is contained in $A'$ since $u \in
A'$.  Hence $A' = \mathcal{I}_{G}(\{u \} \cup A)$.  This proves
assertion (i).

Let $A'$ be a minimal convex extension of $A$.  By (i), $A' = \mathcal{I}_{G}(\{u \} \cup A)$ for some vertex $u \in N_{G}(A)$.  Let $v$ be the neighbor of $u$ in $A$.  Then, by Claim~3, $W_{uv}^{G[A']} = \mathcal{I}_{G}(\{u \} \cup A)-A = A'-A$, and thus $A = W_{vu}^{G[A']}$.  This proves assertion (ii).
\end{proof}

\begin{pro}\label{P:min.conv.ext./retract}
Let $G$ a Peano partial cube, $F$ a proper non-empty cycle-representative convex subgraph of $G$, and $F'$ a minimal convex extension of $F$.  Then $F$ is a retract of $F'$.
\end{pro}

\begin{proof}
If $uv$ is an edge of $G$ with $v \in V(F)$ and $u \in V(F'-F)$, then $uv$ cannot be $\Theta$-equivalent to an edge of some convex cycle of $G$ of length greater than $4$ which has another edge in $\partial_G(V(F))$, because $F$ is $\Gamma$-closed since it is gated by Corollary~\ref{C:cycle-represent./gated}.  It follows that the set $U_{uv}^{F'}$ is convex, and thus $U_{uv}^{F'} = W_{uv}^{F'} = V(F'-F)$.  Hence, by Lemma~\ref{L:min.conv.ext.}(ii), each vertex $x$ in $V(F'-F)$ has exactly one neighbor, $\phi_{uv}(x)$, in $F$.  Therefore $F$ is a retract of $F'$ by Lemma~\ref{L:Wba.retract}.
\end{proof}

\begin{lem}\label{L:PxPy/(x,y)-geodesic}
Let $ab$ be an edge of a Peano partial cube $G$, $x, y \in \mathcal{I}_G(U_{ab})-U_{ab}$, and $P_x$ and $P_y$ the $U_{ab}$-geodesics associated with $x$ and $y$, respectively.  Then any $(x,y)$-geodesic is a subpath of some geodesic joining an endvertex of $P_x$ to an endvertex of $P_y$.
\end{lem}

\begin{proof}
By (SPS1), $y \in I_G(x,u)$ for some endvertex $u$ of $P_y$, and $x \in I_G(u,v)$ for some endvertex $v$ of $P_x$.  It follows that any $(x,y)$-geodesic is a subpath of some $(u,v)$-geodesic.
\end{proof}

\begin{pro}\label{P:exist.min.conv.ext.}
Any proper convex subgraph of Peano partial cube $G$ has a minimal convex extension.
\end{pro}

\begin{proof}
Let $F$ be a proper convex subgraph of $G$.  Because $F$ is convex, any vertex in $N_G(V(F))$ has exactly one neighbor in $F$, so let $u \in N_G(V(F))$,\; $v$ its unique neighbor in $F$,\; $A := \mathcal{I}_G(U_{vu} \cap V(F))$ and $$F' := G[V(F) \cup \phi_{vu}(A)].$$  We will show that $F'$ is a a minimal convex extension of $F$.

(a)\;  Let $x \in A$.  Then $x \in I_G(a,b)$ for some $a, b \in U_{vu} \cap V(F)$.  Let $P_x$ be the $U_{vu}$-geodesic associated with $x$ in $G$.  Then $P_x$ is a subpath of some $(a,b)$-geodesic in $G$ by (SPS2).  It follows that $P_x$ is a path of $F$ since $F$ is convex in $G$, and that its endvertices belong to $U_{vu} \cap V(F)$.

(b)\; Let $x, y \in A$, and $R$ a $(\phi_{vu}(x),\phi_{vu}(y))$-geodesic in $G$.  Then $\phi_{uv}(R)$ is an $(x,y)$-geodesic in $G$, and thus in $F$ since $F$ is convex.  Let $P_x$ and $P_y$ be the $U_{vu}$-geodesics associated with $x$ and $y$, respectively.  By (a), $P_x$ and $P_y$ are paths of $F$.  By Lemma~\ref{L:PxPy/(x,y)-geodesic}, $\phi_{uv}(R)$ is a subpath of some geodesic joining an endvertex $a$ of $P_x$ to an endvertex $b$ of $P_y$.  Because $a, b \in U_{vu} \cap V(F)$, it follows that $\phi_{vu}(I_G(a,b)) \subseteq V(F')$, and consequently $R$ is a path of $F'$, which proves that $F'$ is convex.

(c)\; Let $x \in V(F'-F)$.  If $x \in \phi_{vu}(U_{vu} \cap V(F))$, then $x \in \mathcal{I}_G(u,\phi_{uv}(x))$.  Suppose that $x \in I_G(a,b)$ for some $a, b \in \phi_{vu} (U_{vu} \cap V(F))$.  Let $P_x$ be the $U_{uv}$-path associated with $\phi_{uv}(x)$.  Then $P_x$ is a path of $F$ by (a), and thus $\Phi{vu}(P_x)$ is a path of $F'$ since $F'$ is convex by (b).  By the minimality of the length of $P_x$ (see Lemma~\ref{L:min.Uab-path=associated}), $\phi_{vu}(P_x)$ is the $U_{uv}$-geodesic associated with $x$.  By (SPS1), $x \in I_G(u,w)$ for some endvertex $w$ of $\phi_{vu}(P_x)$, and thus $x \in I_G(u,\Phi{uv}(w))$.

Consequently $V(F') = \mathcal{I}_G(\{u\} \cup V(F))$, which proves that $F'$ is a minimal convex extension of $F$ by Lemma~\ref{L:min.conv.ext.}(i).
\end{proof}

\begin{proof}[\textnormal{\textbf{Proof of Theorem~\ref{T:conv.=>mooring}}}]
Note that, if $F$ is cycle-representative, then any convex extension of $F$ is also cycle-representative.  It follows that, if $H'$ is a minimal convex-extension of some convex subgraph $H$ of $G$ which contains $F$, then $H$ is a retract of $H'$ by Proposition~\ref{P:min.conv.ext./retract}, and in particular each vertex in $V(H'-H)$ has exactly one neighbor in $H$.

For each ordinal $\alpha$, we inductively construct a gated subgraph
$F_{\alpha}$ of $G$ as follows:

\textbullet\; $F_{0} := F$;

\textbullet\; $F_{\alpha + 1}$ is a minimal convex extension
of $F_{\alpha}$;

\textbullet\; $F_{\alpha} :=
\bigcup_{\beta < \alpha}F_{\beta}$ if $\alpha$ is a limit ordinal.

Note that $F_{\alpha + 1}$ exists according to Proposition~\ref{P:exist.min.conv.ext.} since $F_{\alpha}$ is gated by the induction hypothesis, and that $F_{\alpha}$ is also a gated subgraph of $G$ if
$\alpha$ is a limit ordinal because the set $\{F_{\beta}: \beta <
\alpha \}$ is a set of gated subgraphs totally ordered by inclusion.  For each $x \in V(G)$ we denote by
$\alpha(x)$ the smallest ordinal $\alpha$ such that $x \in
V(F_{\alpha})$.

Define the self-map $\mu$ of $V(G)$ such that $\mu(x)$ is $x$
if $\alpha(x) = 0$ and is the only neighbor of $x$ in
$F_{\alpha(x)-1}$ if $\alpha(x) > 0$.  It suffices to prove that
$\mu$ is a contraction to show that it is a mooring of $G$ on
$F$.  Let $x$ and $y$ be two adjacent vertices of $G$ with $\alpha(x)
\leq \alpha(y)$, we have to show that $\mu(x)$ and $\mu(y)$
are equal or adjacent.  We are done if $\alpha(x) = \alpha(y) = 0$.
If $\alpha(x) = \alpha(y) \neq 0$, then $\mu(x)$ and $\mu(y)$
are adjacent because $\mu(y) \in U_{\mu(x)x}^{G_{\alpha(x)}}$.
If $\alpha(x) < \alpha(y)$, then $x = \mu(y)$ by the definition of
$\mu$, and thus $\mu(x)$ and $\mu(y)$ are equal or
adjacent depending on whether $\alpha(x)$ is or is not equal to $0$.
\end{proof}

\subsection{Retracts, hom-retracts and convex sets}\label{SS:ret.hom-ret./convex}

We recall that a contraction $f : G \to H$ which preserves the edges is called a
\emph{homomorphism}\ of $G$ into $H$.  If a retraction $f : G \to H$ is
a homomorphism, then we will say that $f$ is a \emph{hom-retraction}
and that $H$ is a \emph{hom-retract}.

Any retract of a partial cube is a faithful subgraph of this graph.  By \cite[Theorem 4.5]{P05-2}, any faithful subgraph of a netlike partial cube $G$ is both a netlike partial cube and a retract of $G$.  The corresponding result for median graphs was proved by Bandelt~\cite[Theorem 1]{B84}.  This property is generally not true for Peano partial cubes as is shown by the example of Figure~\ref{F:faithful}.  In this example the graph depicted by the thick edges and the big vertices is a faithful subgraph of the prism $C_{6} \Box K_{2}$, but it is not ph-homogeneous, and thus not a retract of $C_{6} \Box K_{2}$ by Theorem~\ref{T:retract(hypernet.)}; note that its pre-hull number is $2$.  We will show that any convex subgraph of a Peano partial cube is a retract of this graph.

On the other hand for any netlike partial cube $G$, every non-trivial retract is a hom-retract of $G$ by \cite[Theorem 4.5]{P05-2}.  However, this is not the case for any Peano partial cube.  Take for example the prism $G = B \Box K_{2}$ over the benzenoid graph $B$ that is the the union of two $6$-cycles having exactly one edge in common.  Then any of the two $B$-layers of $G$ is a convex subgraph of $G$, and also a retract of $G$, but it is straightforward to check that it cannot be a hom-retract of $G$.

\begin{thm}\label{T:conv.=>retract}
Let $G$ be a Peano partial cube (resp. such that, for each edge $ab$ of $G$, any two convex cycles of length greater than $4$ of the subgraph of $G$ induced by $U_{ab}$ are $\Theta$-equivalent).  Then any non-empty (resp. non-trivial) convex subgraph of $G$ is a retract (resp. hom-retract) of $G$.
\end{thm}

Note that the converse is false since, for any positive integer $n$, a path of length $n$ of a $2n$-cycle $C$ is a retract of $C$, but is not a convex subgraph of $C$.  Moreover the condition to have a hom-retract in the above theorem is obviously not necessary since the subgraph induced by any two adjacent vertices of a bipartite graph is a hom-retract of this graph.  We need the following supplement to Proposition~\ref{P:min.conv.ext./retract}.

\begin{pro}\label{P:min.conv.ext./hom-retract}
Let $G$ be a Peano partial cube such that, for any edge $ab$ of $G$, all convex cycles of length greater than $4$ of the subgraph of $G$ induced by $U_{ab}$ are $\Theta$-equivalent, $F$ a proper non-empty convex subgraph of $G$, and $F'$ a minimal convex extension of $F$.  Then $F$ is a hom-retract of $F'$.
\end{pro}

\begin{proof}
We already know that $F$ is a retract of $F'$ (see Proposition~\ref{P:min.conv.ext./retract}).  Let $u \in V(F'-F)$ and $v$ its neighbor in $F$.  By Lemma~\ref{L:min.conv.ext.}, $V(F) = W_{vu}^{F'}$ and $V(F'-F) = W_{uv}^{F'}$.  Moreover, by Lemma~\ref{L:prop.isom subgr.}, $\mathcal{I}_F(U_{vu}^F) = \mathcal{I}_F(U_{vu}) \cap V(F)$ since $F$ is convex.  Hence $\mathcal{I}_F(U_{vu}^F)$ is convex in $G$, and thus is ph-homogeneous.  It follows that a cycle of $F[\mathcal{I}_F(U_{vu}^F)]$ is convex or isometric in this subgraph if and only if it is convex or isometric in $G$, respectively.  Therefore any convex cycle of $F[\mathcal{I}_F(U_{vu}^F)]$ of length greater than $4$ is cycle-representative.  We distinguish two cases.

(a)\; If $U_{vu}^F = \{v\}$, then the map $f : V(F') \to V(F)$ such that $f(x) = x$ if $x \in V(F)$ and $f(u) = w$, where $w$ is a neighbor of $v$ in $F$, is a hom-retraction of $F'$ onto $F$.  

(b)\; Assume now that $v$ is not the only element of $U_{vu}^F$.  We have two subcases.

(b.1)\; Suppose that $F[\mathcal{I}_F(U_{vu}^F)]$ contains no convex cycle of length greater than $4$.  Then $F[\mathcal{I}_F(U_{vu}^F)]$ is a median graph by Proposition~\ref{P:hypernet./med.gr.}.  On the other hand, $v$ has a neighbor $w \in \mathcal{I}_F(U_{vu}^F)$ since $\mathcal{I}_F(U_{vu}^F)$ is convex in $V(F)$, and $F[w]$ is a convex subgraph of $F[\mathcal{I}_F(U_{vu}^F)]$.  Hence, by Proposition~\ref{P:mooring/med.gr.}, there exists a mooring $\mu$ of $F[\mathcal{I}_F(U_{vu}^F)]$ on $F[w]$.  It clearly follows that the map $f : V(F') \to V(F)$ such that $f(x) = x$ if $x \in V(F)$, \; $f(x) = \mu(\phi_{uv}(x))$ if $x \in V(F'-F)$, and $f(\phi_{vu}(w)) = v$, is a hom-retraction
of $F'$ onto $F$.

(b.2)\; If $F[\mathcal{I}_F(U_{vu}^F)]$ contains a convex cycle $C = \langle c_1,\dots,c_{2n},c_1\rangle$ of length greater than $4$, then, by Theorem~\ref{T:conv.=>mooring}, there exists a mooring $\mu$ of $F[\mathcal{I}_F(U_{vu}^F)]$ on $C$.  Then the map $f : V(F') \to V(F)$ such that $f(x) = x$ if $x \in V(F)$,\; $f(x) = \mu(\phi_{uv}(x))$ if $x \in V(F'-(F \cup \phi_{vu}(C)))$, and $f(\phi_{vu}(c_i)) = c_{i+1}$ for $1 \leq i \leq 2n$ with $c_{2n+1} := c_1$, is a hom-retraction
of $F'$ onto $F$.
\end{proof}

\begin{proof}[\textnormal{\textbf{Proof of Theorem~\ref{T:conv.=>retract}}}]
For each ordinal $\alpha$, we inductively construct a convex subgraph
$F_{\alpha}$ of $G$ as follows:

\textbullet\; $F_{0} := F$;

\textbullet\; $F_{\alpha + 1}$ is a minimal convex extension
of $F_{\alpha}$;

\textbullet\; $F_{\alpha} :=
\bigcup_{\beta < \alpha}F_{\beta}$ if $\alpha$ is a limit ordinal.

Note that $F_{\alpha + 1}$ exists according to Proposition~\ref{P:exist.min.conv.ext.} since $F_{\alpha}$ is convex by the induction hypothesis, and that $F_{\alpha}$ is also a convex subgraph of $G$ if
$\alpha$ is a limit ordinal because the set $\{F_{\beta}: \beta <
\alpha \}$ is a set of convex subgraphs totally ordered by inclusion.  Let $\gamma$ be the least ordinal such that $F_{\gamma}
= G$.

Now, for each ordinal $\alpha \leq \gamma$, we construct a
retraction (resp. hom-retraction) $f_{\alpha}$ of $F_{\alpha}$ onto $F_{0}$.  Let $f_{0}$
be the identity function on $V(F_{0})$.  Let $\alpha \geq 0$.  Suppose that
$f_{\beta}$ has already been constructed for every $\beta < \alpha$.
If $\alpha = \beta + 1$ for some ordinal $\beta$, then $f_{\alpha} :=
f_{\beta} \circ f_{F_{\alpha}}$ where $f_{F_{\alpha}}$ is a
retraction (resp. hom-retraction)  of $F_{\alpha}$ onto $F_{\beta}$ induced by
Proposition~\ref{P:min.conv.ext./retract}.  Then $f_{\alpha}$ is
obviously a retraction of $F_{\alpha}$ onto $F_{0}$.

Suppose that $\alpha$ is a limit ordinal.  Let $f_{\alpha} :=
\bigcup_{\beta < \alpha}f_{\beta}$, i.e., $f_{\alpha}$ is the map of
$F_{\alpha}$ onto $F_{0}$ such that, for each vertex $x$ of
$F_{\alpha}$,\; $f_{\alpha}(x) := f_{\beta}(x)$, where $\beta$ is the
least ordinal such that $x \in V(F_{\beta})$.  In particular
$f_{\alpha}(x) = x$ if $x \in V(F_{0})$.  It remains to prove that
$f_{\alpha}$ is a contraction (resp. homomorphism).  Let $x, y$ be two adjacent vertices
of $F_{\alpha}$.  Then there is an ordinal $\beta < \alpha$ such that
$x, y \in V(F_{\beta})$.  Therefore $f_{\alpha}(x) = f_{\beta}(x)$ and
$f_{\alpha}(y) = f_{\beta}(y)$.  It follows that $f_{\alpha}(x)$ and
$f_{\alpha}(y)$ are adjacent or equal (resp. adjacent) because $f_{\beta}$ is a contraction (resp. homomorphism) by
the induction hypothesis.  Consequently $f_{\alpha}$ is a
retraction (resp. hom-retraction) of $F_{\alpha}$ onto $F_{0}$.

We deduce that $f_{\gamma}$ is the desired retraction (resp. hom-retraction) of $F$ onto
$G$.
\end{proof}

\subsection{Strongly faithful subgraphs}\label{SS:strong.faithf.}

By \cite[Proposition 4.4]{P05-2}, any faithful subgraph of a netlike partial cube $G$ is a netlike partial cube.  However, a faithful subgraph of a Peano partial cube is generally not ph-homogeneous as we saw by the example of Figure~\ref{F:faithful}.

\begin{defn}\label{D:strong.med.stable}
A subgraph $F$ of a Peano partial cube $G$ is said to be \emph{strongly median-stable} in $G$ if, for any triple $(u,v,w)$ of vertices of $F$ that has a median $m$ or a hyper-median $(x,y,z)$ in $G$, then $m$ or $x, y, z$ are vertices of $F$.
\end{defn}

\begin{defn}\label{D:strong.faithf.}
A subgraph of a Peano partial cube $G$ is said to be \emph{strongly faithful} if it is both isometric and strongly median-stable in $G$.
\end{defn}

A convex subgraph of a Peano partial cube is clearly strongly faithful.

\begin{pro}\label{P:strong.faithf./hyper-median}
Let $F$ be a strongly faithful subgraph of a Peano partial cube $G$, and $(u,v,w)$ a triple of vertices of $F$.  Then:

\textnormal{(i)}\; If $(u,v,w)$ has a median $m$ in $G$, then $m$ is the median of $(u,v,w)$ in $F$.

\textnormal{(ii)}\; If $(u,v,w)$ has a hyper-median $(x,y,z)$ in $G$, then $(x,y,z)$ is a hyper-median of $(u,v,w)$ in $F$, and moreover $co_F(x,y,z) = co_G(x,y,z)$.
\end{pro}

\begin{proof}
Because $F$ is isometric in $G$, it follows that, if $m$ is the median of $(u,v,w)$ is $G$, then it is the median of $(u,v,w)$ in $F$.  Assume that $(u,v,w)$ has a hyper-median $(x,y,z)$ in $G$.  Then $H := G[co_G(x,y,z)]$ is a hypertorus.  Moreover there exists an isometric cycle $C$ which passes through $x, y, z$, and thus which is such that $co_G(C) = H$.  By Proposition~\ref{P:TPQ/prop.isom.cycle}, $\mathrm{idim}(C) = \mathrm{idim}(H)$.  Because $F$ is isometric in $G$, it follows that there exists a cycle $C'$ in $F$ which passes through $x, y, z$, and thus which has the same $\Theta$-classes as $C$.  Once again by Proposition~\ref{P:TPQ/prop.isom.cycle}, $H$ is the median-closure of $C'$ in $G$.  Therefore, $H$ is the median-closure of $C'$ in $F$ since $F$ is median-stable.  Hence $co_F(x,y,z) = H$, which proves that $(x,y,z)$ is a hyper-median of $(u,v,w)$ in $F$.
\end{proof}

We immediately deduce the following corollary:

\begin{cor}\label{C:str.faithf.str.faithf.}
Let $F$ be a strongly faithful subgraph of a Peano partial cube $G$.  Then any strongly faithful subgraph of $F$ is a strongly faithful subgraph of $G$.
\end{cor}

The following result is analogous to \cite[proposition 4.4]{P05-2} stating that a faithful subgraph of a netlike partial cube is also a netlike partial cube.

\begin{thm}\label{T:strong.faithf.=>ph-homogeneous}
Any strongly faithful subgraph $F$ of a Peano partial cube $G$ is a Peano partial cube.   Moreover $U_{ab}^F = U_{ab} \cap \mathcal{I}_G(U_{ab}^F)$ for each edge $ab$ of $F$.
\end{thm}

\begin{proof}
(a)\; Let $ab \in E(F)$ and $u \in \mathcal{I}_F(U_{ab}^F) - U_{ab}^F$.  Then $u \in I_F(x,y)$ for some $x, y \in U_{ab}^F$.  Because $U_{ab}^F \subseteq U_{ab} \cap V(F)$ by Lemma~\ref{L:prop.isom subgr.}, it follows that $u \in \mathcal{I}_F(U_{ab}) - U_{ab}$ and $x, y \in U_{ab}$.  Because $G$ is ph-homogeneous, there is an $ab$-cycle $C$ which associated with $u$.  This cycle is gated by Lemma~\ref{L:ab-cycle/gated} and by the fact that the subgraph $\mathbf{Cyl}[X]$ is gated, where $X$ is the bulge of $co_G(U_{ab})$ which contains $u$.

Let $x'$ and $y'$ be the neighbors of $x$ and $y$ in $U_{ba}$, respectively.  Clearly the triple $(u,v',w')$ of the gates in $C$ of $u, x', y'$ is a hyper-median of $(u,x',y')$.  Because $F$ is strongly faithful in $G$, it follows that $u, v', w' \in V(F)$, and thus $(u,v',w')$ is a hyper-median of $(u,x',y')$ in $F$.  Moreover $C$ is a gated cycle of $F$ since $F$ is isometric in $G$.  It follows that $P := C-W_{ba}^F$ is a convex $ab$-path in $F$ which passes through $u$.  This path $P$ satisfies the properties (SPS1) and (SPS2) in $G$, since $G$ is ph-homogeneous.  Then, because $F$ is an isometric subgraph of $G$, it easily follows that $P$ satisfies the properties (SPS1) and (SPS2) in $F$.  Hence $U_{ab}^F$ is strongly ph-stable, and analogously $U_{ba}^F$ is strongly ph-stable.  Consequently $F$ is ph-homogeneous by the Characterization Theorem.

(b)\; Because $U_{ab}^F \subseteq U_{ab}$, it follows that $U_{ab}^F \subseteq U_{ab} \cap \mathcal{I}_F(U_{ab}^F)$.  Let $u \in U_{ab} \cap \mathcal{I}_F(U_{ab}^F)$.  Then $u \in I_G(v,w)$ for some $v, w \in \mathcal{I}_F(U_{ab}^F)$.  Because $u \in U_{ab}$ and since $F$ is median-stable, we have $$\phi_{ab}(u) = m_G(\phi_{ab}(v),\phi_{ab}(w),u) = m_F(\phi_{ab}(v),\phi_{ab}(w),u).$$  Hence $\phi_{ab}(u) \in V(F)$, and thus $u \in U_{ab}^F$.  Consequently $U_{ab}^F = U_{ab} \cap \mathcal{I}_G(U_{ab}^F)$.
\end{proof}

As an immediate consequence of Proposition~\ref{P:strong.faithf./hyper-median} and Theorem~\ref{T:strong.faithf.=>ph-homogeneous}, we have:

\begin{cor}\label{C:strong.faith.hypermed.=> hypermed.}
Any strongly faithful subgraph of a hyper-median partial cube is hyper-median.
\end{cor}

Note that a strongly faithful subgraph of a Peano partial cube $G$ is not necessarily convex in $G$.  For example, a $3$-path in a $6$-cycle is clearly strongly faithful, but is not convex.  Also note that an isometric subgraph of a Peano partial cube $G$ which is ph-homogeneous in its own right, is not necessarily strongly median-stable, as is shown by the example of an isometric $6$-cycle of a $3$-cube. 

\begin{rem}\label{R:strong.faithf./conv.}
 According to Theorem~\ref{T:strong.faithf.=>ph-homogeneous} and the facts that every convex subgraph of a partial cube is strongly faithful, and that the pre-hull number of a Peano partial cube is at most equal to $1$, we infer immediately that \emph{a partial cube $G$ is ph-homogeneous if and only if the pre-hull number of any finite strongly faithful subgraph of $G$ is at most equal to $1$}.  Contrary to Propositions~\ref{P:netl./fin.faithf.} and~\ref{P:cube-free netlike/isom.subgr.}, and even if a ‘‘strongly faithful subgraph’’ is a strictly weaker concept than that of ‘‘convex subgraph’’, the property that ‘‘the pre-hull number of any finite strongly faithful subgraph of $G$ is at most equal to $1$’’ characterizes any Peano partial cube, and not some special ones, as one may have expected.
 \end{rem}

\begin{thm}\label{T:retract.ph-homog./strong.faithf.}
Any retract of a Peano partial cube $G$ is strongly faithful in $G$.
\end{thm}

\begin{proof}
Let $f$ be a retraction of $G$ onto one of its subgraph $F$.  Then $F$ is faithful by Lemma~\ref{L:Gf}.  Let $(u,v,w)$be a triple of vertices of $F$ which has a hyper-median $(x,y,z)$ in $G$.  Then $(x,y,z)$satisfies the following equalities:
\begin{align*}
I_G(u,v) \cap I_G(u,w) = I_G(u,x),\\
I_G(v,x) \cap I_G(v,w) = I_G(v,y),\\
I_G(w,x) \cap I_G(w,y) = I_G(w,z).
\end{align*}
Because $F$ is isometric in $G$, it follows that $(f(x),f(y),f(z))$ satisfies:
\begin{align*}
I_G(u,v) \cap I_G(u,w) = I_G(u,f(x)),\\
I_G(v,f(x)) \cap I_G(v,w) = I_G(v,f(y)),\\
I_G(w,f(x)) \cap I_G(w,f(y)) = I_G(w,f(z)).
\end{align*}

Then the three intervals $I_G(f(x),f(y)),\; I_G(f(y),f(z)),\; I_G(f(z),f(x))$ pairwise intersect in their common endvertices.  Hence $(f(x),f(y),f(z))$ is a quasi-median of $(u,v,w)$.  By the uniqueness of the quasi-median in a Peano partial cube (see Proposition~\ref{P:quasi-med./uniqueness}), it follows that $f(x) = x,\; f(y) = y$ and $f(z) = z$.  Therefore $F$ is strongly faithful in $G$.
\end{proof}

We recall that, by \cite[Theorem 1]{B84} (resp. \cite[Theorem 4.5]{P05-2}), any faithful subgraph of a median graph (resp. a netlike cube) $G$ is a retract of $G$.  This is generally not true for any strongly faithful subgraph of a Peano partial cube, and even of a hyper-median partial cube as is shown by the following example.  Let $F$ be the gated amalgam of $C_4$ and $K_2$ along a vertex.  Then $F$ is a strongly faithful subgraph of $Q_3 = C_4 \Box K_2$, but it is clearly not a retract of $Q_3$.  This example is a counter-example, and is in fact the smallest, of the property that a gated amalgam of Peano partial cubes is generally not a retract of the Cartesian product of its constituents.

This also shows that a hyper-median partial cube is generally not a retract of some hypertorus.  Indeed a hypertorus that contains $F$ as a subgraph must also contain $Q_3$.  Hence $F$ cannot be a retract of $H$ since it is not even a retract of $Q_3$.  Recall that in~\cite{CKM19} it was conjectured that a partial cube is hyper-median if and only if it is a retract of a Cartesian product of bipartite cellular graphs.

\begin{que}\label{Q:str.faithf./retract}
What are the strongly faithful subgraphs of a Peano partial cube which are retracts of this graph?
\end{que}

We only give a necessary condition for a subgraph of a Peano partial cube to be a retract of this graph.  First we need a lemma.

\begin{lem}\label{L:strong.med.st.}
Let $F$ be a subgraph of a Peano partial cube $G$.  We have the following properties:

\textnormal{(i)}\; If $F$ is median-stable, then $F[U_{ab}^F]$ is convex in $F[U_{ab}^G \cap V(F)]$.

\textnormal{(ii)}\; If $F$ is strongly faithful, then any isometric cycle of $F[U_{ab}^G \cap V(F)]$ that has at least two vertices in $U_{ab}^F$ has all its vertices in $U_{ab}^F$.
\end{lem}

\begin{proof}
For each vertex $u \in U_{ab}^G$, we denote by $u'$ the neighbor of $u$ in $U_{ba}^G$.
 
(i)\; Let $x, y \in U_{ab}^F$, and $\langle x_0,\dots,x_n\rangle$ be an $(x,y)$-geodesic in $F[U_{ab}^G \cap V(F)]$ with $x_0 = x$ and $x_n = y$.  Then, for every $i$ with $0 < i < n$, $x'_i$ is the median of the triple $(x'_0,x_i,x'_n)$.  Because $F$ is median-stable in $G$, it follows that $x'_i \in V(F)$, and thus $x_i \in U_{ab}^F$.

(ii)\; Let $C = \langle x_0,\dots,x_{2n-1},x_0\rangle$ be an isometric cycle of $F[U_{ab}^G \cap V(F)]$ with $n \geq 2$ such that two of its vertices, say $x_0$ and $x_i$, $i < 2n$, belong to $U_{ab}^F$.

Assume that $i = n$.  Then $x'_0$ and $x'_{n}$ belong to $U_{ba}^F$.  Hence, by (i), because $F$ is median-stable and since $C$ is the union of two $(x_0,x_n)$-geodesics in $F[U_{ab}^G \cap V(F)]$, any vertex of $C$ belongs to $U_{ab}^F$.

Now assume that $i$ is distinct from $n$.  Without loss of generality, suppose that $i < n$.  Let $j = \lfloor (i+2n \rfloor)/2$.  Because $C$ is isometric in $G$, $(x'_0,x'_j,x'_i)$ is the hyper-median of $(x'_0,x_j,x'_i)$ in $G$, and thus in $F$ since $F$ is strongly faithful in $G$.  Therefore, by Proposition~\ref{P:strong.faithf./hyper-median}(ii), the isometric cycle $\langle x'_0,\dots,x'_{2n-1},x'_0\rangle$ is a cycle of $F$, and thus $V(C) \subseteq U_{ab}^F$.
\end{proof}

\begin{pro}\label{P:retract/strong.faithf.}
If a subgraph $F$ of a Peano partial cube $G$ is a retract of $G$, then we have the following two properties:

\textnormal{(i)}\; $F$ is a strongly faithful subgraph of $G$.

\textnormal{(ii)}\; For each edge $ab \in \partial_G(V(F))$ with $a \in V(F)$, if there exist two vertices $u \in U_{ab}^F$ and $v \in U_{ab}^G \cap V(F)$ such that there are two $(u,v)$-geodesics in $F[U_{ab}^G \cap V(F)]$ that pass through distinct neighbors of $v$, then $v \in U_{ab}^F$.
\end{pro}

\begin{proof}
Let $F$ be a retract of $G$.  Then $F$ must be a strongly faithful subgraph of $G$ by Theorem~\ref{T:retract.ph-homog./strong.faithf.}.  Let $f$ be a retraction of $G$ onto $F$ and $ab \in \partial_G(V(F))$ with $a \in V(F)$.  Let $\langle x_0,\dots,x_n\rangle$ and  $\langle y_0,\dots,y_n\rangle$ be two $(u,v)$-geodesics in $F[U_{ab}^G \cap V(F)]$ with $n \geq 2$, $x_0 = y_0 = u$, $x_n = y_n = v$ and $x_{n-1} \neq y_{n-1}$.

By Lemma~\ref{L:strong.med.st.}(i), because $F$ is median-stable in $G$, $I_{F[U_{ab}^G]}(z_0,z_1) \subseteq U_{ab}^F$ for all vertices $z_0, z_1 \in U_{ab}^F$.  Suppose that $v \notin U_{ab}^F$.  Let $i_0$ and $j_0$ be the smallest integers such that $x_{i_0}$ and $y_{j_0}$ do not belong to $U_{ab}^F$.  Then $i_0 \geq 1$ and $j_0 \geq 1$.  For each vertex $w \in U_{ab}^G$, denote by $w'$ the neighbor of $w$ in $U_{ba}^G$.  Then, because $f$ is a retraction,  $f(x'_{i_0}) = x_{i_0-1}$ and $f(y'_{i_0}) = y_{i_0-1}$, and then $f(x'_{i_0+1}) = x_{i_0}$ and $f(y'_{i_0+1}) = y_{i_0}$, and so on.  It follows that $x_{n-1} = f(x'_n) = f(y'_n) = y_{n-1}$, contrary to the hypothesis $x_{n-1} \neq y_{n-1}$.  Hence $v \in U_{ab}^F$.
\end{proof}

We clearly have the following consequence, which is in some ways stronger than Lemma~\ref{L:strong.med.st.}(ii).

\begin{cor}\label{C:retract/cycle}
If a subgraph $F$ of a Peano partial cube $G$ is a retract of $G$, then, for each edge $ab \in \partial_G(V(F))$ with $a \in V(F)$,  any isometric cycle of $F[U_{ab}^G \cap V(F)]$ that has a vertex in $U_{ab}^F$ has all its vertices in $U_{ab}^F$.
\end{cor}

\section{Fixed subgraph properties}\label{S:FSP}

Fixed finite subgraph theorems, which are far-reaching outgrowths of
metric fixed point theory, have been a flourishing topic in the
literature on metric graph theory.  In this section we prove several fixed subgraph properties that generalize analogous results on median graphs and netlike partial cubes.  As a side result, we show that the intersection graph of the maximal gated regular subgraphs of a finite Peano partial cube is dismantlable, in other words, cop-win.

\subsection{Finite Peano partial cubes}\label{SS:FSP/finite}

We recall that the gated regular subgraphs of a Peano partial cube $G$ are the convex quasi-hypertori of $G$.  Hence $\mathbf{Tor}(G)$ is the set of these subgraphs, and $\mathbf{Tor}(G,ab)$ the subset of these subgraphs that contain an edge $\Theta$-equivalent to a given edge $ab$ of $G$.

Let $G$ be a partial cube.  
We denote by $\Gamma(\mathbf{Tor}(G))$ the intersection graph of the maximal gated quasi-hypertori of $G$, i.e., the graph whose vertex set is
the set of all maximal gated quasi-hypertori of $G$, and
such that two vertices are adjacent if and only if they have a
non-empty intersection.  We also denote by $G^{\Diamond}$ the graph
having the same vertex set as $G$ and where two vertices are adjacent
if and only if they belong to a common gated
quasi-hypertori of $G$.  The graph $\Gamma(\mathbf{Tor}(G))$
is the clique graph of $G^{\Diamond}$, that is, the intersection graph
of the maximal simplices (i.e., complete subgraphs) of $G^{\Diamond}$.

We recall that, if $x$ and $y$ are two vertices of a finite graph $G$,
then $x$ is said to be \emph{dominated} by $y$ in $G$ if $N_{G}[x]
\subseteq N_{G}[y]$.  We say that a finite graph $G$ is \emph{dismantlable}
if its vertices can be linearly ordered $x_{0}, \ldots, x_{n}$ so
that, for each $i < n$, the vertex $ x_{i}$ is dominated by $x_{i+1}$
in the subgraph of $G$ induced by $\{x_{i},\ldots,x_{n} \}$.  The
enumeration $x_{0}, \ldots, x_{n}$ is called a \emph{dismantling
enumeration} of the vertices of $G$.

\begin{pro}\label{P:BP91}\textnormal{(Bandelt and Prisner 
    \cite[Proposition 2.6]{BP91})}
The clique graph of a dismantlable graph is also dismantlable.
\end{pro}

\begin{lem}\label{L:gated.H}
Let $W_{ab}$ be a semi-periphery of a Peano partial cube $G$, and
let $H \in \mathbf{Tor}(G,ab)$ be such that $V(H) \subseteq U_{ab}
\cup U_{ba}$.  If $H \cap G[W_{ab}]$ is gated, then $H$ is gated.
\end{lem}

\begin{proof}
Note that $H$ is isomorphic to the Cartesian product of $H' := H \cap
G[U_{ab}]$ with $K_{2}$.  Let $u \in V(G)$ and $x$ be its gate in $H'$.
Clearly $x$ or its neighbor $x'$ in $U_{ba}$ is the gate of $u$ in $H$
depending on whether $u \in W_{ab}$ or $u \in W_{ba}$.
\end{proof}

\begin{thm}\label{T:dismantlable}
If $G$ is a finite Peano partial cube, then the
graphs $G^{\Diamond}$ and $\Gamma(\mathbf{Tor}(G))$ are dismantlable.
\end{thm}

\begin{proof}
By Proposition~\ref{P:BP91}, it suffices to prove that $G^{\Diamond}$
is dismantlable.  The proof will be by induction on the order $|V(G)|$
of $G$.  This is obvious if $|V(G)| = 1$.  Suppose that this holds for
any Peano partial cube whose order is at most
$n$, for some positive integer $n$.  Let $G$ be an
Peano partial cube such that $|V(G)| = n+1$.

Let $W_{ab}$ be a semi-periphery of $G$.  If $W_{ab} \neq
U_{ab}$, we first consider the elements of $W_{ab} - U_{ab}$.  Let $x
\in W_{ab} - U_{ab}$.  Then $x$ is a vertex of some bulge $X$ of $G$, and thus of $\mathbf{Cyl}[X]$.  Therefore $x$ is a vertex of some convex cycle $C \in \mathbf{C}(G,ab)$, and $C$ is clearly a cycle of any convex quasi-hypertorus that contains $x$.  It follows that $x$ is dominated in $G^{\Diamond}$ by any vertex
of $C$, and in particular by those that belong to $U_{ab}$.

Let $x_{0}, \ldots, x_{i}$ be an enumeration of the vertices of
$W_{ab} - U_{ab}$.  In the subgraph $G - \{x_{0}, \ldots, x_{i} \}$
(that is, in $G$ if $W_{ab} = U_{ab}$), each vertex $u$ of $U_{ab}$ is
clearly dominated by its neighbor $u'$ in $U_{ba}$ because, by
the properties of $\mathbf{Cyl}[X]$ or by Lemma~\ref{L:gated.H}, $u'$
belongs to every maximal gated element  of
$\mathbf{Tor}(G,ab)$ to which belongs $u$.  Let $x_{i+1}, \ldots,
x_{j}$ be an enumeration of the vertices in $U_{ab}$, and let $H := G
- \{x_{0}, \ldots, x_{j} \}$.

This subgraph $H$, which a convex subgraph of $G$, is then
a Peano partial cube.  Consequently, by the induction
hypothesis, $H^{\Diamond}$ is dismantlable.  Let $x_{j+1}, \ldots,
x_{n+1}$ be a dismantling enumeration of $V(H)$.  Then $x_{0}, \ldots,
x_{n+1}$ is a dismantling enumeration of the vertices of $G$.
\end{proof}

We say that a self-contraction $f$ of a graph $G$ \emph{fixes} a
subgraph $H$ of $G$ if $f(H) = H$.

\begin{thm}\label{T:autom./fin.hypernet.}
Any finite Peano partial cube $G$ contains a
gated quasi-hyper-torus which is fixed by all
automorphisms of $G$.
\end{thm}

\begin{proof}
Each automorphism of $G$ clearly induces an automorphism of the graph 
$\Gamma(\mathbf{Tor}(G))$.  By Theorem~\ref{T:dismantlable} and
\cite[Theorem 4.8]{P95}, there exists a finite simplex $\mathcal{S}$
of $\Gamma(\mathbf{Tor}(G))$ which is fixed by every automorphism of
this graph.  By the definition of $\Gamma(\mathbf{Tor}(G))$, the
elements of $\mathcal{S}$ are gated and pairwise non-disjoint.  Hence,
by \cite[Proposition 2.4]{B89}, they have a non-empty intersection
$H$.  Because the intersection of gated subgraphs is gated, and thus convex, it follows by the Convex Subgraph Property of Cartesian product that $H$ is the Cartesian product of even cycles of length greater than $4$ and of paths.  Hence $H$ is either a quasi-hypertorus or a median graph or the Cartesian product of hypertorus with a median graph.  Moreover, $H$ is clearly fixed by every
automorphism of $G$.

It is sufficient to prove that $H$ contains a gated quasi-hypertorus which is fixed by every automorphism of $G$.  We are done if $H$ is regular.  Assume that $H$ is not regular.  We 
have two cases.

\emph{Case 1.}\; $H$ is a median graph.

Then, by \cite{BV87}, $H$ contains a hypercube $H'$ which is fixed by
all automorphisms of $H$, and thus of $G$.  Moreover $H'$ is gated
since it is a hypercube.

\emph{Case 2.}\; $H$ is not a median graph.

Then $H$ is the Cartesian product of a hypertorus $T$ with a finite
median graph $M$.  As in Case~1, $M$ contains a hypercube $M'$ which
is fixed by all automorphisms of $M$.  Hence $H' := T \Box M'$ is a
convex quasi-hypertorus which is fixed by all automorphisms
of $H$, and thus of $G$.  Moreover, $H'$ is convex in $G$ since so is
$H$, and thus it is gated in $G$ by
Theorem~\ref{T:conv.reg.H-subgr.=>gated}.
\end{proof}

The above theorem gives as a
particular case the result \cite[Theorem 4.3]{P08-4} stating that
\emph{in any finite netlike partial cube $G$ there exists a gated
cycle or a hypercube which is fixed by all automorphisms of $G$}.

\subsection{Compact Peano partial cubes}\label{SS:FSP/compact}

Recall that, by Corollary~\ref{C:comp.hyp.=no.isom.rays}, \emph{a Peano partial cube is compact if and only if it contains no isometric rays}.  We prove three fixed subgraph properties for compact
Peano partial cubes that are of the same type as
those which were proved in~\cite{P09-2}.  We will use Proposition~\ref{P:Gf.hypernet.} and the notations introduced before the statement of this proposition.  
We need two more results.

\begin{lem}\label{L:inter(Gi)/hypernet.}
Let $G$ be a Peano partial
cube.  Let $(G_{i})_{i \in I}$ be a family of faithful subgraphs of $G$ that are ph-homogeneous and such that, for each
finite $J \subseteq I$, $G_{J} := \bigcap_{j \in J}G_{j}$ is a
faithful subgraph of $G$ which is ph-homogeneous.
Then $G_{I} := \bigcap_{i \in I}G_{i}$ is a faithful subgraph of $G$ which is ph-homogeneous.
\end{lem}

\begin{proof}
By \cite[Proposition 3.10]{P09-2}, $G_{I}$ is a faithful subgraph of
$G$.  We will prove that it is ph-homogeneous.  First note that, for all finite $J \subseteq J' \subseteq I$,\; $G_{J'}$ is a faithful subgraph of $G_J$.

Let $F$ be a finite convex subgraph of $G_I$.  Denote by $F_G$ and $F_{G_J}$ the convex hulls of $F$ in $G$ and $G_J$, respectively, for every finite $J \subseteq I$.  The subgraph $F_G$ is finite and $F_{G_J} \subseteq F_G$ for every finite $J \subseteq I$.  Moreover $F_{G_{J'}} \subseteq F_{G_J}$ for every finite $J' \subseteq I$ such that $J \subseteq J'$, since $G_{J'}$ is a faithful subgraph of $G_J$.  It follows that there exists some finite $J \subseteq I$ such that $F_{G_J} = F_{G_{J'}}$ for every finite $J' \subseteq I$ with $J \subseteq J'$.  Therefore $F = F_J$, and thus $ph(F) \leq 1$ since $G_J$ is a Peano partial cube.
\end{proof}

\begin{pro}\label{P:fixed.fin.set./autom.}\textnormal{(Polat~\cite[Corollary
3.4]{P09-2})}
Let $G$ be a compact partial cube.  Then there exists a non-empty
finite convex subgraph of $G$ which is fixed by every automorphism of
$G$.
\end{pro}

\begin{thm}\label{T:fix.subgr./comp.hypernet.}
Let $G$ be a compact Peano partial cube.  We
have the following properties:

\textnormal{(i)}\; $G$ contains a gated quasi-hypertorus which is fixed by all automorphisms
of $G$.

\textnormal{(ii)}\; Any self-contraction of $G$ fixes a
gated quasi-hypertorus.

\textnormal{(iii)}\; For any commuting family $\mathcal{F}$ of
self-contractions of $G$, there exists a gated
quasi-hypertorus which is fixed by every
element of
$\mathcal{F}$.
\end{thm}

\begin{proof}
(i)\; By Proposition~\ref{P:fixed.fin.set./autom.}, there is a
non-empty finite convex subgraph $F$ of $G$ which is fixed by every
automorphism of $G$.  Then $F$ is a finite Peano
partial cube.  Hence, by Theorem~\ref{T:autom./fin.hypernet.}, $F$
contains a gated quasi-hypertorus which is fixed by
all automorphisms of $F$, and thus of $G$.  Note that, because $F$ is
convex, a gated quasi-hypertorus
of $F$ is convex in $G$, and thus also gated in $G$ by
Theorem~\ref{T:conv.reg.H-subgr.=>gated}.

(ii)\; Let $f$ be a self-contraction of $G$.  Then, by Lemma~\ref{L:Gf} and Proposition~\ref{P:Gf.hypernet.}, $G_{f}$ is a non-empty
Peano partial cube, which
is faithful in $G$, and whose vertex set is closed and thus compact
since so is $V(G)$.  Clearly the restriction of $f$ to $V(G_{f})$ is
an automorphism of $G_{f}$.  Therefore, by (i), there exists a
convex quasi-hypertorus $F$ of
$G_{f}$ which is fixed by all automorphisms of $G_{f}$, and thus by
$f$.  Then $F$ is a faithful subgraph of $G$, since $G_{f}$ is
faithful in $G$, and hence $F$ is gated in $G$ by
Theorem~\ref{T:faithf.hypertor.=>gated}.

(iii)\; By Lemma~\ref{L:Gf}, for every $f \in \mathcal{F}$, the set $V_{f}$, where $V$
stands for $V(G)$, is non-empty, and $G_{f}$ is a faithful subgraph of $G$ whose vertex set is a closed
and thus compact set of $G$.  Therefore, by Proposition~\ref{P:Gf.hypernet.}, $G_{f}$ is a compact
Peano partial cube.  If $g \in \mathcal{F}$
commutes with $f$ on $V_{f}$, and if $x \in V_{f}$, then $f^{p}(g(x))
= g(f^{p}(x)) = g(x)$ for any $p \geq 0$ such that $f^{p}(x) = x$.
Thus $g(V_{f}) \subseteq V_{f}$.  Hence, since $G_{f}$ is a compact
Peano partial cube, it follows by (ii) that $g$
fixes a non-empty finite gated quasi-hypertorus of
$G_{f}$.  Therefore, by Lemma~\ref{L:Gf}, $V_{f} \cap V_{g} =
(V_{f})_{g}$ ($= (V_{g})_{f}$) is a non-empty, faithful and closed, and thus
compact, set of vertices of $G_{f}$, and thus of $G$, and moreover, by Proposition~\ref{P:Gf.hypernet.}, 
$G[V_{f} \cap V_{g}] = (G_{f})_{g}$ is a faithful subgraph of $G_{f}$,
and thus of $G$, which is ph-homogeneous.  Note that $[x]_{f} \cup [x]_{g} \subseteq V_{f}
\cap V_{g}$ for every $x \in V_{f} \cap V_{g}$.  Hence the
restrictions of $f$ and $g$ to $V_{f} \cap V_{g}$ are automorphisms
of $G[V_{f} \cap V_{g}]$.  Inductively, for any non-empty finite
$\mathcal{K} := {f_{1},\dots,f_{n}} \subseteq \mathcal{F}$, the set
$V_{\mathcal{K}} := \bigcap_{f \in \mathcal{K}} V_{f} = (\dots
(V_{f_{1}}) \dots)_{f_{n}}$ is a non-empty, faithful,
 and closed, and thus compact subset of $V(G)$, and $G_{\mathcal{K}} :=
G[V_{\mathcal{K}}]$ is a Peano partial cube by Proposition~\ref{P:Gf.hypernet.}.  Therefore $V_{\mathcal{F}}
:= \bigcap_{f \in \mathcal{F}} V_{f} \neq \emptyset$ since the space
$V(G)$ is compact and the sets $V_{f}$'s are closed.  Then
$G_{\mathcal{F}} := G[V_{\mathcal{F}}]$, being the intersection of all
$G_{f}$'s, which are faithful and non-empty, is also a non-empty faithful subgraph of
$G$.   Hence $G_{\mathcal{F}}$ is a non-empty Peano partial cube by Lemma~\ref{L:inter(Gi)/hypernet.}, whose vertex
set is closed in $G$, and thus which is compact since so is $V(G)$.
Because the restriction of every $f \in \mathcal{F}$ to
$V_{\mathcal{F}}$ is an automorphism of $G_{\mathcal{F}}$, it follows
by (i) that $G_{\mathcal{F}}$ contains a gated quasi-hypertorus $F$ which is fixed by every element of
$\mathcal{F}$.    Then $F$ is a faithful subgraph of $G$, since $G_{\mathcal{F}}$ is
faithful in $G$, and hence $F$ is gated in $G$ by
Theorem~\ref{T:faithf.hypertor.=>gated}.
\end{proof}

The result above generalizes the corresponding results \cite[Theorem
1.2]{T96} for median graphs and \cite[Theorems 6.5, 6.6 and
6.8]{P08-4} for netlike partial cubes.

By ~\cite[Proposition 4.1]{P09-2}, if a graph contains an isometric ray, then there exists a self-contraction of this graph which fixes no non-empty finite set of vertices.  Hence we can state the following improvement of property (ii) of the above theorem.

\begin{cor}\label{C:NSC}
Any self-contraction of a Peano partial cube $G$ fixes a finite gated quasi-hypertorus if and only if $G$ contains no isometric rays.
\end{cor}

We complete this subsection with a result which holds for infinite Peano partial cubes that are not necessarily compact, and which is a consequence of Theorem~\ref{T:fix.subgr./comp.hypernet.}.

\begin{pro}\label{P:fixed.fin;subgr./quasi-hypert.}
Let $f$ be a self-contraction of a Peano partial cube $G$.  If $f$ fixes a finite subgraph of $G$, then $f$ fixes some gated quasi-hypertorus.
\end{pro}

\begin{proof}
Let $F$ be a finite subgraph of $G$ which is fixed by $f$.  Clearly $f$ fixes the convex hull $\overline{F}$ of $F$, which is a finite Peano partial cube.  Then the restriction $f'$ of $f$ to $V(\overline{F})$ is a self-contraction of $\overline{F}$.  Hence, by Theorem~\ref{T:fix.subgr./comp.hypernet.}(ii), $\overline{F}$
contains a gated quasi-hypertorus $H$ which is fixed by
$f'$, and thus by $f$.  Moreover $H$ is gated in $G$ by
Theorem~\ref{T:conv.reg.H-subgr.=>gated}.
\end{proof}

\section{Convex invariants}\label{S:conv.inv.}

In this section we study two convex invariants of the geodesic convexity of a Peano partial cube: the Helly number and the depth, i.e., the height of the poset of the non-trivial half-spaces ordered by inclusion.

\subsection{Helly number}\label{SS:Helly}

The \emph{Helly number} $h(G)$ of a graph $G$ is the smallest integer, if it exists, such that any finite family of $h(G)$-wise non-disjoint convex sets of $G$ has a non-empty intersection.  

As an immediate consequence of a result of Bandelt and Chepoi~\cite[Theorem 2]{BC96-2} stating that \emph{the Helly number of a discrete geometric weakly modular space is equal to its clique number}, we have the following result.

\begin{pro}\label{P:Helly/med.gr.}
If $G$ is a median graph with at least two vertices, then $h(G) =~2$.
\end{pro}

This result does not hold if $G$ is any netlike partial cube.  Take for example an even cycle $C$ of length $6$, and let three paths of $C$ of length $2$ that pairwise have exactly one vertex in common, and thus whose union is $C$.  Then these three paths are convex and have an empty intersection, which proves that $h(C) \geq 3$.  More precisely and more generally we have:

\begin{thm}\label{T:Helly/ph-homogeneous}
The Helly number of a Peano partial cube $G$ is at most 3, with the equality if and only if $G$ is not a median graph.
\end{thm}

\begin{proof}
Let $G$ be a Peano partial cube with at least two vertices.  By Proposition~\ref{P:Helly/med.gr.}, $h(G) = 2$ if $G$ is a median graph.  Assume that $G$ is not a median graph.  Then, by Proposition~\ref{P:hypernet./med.gr.}, $G$ contains a convex cycle of length greater than $4$.  It follows, by what we saw above, that $h(G) \geq 3$.  To prove that $3$ is sufficient, it suffices to show that, for every integer $n \geq 3$, any family of $n$ convex sets of $G$ that are $3$-wise non-disjoint has a non-empty intersection.  The proof will be by induction on $n$.

This is obvious if $n = 3$.  Suppose that this is true for some $n \geq 3$.  Let $(F_i)_{0 \leq i \leq n}$ be a family of $n+1$ convex sets of $G$ that are $3$-wise non-disjoint.  Suppose that $F_0 \cap \dots \cap F_n = \emptyset$.

By the induction hypothesis, the $F_i$'s are $n$-wise non-disjoint.  Hence $F := F_1 \cap \dots \cap F_n$ is a non-empty convex set.  By the definition of a Peano partial cube, $G$ has the separation property $\mathrm{S}_{4}$.  Therefore there exits a half-space $H$ containing $F_0$ and disjoint from $F$, which is maximal with respect to these two conditions.  Because $G$ is a partial cube, $H = W_{ba}$ for some edge $ab$ of $G$.  The set $W_{ba} \cup \mathrm{Att}(W_{ba})$ is convex by definition.  Hence $F \cap \mathrm{Att}(W_{ba}) \neq \emptyset$, since otherwise, by the separation property $\mathrm{S}_{4}$, there would exist a half-space $H'$ containing $W_{ba} \cup \mathrm{Att}(W_{ba})$ and disjoint from $F$, contrary to the maximality of $H$.  Let $u \in F \cap \mathrm{Att}(W_{ba})$.  We distinguish two cases.

\emph{Case 1.}\; $u \in U_{ab}$.

Then the neighbor $u'$ of $u$ in $U_{ba}$ belongs to $F_i$ for any $i$ with $1 \leq i \leq n$, since $F_i \cap W_{ba} \supseteq F_i \cap F_0 \neq \emptyset$.  It follows that $u' \in F$, contrary to the hypothesis that $F \subseteq W_{ab}$.

\emph{Case 2.}\; $u \notin U_{ab}$.

$U_{ab}$ is strongly ph-stable since $G$ is ph-homogeneous.  
Let $P_u$ be the $U_{ab}$-geodesic associated with $u$, and let $v$ and $w$ be its endvertices.  Then, by Lemma~\ref{L:I(u,x)/v,w}, for any vertex $x \in U_{ab}$, and thus for any $x \in W_{ba}$, $v$ or $w$ belongs to $I_G(u,x)$. 

On the other hand, by the induction hypothesis, the $F_i$'s are $n$-wise non-disjoint.  Hence the elements of $(F_i \cap W_{ba})_{0 \leq i \leq n}$, which are non-empty convex sets since $W_{ba}$ is convex, are also $n$-wise non-disjoint.  It follows that the elements of $(F_i \cap W_{ba})_{1 \leq i \leq n}$ are $(n-1)$-wise non-disjoint.

It follows that, for each $i$,\; $1 \leq i \leq n$,\; $v$ or $w$ belongs to $F_i$ since this set is convex and because $v$ or $w$ belongs to $I_G(u,x_i)$ for some $x_i \in F_i \cap W_{ba}$.  Suppose that $w \notin F$.  Then $w \notin F_i$ for some $i$ with $1 \leq i \leq n$.  It follows that $v \in F_i$.  Let $j \neq i$ with $1 \leq j \leq n$.  Then $F_i \cap F_j \cap W_{ba}$ is non-empty since $n \geq 3$.  Let $x_{ij}$ be an element of this intersection.  Then $v \in I_G(u,x_{ij})$ because $w \notin F_i$.  It follows that $v \in F_j$ by the convexity of this set.  We infer that $v \in F$.

Therefore, in any case, $v$ or $w$ belongs to $F$.  By Case 1, this yields a contradiction with the hypothesis that $F \subseteq W_{ab}$.  Consequently $F_0 \cap \dots \cap F_n \neq \emptyset$.

We deduce that $h(G) = 3$.
\end{proof}

\begin{cor}\label{C:med.gr.=ph-homog.Helly2}
A Peano partial cube $G$ is a median graph if and only if $h(G) =~2$.
\end{cor}

\subsection{Depth}\label{SS:depth}

In this subsection we study the depth of the geodesic convex structure of a Peano partial cube, and mainly of a hyper-median partial cube.

In~\cite{BV91}, Bandelt and van de Vel introduced an invariant of 
convex structures---the depth---to study the structure of finite 
median graphs.

\begin{defn}\label{D:depth}
The \emph{depth} of a convex structure is the supremum length of a
chain of non-trivial half-spaces.
\end{defn}

We will prove several results about the depth of (non-necessarily finite) Peano partial cubes.  For all these results but Proposition~\ref{P:depth=1} we have to require that these partial cubes are tricycle-free.  We begin by the few following remarks:

\textbullet\; If a partial cube $G$ is compact, then any chain of half-spaces of $G$ is finite (Proposition~\ref{P:half-spaces/compact}).  The converse is false as is shown by an infinite hypercube. 

\textbullet\; If a partial cube has finite diameter, and a fortiori if it is finite, then it obviously contains no isometric rays, and thus is compact by Corollary~\ref{C:comp.hyp.=no.isom.rays}, and moreover it has finite depth.

\textbullet\; If a partial cube $G$ contains no isometric rays, then any quasi-hypertorus of $G$ is finite.

\textbullet\; Any convex subgraph of a Peano partial cube is a Peano partial cube, and thus its geodesic convexity has the separation property $\mathrm{S}_{4}$ by Theorem~\ref{T:Peano/ph}.

\textbullet\; Any convex subgraph of a compact hyper-median partial cube is also a compact hyper-median partial cube.

\textbullet\; By Theorem~\ref{T:C-subgr./gated}, if $G$ is a Peano partial cube, then any element of $\mathbf{Cyl}[G]$ is gated in $G$.\\

Clearly a partial cube has depth $1$ if and only if it is strongly semi-peripheral.  Therefore we deduce from Theorem~\ref{T:reg. hypernet./str.sem.-periph./quasi-hypert.} the following result.

\begin{pro}\label{P:depth=1}
Let $G$ be a compact Peano partial cube.  Then $G$ has depth $1$ if and only if it is a quasi-hypertorus.
\end{pro}

We now recall the main result of~\cite{BV91}.

\begin{pro}\label{P:Band.v.d.Vel}\textnormal{(Bandelt and van de Vel \cite[Theorem 2.4]{BV91})}
A finite median graph $G$ has depth $k \geq 2$ if and only if there is a convex set $C \subseteq V(G)$ of depth $k-2$ meeting each maximal cube of $G$.  All convex sets meeting each maximal cube of $G$ and that are minimal with respect to this property are isomorphic.
\end{pro}

Next theorem generalizes the above result and the analogous one \cite[Theorem 6.4]{P07-5} about tricycle-free netlike partial cubes.  We recall that $\mathbf{Cyl}^+[G]$ denotes the set of all subgraphs of $G$ that are either elements of $\mathbf{Cyl}[G]$ or maximal hypercubes of $G$.

\begin{thm}\label{T:depth/subgr.}
Let $G$ be a compact hyper-median partial cube whose depth is finite.  Then $G$ has depth $k \geq 2$ if and only if
there is a gated subset of $V(G)$ of depth $k-2$ that meets
each element of $\mathbf{Cyl}^+[G]$ and that is minimal with respect to this property.
\end{thm}

We need two lemmas.

\begin{lem}\label{L:Fgated.G.2edges}
Let $F$ be a gated subgraph of a partial cube $G$, and let $a_{1}b_{1}$ and $a_{2}b_{2}$ be two edges of $F$.  If $W_{b_{1}a_{1}}^{F} \subset W_{b_{2}a_{2}}^{F}$, then $W_{b_{1}a_{1}}^{G} \subset W_{b_{2}a_{2}}^{G}$.
\end{lem}

\begin{proof}
Assume that $W_{b_{1}a_{1}}^{F} \subset W_{b_{2}a_{2}}^{F}$, and suppose that 
$a_{1}b_{1}$ is $\Theta$-equivalent to an edge $uv$ of $G[co_{G}(U_{a_{2}b_{2}}^{G})]$.  The gates $u'$ and $v'$ in $F$ of $u$ and $v$, respectively, clearly belong to $co_{F}(U_{a_{2}b_{2}}^{F})$.  Moreover, because $uv$ is $\Theta$-equivalent to the edge $a_{1}b_{1}$ of $F$, we infer that $u'v'$ is an edge which is $\Theta$-equivalent to $uv$, and thus to $a_{1}b_{1}$, contrary to the assumption.

Consequently $a_{1}b_{1}$ is not $\Theta$-equivalent to an edge of $co_{G}(U_{a_{2}b_{2}}^{G})$, and thus of $W_{a_{2}b_{2}}^{G}$ by Lemma~\ref{L:E(G[Wab])}, which implies that $W_{b_{1}a_{1}}^{G} \subset W_{b_{2}a_{2}}^{G}$.
\end{proof}

The following result is obvious.

\begin{lem}\label{L:inter.two.gated}
Let $A$ and $B$ be two non-disjoint sets of vertices of a graph $G$ such that $A$ is convex and $B$ is gated.  Then $A \cap B$ is the set of the gates in $B$ of all elements of $A$.
\end{lem}

\begin{proof}[\textnormal{\textbf{Proof of
Theorem~\ref{T:depth/subgr.}}}]
(a)\; Assume that the depth of $G$ is $k \geq 2$.  Note that all hypercubes are finite since $G$ contains no isometric rays by Corollary~\ref{C:comp.hyp.=no.isom.rays}.

(a.1)\; We denote by $\mathcal{S}(G)$ the set of all maximal sequences $\sigma = (a_{i}b_{i})_{1 \leq i \leq h}$ of length $h \geq 2$ of edges of $G$ such that $W_{b_{1}a_{1}}^{G} \subset \dots \subset W_{b_{h}a_{h}}^{G}$ and $V(G) = W_{b_{h}a_{h}}^{G} \cup co_{G}(U_{a_{h}b_{h}}^{G})$, and moreover we put $A_{\sigma} := W_{a_{2}b_{2}}^{G} \cup co_{G}(U_{b_{2}a_{2}}^{G})$.  Note that the length of any sequence in $\mathcal{S}(G)$ is at most $k$.  Put $$A := \bigcap_{\sigma \in \mathcal{S}(G)}A_{\sigma}.$$  

For any $\sigma \in \mathcal{S}(G)$,\; $A_{\sigma}$ is gated by Lemma~\ref{L:tricycle/Gab}, since $G$ is tricycle-free by Theorem~\ref{T:charact.hyper-median}.     
Clearly $A_{\sigma} \cap A_{\sigma'} \neq \emptyset$ for all $\sigma, \sigma' \in \mathcal{S}(G)$.  Moreover $G$ is compact.  It follows that $A$ is a non-empty gated set since the gated sets of $G$ have the strong Helly property by Proposition~\ref{P:gated.sets/Helly}, and since the intersection of gated sets is gated.

(a.2)\; Let $H \in \mathbf{Cyl}^+[G]$.  Suppose that $A_{\sigma} \cap V(H) = \emptyset$ for some $\sigma \in \mathcal{S}(G)$.  Let $x \in A_{\sigma}$ and $y \in V(H)$ be such that $d_{G}(x,y) = d_{G}(A_{\sigma},V(H))$.  Then $x$ is the gate of $y$ in $A_{\sigma}$.  Let $xx' \in \partial_{G}(A_{\sigma})$ be an edge of some $(x,y)$-geodesic.  
By Lemma~\ref{L:gen.propert.}(ix), $xx'$ is not $\Theta$-equivalent to an edge of $G[A_{\sigma}]$.  By the properties of $S$,\; $W_{x'x}^{G}$ is a semi-periphery of $G$.  Moreover $y \in W_{x'x}^{G}$.  Hence $H \in \mathbf{Cyl}[G,x'x]$.  It follows that, if $x'$ were a vertex of $H$, then $x$ would be a vertex of $H$ as well, contrary to the hypothesis.  Hence $x' \notin V(H)$, and thus $x'$ cannot belong to $I_{G}(x,V(H))$ since $H \in \mathbf{Cyl}[G,x'x]$, contrary to the choice of $x'$.  Consequently $A_{\sigma} \cap V(H)$ is non-empty.  By Lemma~\ref{L:inter.two.gated}, $A_{\sigma} \cap V(H)$ is then a gated subset of $V(H)$ which contains the gate in $H$ of each element of $A_{\sigma}$.  Since $A$ is non-empty, it follows that $A \cap V(H)$ is a non-empty gated set which contains the gate in $H$ of each element of $A$.

(a.3)\; Denote by $\mathcal{A}$ the set of all gated subset of $A$ that meet each element of $\mathbf{Cyl}^+[G]$.  Let $\mathcal{C}$ be a descending chain of elements of $\mathcal{A}$.  Then $C \cap H$ is convex for any $C \in \mathcal{C}$, since so is any element of $\mathbf{Cyl}^+[G]$ and of $\mathcal{A}$.  It follows that the intersection of $\bigcap \mathcal{C}$ with $H$ is non-empty because $G$ is compact.  Moreover $\bigcap \mathcal{C}$ is gated as an intersection of gated sets.  
Therefore $\bigcap \mathcal{C} \in \mathcal{A}$.  Consequently, by Zorn's lemma, $\mathcal{A}$ has a minimal element.  Denote by $A^{*}$ such a minimal gated subset of $A$.

(a.4)\; We now prove that both $A$ and $A^{*}$ have depth $k-2$.  First, let $\sigma = (a_{i}b_{i})_{1 \leq i \leq h}$ be a sequence of edges of $G[A]$ such that $W_{b_{1}a_{1}}^{G[A]} \subset \dots \subset W_{b_{h}a_{h}}^{G[A]}$.  Then, by Lemma~\ref{L:Fgated.G.2edges},\; $W_{b_{1}a_{1}}^{G} \subset \dots \subset W_{b_{h}a_{h}}^{G}$.  Therefore $\sigma$ is a subsequence of some element of $\mathcal{S}(G)$.  
It follows, by the construction of $A$, that $h \leq k-2$, and thus the depth of $A$ is at most $k-2$.

On the other hand, because the depth of $G$ is $k$,  there exists a sequence $(a_{i}b_{i})_{1 \leq i \leq k} \in \mathcal{S}(G)$.  Then, because $A$ meets each element of $\mathbf{Cyl}[G,a_{i}b_{i}]$ for $1 \leq i \leq k$, it follows that $W_{b_{i}a_{i}}^{G[A]}$ is non-empty for every $i$ with $2 \leq i \leq k-1$.  Therefore the depth of $A$ is at least $k-2$, and thus exactly $k-2$ by the above inequality.

By replacing $A$ by $A^{*}$ in the above proof, we would obtain that the depth of $A^{*}$ is also $k-2$.\\

(b)\;  Conversely assume that there exists a gated subset $B$ of $V(G)$ of depth $k-2$ which meets
each element of $\mathbf{Cyl}^+[G]$ and which is minimal with respect to this property.  Let $l$ be the depth of $G$.  By the first part of (a.4) and since $B$ meets
every element of $\mathbf{Cyl}^+[G]$, each sequence in $\mathcal{S}(G[B])$ of length $k-2$ is a subsequence of an element of $\mathcal{S}(G)$ of length at most $k$.  On the other hand, each sequence in $\mathcal{S}(G)$ of length $l$ gives a sequence in $\mathcal{S}(G[B])$ of length at most $k-2 \leq l$.  Hence $k-2 \leq l \leq k$.  

Suppose that $l < k$.  Let $A$ be the gated set constructed in (a).  Then $B \nsubseteq A$ since, by (a.4), any gated subset of $A$ meeting each element of $\mathbf{Cyl}^+[G]$ and minimal with respect to this property has depth $l-2 < k-2$.
Hence, by the construction of $A$, there is a sequence $(a_{i}b_{i})_{1 \leq i \leq h} \in \mathcal{S}(G)$ such that $B' := B \cap C$, where $C := W_{a_{2}b_{2}}^{G} \cup co_{G}(U_{b_{2}a_{2}}^{G})$, is a non-empty proper gated subset of $B$.  Let $H \in \mathbf{Cyl}^+[G]$.  If $V(H) \subseteq C$, then $B' \cap V(H) = B \cap V(H) \neq \emptyset$.  Suppose that $V(H) \nsubseteq C$.  Clearly $C \cap V(H) \neq \emptyset$ since $A \subseteq C$ and $A \cap V(H) \neq \emptyset$.  Let $x \in B'$, and let $y$ be the gate of $x$ in $H$.  Then, by Lemma~\ref{L:inter.two.gated}, $$y \in (B \cap V(H)) \cap (C \cap V(H)) = B' \cap V(H).$$  Hence $B' \cap V(H)$ is non-empty.  This yields a contradiction with the hypothesis that $B$ is minimal with respect to the property of meeting all element of $\mathbf{Cyl}^+[G]$.  Consequently $l = k$.
\end{proof}

Because any convex set of a median graph is gated, this result gives the main part of Proposition~\ref{P:Band.v.d.Vel}.  For a  compact Peano partial cube that is not tricycle-free, the result may be different.  Take for example the benzenoid graph $G$ in Figure~\ref{F:counter-example}.  $G$ has depth $4$, and contains no proper gated set meeting all $6$-cycles.  On the other hand there is exactly one convex set of depth $2$ meeting all $6$-cycles and which is minimal with respect to both these properties; this set is depicted by the big points.  Moreover there also exists exactly one minimal convex set meeting all $6$-cycles, it has depth $3$ and is depicted by the encircled big points in the figure.

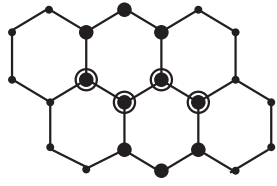
\begin{figure}[!h]
    \centering
{\tt    \setlength{\unitlength}{0.50pt}
\begin{picture}(267,152)
 \linethickness{0,3mm}    
              \put(69,140){\line(5,-3){29}}
              \put(68,139){\line(-5,-3){26}}
              \put(42,122){\line(0,-1){36}}
              \put(127,140){\line(5,-3){29}}
              \put(126,139){\line(-5,-3){26}}
              \put(98,120){\line(0,-1){36}}
              \put(155,121){\line(0,-1){36}}
              \put(41,87){\line(5,-3){29}}                                           
              \put(183,33){\line(-5,-3){26}}
              \put(183,67){\line(0,-1){36}}
              \put(128,32){\line(5,-3){29}}                                          
              \put(98,86){\line(-5,-3){27}}
              \put(71,69){\line(0,-1){36}}
              \put(72,33){\line(2,-1){27}}
              \put(97,18){\line(2,1){29}}               
              \put(127,69){\line(0,-1){36}}
              \put(183,33){\line(5,-3){29}}
              \put(183,139){\line(-5,-3){29}}
              \put(211,86){\line(-5,-3){29}}
              \put(183,139){\line(5,-3){29}}
              \put(211,122){\line(0,-1){34}}
              \put(211,86){\line(5,-3){29}}
              \put(236,33){\line(-5,-3){29}}
              \put(238,69){\line(0,-1){34}}              
              \put(99,86){\line(5,-3){29}}
              \put(156,85){\line(5,-3){29}}
              \put(155,85){\line(-5,-3){26}}             
              
              \put(238,35){\circle*{6}}
              \put(238,69){\circle*{6}}
              \put(211,86){\circle*{6}}
              \put(211,122){\circle*{6}}
              \put(183,139){\circle*{6}}
              \put(211,17){\circle*{6}}                                          
              \put(127,69){\circle*{11}}
              \put(127,69){\circle{16}}
              \put(71,69){\circle*{6}}
              \put(98,86){\circle*{11}}
              \put(98,18){\circle*{6}}
              \put(71,35){\circle*{6}}              
              \put(127,33){\circle*{11}}
              \put(42,122){\circle*{6}}
              \put(70,139){\circle*{6}}
              \put(155,122){\circle*{11}}
              \put(98,122){\circle*{11}}
              \put(98,86){\circle{16}}
              \put(127,139){\circle*{11}}
              \put(42,86){\circle*{6}}
              \put(183,69){\circle*{11}}
              \put(183,69){\circle{16}}
              \put(155,86){\circle*{11}}
              \put(155,86){\circle{16}}
              \put(183,33){\circle*{11}}
              \put(155,17){\circle*{11}}                                          
\end{picture}}
\caption{A minimal convex set meeting all $6$-cycles and a minimal convex set of depth $2$ meeting all $6$-cycles.}
\label{F:counter-example}
\end{figure}

We give a first consequence of Theorem~\ref{T:depth/subgr.}.

\begin{cor}\label{C:depth2}
The depth of a compact hyper-median partial cube $G$ is $n \leq 2$ if and
only if there exists a vertex which is common to all elements of $\mathbf{Cyl}^+[G]$.
\end{cor}

\begin{proof}
This is obvious if $n = 0$ since $G$ has then exactly one vertex, and 
if $n = 1$ since $G$ is then either a finite hypercube or a hypertorus or a prism by Proposition~\ref{P:depth=1}.  
For $n = 2$, the result is a consequence of Theorem~\ref{T:depth/subgr.}.
\end{proof}

The following theorem shows that the gated hull of a finite set of vertices of some infinite partial cube, even if it is generally not finite, may have a finiteness property.

\begin{thm}\label{T:gated.hull(finite.set)}
Let $G$ be a compact hyper-median partial cube.  Then the gated hull of any finite set of vertices of $G$ has finite depth.
\end{thm}

\begin{proof}
Let $S$ be a finite set of vertices of $G$.  We are done if the depth of $G$ is finite.  Assume that the depth of $G$ is infinite.  Then, because $S$ is finite, there exist infinitely many edges $ab$ of $G$ such that $S \subseteq W_{ab}^{G} \cup co_{G}(U_{ba}^{G}) = : A_{ab}$.  Denote by $\mathcal{E}$ the set of all these edges, and let $$A := \bigcap_{ab \in \mathcal{E}}A_{ab}.$$  For any $ab \in \mathcal{E}$, $A_{ab}$ is gated by Lemma~\ref{L:tricycle/Gab} since $G$ is tricycle-free.    
Clearly $A_{ab} \cap A_{a'b'} \neq \emptyset$ for all $ab, a'b' \in \mathcal{E}$.  Therefore, because $G$ is compact, it follows that $A$ is a non-empty gated set since the gated sets of $G$ have the strong Helly property by Proposition~\ref{P:gated.sets/Helly}.

Suppose that the depth of $A$ is infinite.  Then there exists a sequence $(a_{i}b_{i})_{0 \leq i \leq k}$ of edges of $G[A]$ such that $W_{b_{0}a_{0}}^{G[A]} \subset \dots \subset W_{b_{k}a_{k}}^{G[A]}$ and $S \cap W_{b_{0}a_{0}}^{G[A]} = \emptyset$.  Then $S \subseteq W_{a_{1}b_{1}}^{G[A]} \cup co_{G}(U_{b_{1}a_{1}}^{G[A]}) = : B$.  By Lemma~\ref{L:Fgated.G.2edges}, $W_{b_{0}a_{0}}^{G} \subset \dots \subset W_{b_{k}a_{k}}^{G}$, and $S \subseteq W_{a_{1}b_{1}}^{G} \cup co_{G}(U_{b_{1}a_{1}}^{G})$.  Therefore $a_{1}b_{1} \in \mathcal{E}$, contrary to the facts that $A \cap A_{a_{1}b_{1}} = B$ and $B$ is a proper subset of $A$.  It follows that the depth of $A$ is finite.  Let $C$ be the gated hull of $S$.  Then $C \subseteq A$ since $A$ is gated and contains $S$.  It follows, by a proof analogous to part (a.4) of the proof of Theorem~\ref{T:depth/subgr.}, that the depth of $C$ is at most that of $A$, and thus is finite.
\end{proof}

Note that the above result does not hold for a Peano graph $G$ that is not tricycle-free and not compact, such as for example the infinite hexagonal grid $H$.  Indeed the gated hull of a finite set $S$ of vertices of $H$ is either the vertex set of some $6$-cycle $C$ or $V(H)$ depending on whether $S$ is or is not a subset of some $C$.

We can now state the second consequence of Theorem~\ref{T:depth/subgr.} which generalizes \cite[Corollary 2.5]{BV91} about finite median graphs.

\begin{cor}\label{C:depth.2vert.}
Let $G$ be a compact hyper-median partial cube, and let $u, v \in V(G)$.  Then the depth of the subgraph of $G$ induced by the gated hull of the set $\{u,v\}$ is equal to the minimal length of a chain of elements of $\mathbf{Cyl}^+[G]$ joining $u$ and $v$.
\end{cor}

The proof is the same as that of \cite[Corollary 2.5]{BV91} except for some obvious modifications and the use of Theorems~\ref{T:depth/subgr.} and \ref{T:gated.hull(finite.set)} and Proposition~\ref{P:depth=1}.  If $G$ is a median graph, then the gated hull of the set $\{u,v\}$ is equal to $co_{G}(\{u,v\})$, and thus to the interval $I_{G}(u,v)$.  In particular, 
if $G$ is a tree, the depth of $I_{G}(u,v)$ is then equal to $d_{G}(u,v)$.

Note that the above result is not true for the graph $G$ in Figure~\ref{F:counter-example}.  For example the minimal convex set $A$ meeting all $6$-cycles, which is depicted by the encircled points, induces a path $P$ of length $3$, which is then the depth of $A$.  On the other hand, the gated hull of the set $\{u,v\}$ of endvertices of $P$ is equal to $V(G)$, and thus has depth $4$.  Hence neither the depth of the convex hull nor that of the gated hull of $\{u,v\}$ are equal to $2$, which is the minimal length of a chain of $6$-cycles joining $u$ and $v$.

\section{Euler-type properties}\label{S:Euler}

In this section we derive Euler-type formulas and formulas giving the isometric dimension of the finite  Peano partial cubes.  These results generalize analogous properties of median graphs and of some special netlike partial cubes.

\subsection{Euler characteristic of a finite Peano partial cube}\label{SS:Euler.charact.}

For a partial cube $G$ and a non-negative integer $i$, we denote by $\alpha_{i}(G)$ the number of $i$-cubes of $G$.  We will generalize the following result:

\begin{pro}\label{P:S.C.S}\textnormal{(Soltan and Chepoi~\cite{SoChe87}, \v{S}krekovski~\cite{Sk01})}
Let $G$ be a finite median graph.    Then
\begin{equation}
\label{E:Sa=1}
\sum_{i \in \mathbb{N}}(-1)^{i}\alpha_{i}(G) = 1.
\end{equation}
\end{pro}

\begin{pro}\label{P:reg.netl.}
Let $G$ be a finite netlike partial cube.  The following assertions are equivalent:

\textnormal{(i)}\; $G$ is a quasi-hypertorus.

\textnormal{(ii)}\; $G$ is an even cycle or a hypercube.

\textnormal{(iii)}\; $G$ is regular.
\end{pro}

\begin{proof}
(i) $\Rightarrow$ (ii):\; Suppose that $G$ is a quasi-hypertorus.  Then, by Proposition~\ref{P:TPQ/prop.isom.cycle}, $G$ contains an isometric cycle whose convex hull is $G$ itself.  Hence $G$ is an even cycle or a hypercube by Proposition~\ref{P:hypernet./netl.}.

(ii) $\Rightarrow$ (i) is trivial.

(i) $\Leftrightarrow$ (iii) is a consequence of Theorem~\ref{T:reg. hypernet./str.sem.-periph./quasi-hypert.} and Proposition~\ref{P:hypernet./netl.}.
\end{proof}

Let $G$ be a finite Peano partial cube.  By Theorem~\ref{T:reg. hypernet./str.sem.-periph./quasi-hypert.}, the convex regular subgraphs of $G$ are the convex quasi-hypertori of $G$.  We extend the concept of dimension of a hypercube to any quasi-hypertorus of $G$ as follows.  A hypertorus which is the Cartesian product of $n$ cycles is said to be of dimension $2n$, and a prism over a hypertorus of dimension $2n$ is said to be of dimension $(2n+1)$.

For every non-negative integer $n$, we denote by $\beta_{n}(G)$ the number of convex quasi-hypertori of dimension $n$ of $G$.  In particular, \emph{$\beta_{n}(G) = \alpha_{n}(G)$ for $n = 0, 1$}.  Moreover, \emph{if $G$ is a median graph, then $\beta_{n}(G) = \alpha_{n}(G)$ for any $n$}; and \emph{if $G$ is a netlike partial cube, then $\beta_{n}(G) = \alpha_{n}(G)$ for any $n \neq 2$} since, by Proposition~\ref{P:reg.netl.}, any quasi-hypertorus of a netlike partial cube is either an even cycle or a hypercube.

\begin{lem}\label{L:edge.C-layer}
Let $ab$ be an edge of a Peano partial cube $G$,\;  $X$ a bulge of $co_G(U_{ab})$, and $H = C \Box A_X := \mathbf{Cyl}[X]$.  Then any edge of $G_{ab}$ that is $\Theta$-equivalent to an edge of $C$ is an edge of some $C$-layer of $H$.
\end{lem}

\begin{proof}
Let $uv$ be an edge that is $\Theta$-equivalent to an edge $u'v'$ of $C$.  Denote by $P_u$ and $P_v$ a $(u,u')$-geodesic and a $(v,v')$-geodesic, respectively, and let $u''$ and $v''$ be the vertices of $P_u \cap A_X$ and $P_v \cap A_X$, respectively, that are the closest from $u$.  Let $Q$ be a $(u'',v'')$-geodesic.  
Then, because $u'' \in I_G(u,v')$ and $v'' \in I_G(v,u')$ since the edges $uv$ and $u'v'$ are $\Theta$-equivalent, we have $u \in I_G(v,u'')$ and $v \in I_G(u,v'')$.  Therefore the paths $P_u[u,u'']$ and $P_v[v,v'']$ contains no edge $\Theta$-equivalent to $uv$.  It follows that $Q$ contains an edge $cd$ that is $\Theta$-equivalent to $uv$.  This edge is an edge of $A_X$ since this subgraph is convex.  Therefore $u'v'$ and $cd$ cannot be $\Theta$-equivalent since $H = C \Box A_X$.  Consequently $uv$ cannot be $\Theta$-equivalent to an edge $u'v'$ of $C$.
\end{proof}

\begin{lem}\label{L:bulge/sep.}
Let $G$ be a Peano partial cube, $ab$ an edge of $G$, $X$ a bulge of $co_G(U_{ab})$, and $H := \mathbf{Cyl}[X] = C \Box H'$, where $C \in \mathbf{C}(G,ab)$.  Then:

\textnormal{(i)}\; Any edge of $G[W_{ab}]$ that is $\Theta$-equivalent to an edge of a $C$-layer of $H$ is also an edge of a $C$-layer of $H$.

\textnormal{(ii)}\; Let $e$ be an edge of the intersection of $X$ with some $C$-layer of $H$.  Then the set $A_{e}$ of all edges of $X$ that are $\Theta$-equivalent to $e$ is an edge cut of the subgraph $G_{ab}$.
\end{lem}

\begin{proof}
(i):\; Let $uv$ be an edge that is $\Theta$-equivalent to an edge $u'v'$ of $C$.  Denote by $P_u$ and $P_v$ a $(u,u')$-geodesic and a $(v,v')$-geodesic, respectively.  These geodesics pass through vertices of one of the component of $X[U_{ab}]$, say $A_X$.  Let $u''$ and $v''$ be the vertices of $P_u \cap A_X$ and $P_v \cap A_X$, respectively, that are the closest from $u$.  Let $Q$ be a $(u'',v'')$-geodesic.  
Then, because $u'' \in I_G(u,v')$ and $v'' \in I_G(v,u')$ since the edges $uv$ and $u'v'$ are $\Theta$-equivalent, we have $u \in I_G(v,u'')$ and $v \in I_G(u,v'')$.  Therefore the paths $P_u[u,u'']$ and $P_v[v,v'']$ contains no edge $\Theta$-equivalent to $uv$.  It follows that $Q$ contains an edge $cd$ that is $\Theta$-equivalent to $uv$.  This edge is an edge of $A_X$ since this subgraph is convex.  Therefore $u'v'$ and $cd$ cannot be $\Theta$-equivalent since $H = C \Box A_X$.  Consequently $uv$ cannot be $\Theta$-equivalent to an edge $u'v'$ of $C$.

(ii):\; Let $H'_{0}$ and $H'_{1}$ be two disjoint $H'$-layers of $H$ whose vertex sets are contained in $U_{ab}$, i.e., $H'_{0}$ and $H'_{1}$ are the two component of $X[U_{ab}]$.  Suppose that $A_{e}$ is not an edge cut of the subgraph $G[W_{ab}]$.  Then there exists a path $P$ of minimal length which joins a vertex $x_{0}$ of $H'_{0}$ to a vertex $x_{1}$ of $H'_{1}$, such that $E(P) \cap A_{e} = \emptyset$.  Let $R$ be an $(x_{0},x_{1})$-geodesic in $H$.  By the definition of $H$,\; $A_{e}$ is an edge cut of $H$, and thus $R$ contains an edge $e' \in A_{e}$.  By the minimality of the length of $P$,\; $P \cup R$ is a cycle of $G$, and thus it contains an edge $e'' \neq e'$ which is $\Theta$-equivalent to $e'$, and thus to $e$.  Because $R$ is a geodesic in $H$ and since the length of $P$ is minimal, it follows that $e'' \in E(P-H)$.  On the other hand, $e''$ is an edge of $H$ by (i), and thus $e'' \in A_{e}$, contrary to the hypothesis on $P$.
\end{proof}

\begin{thm}\label{T:S=1}
Let $G$ be a finite Peano partial cube.    Then
\begin{equation}
\label{E:S=1}
\sum_{i \in \mathbb{N}}(-1)^{i}\beta_{i}(G) = 1.
\end{equation}
\end{thm}

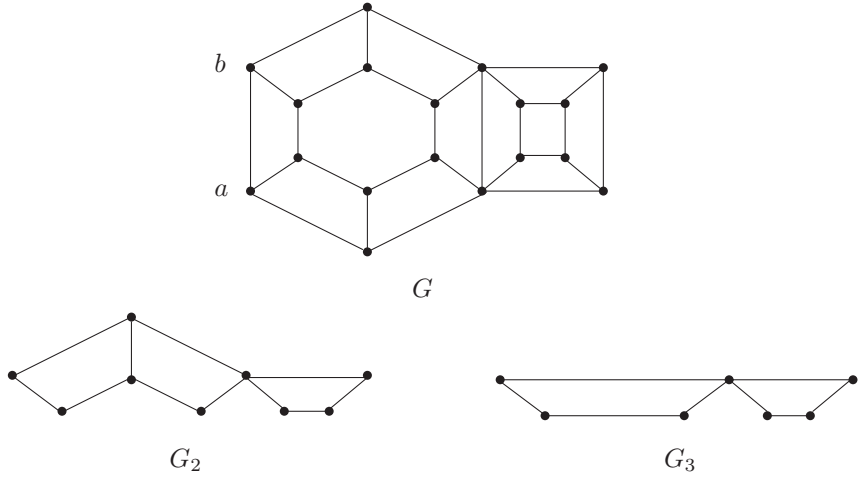
\begin{figure}[!h]
    \centering
{\tt    \setlength{\unitlength}{0.85pt}
\begin{picture}(419,236)
\thinlines    \put(116,196){\circle*{4}}
              \put(168,223){\circle*{4}}
              \put(219,196){\circle*{4}}
              \put(137,180){\circle*{4}}
              \put(168,196){\circle*{4}}
              \put(198,180){\circle*{4}}
              \put(273,196){\circle*{4}}
              \put(236,180){\circle*{4}}
              \put(256,180){\circle*{4}}
              \put(116,141){\circle*{4}}
              \put(137,156){\circle*{4}}
              \put(168,141){\circle*{4}}
              \put(198,156){\circle*{4}}
              \put(236,156){\circle*{4}}
              \put(256,156){\circle*{4}}
              \put(219,141){\circle*{4}}
              \put(273,141){\circle*{4}}
              \put(168,114){\circle*{4}}
              \put(63,85){\circle*{4}}
              \put(10,59){\circle*{4}}
              \put(63,57){\circle*{4}}
              \put(32,43){\circle*{4}}
              \put(114,59){\circle*{4}}
              \put(94,43){\circle*{4}}
              \put(168,59){\circle*{4}}
              \put(131,43){\circle*{4}}
              \put(151,43){\circle*{4}}
              \put(227,57){\circle*{4}}
              \put(329,57){\circle*{4}}
              \put(384,57){\circle*{4}}
              \put(247,41){\circle*{4}}
              \put(309,41){\circle*{4}}
              \put(346,41){\circle*{4}}
              \put(365,41){\circle*{4}}
              \put(116,197){\line(2,1){53}}
              \put(168,114){\line(2,1){53}}
              \put(219,197){\line(-2,1){51}}
              \put(166,115){\line(-2,1){51}}
              \put(116,196){\line(0,-1){55}}
              \put(219,196){\line(0,-1){55}}
              \put(137,179){\line(0,-1){23}}
              \put(198,181){\line(0,-1){24}}
              \put(168,196){\line(0,1){25}}
              \put(168,141){\line(0,-1){26}}
              \put(168,196){\line(-2,-1){31}}
              \put(168,196){\line(2,-1){30}}
              \put(168,141){\line(-2,1){31}}
              \put(168,141){\line(2,1){30}}
              \put(116,196){\line(4,-3){21}}
              \put(219,196){\line(-4,-3){21}}
              \put(219,141){\line(-4,3){21}}
              \put(116,141){\line(3,2){22}}
              \put(219,196){\line(6,-5){17}}
              \put(219,141){\line(6,5){17}}
              \put(273,196){\line(-6,-5){17}}
              \put(273,141){\line(-6,5){17}}
              \put(236,180){\line(0,-1){23}}
              \put(256,180){\line(0,-1){23}}
              \put(236,180){\line(1,0){20}}
              \put(236,157){\line(1,0){20}}
              \put(219,196){\line(1,0){54}}
              \put(219,141){\line(1,0){54}}
              \put(273,196){\line(0,-1){55}}
              \put(114,58){\line(1,0){54}}
              \put(131,43){\line(1,0){20}}
              \put(168,58){\line(-6,-5){17}}
              \put(114,58){\line(6,-5){17}}
              \put(114,58){\line(-4,-3){21}}
              \put(11,58){\line(4,-3){21}}
              \put(63,58){\line(2,-1){30}}
              \put(63,58){\line(-2,-1){31}}
              \put(63,58){\line(0,1){25}}
              \put(114,59){\line(-2,1){51}}
              \put(11,59){\line(2,1){53}}
              \put(329,57){\line(1,0){54}}
              \put(346,41){\line(1,0){20}}
              \put(383,57){\line(-6,-5){17}}
              \put(329,57){\line(6,-5){17}}
              \put(329,57){\line(-4,-3){21}}
              \put(226,57){\line(4,-3){21}}
              \put(405,57){\line(2,-3){2}}
              \put(246,41){\line(1,0){63}}
              \put(226,57){\line(1,0){105}}
              \put(188,94){$G$}
              \put(80,18){$G_{2}$}
              \put(300,18){$G_{3}$}
              \put(100,194){$b$}
              \put(100,139){$a$}
\end{picture}}
\caption{Proof of Theorem~\ref{T:S=1}.}
\label{F:n-m+c-p=1}
\end{figure}

\begin{proof}
The proof is by induction on the number of vertices of $G$.  This is
trivial if $G$ has only one vertex.  Let $r \geq 1$.  Suppose that
this is true for any Peano partial cube partial cube
having at most $r$ vertices.  Let $G$ be a Peano partial
cube partial cube having $n = r+1$ vertices.

Because $G$ is finite and ph-homogeneous, there is an edge $ab$ of $G$ such
that $W_{ab}$ is a semi-periphery.  As a convex subgraph of $G$,\; $G_{1}
:= G[W_{ba}]$ is ph-homogeneous.  Moreover it has at most
$r$ vertices.  Hence, by the induction hypothesis, we have
\begin{equation}
    \label{E:S'=1}
\sum_{i \in \mathbb{N}}(-1)^{i}\beta_{i}(G_{1}) = 1.
\end{equation}

The graph $G_{2} := G[W_{ab}]$ is also a convex subgraph of $G$.  Therefore $G_{2}$ is a Peano partial cube with at most $r$ vertices.  It follows by the induction hypothesis that
\begin{equation}
    \label{E:S"=1}
\sum_{i \in \mathbb{N}}(-1)^{i}\beta_{i}(G_{2}) = 1.
\end{equation}

Now let $G_{3}$ be the graph obtained from $G_{2}$ as follows (Figure~\ref{F:n-m+c-p=1}): for each convex path $W$ of $G_{2}$ of length greater than $1$ whose only vertices in $U_{ab}$ are its endvertices, we replace $W$ by an edge joining its endvertices.  Note that $W = C-W_{ba}$ for some convex $C \in \mathbf{C}(G,ab)$.\\

\emph{Claim.\; $G_{3}$ is a Peano partial cube.}

For each subgraph $H \in \mathbf{Cyl}[G,ab]$, let $\varphi(H) := G_{3}[V(H) \cap V(G_{3})]$, that is, the graph obtained from $H-W_{ba}$ by replacing each convex path of this graph of length greater than $1$ whose only vertices in $U_{ab}$ are its endvertices by an edge joining its endvertices.  Then denote by $\varphi_{H}(G_{2})$ the graph obtained from $G_{2}$ by replacing $H-W_{ba}$ by $\varphi(H)$.

Put $A_{H} := E(\varphi(H))-E(H)$.  Then all elements of $A_{H}$ are pairwise $\Theta$-equivalent.  Moreover, we clearly infer from the application of Lemma~\ref{L:bulge/sep.} to $H$, that $A_{H}$ is an edge cut of $\varphi_{H}(G_{2})$.  It follows that any edge of $\varphi_{H}(G_{2})$ that is $\Theta$-equivalent to an edge in $A_{H}$ also belongs to $A_{H}$ (also see Lemma~\ref{L:bulge/sep.}).  Hence $A_{H}$ is a $\Theta$-class in $E(\varphi_{H}(G_{2}))$.

The set $\mathbf{Cyl}[G,ab]$ is finite since so is $G$.  Put $$\mathbf{Cyl}[G,ab] = \{H_{i}: 1 \leq i \leq n\},$$ and let $G'_{0} := G_{2}$ and $G'_{i+1} := \varphi_{H_{i+1}}(G'_{i})$ for every $i < n$.  Then $G_{3} = G'_{n}$.  We prove by induction on $i$ that $G'_{i}$ is a Peano partial cube.  This clear if $i = 0$.  Suppose that $G'_{i}$ is a  Peano partial cube for some $i < n$.

We have $H_{i+1} = C_{i+1} \Box H'_{i+1}$, where $C_{i+1} \in \mathbf{C}(G,ab)$ and $H'_{i+1}$ is some Peano partial cube.  Then $P_{i+1} := C_{i+1}-W_{ba}$ is a geodesic of $G'_{i}$ of length $p \geq 2$, and $B := H_{i+1}-W_{ba} = P_{i+1} \Box H'_{i+1}$.  Moreover $\varphi(B) := \varphi(H_{i+1})-W_{ba} = K_{2} \Box H'_{i+1}$.

Because $W_{ab}$ is a semi-periphery of $G$, it follows that $N_{G}(x) \subseteq V(X)$ for any bulge $X$ of $W_{ab}$, and in particular for $X = B$, and every $x \in V(X)-U_{ab}$.  Moreover, by (HNB2), $V(B)-U_{ab}$ is a separator of $V(G'_{i})$.  Hence we clearly have the following properties:

\begin{enumerate}
\item[(a)]\; $I_{G'_{i+1}}(x,y) = I_{G'_{i}}(x,y) \cap V(G'_{i+1})$ for all $x, y \in V(G'_{i+1})$.

\item[(b)]\; If $X$ is a convex subset of $V(G'_{i})$, then $X \cap V(G'_{i+1})$ is convex in $G'_{i+1}$.

\item[(c)]\; For all $x, y \in V(G'_{i+1})$,\; $d_{G'_{i+1}}(x,y) = d_{G'_{i}}(x,y)$ or $d_{G'_{i}}(x,y)-p+1$ depending on whether $x$ and $y$ belong or do not belong to the same component of $V(G'_{i})-(V(B)-U_{ab})$.
\end{enumerate}

By (c), any two edges in $E(G'_{i}) \cap E(G'_{i+1})$ that are $\Theta$-equivalent in $G'_{i}$ are $\Theta$-equivalent in $G'_{i+1}$ as well.  Therefore, because $A_{H_{i+1}}$ is a $\Theta$-class, it follows that the relation $\Theta$ is transitive, and thus
 $G'_{i+1}$ is a partial cube by Theorem~\ref{T:Djokovic-Winkler}.
 
 We will show that the pre-hull number of any finite convex subgraph of $G'_{i+1}$ is at most $1$.  Let $F$ be a convex subgraph of $G'_{i+1}$.  Then $\varphi(B_{F}) := F \cap \varphi(B) = K_{2} \Box H''_{i+1}$, where $H''_{i+1}$ is a convex subgraph of $H'_{i+1}$.  It follows that $$F' := (F-\varphi(B_{F})) \cup B_{F},$$ where $B_{F} := P_{i+1} \Box H''_{i+1}$, is a convex subgraph of $G'_{i}$ such that $V(F) = V(F') \cap V(G'_{i+1})$.  By the induction hypothesis, $G'_{i}$ is ph-homogeneous.  Let $uv$ be an edge of $F$.  We will show that $U_{uv}^{F}$ is ph-stable.

Assume that $uv$ is also an edge of $F'$.  Then, by (c), we infer that:

\textbullet\; $W_{uv}^{F} = W_{uv}^{F'} \cap V(F)$\; and\; $W_{vu}^{F} = W_{vu}^{F'} \cap V(F)$

\textbullet\; $U_{uv}^{F} = U_{uv}^{F'} \cap V(F)$.

It follows, by the definition of $F'$, that:
\begin{align*}
co_{F}(U_{uv}^{F}) &= co_{F}(U_{uv}^{F-\varphi(B_{F})}) \cup co_{F}(U_{uv}^{\varphi(B_{F})})\\
 &= co_{F'}(U_{uv}^{F'-B_{F}}) \cup (co_{F'}(U_{uv}^{B_{F}}) \cap V(F))\\
&= co_{F'}(U_{uv}^{F'}) \cap V(F).
\end{align*}
Note that, if $uv$ is not $\Theta$-equivalent to an edge of some $H''_{i+1}$-layer of $B_{F}$, then $U_{uv}^{F} = U_{uv}^{F'}$ and $co_{F}(U_{uv}^{F}) = co_{F'}(U_{uv}^{F'})$.

Then, from the hypothesis that $U_{uv}^{F'}$ is ph-stable, we infer that $U_{uv}^{F}$ is also ph-stable. 

Now assume that $uv$ is not an edge of $F'$.  Then $uv \in A_{H_{i+1}}$.  Let $P$ be the $(u,v)$-geodesic of $B_F$ whose only vertices in $U_{ab}$ are its endvertices $u$ and $v$.  $P$ is then some $P_{i+1}$-layer of $B_{F}$.  Let $w$ be the neighbor of $u$ in $P$.  Then, because, by Lemma~\ref{L:bulge/sep.}(i), any edge of $G'_{i}$ that is $\Theta$-equivalent to $uw$ is an edge of $H_{i+1}-W_{ba}$, it follows that $$U_{uv}^{F} = U_{uw}^{F'}.$$  Moreover, by the properties of the hypercylinder $H_{i+1}$,\; $\mathcal{I}_{F'}(U_{uw}^{F'})$ is a periphery of $G'_i$.  Therefore $$\mathcal{I}_F(U_{uv}^{F}) = \mathcal{I}_{F'}(U_{uw}^{F'}) = U_{uw}^{F'} = U_{uv}^{F}.$$  Hence $\mathcal{I}_F(U_{uv}^{F})$ is a periphery of $G'_{i+1}$, and thus $U_{uv}^{F}$ is ph-stable.

Consequently $ph(F) \leq 1$, and thus $G'_{i+1}$ is a Peano partial cube.  Finally $G_{3}$, which is equal to $G'_{n}$, is then a Peano partial cube, which proves the claim.\\

Because $G_{3}$ has at most $r$ vertices, we have by the induction hypothesis
\begin{equation}
    \label{E:S"'=1}
\sum_{i \in \mathbb{N}}(-1)^{i}\beta_{i}(G_{3}) = 1.
\end{equation}

Let $H$ be a convex quasi-hypertorus of $G_{2}[U_{ab}]$ which is not contained in a bulge of $W_{ab}$.  Because $\mathcal{I}_{G}(U_{ba})$ is a semi-periphery of $G[\mathcal{I}_{G}(U_{ba}) \cup W_{ab}]$, it follows that the projection $H'$ of $H$ onto $G_{1}[U_{ba}]$ is also convex and is not contained in a bulge of $co_{G}(U_{ba})$.  Therefore $G[V(H \cup H')] = H \Box K_{2}$, and thus $G[V(H \cup H')]$ is a convex quasi-hypertorus of $G$ of dimension $d+1$ if $H$ is a convex quasi-hypertorus of $G_{2}[U_{ab}]$ of dimension $d$.

Consequently we infer, by the construction of $G_{3}$, that there is a bijection between the quasi-hypertori of $G_{3}$ of dimension $i$ and the elements of $\mathbf{Tor}(G,ab)$ of dimension $i+1$.  It follows that $$\beta_{i}(G) = \beta_{i}(G_{1}) + \beta_{i}(G_{2}) + \beta_{i-1}(G_{3})$$ for every non-negative integer $i > 1$.  Consequently, by \eqref{E:S'=1}, \eqref{E:S"=1} and \eqref{E:S"'=1}, we have
\begin{align*}
\sum_{i \in \mathbb{N}}(-1)^{i}\beta_{i}(G) &= \beta_{0}(G_{1}) + \beta_{0}(G_{2}) + \sum_{i \geq 1}(-1)^{i}(\beta_{i}(G_{1}) + \beta_{i}(G_{2}) + \beta_{i-1}(G_{3}))\\
&= \sum_{i \in \mathbb{N}}(-1)^{i}\beta_{i}(G_{1}) + \sum_{i \in \mathbb{N}}(-1)^{i}\beta_{i}(G_{2}) - \sum_{i \in \mathbb{N}}(-1)^{i}\beta_{i}(G_{3})\\
&= 1.
\end{align*}
\end{proof}

\subsection{Isometric dimension of a finite Peano partial cube}\label{SS:iso.dim.}

\v{S}krekovski~\cite{Sk01} also obtained the following result which gives the the number of $\Theta$-classes of a finite median graph, in other words, the isometric dimension of a finite median graph.

\begin{pro}\label{P:dimSk}\textnormal{(\v{S}krekovski~\cite{Sk01})}
The isometric dimension of a finite median graph $G$ is
\begin{equation}
\label{E:dimSk}
\mathrm{idim}(G) = -\sum_{i \in \mathbb{N}}(-1)^{i}i\alpha_{i}(G).
\end{equation}
\end{pro}

We will show that finite Peano partial cubes satisfy an analogous property which generalizes \eqref{E:dimSk}. We need some definition and notation.

We define the \emph{circumference number} $\gamma(H)$ of a finite quasi-hypertorus $H$ as follows.  If $H = K_{2}$, then $\gamma(H) := 0$.  If $H$ is a cycle of length $2n$ for some $n \geq 2$, then $\gamma(H) := n-2$.  If $H = \cartbig_{0 \leq i \leq k}H_{i}$, then $\gamma(H) := \sum_{0 \leq i \leq k}\gamma(H_{i})$.

Let $G$ be a finite Peano partial cube.  For all integers $i \geq 2$ and $j \in \mathbb{N}$, we denote by $\beta_{i}^{j}(G)$ the number of convex quasi-hypertori of $G$ whose dimension is $i$ and circumference number is $j$.  Moreover we put:

\textbullet\; $\beta_{i}^{j}(G) := 0$ if $i < 0$ or $j < 0$;

\textbullet\; for $i = 0, 1$, $\beta_{i}^{0}(G) := \alpha_{i}(G)$, and $\beta_{i}^{j}(G) := 0$ if $j \neq 0$.

\begin{thm}\label{T:dim}
The isometric dimension of a finite Peano partial cube $G$ is
\begin{equation}
\label{E:dim}
\mathrm{idim}(G) = -\sum_{i \in \mathbb{N}}(-1)^{i}(\sum_{j \in \mathbb{N}}(i+j)\beta_{i}^{j}(G)).
\end{equation}
\end{thm}

We need two lemmas in which we will use the following notation.  Let $G$ be a finite Peano partial cube, and $ab$ an edge of $G$ such that $W_{ab}$ is a semi-periphery.  Then:

\textbullet\; $\kappa(G)$ is the greatest $k$ such that $G$ contains a quasi-hypertorus of dimension $k$.

\textbullet\; $\beta_{i}^{j}(G,ab)$ is the number of convex elements of $\mathbf{Tor}(G,ab)$ whose dimension is $i$ and circumference number is $j$.  In particular $\beta_{1}^{0}(G,ab)$ is the number of edges of $G$ that are $\Theta$-equivalent to $ab$.

\textbullet\; $\mathcal{A}(G,ab) := -\sum_{i \in \mathbb{N}}(-1)^{i}(\sum_{j \in \mathbb{N}}(i+j)\beta_{i}^{j}(G,ab))$.

\textbullet\;  $\mathcal{B}(G[W_{ab}]) := \sum_{i \in \mathbb{N}}(-1)^{i}(\sum_{j \in \mathbb{N}}(i+j+1)\beta_{i}^{j}(G[W_{ab}]))$.

\begin{lem}\label{L:A=B(TuP)}
Let $G = C \Box H$, where $C$ is a $2p$-cycle with $p \geq 3$ and $H$ is a finite Peano partial cube such that $\kappa(H) = k$ for some non-negative integer $k$, and let $ab$ be an edge of a $C$-layer of $G$.  Then
\begin{equation}
\label{E:A=B(CQ)}
\mathcal{A}(G,ab) = \mathcal{B}(G[W_{ab}]) = \sum_{i=0}^{k}(-1)^{i}(\sum_{j \in \mathbb{N}}(i+j-p+2)\beta_{i}^{j}(H)).
\end{equation}
\end{lem}

\begin{proof}
First note that $j = 2$ or $p$.  
Let $H_{1}$ and $H_{2}$ be the two disjoint $H$-layers of $G$ whose vertex sets are contained in $U_{ab}$.  Because an element of $\mathbf{Tor}(G,ab)$ whose dimension is $i$ and circumference number is $j$ is either the Cartesian product of a quasi-hypertorus of $H_{1} \cup H_{2}$ whose dimension is $i-1$ and circumference number is $j$ with $K_{2}$ such that the edge of each of its $K_{2}$-layers is $\Theta$-equivalent to $ab$, or the Cartesian product of a quasi-hypertorus of $H_{1}$ whose dimension is $i-2$ and circumference number is $j-p+2$ by $C$, and because of the convention we made on the $\beta_{i}^{j}$'s, we clearly have:
\begin{align*}
\beta_{i}^{j}(G,ab) &= \beta_{i-1}^{j}(H_{1} \cup H_{2}) + \beta_{i-2}^{j}(H_{1})\\
&= 2\beta_{i-1}^{j}(H) + \beta_{i-2}^{j}(H).
\end{align*}

Therefore (note that $\beta_{k+1}^{j}(H) = 0$ since $\kappa(H) = k$)
\begin{align*}
\mathcal{A}(G,ab) &= -\sum_{i \in \mathbb{N}}(-1)^{i}(\sum_{j \in \mathbb{N}}(i+j)\beta_{i}^{j}(G,ab))\\
&= -\sum_{i=0}^{k+2}(-1)^{i}(\sum_{j \in \mathbb{N}}(i+j)(2\beta_{i-1}^{j}(H) + \beta_{i-2}^{j}(H)))\\
&= -\sum_{i=0}^{k+1}(-1)^{i}(\sum_{j \in \mathbb{N}}(2(i+j)-(i+1+j+p-2))\beta_{i-1}^{j}(H))\\
&= -\sum_{i=0}^{k+1}(-1)^{i}(\sum_{j \in \mathbb{N}}(i+j-p+1)\beta_{i-1}^{j}(H))\\
&= \sum_{i=0}^{k}(-1)^{i}(\sum_{j \in \mathbb{N}}(i+j-p+2)\beta_{i}^{j}(H).
\end{align*}

On the other hand, because an element of $\mathbf{Tor}(G[W_{ab}])$ whose dimension is $i$ and circumference number is $j$ is either a quasi-hypertorus of some $H$-layer of $G-W_{ba}$ whose dimension is $i$ and circumference number is $j$, or the Cartesian product of a quasi-hypertorus of $H$ whose dimension is $i-1$ and circumference number is $j$ with $K_{2}$ such that the edge of each of its $K_{2}$-layers is an edge of $C-W_{ba}$, and because of the convention we made on the $\beta_{i}^{j}$'s, we clearly have:

$\beta_{i}^{j}(G[W_{ab}]) = p\beta_{i}^{j}(H)+(p-1)\beta_{i-1}^{j}(H)$.

Then
\begin{align*}
\mathcal{B}(G[W_{ab}]) &= \sum_{i \in \mathbb{N}}(-1)^{i}(\sum_{j \in \mathbb{N}}(i+j+1)\beta_{i}^{j}(G[W_{ab}]))\\
&= \sum_{i=0}^{k+1}(-1)^{i}(\sum_{j \in \mathbb{N}}(i+j+1)(p\beta_{i}^{j}(H)+(p-1)\beta_{i-1}^{j}(H)))\\
&= \sum_{i=0}^{k+1}(-1)^{i}(\sum_{j \in \mathbb{N}}(i+j+1)p-(i+j+2)(p-1))\beta_{i}^{j}(H)))\\
&= \sum_{i=0}^{k+1}(-1)^{i}(\sum_{j \in \mathbb{N}}(i+j-p+2)\beta_{i}^{j}(H))).
\end{align*}

This proves the lemma.
\end{proof}

\begin{lem}\label{L:A=B}
Let $G$ be a finite Peano partial cube, and $ab$ an edge of $G$ such that $W_{ab}$ is a semi-periphery.  Then
\begin{equation}
\label{E:A=B}
\mathcal{A}(G,ab) = \mathcal{B}(G[W_{ab}]).
\end{equation}
\end{lem}

\begin{proof}
The proof will be by induction on the order of $G$.  Note that the order of $G$ is necessarily even.  
The result is obvious if the order of $G$ is $2$.  Suppose that it holds if the order of $G$ is at most $2n$ for some positive integer $n$.  Assume that the order of $G$ is $2(n+1)$.  Without loss of generality we can suppose that the edge $ab$ is such that $W_{ab}$ and $W_{ba}$ are semi-peripheries.  The proof is straightforward if $W_{ab}$ (and thus $W_{ba}$) is a periphery.  Suppose that $W_{ab}$ is not a periphery.  Let $X$ be a bulge of $W_{ab}$, and $H := \mathbf{Cyl}[X]$.  Then $H = C \Box H'$ where $C$ is a cycle of length greater than $4$ and $H'$ a finite Peano partial cube such that $\kappa(H') = k$ for some non-negative integer $k$.  We are done if $G = H$ by Lemma~\ref{L:A=B(TuP)}.  Assume that $G \neq H$.

In order to avoid the introduction of unnecessary edges, if $G'$ is a subgraph of $G$ that 
has an edge $uv$ $\Theta$-equivalent to $ab$, but not necessarily the edge $ab$ itself, we will still denote the set $W_{uv}^{G'}$ by $W_{ab}^{G'}$.

By (HNB2), $V(X)-U_{ab}$ is a separator of $G[W_{ab}]$, and thus $V(H)-(U_{ab} \cup U_{ba})$ is a separator of $G$.  Let $F$ be the connected component of $H \cap G[U_{ab} \cup U_{ba}]$  such that there exists a vertex of $G-H$ which is adjacent to some vertex of $F$.  Then $\kappa(F) = k+1$.  Denote by $G_{0}$ the component of $G-(H-U_{ab} \cup U_{ba})$  containing $F$, and let $G_{1} := G-(G_{0}-F)$.

For $i = 0, 1$, $G_{i}$ is a convex subgraph of $G$, and thus is ph-homogeneous, and also clearly so is $F$.  Moreover $W_{ab}^{G_{i}}$ and $W_{ba}^{G_{i}}$ are semi-peripheries of $G_{i}$.  Because the orders of $G_{0}$, $G_{1}$ and $F$ are less than $2(n+1)$, we have by the induction hypothesis:

\textbullet\; $\mathcal{A}(G_{i},ab) = \mathcal{B}(G_{i}[W_{ab}^{G_{i}}])$ for $i = 0, 1$

\textbullet\; $\mathcal{A}(F,ab) = \mathcal{B}(F[W_{ab}^{F}])$.

Then $\mathcal{A}(G,ab) = \mathcal{B}(G[W_{ab}])$ is a straightforward consequence of these equalities and of the obvious following facts:

\textbullet\; $\beta_{i}^{j}(G,ab) = \beta_{i}^{j}(G_{0},ab)+\beta_{i}^{j}(G_{1},ab)-\beta_{i}^{j}(F,ab)$;

\textbullet\; $\beta_{i}^{j}(G[W_{ab}]) = \beta_{i}^{j}(G_{0}[W_{ab}^{G_{0}}])+\beta_{i}^{j}(G_{1}[W_{ab}^{G_{1}}])-\beta_{i}^{j}(F[W_{ab}^{F}])$.
\end{proof}

\begin{proof}[\textnormal{\textbf{Proof of Theorem~\ref{T:dim}}}]
The proof is by induction on the number of vertices of $G$.  This is
trivial if $G$ has only one vertex.  Let $r \geq 1$.  Suppose that
this is true for any Peano partial cube partial cube
having at most $r$ vertices.  Let $G$ be a Peano partial
cube partial cube having $r+1$ vertices.  Because $G$ is finite ph-homogeneous, there is an edge $ab$ of $G$ such
that $W_{ab}$ is a semi-periphery.  As a convex subgraph of $G$,\; $G_{1}
:= G[W_{ba}]$ is also ph-homogeneous.    Moreover it has at most
$r$ vertices.  Hence, by the induction hypothesis, we have
\begin{equation}
    \label{E:dim1}
\mathrm{idim}(G_{1}) = -\sum_{i \in \mathbb{N}}(-1)^{i}(\sum_{j \in \mathbb{N}}(i+j)\beta_{i}^{j}(G_{1})).
\end{equation}

The graph $G_{2} := G[W_{ab}]$ is also a convex subgraph of $G$.  Therefore $G_{2}$ is a Peano partial cube.  Hence, by \eqref{E:S=1}
\begin{equation}
\label{E:S(G[W])}
\sum_{i \in \mathbb{N}}(-1)^{i}\beta_{i}(G_{2})=1.
\end{equation}

On the other hand the $\Theta$-classes of $G$ are the $\Theta$-classes of the edges of $G_{1}$ with the $\Theta$-class of $ab$, because any edge of $G_{2}$ is $\Theta$-equivalent to an edge of $G_{1}$.  Hence 
\begin{equation}
\label{E:dim2}
\mathrm{idim}(G) = \mathrm{idim}(G_{1})+1.
\end{equation}

Because $W_{ab}$ is a semi-periphery we have:
$$\beta_{i}^{j}(G) = \beta_{i}^{j}(G_{1})+\beta_{i}^{j}(G[W_{ab}])+\beta_{i}^{j}(G,ab).$$

Let
$$\Delta(G') := -\sum_{i \in \mathbb{N}}(-1)^{i}(\sum_{j \in \mathbb{N}}(i+j)\beta_{i}^{j}(G'))$$ where $G'$ is either $G$ or $G_{1}$ or $G[W_{ab}]$ or $(G,ab)$.

We then infer from the above equalities and from the fact that $\beta_{i}(G[W_{ab}]) = \sum_{j \in \mathbb{N}}\beta_{i}^{j}(G[W_{ab}])$, that
\begin{align*}
\Delta(G) &= \Delta(G_{1})+\Delta(G[W_{ab}])+\Delta(G,ab)\\
&= \mathrm{idim}(G_{1}) -\sum_{i \in \mathbb{N}}(-1)^{i}(\sum_{j \in \mathbb{N}}(i+j)\beta_{i}^{j}(G[W_{ab}])) + \mathcal{A}(G,ab) \text{\qquad by \eqref{E:A=B}}\\
&= \mathrm{idim}(G_{1}) -\sum_{i \in \mathbb{N}}(-1)^{i}(\sum_{j \in \mathbb{N}}(i+j)\beta_{i}^{j}(G[W_{ab}])) + \mathcal{B}(G[W_{ab}])\; \text{\quad by \eqref{E:dim1}}\\
&= \mathrm{idim}(G_{1}) -\sum_{i \in \mathbb{N}}(-1)^{i}(\sum_{j \in \mathbb{N}}(i+j)\beta_{i}^{j}(G[W_{ab}]))\\
&\quad + \sum_{i \in \mathbb{N}}(-1)^{i}(\sum_{j \in \mathbb{N}}(i+j+1)\beta_{i}^{j}(G[W_{ab}]))\\
&= \mathrm{idim}(G_{1}) -\sum_{i \in \mathbb{N}}(-1)^{i}(\sum_{j \in \mathbb{N}}\beta_{i}^{j}(G[W_{ab}]))\\
&= \mathrm{idim}(G_{1}) + \sum_{i \in \mathbb{N}}(-1)^{i}\beta_{i}(G[W_{ab}])\\
&= \mathrm{idim}(G_{1})+1= \mathrm{idim}(G)\mspace{200mu}  \text{ by \eqref{E:S(G[W])} and \eqref{E:dim2}}.
\end{align*}
\end{proof}

\subsection{Cube-free netlike partial cubes}\label{SS:cube-free}

Let $G$ be a finite partial cube, and $\mathcal{C}(G)$ the set of its convex cycles.  The sum $e(G):= \sum_{C \in \mathcal{C}(G)}\frac{\vert C\vert - 4}{2}$ is called the \emph{convex-excess} of $G$ by Klav\v{z}ar and Shpectorov~\cite{KSh08}.  Let $\phi$ be a $\Theta$-class of edges of $G$.  Then the \emph{$\phi$-zone graph} is the graph $Z_{\phi}$ whose vertex set is $\phi$, and such that $f, f' \in \phi$ are adjacent if they belong to a convex cycle of $G$.

\begin{thm}\label{T:KSh}\textnormal{(\cite[Theorem 1.1]{KSh08})}
If $G$ is a partial cube, then
\begin{equation}
\label{E:KSh}
2\alpha_{0}(G)-\alpha_{1}(G)-\mathrm{idim}(G)-e(G) \leq 2.
\end{equation}

\noindent Moreover the equality holds if and only if all zone graphs of $G$ are trees.
\end{thm}

Furthermore Klav\v{z}ar and Shpectorov noted that any zone graph of a partial cube is connected.  It follows that a zone graph is a tree if and only if it contains no cycle.

For a finite Peano partial cube $G$, we have $e(G) = \sum_{j \in \mathbb{N}}j\beta_{2}^{j}(G)$.  By Proposition~\ref{P:cube-free netlike/fin.conv.subgr.}, to characterize cube-free netlike partial cubes, it is sufficient to consider the case of finite partial cubes only.

\begin{thm}\label{T:cube-free.netl.p.c.}
Let $G$ be finite Peano partial cube.  The 
following assertions are equivalent:

\textnormal{(i)}\; $G$ is cube-free netlike partial cube.

\textnormal{(ii)}\; $G$ contains no convex prism over an even cycle.

\textnormal{(iii)}\; $G$ satisfies the equalities
\begin{gather}
\label{E:n-m+c=1}
\beta_{0}(G)-\beta_{1}(G)+\beta_{2}(G)=1\\
\label{E:lin/dim}
\mathrm{idim}(G)=\beta_{1}(G)-2\beta_{2}(G)-e(G).
\end{gather}

\textnormal{(iv)}\; $G$ satisfies the equality
\begin{equation}
\label{E:(iv)}
2\beta_{0}(G)-\beta_{1}(G)-\mathrm{idim}(G)-e(G) = 2.
\end{equation}

\textnormal{(v)}\; All zone graphs of $G$ are trees.
\end{thm}

\begin{proof}
(i) $\Rightarrow$ (ii) is obvious because $G$ contains no convex quasi-hypertorus of dimension greater than $2$.

(ii) $\Rightarrow$ (i) is clear since any convex quasi-hypertorus of dimension at least $3$ contains a convex prism over a cycle.  Hence any isometric cycle of $G$ is convex, and thus $G$ is a cube-free netlike partial cube by Proposition~\ref{P:hypernet./cube-free netl.}.

(i) $\Rightarrow$ (iii):\; The equalities \eqref{E:n-m+c=1} and \eqref{E:lin/dim} are the particular cases of the equalities \eqref{E:S=1} and \eqref{E:dim} with $\beta_i(G) = \beta_i^j(G) = 0$ for all $i \geq 3$.

(iii) $\Rightarrow$ (iv):\; Suppose that $G$ satisfies (iii).  Then, by \eqref{E:n-m+c=1}, $\beta_{2}(G)=-\beta_{0}(G)+\beta_{1}(G)+1$.  Hence, by \eqref{E:lin/dim}, 
\begin{equation}
\mathrm{idim}(G)=2\beta_{0}(G)-\beta_{1}(G)-e(G)-2.
\end{equation}
Therefore $G$ satisfies (iv).

(iv) $\Rightarrow$ (v) is a consequence of Theorem~\ref{T:KSh}.

(v) $\Rightarrow$ (ii):\; Suppose that $G$ contains a convex prism $P = C \Box K_2$ over an even cycle $C$.  Let $uv$ be an edge of some $K_2$-layer of $P$. Then, because any $4$-cycle of $G$ is convex, the $\Theta$-class in $P$ of the edge $uv$ induces a cycle in the zone graph $Z_{\phi}$, where $\phi$ is the $\Theta$-class in $G$ of $uv$.  Therefore $G$ does not satisfy (v).
\end{proof}

In the result below we characterize cube-free netlike partial cubes by using the properties of the cycle space of a graph, that is, of the linear space over \textbf{GF(2)} with
all finite eulerian subgraphs of this graph as elements and the
symmetric difference as addition.

\begin{thm}\label{T:cube-free.netl./cycle.space}
A finite Peano partial cube $G$ is a cube-free netlike partial cube if and only if the set of all convex cycles of any convex subgraph $F$ of $G$ is a basis of the cycle space of $F$.
\end{thm}

\begin{proof}
Suppose that $G$ is a cube-free netlike partial cube, and let $F$ be a convex subgraph of $G$.  Then $F$ is also a cube-free netlike partial cube, and thus, by Theorem~\ref{T:cube-free.netl.p.c.}, $F$ satisfies the equality (\ref{E:n-m+c=1}).  The dimension of the cycle space (the cyclomatique
number) of $F$ is $\beta_{1}(F)-\beta_{0}(F)+1$.  Hence, by \eqref{E:n-m+c=1}, this
number is $\beta_{2}(G)$, that is, the number of convex cycles of $F$.  On the other hand, by \cite[Theorem 6.15]{HLS13}, the cycle space of $F$ has a basis of convex cycles.  Therefore the set of all convex cycles of $F$ is a basis of the cycle space of $F$.

Conversely, suppose that $G$ is not a cube-free netlike partial cube.  Then, by Theorem~\ref{T:cube-free.netl.p.c.}, $G$ contains a convex prism $P $ over an even cycle $C_{2n}$.  The cyclomatique number of $P$ is $\beta_{1}(P)-\beta_{0}(P)+1 = 2n+1$, whereas the number of convex cycles of $P$ is $\beta_2(P) = 2n+2$.  This proves that the convex cycles of $P$ are not linearly 
independent, and thus they cannot constitute a basis of the 
cycle space of $P$.
\end{proof}

\end{document}